\documentclass[12pt,oneside]{book}

\usepackage[utf8]{inputenc} 
\usepackage[T1]{fontenc} 

\usepackage{color}
\usepackage{longtable}
\usepackage{supertabular}

\usepackage[a4paper]{geometry}
\usepackage{cite}
\usepackage{graphicx}
\usepackage{amsthm, amsmath, amssymb, amsfonts, stmaryrd, multirow}
\usepackage{float}
\usepackage{array}
\usepackage{cases}
\usepackage{hyperref}
\newtheorem{proposition}{Proposition}[chapter]
\newtheorem{theorem}[proposition]{Theorem}
\newtheorem{lemma}[proposition]{Lemma}
\newtheorem{definition}[proposition]{Definition}
\theoremstyle{remark}
\newtheorem{remark}[proposition]{Remark}
\newtheorem{corollary}[proposition]{Corollary}
\newtheorem{assumption}[proposition]{Assumption}
\newtheorem{example}[proposition]{\textit{Example}}
\newtheorem{conjecture}[proposition]{Conjecture}
\usepackage{rotating}
\usepackage{tablefootnote}
\usepackage{algorithm}
\usepackage[noend]{algpseudocode}
\usepackage{fixltx2e}
\usepackage{stfloats}

\numberwithin{equation}{chapter}

\usepackage{epsfig} 
\usepackage{epstopdf}

\newcommand{\qcqp}{\mbox{QCQP-$\mathbb{C}$}}
\newcommand{\csdpr}{\mbox{CSDP-$\mathbb{R}$}}
\newcommand{\sdpr}{\mbox{SDP-$\mathbb{R}$}}
\newcommand{\sdpc}{\mbox{SDP-$\mathbb{C}$}}
\newcommand{\socpc}{\mbox{SOCP-$\mathbb{C}$}}
\newcommand{\socpr}{\mbox{SOCP-$\mathbb{R}$}}
\newcommand{\csocpr}{\mbox{CSOCP-$\mathbb{R}$}}
\newcommand{\matpower}{M{\sc atpower}}
\newcommand\NoDo{\renewcommand\algorithmicdo{}}
\allowdisplaybreaks

\usepackage{cite}
\usepackage{url}
\usepackage{tablefootnote}

\begin{document}

\begin{titlepage}{\bfseries

\begin{center}

\includegraphics[height=2cm]{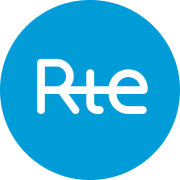}
\hspace{3cm}~
\includegraphics[height=2cm]{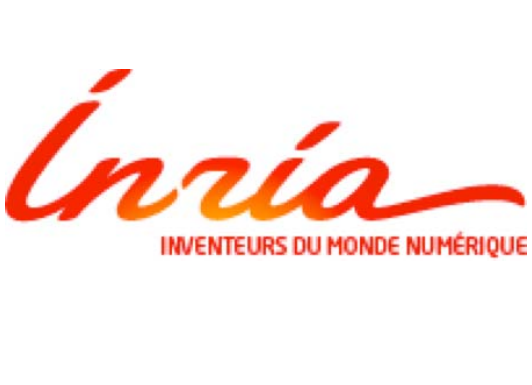}
\hspace{3cm}~
\includegraphics[height=2cm]{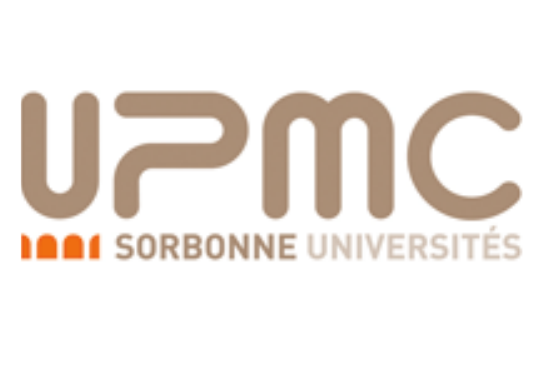}
\vspace{0.3cm}
\text{}
\vspace{0.1cm}
\textcolor{blue}{Dissertation in fulfillment of the degree of}\\
\vspace{0.2cm}
{{\scshape Doctor of Philosophy in Applied Mathematics}}\\
{{\scshape University Paris VI}}\\
\vspace{0.1cm}
{\itshape Universit\'e Pierre et Marie Curie}\\
{\itshape \'Ecole Doctorale de Sciences Math\'ematiques de Paris Centre}\\
\vspace{0.6cm}
\hrule~\\
\vspace{0.3cm}
{\LARGE Application of Polynomial Optimization to Electricity Transmission Networks}\\
\vspace{0.3cm}
\hrule~\\
\vspace{0.20cm}
\vspace{0.20cm}
{C\'edric {\scshape Josz}}\\
\vspace{0.20cm}
\textcolor{blue}{under the supervision of}\\
\vspace{0.20cm}
{ Jean Charles {\scshape Gilbert}}\\
French Institute for Research in Computer Science and Automation\\
\textit{Institut National de Recherche en Informatique et en Automatique}\\
\vspace{0.20cm}
{ Jean {\scshape Maeght}} and {Patrick {\scshape Panciatici}}\\
French Transmission System Operator\\
\textit{R\'eseau de Transport d'\'Electricit\'e}\\
\vspace{0.20cm}
\textcolor{blue}{funded by}\\
\vspace{0.20cm}
French Transmission System Operator\\
\vspace{0.20cm}
French Ministry of Higher Education and Research\\
CIFRE ARNT contract 2013/0179\\
\vspace{0.3cm}
\textcolor{blue}{examined on July 13\textsuperscript{th} 2016 by}\\
\vspace{0.35cm}
\begin{tabular}{ll}
Patrick Combettes & University Paris VI\\
St\'ephane Gaubert & Ecole Polytechnique \\
Jean Charles Gilbert & INRIA Paris \\
Jean Bernard Lasserre & CNRS Toulouse \\
Patrick Panciatici & RTE Versailles\\
Mihai Putinar & UC Santa Barbara \\
Markus Schweighofer & University of Konstanz \\
Pascal Van Hentenryck & University of Michigan
\end{tabular}
\end{center}

}\end{titlepage}
\thispagestyle{empty}
\text{ }
\newpage
\thispagestyle{empty}
\vspace*{0.475\textheight}

\noindent {\itshape « Cherche et tu trouveras. » } \bigbreak

\hfill Ang\'elique Haenecour Josz

\newpage
\thispagestyle{empty}
\section*{\centering Acknowledgements}
\text{}
\indent 
I wish to thank my thesis advisor Jean Charles Gilbert for his support and guidance throughout my doctoral project. He has given me an invaluable insight into optimization and the world of research. He encouraged me to pursue my ideas and to publish them, as well to attend conferences and to meet with the scientific community. It was a pleasure to be his teaching assistant at the University of Paris-Saclay, ENSTA ParisTech, for two consecutive years. It is thanks to his teachings that I developed a strong background in optimization when I was a master's student at ENSTA ParisTech. Jean Charles mentionned the Lasserre hierarchy at the beginning of our first meeting three years ago. At the time, I didn't understand its relevance for electricity, but it turned out to be the cornerstone of this dissertation!

I would like to thank Jean Maeght, my scientific advisor at RTE, for mentoring me over the three years of a half years I spent at RTE. I was fortunate to share my office with someone who works on the optimal power flow problem in a European project and who has a very strong background in mathematics. He invested a lot of time in me to teach me about the company, its various projects, its key people, its history, and the way it functions on a daily basis. He always made sure that I was in the loop concerning events and meetings at RTE. I would have liked to spend more time sharing the office, but as he said so often after a long day of day of work, ``\textit{toute bonne chose ayant une fin ...}''.

I am grateful to Patrick Panciatici, the chief scientific advisor of R\&D at RTE, for being the main promoter of my doctoral project. Thanks to his daring vision, I was able to work on new and exciting ideas in the field of power systems. He gave me the opportunity to participate in a workshop in Dublin at IBM Research Ireland and to meet with top researchers in various universities and laboratories accross the world. These include ETH Zürich, Berkeley, Caltech, and Tokyo Tech. It was thrilling to work with someone so passionate and knowledgeable about energy.

I thank my three advisors for the trust they put in me during the entire project. I was given a great amount freedom to work on what pleased me most and to start collaborations with researchers in my field.

I wish to thank Stéphane Fliscounakis, research engineer at RTE, for his collaboration and assistance during my doctoral project. Thanks to his expertise on the issues and complexity of modeling transmission networks, we successfully made a large-scale representation of the European network easily available to the public for the first time. He taught me a great deal about the applications of optimization to power systems, and about the problems that arise when dealing with data in industry. I truly hope to pursue our joint work in the future.

I would like to thank Robert Gonzalez, optimization expert at RTE, for allowing me to participate in the optimization workshops he organizes. It was a great opportunity to learn about how optimization is used by RTE, as well as a chance to present my work in detail. That allowed me to get valuable feedback on my research.

I am grateful to my managers Fr\'ed\'erique Verrier and Lucian Bal\'ea at RTE for supporting my project. They always made sure that I was well integrated within my group at RTE. They gave me the opportunity to share my work with the R\&D through various presentations.

I wish to thank Fran\c{c}oise Sericola, Lydie Pendu and Nathalie Lucazeau from HR at RTE for being very helpful on numerous occasions. They were crucial in setting up my contract as well as organizing my travels for RTE. I also wish to thank Nathalie Bonte from HR at INRIA for helping me get started quickly at INRIA.

I wish to thank Jean-Pierre Restoux, head of IT at RTE Versailles, as well as Gauthier Plouvier, L\'eon Amirkhanian, and Yacine Chaoui at IT support from RTE for being a great help throughout my years at RTE. Thanks to them, that I was able to work efficiently on my computer and use the latest versions of the softwares I needed.

I would like to thank all my colleagues at RTE and INRIA for creating an energetic environment to work in. I wish to thank Alexandre Debetencourt for his relentless support during my doctoral project. He stunned me by his natural curiosity and interest in my work. 

I am fortunate to have collaborated with Professor Didier Henrion from the University of Toulouse. Working with a leader in the domain of optimization and control was an enriching experience. 
I am also fortunate to have collaborated with post-doctoral fellow Daniel K. Molzahn (now a research engineer at Argonne National Laboratory) and Professor Ian A. Hiskens from the University of Michigan. Many of the results in this dissertation were found thanks to our close collaboration. It was a very exciting experience that I hope to build on in the future. 

I wish to thank Professors Pascal Gourdel and Bruno Nazaret for giving me the opportunity to teach at the University of Paris I Sorbonne-Tolbiac. It was a great opportunity for me to learn how to teach mathematics. Also from that university, I would like to thank Professor Jean-Bernard Baillon for his kindness and his help on some tough mathematics. 

Many thanks to Professor Mihai Putinar at the University of California, Santa Barbara, for his encouragements, his valuable advice, and his generous help with my work. 


Last but not least, many thanks to my brothers Tanguy and J\'er\^ome, to my parents, and to my friends for their support. 

%
%

\bigbreak
\newpage

\thispagestyle{empty}
\vspace*{0.4\textheight}
\section*{\centering Abstract}
\text{}
\indent 
Transmission system operators need to adapt their decision-making tools to the technological evolutions of the twenty first century. A computation inherent to most tools seeks to find alternating-current power flows that minimize power loss or generation cost. Mathematically, it consists in an optimization problem that can be described using only addition and multiplication of complex numbers. The objective of this thesis is to find global solutions, in other words the best solutions to the problem. One of the outcomes of this highly collaborative doctoral project is to use recent results from algebraic geometry to compute globally optimal power flows in the European high-voltage transmission network.
\\\\
\textbf{Keywords:}
polynomial optimization, semidefinite optimization, optimal power flow, Lasserre hierarchy
\bigbreak
\newpage
\thispagestyle{empty}
\section*{\centering Summary}
\text{}
\indent This dissertation is motivated by an encouraging discovery made in the field of power systems during the first decade of the twenty-first century. Numerical experiments on several benchmark transmission networks showed that it is possible to find global solutions to the optimal power flow problem using semidefinite optimization. The optimal power flow problem seeks to find a steady-state operating point of an alternating-current transmission network that is optimal under some criteria such as power loss or generation costs. After five decades of research on this highly nonconvex problem, a method for finding global solutions was thought to be out of reach. The concept used was to omit nonconvexities and solve a convex problem instead. This is known as the Shor relaxation, in reference to the Ukrainian mathematician Naum Zuselevich Shor. However, the Shor relaxation does not provide global solutions to many networks of interest. Bridging this gap is the starting point of this dissertation.

The first step that was achieved (cf. Chapter \ref{sec:Lasserre hierarchy for small-scale networks}) was to show that low orders of the Lasserre hierarchy find the global solution to small-scale networks that the Shor relaxation cannot solve. To do so, we realized that the optimal power flow problem is a particular instance of polynomial optimization. Thankfully, any polynomial optimization problem with a bounded feasible set can be approximated as close as desired by a sequence of semidefinite optimization problems. This sequence is called the Lasserre hierarchy, in reference to the French mathematician Jean Bernard Lasserre. This is remarkable because polynomial optimization problems encompass many non-deterministic polynomial-time hard problems such as quadratically-constrainted quadratic programming, mixed-integer linear programming, and in particular the traveling salesman problem. 

To further prove the numerical applicability of the Lasserre hierarchy, we proved that there is zero duality gap in each semidefinite optimization problem in the hierarchy in the case of the optimal power flow problem (cf. Chapter \ref{sec:Zero duality gap in the Lasserre hierarchy}). This property is essential for efficient solvers to work. More generally, we proved that for any polynomial optimization problem containing a ball constraint, there is no duality gap. Adding a redundant ball constraint to a problem with a bounded feasible set guarantees the global convergence of the Lasserre hierarchy, hence the relevance of our result.

Having shown the applicability of the Lasserre hierarchy to small instances, the next task was to be able to tackle large-scale problems. However, there were few large-scale benchmark networks on which to test new approaches. Network are considered large-scale if they contain several thousand buses. We filled this gap by providing data for the entire European synchronous grid, with a little over 9,000 buses. To make it possible to work progressively on the data, we provided four instances corresponding to larger and larger parts of the European network (cf. Chapter \ref{sec:Data of European high-voltage transmission network}). The data stems from a European project involving many transmission system operators whose purpose was to develop new tools for the pan-European grid.

Since it had been discovered in 2000, the Lasserre hierarchy had never been able to solve practical problems with more than several dozens of variables. This changed when Daniel K. Molzahn and Ian A. Hiskens at the University of Michigan developed an algorithm to exploit sparsity in the Lasserre hierarchy for the optimal power flow problem. This enabled them to solve networks with several hundred buses. At around the same time, Ramtin Madani, Morteza Ashraphijuo and Javad Lavaei at the University of Columbia showed that the Shor relaxation succeeded on some large-scale networks provided two penalty terms were added to the objective function. We proposed to combine both approaches to systematically provide nearly global solutions to large-scale networks (cf. Chapter \ref{sec:Penalized Lasserre hierarchy for large-scale networks}). This work was carried out in collaboration with the University of Michigan. In the combined approach, only one penalty parameter has to be specified, instead of two. In the case of active power loss minimization, the objective function is convex (in function of the voltage variables) and we observed that no penalization term is needed. This means that the approach finds the global solution. In the case of generation cost minimization, the objective is not a convex function and a penalty parameter must be specified, yielding a nearly global solution.

Specifying a penalization parameter is problematic because there is no general method for doing so. To overcome this, we realized that successful penalizations of the optimal power flow were related to the Laplacian matrix of the graph of the power network. We thus proposed a Laplacian-based Shor relaxation to obtain nearly global solutions without the need to specify any parameter (cf. Chapter \ref{sec:Laplacian matrix gets rid of penalization parameter}). An issue that emerged while trying to solve large-scale optimal power flow problems is that the data are ill-conditionned. Some power lines have very low impedance, i.e., opposition to current, while others have up to one thousand times larger impedance. As a result, in all large-scale numerical experiments in this dissertation, the data is preprocessed to have more homogenous line characteristics.

Having shown the applicability of the Lasserre hierarchy to large-scale networks, we next enhanced its tractability by transposing it from real to complex numbers (cf. Chapter \ref{sec:Complex hierarchy for enhanced tractability}). What prompt us to do so is that the optimal power flow problem is written using complex numbers. They are used to model an oscillatory phenomena, namely alternating-current. We realized that omitting nonconvexities and converting from complex to real numbers are two non-commutative operations. This lead us to propose a general approach for finding global solutions to polynomial optimization problems with bounded feasible sets where variables and data are complex numbers. It is based on recent results in algebraic geometry concerning positive polynomials with complex indeterminates. By exploiting sparsity, it succeeds in finding global solutions to problems with several thousand complex variables. In addition to the operation and planning of future power systems, the complex moment/sum-of-squares hierarchy we developed can be applied to signal processing, imaging science, automatic control, and quantum mechanics.
\\\\
The dissertation is organized as follows.
\\\\
Chapter \ref{sec:Background and motivations} describes the optimal power flow problem and the underlying mathematical concepts.\\\\
Chapter \ref{sec:Lasserre hierarchy for small-scale networks} numerically illustrates that low orders of the Lasserre hierarchy find the global solution to small-scale networks. Associated publication: 
{\scshape C. Josz, J. Maeght, P. Panciatici, and J.C. Gilbert}, \textit{Application of the Moment-SOS
Approach to Global Optimization of the OPF Problem}, Institute of Electrical and Electronics Engineers, Transactions on Power Systems, 30, pp. 463–470, May 2014.  \href{http://dx.doi.org/10.1109/TPWRS.2014.2320819}{[doi]} \href{http://arxiv.org/pdf/1311.6370v1.pdf}{[preprint]}
\\\\
Chapter \ref{sec:Zero duality gap in the Lasserre hierarchy} proves that there is no duality gap between the primal and dual versions of an instance of the Lasserre hierarchy in the presence of a ball constraint in the original polynomial problem. Associated publication: 
{\scshape C. Josz and D. Henrion}, \textit{Strong Duality in Lasserre’s Hierarchy for Polynomial Optimization},
Springer Optimization Letters, February 2015. \href{http://dx.doi.org/10.1007/s11590-015-0868-5}{[doi]} \href{https://docs.google.com/viewer?a=v&pid=sites&srcid=ZGVmYXVsdGRvbWFpbnxjZWRyaWNqb3N6fGd4OjY1M2E5NDAyMjg2M2U2Y2Q}{[preprint]}
\\\\
Chapter \ref{sec:Data of European high-voltage transmission network} provides data of large-scale networks representing the European high-voltage transmission network. Associated public data: 
{\scshape C. Josz, S. Fliscounakis, J. Maeght, and P. Panciatici}, \textit{Power Flow Data of the European High-Voltage Transmission Network: 89, 1354, 2869, and 9241-bus PEGASE Systems}, MATPOWER 5.1, March 2015. \href{http://www.pserc.cornell.edu//matpower/}{[link]}\\\\
Chapter \ref{sec:Penalized Lasserre hierarchy for large-scale networks} computes nearly global solutions to large-scale networks using the Lasserre hierarchy and a penalization parameter. Associated publication: 
{\scshape D.K. Molzahn, C. Josz, I.A. Hiskens, and P. Panciatici}, \textit{Solution of Optimal Power Flow Problems using Moment Relaxations Augmented with Objective Function Penalization}, 54th Annual Conference on Decision and Control, Osaka, December 2015. \href{http://arxiv.org/pdf/1508.05037v1.pdf}{[preprint]} \\\\ 
Chapter \ref{sec:Laplacian matrix gets rid of penalization parameter} computes nearly global solutions to large-scale networks using Laplacian matrices instead of a penalization parameter. Associated preprint:
{\scshape D.K. Molzahn, C. Josz, I.A. Hiskens, and P. Panciatici}, \textit{A Laplacian-Based Approach for Finding Near Globally Optimal Solutions to OPF Problems}, submitted to Institute of Electrical and Electronics Engineers, Transactions on Power Systems. \href{http://arxiv.org/pdf/1507.07212v1.pdf}{[preprint]}
\\\\
Chapter \ref{sec:Complex hierarchy for enhanced tractability} transposes the Lasserre hierarchy to complex numbers to enhance its tractability when dealing with complex variables instead of real ones. Associated preprint: 
{\scshape C. Josz, D. K. Molzahn}, \textit{Moment/Sum-of-Squares Hierarchy for Complex Polynomial Optimization}, submitted to Society for Industrial and Applied Mathematics, Journal on Optimization. \href{http://arxiv.org/pdf/1508.02068v1.pdf}{[preprint]}\\\\
Chapter \ref{sec:Conclusion and perspectives} suggests future research directions and is followed by references.\\\\
The abstract and summary are translated in French in the following pages.

\newpage
\thispagestyle{empty}
\vspace*{0.4\textheight}
\section*{\centering Abstract}
\text{}
\indent 
Les gestionnaires des réseaux de transport d'électricité doivent adapter leurs outils d'aide à la décision aux avancées technologiques du XXI\textsuperscript{i\`eme} si\`ecle. Une opération sous-jacente à beaucoup d'outils est de calculer les flux en actif/r\'eactif qui minimisent les pertes ou les co\^uts de production. Math\'ematiquement, il s'agit d'un probl\`eme d'optimisation qui peut \^etre d\'ecrit en utilisant seulement l'addition et la multiplication de nombres complexes. L'objectif de cette th\`ese est de trouver des solutions globales. Un des aboutissements de ce projet doctoral hautement collaboratif est d'utiliser des r\'esultats r\'ecents en g\'eom\'etrie alg\'ebrique pour calculer des flux optimaux dans le r\'eseau Europ\'een à haute tension.
\\\\
\textbf{Mots-clefs:}
hi\'erarchie de Lasserre, r\'eseau de transport d'\'electricit\'e, optimisation polynomiale, optimisation semid\'efinie
\bigbreak
\newpage
\thispagestyle{empty}
\section*{\centering R\'esum\'e}
\text{}
\indent Cette th\`ese est motiv\'ee par une d\'ecouverte encourageante faite dans le domaine des r\'eseaux \'electriques durant la premi\`ere d\'ecennie du XXI\textsuperscript{i\`eme} si\`ecle. Des exp\'eriences num\'eriques sur certains cas tests ont montr\'e qu'il \'etait possible de trouver des solutions globales au probl\`eme d'\'ecoulement des flux en utilisant l'optimisation semid\'efinie positive. Le probl\`eme d'\'ecoulement des flux recherche un point stationnaire du r\'eseau qui est optimal au sens des pertes d'énergie ou des co\^uts de production. Apr\`es cinquante années de recherches sur ce problème non convexe, une méthode pour trouver des solutions globales semblait hors de port\'ee. Le concept utilisé a été d'omettre les non convexit\'es et de r\'esoudre un probl\`eme convexe à la place. Ce proc\'ed\'e est connu sous le nom de relaxation de Shor, en r\'ef\'erence au math\'ematicien ukrainien Naum Zuselevich Shor. Cependant, la relaxation de Shor ne fournit pas de solutions globales dans tous les cas. Pallier ce manque est le point de d\'epart de cette dissertation.

La premi\`ere \'etape qui a \'et\'e franchie (cf. Chapitre \ref{sec:Lasserre hierarchy for small-scale networks})  a \'et\'e de montrer que l'on peut r\'esoudre des petits r\'eseaux \`a l'aide de la hi\'erarchie de Lasserre avec des ordres faibles lorsque la relaxation de Shor \'echoue. Nous nous sommes en effet aper\c cus que le probl\`eme de calcul des flux optimaux est une instance particuli\`ere d'optimisation polynomiale. Or tout problème d'optimisation polynomiale dont le domaine d'admissibilit\'e est born\'e peut \^etre approch\'e d'aussi pr\`es que l'on veuille par une suite de probl\`emes d'optimisation semid\'efinie positive. Cette suite est connue sous le nom de hi\'erarchie de Lasserre, en r\'ef\'erence au math\'ematicien fran\c cais Jean Bernard Lasserre. Ceci est remarquable car l'optimisation polynomiale englobe de nombreux probl\`emes NP-ardus tels que l'optimisation quadratique sous contraintes quadratiques, l'optimisation lin\'eaire en nombres entiers, et en particulier le problème du voyageur de commerce.

Pour prouver davantage l'applicabilité de la hiérarchie de Lasserre d'un point de vue numérique, nous avons prouvé qu'il n'y a pas de saut de dualité pour chaque probl\`eme d'optimisation semid\'efinie positive dans la hi\'erarchie, pour le cas du probl\`eme d'\'ecoulement des flux (cf. Chapitre \ref{sec:Zero duality gap in the Lasserre hierarchy}). Cette propri\'et\'e est essentielle pour que des solveurs efficaces fonctionnent. Plus g\'en\'eralement, nous avons prouv\'e que pour tout probl\`eme d'optimisation contenant une contrainte de boule, il n'y a pas de saut de dualit\'e. Ajouter une contrainte de boule redondante à un probl\`eme avec un ensemble admissible born\'e garantit la convergence de la hi\'erarchie de Lasserre, d'o\`u la pertinence de notre r\'esultat.

Apr\`es que l'applicabilit\'e de la hi\'erarchie de Lasserre ait \'et\'e d\'emontr\'ee pour des petits r\'eseaux, la prochaine \'etape \'etait de pouvoir traiter des r\'eseaux de grande taille. Cependant, il y avait peu de cas tests sur lequels tester de nouvelles approches. Les r\'eseaux sont consid\'er\'es de grande taille s'il contiennent plusieurs milliers de n\oe{}uds. Nous avons palli\'e ce manque en fournissant des donn\'ees du r\'eseau Europ\'een synchrone, contenant un peu plus de 9.000 n\oe{}uds. Afin de pouvoir travailler progressivement sur ces donn\'ees, nous avons fourni quatre instances correspondant à des parties de plus en plus grandes du r\'eseau Europ\'een (cf. Chapitre \ref{sec:Data of European high-voltage transmission network}). Les donn\'ees \'emanent d'un projet europ\'een impliquant nombreux gestionnaires de r\'eseaux dont le but \'etait de d\'evelopper de nouveaux outils pour le r\'eseau supra-national Europ\'een.

Depuis sa d\'ecouverte en 2000, la hi\'erarchie de Lasserre n'avait jamais r\'esolu des probl\`emes provenant des applications avec plus de quelques dizaines de variables. Ceci changea lorsque Daniel K. Molzahn et Ian A. Hiskens à l'universit\'e du Michigan developp\`erent un algorithme pour exploiter le creux dans la hi\'erarchie Lasserre pour le probl\`eme d'\'ecoulement des flux optimaux. Cela leur permit de s'attaquer à des r\'eseaux avec quelques centaines de n\oe{}uds. A peu pr\`es au m\^eme moment, Ramtin Madani, Morteza Ashraphijuo, et Javad Lavaei à l'universit\'e de Columbia ont montr\'e que la relaxation de Shor permet de r\'esoudre certains r\'eseaux de grande taille à condition d'ajouter deux termes de p\'enalisation à l'objecif. Nous avons propos\'e de combiner les deux approches afin d'apporter des solutions proches de l'optimum global de façon syst\'ematique (cf. Chapitre \ref{sec:Penalized Lasserre hierarchy for large-scale networks}). Ce travail a \'et\'e effectu\'e en collaboration avec l'universit\'e du Michigan. Dans l'approche combin\'ee, seul un param\`etre de p\'enalisation doit être sp\'ecifi\'e, au lieu de deux. Dans le cas de la minimisation des pertes, l'objectif est convexe (en fonction des variables de tensions) et nous avons observ\'e qu'aucun terme de p\'enalisation n'est n\'ecessaire. Cela signifie que l'approche trouve l'optimum global. Dans le cas de la minimisation des co\^uts de production, l'objectif n'est pas convexe et un terme de p\'enalisation doit être sp\'ecifi\'e, ce qui donne lieu à une solution proche de l'optimum global.

Sp\'ecifier un param\`etre de p\'enalisation est probl\'ematique car il n'existe pas de méthode générale pour le faire. Pour contourner ce probl\`eme, nous nous sommes aper\c cus que les p\'enalisations réussites \'etaient li\'ees \`a la matrice de Laplace du graphe du r\'eseau \'electrique. Nous avons donc propos\'e une relaxation de Shor bas\'ee sur la matrice de Laplace afin d'obtenir des solutions proches de l'optimum global sans avoir à sp\'ecifier un param\`etre (cf. Chapitre \ref{sec:Laplacian matrix gets rid of penalization parameter}). Un probl\`eme qui est survenu lorsque nous avons essay\'e de r\'esoudre des probl\`emes de grande taille est que les donn\'ees sont mal conditionn\'ees. Certaines lignes ont des imp\'edences très faibles alors que d'autres ont des impédences jusqu'\`a mille fois plus grandes. En conséquence, dans toutes les exp\'erimentations à grande \'echelle, les donn\'ees subissent un traitement pr\'ealable afin d'avoir des caract\'eristiques de lignes plus homog\`enes.

Ayant prouv\'e l'applicabilit\'e de la hi\'erarchie de Lasserre aux r\'eseaux de grande taille, nous avons ensuite r\'eduit son temps de calcul en la transposant des nombres r\'eels aux nombres complexes (cf. Chapitre \ref{sec:Complex hierarchy for enhanced tractability}). Ce qui nous a pouss\'e à le faire est que le probl\`eme d'\'ecoulement des flux est \'ecrit en nombres complexes. Ceux-ci sont utilis\'es pour mod\'eliser un ph\'enom\`eme oscillatoire, à savoir le courant alternatif. Nous nous sommes aper\c cus qu'omettre les non convexit\'es et convertir des nombres complexes aux r\'eels sont deux op\'erations non commutatives. Cela nous a conduit \`a proposer une approche g\'en\'erale pour trouver des solutions globales \`a des probl\`emes d'optimisation avec un domaine admissible born\'e o\`u les variables et les donn\'ees sont des nombres complexes. Elle est bas\'ee sur des r\'esultats r\'ecents en g\'eom\'etrie alg\'ebrique concernant des polyn\^omes strictement positifs avec des ind\'etermin\'ees complexes. En exploitant l'aspect creux, elle parvient \`a trouver des solutions globales à des probl\`emes à plusieurs milliers de variables complexes. En plus de la gestion et la planification des r\'eseaux d'\'electricit\'e du futur, la hi\'erarchie complexe des moments et sommes de carr\'es que nous avons d\'evelopp\'ee pourra \^etre appliqu\'ee en traitement du signal, en imagerie, en automatique, et en m\'ecanique quantique.
\\\\
La thèse est organis\'ee comme suit.
\\\\
Le chapitre \ref{sec:Background and motivations} d\'ecrit le probl\`eme de l'\'ecoulement des flux dans un r\'eseau de transport et les concepts math\'ematiques sous-jacents.\\\\
Le chapitre \ref{sec:Lasserre hierarchy for small-scale networks} illustre num\'eriquement que des ordres faibles de la hi\'erarchie de Lasserre permettent de r\'esoudre des r\'eseaux de petite taille. Publication associ\'ee: 
{\scshape C. Josz, J. Maeght, P. Panciatici, and J.C. Gilbert}, \textit{Application of the Moment-SOS
Approach to Global Optimization of the OPF Problem}, Institute of Electrical and Electronics Engineers, Transactions on Power Systems, 30, pp. 463–470, May 2014.  \href{http://dx.doi.org/10.1109/TPWRS.2014.2320819}{[doi]} \href{http://arxiv.org/pdf/1311.6370v1.pdf}{[preprint]}
\\\\
Le chapitre \ref{sec:Zero duality gap in the Lasserre hierarchy} prouve qu'il n'y a pas de saut de dualit\'e entre les versions primales et duales de la hi\'erarchie de Lasserre en pr\'esence d'une contrainte de boule dans le probl\`eme d'optimisation initial. Publication associ\'ee: 
{\scshape C. Josz and D. Henrion}, \textit{Strong Duality in Lasserre’s Hierarchy for Polynomial Optimization},
Springer Optimization Letters, February 2015. \href{http://dx.doi.org/10.1007/s11590-015-0868-5}{[doi]} \href{https://docs.google.com/viewer?a=v&pid=sites&srcid=ZGVmYXVsdGRvbWFpbnxjZWRyaWNqb3N6fGd4OjY1M2E5NDAyMjg2M2U2Y2Q}{[preprint]}
\\\\
Le chapitre \ref{sec:Data of European high-voltage transmission network} fournit des donn\'ees de grande taille repr\'esentant le r\'eseau Europ\'een à haute tension. Donn\'ees publiques associ\'ees: 
{\scshape C. Josz, S. Fliscounakis, J. Maeght, and P. Panciatici}, \textit{Power Flow Data of the European High-Voltage Transmission Network: 89, 1354, 2869, and 9241-bus PEGASE Systems}, MATPOWER 5.1, March 2015. \href{http://www.pserc.cornell.edu//matpower/}{[link]}\\\\
Le chapitre \ref{sec:Penalized Lasserre hierarchy for large-scale networks} calcule des solutions proches de l'optimum global pour des r\'eseaux de grande taille à l'aide de la hi\'erarchie de Lasserre et d'un paramètre de pénalisation. Publication associ\'ee: 
{\scshape D.K. Molzahn, C. Josz, I.A. Hiskens, and P. Panciatici}, \textit{Solution of Optimal Power Flow Problems using Moment Relaxations Augmented with Objective Function Penalization}, 54th Annual Conference on Decision and Control, Osaka, December 2015. \href{http://arxiv.org/pdf/1508.05037v1.pdf}{[preprint]} \\\\ 
Le chapitre \ref{sec:Laplacian matrix gets rid of penalization parameter} calcule des solutions proches de l'optimum global pour des r\'eseaux de grande taille à l'aide de matrices de Laplace au lieu d'un param\`etre de p\'enalisation. Papier soumis associ\'e:
{\scshape D.K. Molzahn, C. Josz, I.A. Hiskens, and P. Panciatici}, \textit{A Laplacian-Based Approach for Finding Near Globally Optimal Solutions to OPF Problems}, submitted to Institute of Electrical and Electronics Engineers, Transactions on Power Systems. \href{http://arxiv.org/pdf/1507.07212v1.pdf}{[preprint]}
\\\\
Le chapitre \ref{sec:Complex hierarchy for enhanced tractability} transpose la hi\'erarchie de Lasserre aux nombres complexes afin de r\'eduire les temps de calculs lorsqu'on s'int\'eresse à des variables complexes au lieu de variables r\'eelles. Papier soumis associ\'e:
{\scshape C. Josz, D. K. Molzahn}, \textit{Moment/Sum-of-Squares Hierarchy for Complex Polynomial Optimization}, submitted to Society for Industrial and Applied Mathematics, Journal on Optimization. \href{http://arxiv.org/pdf/1508.02068v1.pdf}{[preprint]}\\\\
Le chapitre \ref{sec:Conclusion and perspectives} sugg\`ere des pistes de recherches futures et est suivi des r\'ef\'erences.

\bigbreak
\newpage
\thispagestyle{empty}
\tableofcontents
\newpage

\chapter{Background and motivations}
\label{sec:Background and motivations}
The industrial problem which motivates this work can be viewed as an optimization problem. In this chapter, the equations that define this problem are presented and their relevance in practice is discussed. Next, optimality conditions are presented for general optimization problems. These are crucial to current methods used by industry, and remain important for the approach investigated in this thesis. This approach consists in solving convex relaxations of the original nonconvex problem. The original problem is written using complex numbers and this thesis advocates the use of convex relaxations in complex numbers. To that end, several definitions of complex numbers are provided.

\section{Optimal power flow problem}
\label{subsec:Optimal power flow problem}

The optimal power flow is a central problem in electric power systems introduced half a century ago by Carpentier~\cite{carpentier-1962}. It seeks to find a steady state operation point of an alternating current transmission network that respects Kirchoff's laws, Ohm's law, and power balance equations. In addition, the point has to be optimal under a criteria such as generation costs. It must also satisfy operational constraints which include narrow voltage ranges around nominal values and line ratings to keep Joule heating to acceptable levels.

While many nonlinear methods~\cite{huneault-1991,pandya-2008,matpower,castillo-2013} have been developed to solve this notoriously difficult problem, there is a strong motivation for producing more robust and reliable tools. 
Firstly, electric power systems are growing in complexity due to the increase in the share of renewables, the increase in the peak load, and the expected wider use of demand response and storage. This could hamper power systems reliability if decision-making tools do not evolve. Costly power interruptions could occur more often. Secondly, new tools are needed to profit from high-performance computing and advances in telecommunications such as phasor measurement units and dynamic line ratings. This will reduce operation costs and help keep power supply affordable at a time when expensive investments are being made for renewables. Lastly, system operators face large-scale optimization problems with combinatorial complexity due to phase-shifting transformers, high-voltage direct current transmission lines, and special protection schemes. Solving the continuous case to global optimality would be of great benefit for a more automated decision process.

Electricity transmission networks are meshed networks in which buses not only inject or retrieve power from the network, but also serve as a relay for other buses. Topologically, there exists cycles in the network. 
This is not the case for distribution networks where the topology of the network is a tree\footnote{Optimization over distribution networks may involve graphs that are not trees however. This is due to different possible configurations of the connections between buses in the network.}. A simple model of high-voltage power lines in transmission networks uses a resitance $R$, an inductance $L$, and a capacitance $C$ (cf. figure \ref{fig:tht}).
\begin{figure}[H]
	\centering
	\includegraphics[width=.6\textwidth]{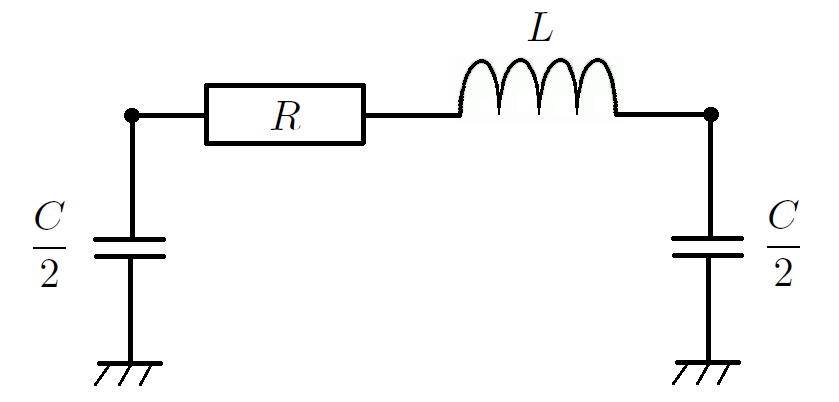}
	\caption{$\Pi$ model of a high voltage transmission line}
	 \label{fig:tht}
\end{figure}

Continental Europe uses alternating current (AC) at a frequency of $50$ Hz $\pm$ $0,5$ Hz, which makes for an angular speed of $\omega \approx 2\pi~\text{rad} \times 50~\text{Hz} \approx 314~\text{rad}.s^{-1}$. The total impedence of a resistance $R$ and an inductance $L$ in series is $R + \text{j} L \omega$. $110~\text{mH}$ is a typical value for inductance in a 100 km long line operating at 400 kV line, so reactance $L\omega$ is roughly equal to $110~\text{mH} \times 314~\text{rad}.s^{-1} \approx 35 ~ \Omega $. Divided by a hundred, the value lies in the range given in table $\ref{tab:Range of physical values in high voltage transmission lines and cables}$.
 
\begin{table}[H]
\begin{tabular}{l|c|c|c|c|c|}
\cline{2-6}
             & \multicolumn{3}{ c| }{\textbf{overhead line}} & \multicolumn{2}{ c| }{\textbf{underground cable}} \\ \cline{2-6}
             & 63 - 90 kV & 225 kV & 400 kV  &   63 kV     &     225 kV \\ \cline{1-6} 
\multicolumn{1}{|c|}{\textbf{resistance} ($\Omega$/km)} & 0.10 - 0.16 & 0.022 - 0.065 & 0.022 - 0.039 & 0.028 - 0.225 & 0.028 - 0.110 \\ \cline{1-6}
\multicolumn{1}{|c|}{\textbf{reactance} ($\Omega$/km)} & 0.4 & 0.29 - 0.41 & 0.32 - 0.43 & 0.104 - 0.134 & 0.107 - 0.134 \\ \cline{1-6}  
\multicolumn{1}{|c|}{\textbf{capacitance} (nF/km)} & 9.1 - 9.5 & 8.9 - 12.5 & 8.7 - 11.5 & 158 - 289 & 131 - 320 \\ \cline{1-6}         
\end{tabular}
\caption{Range of physical values in high voltage transmission lines and cables~\cite{blanchet}}
\label{tab:Range of physical values in high voltage transmission lines and cables}
\end{table}

In order to switch from one of the voltage levels shown in table $\ref{tab:Range of physical values in high voltage transmission lines and cables}$ to another, electricity transmission networks are equipped with an electrical device called \textit{transformer}. 
In this work, it is assumed that power entering a transformer is equal to power exiting it. Such a transformer is said to be an \textit{ideal transformer}. It is modeled by a complex number called \textit{ratio}. The output voltage is equal to the input voltage divided by the ratio while the output current is equal to the input current multiplied by the conjugate of the ratio. This is visible in figure \ref{fig:transmission line with pst} (where $\left(\cdot\right)^H$ denotes the conjugate transpose). Regular transformers have a real ratio and some special transformers called \textit{phase-shifting transformers} have a complex ratio.

Consider a non-zero integer $n \in \mathbb{N}^*$. We model an electricity transmission network by a set of buses $\mathcal{N}: = \{1,\hdots,n\}$ of which a subset $\mathcal{G} \subset \mathcal{N}$ is connected to generators. Let $s_k^\text{gen} = p_k^\text{gen} + \text{j} q_k^\text{gen} \in \mathbb{C}$ denote generated power at bus $k \in \mathcal{G}$. All buses are connected to a load (i.e. power demand). Let $s_k^\text{dem} = p_k^\text{dem} + \text{j} q_k^\text{dem} \in \mathbb{C}$ denote power demand at bus $k \in \mathcal{N}$. Let $v_k \in \mathbb{C}$ denote voltage at bus $k\in \mathcal{N}$ and $i_k \in \mathbb{C}$ denote current injected into the network at bus $k\in \mathcal{N}$. The convention used for current means that $v_k i_k^H$ is the power injected into the network at bus $k\in \mathcal{N}$. This means that $v_k i_k^H = -s_k^\text{dem}$ at bus $k \in \mathcal{N} \setminus \mathcal{G}$ and $v_k i_k^H = s_k^\text{gen} - s_k^\text{dem}$ at bus $k\in \mathcal{G}$.

The network links buses to one another through a set of lines $\mathcal{L} \subset \mathcal{N} \times \mathcal{N}$. A link between two buses is described in figure $\ref{fig:transmission line with pst}$. In this figure, $y_{lm} \in \mathbb{C}$ denotes the mutual admittance between buses $(l,m)\in \mathcal{L}$ ($y_{ml} = y_{lm}$ for all $(l,m)\in \mathcal{L}$); $y^\text{gr}_{lm} \in \mathbb{C}$ denotes the admittance-to-ground at end $l$ of line $(l,m)\in \mathcal{L}$; $\rho_{lm} \in \mathbb{C}$ denotes the ratio of the phase-shifting transformer at end $l$ of line $(l,m) \in \mathcal{L}$ ($\rho_{lm} = 1$ if there is no transformer); and $i_{lm} \in \mathbb{C}$ denotes current injected in line $(l,m)\in \mathcal{L}$ at bus $l$. 
\begin{figure}[H]
	\centering
	\includegraphics[width=1\textwidth]{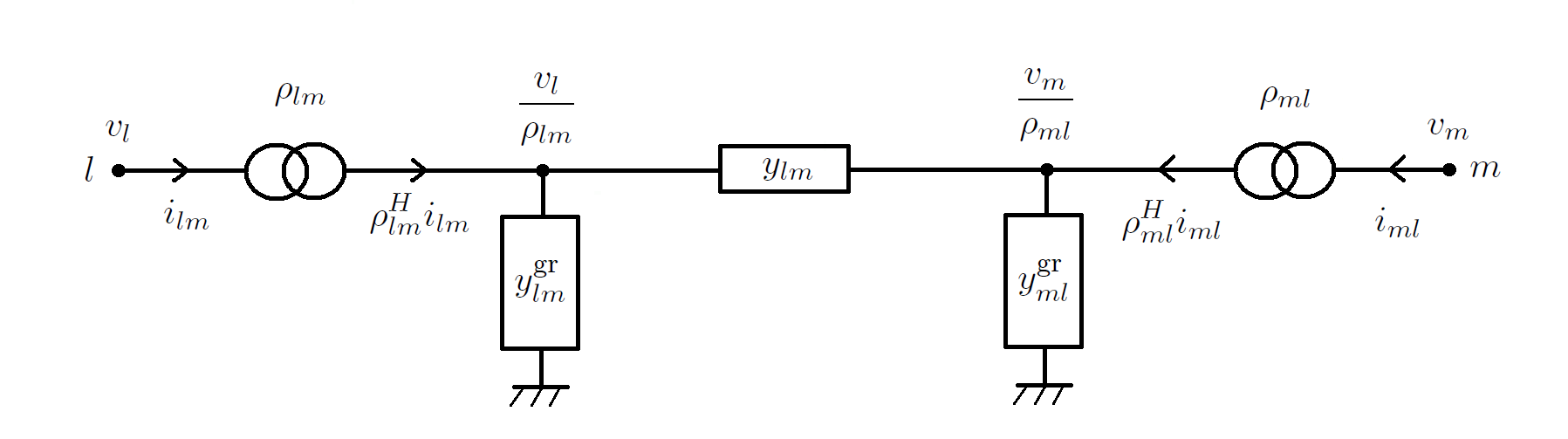}
	\caption{Link between buses $l$ and $m$}
	 \label{fig:transmission line with pst}
\end{figure}
A formulation of the optimal power flow problem is given in tables $\ref{tab:Objective, variables, and data of OPF problem}$ (where $a_k,b_k,c_k \in \mathbb{R}$) and $\ref{tab:Constraints of OPF problem}$.
\begin{table}[H]
\centering
\begin{tabular}{|c|c|}
\hline
\textbf{objective} & \textbf{description} \\
\hline
$\min \sum_{k \in \mathcal{G}} a_k (p^\text{gen})^2 + b_k p^\text{gen} + c_k$ & generation cost \\
\hline
\textbf{variables} & \textbf{description} \\
\hline
$(i_k)_{k\in \mathcal{N}}$ & injected current\\
$(i_{lm})_{(l,m) \in \mathcal{L}}$ & current flow\\
$(p_k^\text{gen})_{k\in \mathcal{N}}$ & active generation\\
$(q_k^\text{gen})_{k\in \mathcal{N}}$ & reactive generation\\
$(v_k)_{k\in \mathcal{N}}$ & voltage\\
\hline
\textbf{data} & \textbf{description} \\
\hline
$(y_{lm})_{(l,m) \in \mathcal{L}}$ & mutual admittance\\
$( y_{lm}^\text{gr})_{(l,m) \in \mathcal{L}}$ & admittance-to-ground\\
$(\rho_{lm})_{(l,m) \in \mathcal{L}}$ & ratio of (phase-shifting) transformer\\
$(p_k^\text{dem})_{k\in \mathcal{N}}$ & active power demand\\
$(q_k^\text{dem})_{k\in \mathcal{N}}$ & reactive power demand\\
$v_k^{\text{min}},v_k^{\text{max}}, p_k^{\text{min}},p_k^{\text{max}},q_k^{\text{min}},q_k^{\text{max}}$ & bounds at buses\\
$i_{lm}^{\text{max}}, v_{lm}^{\text{max}},s_{lm}^{\text{max}},p_{lm}^{\text{max}}$  & bounds on line flow\\
\hline 
\end{tabular}
\caption{Objective, variables, and data}
\label{tab:Objective, variables, and data of OPF problem}
\end{table}

\begin{table}[H]
\centering
\begin{tabular}{|c|c|c|}
\hline
$\begin{array}{c} k\in ... \\ (l,m) \in ... \end{array}$ & \textbf{constraints} & \textbf{description} \\
\hline
$\mathcal{N}$ & $i_l = \sum_{m \in \mathcal{N}(l)} i_{lm}$ & Kirchoff's first law \\
$\mathcal{L}$ & $\rho_{lm}^H i_{lm} =  y_{lm}^\text{gr} \frac{v_l}{\rho_{lm}} + y_{lm}(\frac{v_l}{\rho_{lm}}  - \frac{v_m}{\rho_{ml}})$ & Kirchoff's first law and Ohm's law \\
$\mathcal{N} \setminus \mathcal{G}$  & $v_k i_k^H = -  p^\text{dem}_k  - \text{j} q^\text{dem}_k$ & power demand \\
$\mathcal{G}$ & $v_k i_k^H = p^\text{gen}_k -  p^\text{dem}_k + \text{j}( q^\text{gen}_k - q^\text{dem}_k )$ & power demand and generation\\
$\mathcal{G}$ & $p_k^{\text{min}} \leqslant p^\text{gen}_k \leqslant p_k^{\text{max}}$ & bounds on active generation \\
$\mathcal{G}$ & $q_k^{\text{min}} \leqslant q^\text{gen}_k \leqslant q_k^{\text{max}}$ & bounds on reactive generation \\
$\mathcal{N}$ & $v_k^{\text{min}} \leqslant |v_k| \leqslant v_k^{\text{max}}$ & bounds on voltage amplitude \\
$\mathcal{L}$ & $|v_l - v_m| \leqslant v_{lm}^{\text{max}}$ & bound on voltage difference \\
$\mathcal{L}$ & $|i_{lm}| \leqslant i_{lm}^{\text{max}}$ & bound on current flow \\
$\mathcal{L}$ & $|v_l i_{lm}^H| \leqslant s_{lm}^{\text{max}}$ & bound on apparent power flow \\
$\mathcal{L}$ & $|\text{Re}(v_l i_{lm}^H)| \leqslant p_{lm}^{\text{max}}$ & bound on active power flow \\
\hline
\end{tabular}
\caption{Constraints}
\label{tab:Constraints of OPF problem}
\end{table}
According to the second constraint in table $\ref{tab:Constraints of OPF problem}$, for all $(l,m) \in \mathcal{L}$ :
\begin{equation}
i_{lm} =  \frac{y_{lm} + y_{lm}^\text{gr}}{|\rho_{lm}|^2} v_l  - \frac{y_{lm}}{\rho_{ml} \rho_{lm}^H} v_m
\label{eq:current linear function of voltages}
\end{equation}
Together with the first constraint in table $\ref{tab:Constraints of OPF problem}$, relationship ($\ref{eq:current linear function of voltages}$) yields that for all $l \in \mathcal{N}$ :
$$
\begin{array}{rcl}
i_l & = & \sum_{m \in \mathcal{N} \setminus \{l\}} i_{lm} \\
    & = & \sum_{m \in \mathcal{N} \setminus \{l\}} \frac{y_{lm} + y_{lm}^\text{gr}}{|\rho_{lm}|^2} v_l  - \frac{y_{lm}}{\rho_{ml} \rho_{lm}^H} v_m \\
    & = & \left(\sum_{m \in \mathcal{N} \setminus \{l\}} \frac{y_{lm} + y_{lm}^\text{gr}}{|\rho_{lm}|^2} \right) v_l - \sum_{m \in \mathcal{N} \setminus \{l\}} \frac{y_{lm}}{\rho_{ml} \rho_{lm}^H} v_m
\end{array}
$$
Define the \textit{admittance matrix} $Y$ as the complex matrix of size $n \times n$ by:
$$ Y_{lm} : =
\left\{
\begin{array}{lc}
\sum_{m \in \mathcal{N} \setminus \{l\}} \frac{y_{lm} + y_{lm}^\text{gr}}{|\rho_{lm}|^2} & \text{if} ~~ l = m \\
- \frac{y_{lm}}{\rho_{ml} \rho_{lm}^H} & \text{if} ~~ l \neq m
\end{array}
\right.
$$
Also, define $\textbf{i} : = (i_k)_{k\in \mathcal{N}}$ and $\textbf{v} : = (v_k)_{k\in \mathcal{N}}$. It follows that :
$$ \textbf{i} = Y \textbf{v} $$

\section{Optimality conditions in optimization}
\label{subsec:Local versus global optimality}
As mentionned in Section \ref{subsec:Optimal power flow problem}, current methods for solving the optimal power flow problem use nonlinear optimization techniques. These aim to find at least a solution to the optimality conditions, which we present in this section. Satisfaction of the optimality conditions does not guarantee global optimality for nonconvex problems, but they do for convex problems. The proposed approach in this dissertation uses the optimality conditions to solve convex relaxations of the optimal power flow problem, a concept presented in the next section. 

To discuss optimality conditions, we consider a general framework that encompasses both the nonconvex and convex cases. Consider a finite dimensional normed vector space $\mathbb{E}$ over $\mathbb{R}$ or $\mathbb{C}$ and an objective function $f: \mathbb{E} \rightarrow \mathbb{R}$. Also, consider a feasible set $X \subset \mathbb{E}$ described by a single function $c: \mathbb{E} \rightarrow \mathbb{F}$ where $\mathbb{F}$ is a Hilbert space over $\mathbb{R}$ or $\mathbb{C}$. The feasible set is also defined by a nonempty closed convex cone $K \subset \mathbb{F}$. Using these notations, the problem to be solved can be written:
$$ \boxed{\inf_{x \in \mathbb{E}} ~ f(x) ~~ \text{subject to} ~~ c(x) \in K} $$

Objective function $f$ and constraint function $c$ will be considered twice differentiable. This is valid for the optimal power flow problem and its convex relaxations.\\\\
\textbf{Relationship between local optimality and derivatives}\\
Let's illustrate the relationship between local optimality and derivatives with a simple example. Consider a function $f: \mathbb{R} \rightarrow \mathbb{R}$  with a local minimum in 0 equal to 0. The first order Taylor series reads:
$$ 0 \leqslant f(t) = f'(0)t + o(t) = t\left[f'(0)+o(1)\right] $$
Thus $f'(0)=0$. The second order Taylor series reads:
$$ 0 \leqslant f(t) = f'(0)t + \frac{f''(0)}{2}t^2 + o(t^2) = \frac{1}{2}t^2\left[f''(0)+o(1)\right] $$
Thus $f''(0) \geqslant 0$. If $f''(0) > 0$ then the above line of equations tells us that $0$ is a strict local minimum. Else if $f''(0) = 0$, the third order Taylor series reads:
$$ 0 \leqslant f(t) = f'(0)t + \frac{f''(0)}{2}t^2 + \frac{f'''(0)}{6}t^3 + o(t^3) = \frac{1}{6}t^3\left[f'''(0)+o(1)\right] $$
Thus $f'''(0) = 0$. The fourth order Taylor series will then imply that $f''''(0) \geqslant 0$. We can again distinguish between stricly inequality and equality, and so on $\hdots$ \\\\
In practice, only first order and second order conditions are considered because higher order derivatives are expensive to compute. Moreover, higher order derivatives are unrelated to local optimality unless all lower order derivatives zero out in one point. An example is $f(t) = t^4$ where $f'(0) = f''(0) = f'''(0) = 0$ and $f''''(0) = \frac{1}{4} > 0$. The fourth order strict inequality indicates that 0 is a strict local minimum. On the other hand, if one considers $\alpha \in \mathbb{R}$ and $f(t) = t^2 + \alpha t^3$, this yields $f'(0) = 0$, $f''(0) = 2 > 0$, and $f'''(0) = 6 \alpha$ so that 0 is a strict local minimum regardless of the value of the third order derivative.
\\\\
\textbf{First and second order necessary optimality conditions}\\
Let's go back to the general case where $f: \mathbb{E} \rightarrow \mathbb{R}$.
When there are no constraints, a locally optimal point must be stationnary, that is to say that first order derivatives must zero out in that point. Indeed, consider an optimal point $x$ and write the first order Taylor series for~$h$ close to but different from zero:
$$
0 \leqslant f(x+h) - f(x) = f'(x).h + o(\|h\|) = \|h\| \left[ f'(x)\left(\frac{h}{\|h\|}\right) + o(1) \right]
$$
Hence for all $\|d\|=1$, one has $f'(x).d \geqslant 0$. Taking $d = - \nabla f(x)$  leads to $f'(x)=0$.\\\\
Consider the second order Taylor series for $h$ close to but different from zero:
$$
0 \leqslant f(x+h) - f(x) = f'(x).h + f''(x)(h,h) + o(\|h\|^2) = \|h\|^2 \left[ f''(x) \left( \frac{h}{\|h\|},\frac{h}{\|h\|} \right) + o(1) \right]
$$
Hence for all $\|d\|=1$, one has $f''(x) \left(d,d \right)\geqslant 0$.\\\\ 
When there are constraints, establishing necessary optimality conditions requires the concept of duality in some way or another. 
Let's establish these conditions using the notion of saddle point in min-max duality. This implies an assumption of global  rather than local optimality. 
Min-max duality consists of writing the objective function as the supremum of a coupling function, that is to say : $f(x) = \sup_{y \in Y} \varphi(x,y)$. The optimization problem can thus be written, for some $X \subset \mathbb{E}$ and some $Y \subset \mathbb{F}$ : $$ \inf_{x\in X} \sup_{y \in Y} \varphi(x,y) $$
If it is legitimate to swap inf and sup, i.e. $ \inf_{x\in X} \sup_{y \in Y} \varphi(x,y) = \sup_{y \in Y}\inf_{x\in X}  \varphi(x,y) $ (called \textit{no duality gap}), and there exists a solution to the right-hand problem, then solving the original problem is equivalent to finding a saddle point of the coupling function. Indeed, $(\overline{x},\overline{y})$ is saddle point of $\varphi$ if and only if there is no duality gap and $\overline{x}$ solves $\inf_{x\in X} \sup_{y \in Y} \varphi(x,y)$ and $\overline{y}$ solves $\sup_{y \in Y}\inf_{x\in X}  \varphi(x,y)$.\\\\
The Lagrange function $ \varphi(x,y) = f(x) + \langle c(x),y\rangle $ where $X := \mathbb{E}$ and $Y := K^-:= \{ ~ y \in \mathbb{F} ~|~ \langle z , y \rangle \leqslant 0,~\forall z \in K~\}$ is an example of a coupling function with the enviable property that $\varphi(x,\cdot)$ is an affine function. Let's consider a saddle point $(\overline{x},\overline{y})\in X \times Y$ of the Lagrange function. By definition :
$$ \forall x \in X, \forall y \in Y, ~~~ \varphi(\overline{x},y) \leqslant \varphi(\overline{x},\overline{y}) \leqslant \varphi(x,\overline{y}) $$
The right-hand side inequality means that $\overline{x}$ is an optimal solution of the unconstrained problem $\inf_{x \in \mathbb{E}} \varphi(x,\overline{y})$. Thus $\varphi'_x(\overline{x},\overline{y}) = f'(\overline{x}) + c'(\overline{x})^*~\overline{y} = 0$. 
Proceeding in the same fashion with the left-hand side inequality leads to a constrained optimization problem, so we will proceed differently. The left-hand side inequality implies:
$$\forall y \in K^-,~~~ \langle c(\overline{x}),y \rangle \leqslant  \langle c(\overline{x}),\overline{y} \rangle $$
Plugging in for $y=0$ and $y=2\bar{y}$ yields $\langle c(\overline{x}),\overline{y} \rangle = 0$, known as \textit{complementary slackness}.\\\\
To sum up, if $\overline{x}$ is primal optimal and $\overline{y}$ is dual optimal and there is no duality gap, then:
$$
\left\{
\begin{array}{rl}
\text{first order condition:} & f'(\overline{x}) + c'(\overline{x})^*~\overline{y} = 0 \\

\text{primal feasibility:} & c(\overline{x}) \in K \\
\text{dual feasibility:} & \overline{y} \in K^- \\
\text{complementary slackness:} & \langle c(\overline{x}),\overline{y} \rangle = 0
\end{array}
\right.
$$
Note that min-max duality originates from the minimax theorem proven by Von Neumann in his 1928 paper \textit{Zur Theorie der Gesellschaftsspeile}. A generalization of the minimax theorem states that if $X$ and $Y$ are nonempty convex sets, $X$ is compact, and $\varphi$ is continuous and convex-concave, then:
$$ \inf_{x\in X} \sup_{y \in Y} \varphi(x,y) = \sup_{y \in Y}\inf_{x\in X}  \varphi(x,y) $$
In general, the KKT conditions do not guarantee global optimality, nor even local optimality in fact. If the objective and constraint functions $f$ and  $c$ are convex, then they guarantee global optimality. In the next section, nonconvexities are removed from the optimal power flow problem, yielding a relaxed convex problem. This convex problem is solved using interior-point methods which involve solving for KKT conditions.

\section{Convex relaxation of the optimal power flow problem}
\label{subsec:Convex relaxation of the optimal power flow problem}

Lavaei and Low \cite{lavaei-low-2012} proposed a formulation of the optimal power flow problem where the variables are the real and imaginary parts of voltages at each bus. To do so, they defined a real vector $x = [~\text{Re}(v) ~ ~\text{Im}(v) ~]^T$ where $v \in \mathbb{C}^n$ are the complex voltages. Next they proposed a convex relaxation of the optimal power flow problem. We illustrate their work by considering an example of power loss minimization. 
\begin{figure}[ht]
  \centering
    \includegraphics[width=.65\textwidth]{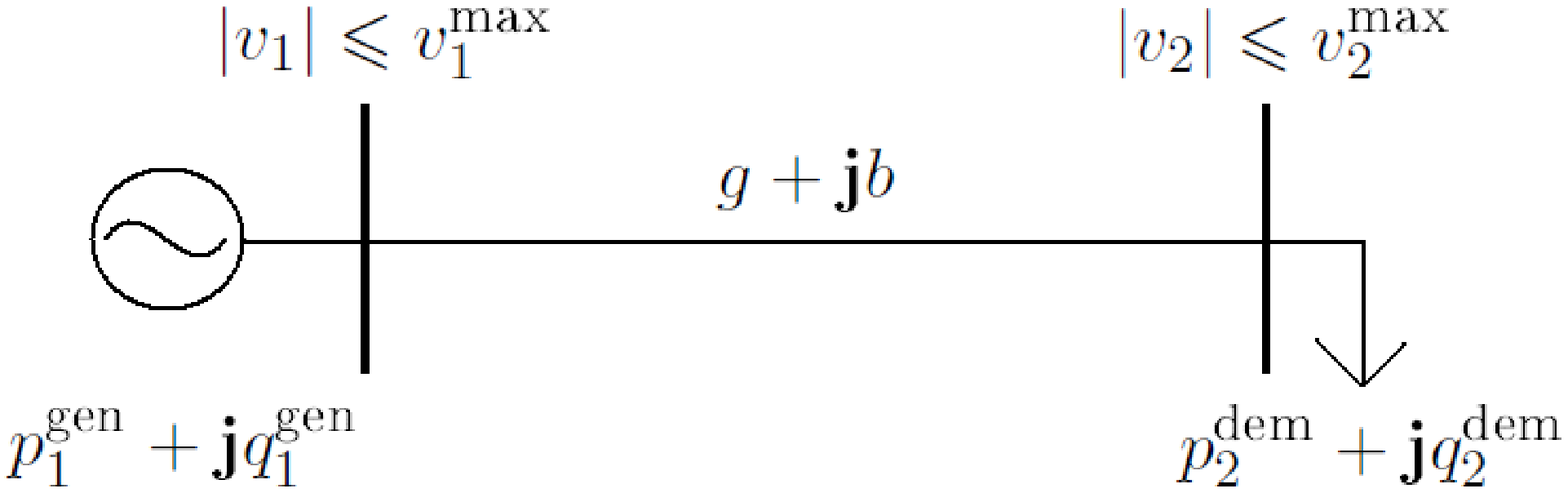}
  \caption{Two-Bus System}
  \label{fig:WB2}
\end{figure}
The system 
of Figure~\ref{fig:WB2} 
links a generator to a load via a line of admittance $g+\textbf{i}b$ while respecting upper voltage constraints. Minimizing power loss reads
\begin{gather}
\inf_{v_1,v_2 \in \mathbb{C}} ~~ g ~ |v_1|^2 - g ~ \overline{v}_1 v_2 - g ~ \overline{v}_2 v_1 + g ~ |v_2|^2, \\
\text{subject to} ~~~~~~~~~~~~~~~~~~~~~~~~~~~~~~~~~~~~~~~~~~~~~~~~~~~~~~~~~~~~~~~~~~~~~~~~~~~~~~~~~~~~~~~~~~~~~~~~~~~~~~~~~~~ \notag \\
 -\frac{g-\textbf{i}b}{2} ~ \overline{v}_1 v_2 -\frac{g+\textbf{i}b}{2} ~ \overline{v}_2 v_1 + g ~ |v_2|^2 = -p_2^\text{dem}, \\
~~~~~ \frac{b+\textbf{i}g}{2} ~ \overline{v}_1 v_2 + \frac{b-\textbf{i}g}{2} ~ \overline{v}_2 v_1 -b ~ |v_2|^2 = -q_2^\text{dem}, ~~ \\
|v_1|^2 \leqslant (v_1^{\text{max}})^2, \\
|v_2|^2 \leqslant (v_2^{\text{max}})^2,
\end{gather}
where $\textbf{i}$ denotes the imaginary number.
Identifying real and imaginary parts of the variables $ v_1 =: x_1 + \textbf{i} x_3 $ and $ v_2 =: x_2 + \textbf{i} x_4 $ leads to
\begin{gather}
\inf_{x_1,x_2,x_3,x_4 \in \mathbb{R}} ~~ gx_1^2 + gx_3^2 - 2gx_1 x_2 - 2g x_3 x_4 - g x_2^2 - gx_4^2, \\
\text{subject to} ~~~~~~~~~~~~~~~~~~~~~~~~~~~~~~~~~~~~~~~~~~~~~~~~~~~~~~~~~~~~~~~~~~~~~~~~~~~~~~~~~~~~~~~~~~~~~~~~~~~~~~~~~~~ \notag \\
- gx_1 x_2 - g x_3 x_4 - b x_1 x_4 + b x_2 x_3 + g x_2^2 + g x_4^2 + p_2^\text{dem} = 0, \\
\hphantom{-}b x_1 x_2 + b x_3 x_4  - g x_1 x_4 + g x_2 x_3 - b x_2^2 - b x_4^2 + q_2^\text{dem} = 0, \\
x_1^2 + x_2^2 \leqslant (v_1^{\text{max}})^2, \\
x_3^2 + x_4^2 \leqslant (v_2^{\text{max}})^2. \\
\end{gather}
This problem can be rewritten as 
\begin{gather}
\inf_{y} ~~ gy_{11} + gy_{33} - 2gy_{12} - 2g y_{34} - g y_{22} - g y_{44}, \\
\text{subject to} ~~~~~~~~~~~~~~~~~~~~~~~~~~~~~~~~~~~~~~~~~~~~~~~~~~~~~~~~~~~~~~~~~~~~~~~~~~~~~~~~~~~~~~~~~~~~~~~~~~~~~~~~~~~ \notag \\
- gy_{12} - g y_{34} - b y_{14} + b y_{23} + g y_{22} + g y_{44} + p_2^\text{dem} = 0, \\
\hphantom{-}b y_{12} + b y_{34}  - g y_{14} + g y_{23} - b y_{22} - b y_{44} + q_2^\text{dem} = 0, \\
y_{11} + y_{22} \leqslant (v_1^{\text{max}})^2, \\
y_{33} + y_{44} \leqslant (v_2^{\text{max}})^2, \\
\left(
\begin{array}{cccc}
y_{11} & y_{12} & y_{13} & y_{14} \\
y_{12} & y_{22} & y_{23} & y_{24} \\
 y_{13} & y_{23} & y_{33} & y_{34} \\
y_{14} & y_{24} & y_{34} & y_{44} 
\end{array}
\right)
\succcurlyeq 0, \\
\text{rank} (y) \leqslant 1.
\end{gather}
Removing the rank constraint leads to a convex relaxation of the optimal power flow problem.

As an extension of Lavaei and Low's work, Sojoudi and Lavaei~\cite{lavaei-sojoudi-2012} studied the theory behind optimization over graphs. To help with future work, we provide fully detailed proofs of results found in \cite{lavaei-sojoudi-2012}. The reader may skip these by moving to Section \ref{subsec:Definitions of complex numbers} and still understand the rest of the dissertation.\\\\
Consider
$$ \mathbb{H}^n_{+1} := \{~\textbf{v}\textbf{v}^H ~|~ \textbf{v} \in \mathbb{C}^n~\}. $$
Let $\textbf{i}$ denote the imaginary number. The notation of set $\mathbb{H}^n_{+1}$ stems from the following proposition:
\begin{proposition}
\label{prop:notation hermitian of rank one}
\normalfont
$$
M \in \mathbb{H}^n_{+1} ~~~ \Longleftrightarrow ~~~ M \in \mathbb{H}^n_+ ~~\text{and}~~ \text{rk}(M) \leqslant 1
\label{prop:positive semidefinite hermitian rank one}
$$
\end{proposition}
\begin{proof}[Proof]
($\Longrightarrow$) Consider a matrix $M \in \mathbb{H}^n_{+1}$, that is, there exists $\textbf{v} \in \mathbb{C}^n$ such that $M = \textbf{v}\textbf{v}^H$. Firstly, observe that $M^H = (\textbf{v}\textbf{v}^H)^H = \textbf{v}\textbf{v}^H = M$. Secondly, for all $\textbf{z} \in \mathbb{C}^n$, $\textbf{z}^H M \textbf{z} = \textbf{z}^H \textbf{v}\textbf{v}^H \textbf{z} = (\textbf{v}^H \textbf{z})^H \textbf{v}^H \textbf{z} = |\textbf{v}^H \textbf{z}|^2 \geqslant 0$. Lastly, each column of $M=\textbf{v}\textbf{v}^H$ is a linear combination of $\textbf{v}$, so that the rank of $M$ is at most 1.\\\\
($\Longleftarrow$) Matrix $M$ is of rank at most 1 so there exists two vectors $\textbf{t},\textbf{u} \in \mathbb{C}^n$ such that $M = \textbf{t}\textbf{u}^H$. Define $\textbf{v} \in \mathbb{C}^n$ such that for all $1\leqslant i \leqslant n$: 
$$v_i: = \sqrt{|t_i u_i|}~~ e^{\textbf{i}\text{arg}(u_i)}$$
Thus : $$
\begin{array}{rcl}
v_i v_j^H & = & \sqrt{|t_i u_i t_j u_j|} ~~ e^{\textbf{i}\text{arg}(u_i)} ~ e^{-\textbf{i}\text{arg}(u_j)} \\
 & = & \sqrt{|t_i u_j t_j u_i |} ~~ e^{\textbf{i}\text{arg}(u_i)} ~ e^{-\textbf{i}\text{arg}(u_j)} \\
 & = & \sqrt{|t_i u_j|^2} ~~ e^{\textbf{i}\text{arg}(u_i)} ~ e^{-\textbf{i}\text{arg}(u_j)} ~~ \text{($|t_i u_j|= |t_j u_i|$ since $\textbf{t} \textbf{u}^H \in \mathbb{H}^n$)} \\
 & = & |t_i u_j| ~~ e^{\textbf{i}\text{arg}(u_i)} ~ e^{-\textbf{i}\text{arg}(u_j)} \\
 & = & |t_i| ~ e^{\textbf{i}\text{arg}(u_i)} ~~ |u_j| ~ e^{-\textbf{i}\text{arg}(u_j)} \\
 & = & |t_i| ~ e^{\textbf{i}\text{arg}(u_i)} ~~ u_j^H \\
 & = & |t_i| ~ e^{\textbf{i}\text{arg}(t_i)} ~~ u_j^H ~~ \text{($\text{arg}(t_i) \equiv \text{arg}(u_i)[2\pi]$ since $t_iu_i^H \in \mathbb{R}_+$ since $\textbf{t}\textbf{u}^H \succcurlyeq 0 $)} \\
 & = & t_i u_j^H \\
\end{array}
$$
Vector $\textbf{v }\in \mathbb{C}^n$ thereby satisfies $\textbf{v} \textbf{v}^H = \textbf{t} \textbf{u}^H = M $.
\end{proof}

Given a set of edges $\mathcal{E} \subset \mathcal{N} \times \mathcal{N}$, define the following $\mathbb{C}$-linear operator :
$$
\begin{array}{rccl}
\phi^\mathcal{E} : & \mathcal{M}_n(\mathbb{C}) & \longmapsto & \mathcal{M}_n(\mathbb{C}) \\
 & M & \longrightarrow & \phi^\mathcal{E}(M)_{ij} = \left\{ \begin{array}{cl} M_{ij} & \text{if} ~~ (i,j) \in \mathcal{E} ~~ \text{or} ~~ i = j \\ 0 & \text{else} \end{array} \right.
\end{array}
$$
Given a graph $\mathcal{L}$, graph theory can be used to decompose the constraint $M \in \mathbb{H}^n_{+1} + \text{Ker}(\phi^\mathcal{L})$ into several smaller constraints. (The graph $\mathcal{L}$ typically corresponds to the sparsity pattern. In a such sparse optimization problem, the constraint $M \in \mathbb{H}^n_{+1}$ may be replaced by $M \in \mathbb{H}^n_{+1} + \text{Ker}(\phi^\mathcal{L})$.) First, two lemmas are presented.

\begin{lemma}
\label{lemma:spanning tree}
\normalfont
\textit{Any undirected connected graph has a spanning tree.}
\end{lemma}
\begin{proof}[Proof] Let $\mathcal{L}$ be the set of edges of an undirected connected graph. Define the following set :
$$\mathcal{S} : = \{ ~ \mathcal{T} \in \mathcal{P}(\mathcal{L})~|~ \mathcal{T} ~ \text{is a tree} ~ \}$$
$(\mathcal{S},\subset)$ is a partially ordered set since $(\mathcal{P}(\mathcal{L}),\subset)$ is a partially ordered set and $\mathcal{S} \subset \mathcal{P}(\mathcal{L})$. Consider a totally ordered subset of $\mathcal{S}$ and name it $\mathcal{U}$. $\cup_{\mathcal{T}\in \mathcal{U}} \mathcal{T}$ is a bound of $\mathcal{U}$ in $\mathcal{S}$. Zorn's lemma implies that $\mathcal{S}$ contains a maximal element. If the maximal element is not a spanning tree, there exists a vertex not contained in it. Since the graph is connected, there exists a path in $\mathcal{L}$ linking this vertex to a vertex in the maximum tree. The union of the maximum tree and that path forms a tree of $\mathcal{L}$ that contradicts the maximality of the maximum tree. Thus the maximum tree is a spanning tree of $\mathcal{L}$.
\end{proof}
\begin{lemma}
\label{lemma:equation over spanning tree}
\normalfont
\textit{Let $\mathcal{T}$ denote a spanning tree of a finite, undirected, and connected graph $(\mathcal{N},\mathcal{L})$ and let $(\theta_{ij})_{(i,j)\in \mathcal{T}}$ denote some real numbers. Assume that $\theta_{ij} + \theta_{ji} \equiv 0 ~[2\pi]$ for all $(i,j)\in \mathcal{T}$. Then there exists some real numbers $(\theta_i)_{i \in \mathcal{N}}$ such that:}
\begin{equation*}
\theta_i - \theta_j \equiv \theta_{ij}~[2\pi] ~~,~~ \forall (i,j) \in \mathcal{T}
\end{equation*}
\end{lemma}
\begin{proof}[Proof]
Simply define $(\theta_i)_{i \in \mathcal{N}}$ by choosing a random real number $\theta_{i_0}$ for some $i_0 \in \mathcal{N}$ and for all $i \in \mathcal{N}\setminus \{i_0\}$, choose $\theta_i \equiv \theta_{i_0} + \sum_{l,m} \theta_{lm} ~[2\pi]$ where the sum is taken over the path in $\mathcal{T}$ linking $i$ to $i_0$. This is possible because $(\mathcal{N},\mathcal{L})$ is a finite graph. Thus defined, the real numbers $(\theta_i)_{i \in \mathcal{N}}$ satisfy for all $(i,j) \in \mathcal{T}$:
$$
\begin{array}{rcll}
\theta_i - \theta_j & \equiv & \theta_{ij}~[2\pi] & \\
\end{array}
$$
because all terms in the sums associated respectively to $\theta_i$ and $\theta_j$ cancel each other out except for $\theta_{ij}$. Indeed, $(i,j)\in \mathcal{T}$ so the paths starting at $i$ and $j$ and both ending at $i_0$ are the same expect that one contains $(i,j)$ or $(j,i)$ and the other doesn't. If the former contains $(i,j)$, the only term that does not cancel out is $\theta_{ij}$. If the latter contains $(j,i)$, the only term that does not cancel out is $-\theta_{ji} \equiv \theta_{ij} ~[2\pi]$.
\end{proof}

Let $(c_i)_{1 \leqslant i \leqslant p}$ denote a cycle basis of $\mathcal{L}$ and let $(b_i)_{1 \leqslant i \leqslant q}$ denote the set of bridge edges of $\mathcal{L}$. Define the following set :
$$ \Omega : = \underbrace{\{c_1,\hdots,c_p\}}_\text{cycle basis} \cup \underbrace{\{b_1,\hdots,b_q\}}_\text{bridge edges} $$
\begin{proposition}
\normalfont
\textit{The following statement holds:}
$$
\mathbb{H}^n_{+1} + \text{Ker}(\phi^\mathcal{L}) = \bigcap_{\mathcal{E} \in \Omega} \mathbb{H}^n_{+1} + \text{Ker}(\phi^\mathcal{E})
$$
\label{prop:matrix completion over cycle basis}
\end{proposition}
\begin{proof}[Proof]
($\subset$) It suffices to see that $\text{Ker}(\phi^\mathcal{L}) \subset \text{Ker}(\phi^\mathcal{E})$ for all $\mathcal{E} \in \Omega$. Indeed, $\mathcal{E} \in \Omega$ implies that $\mathcal{E} \subset \mathcal{L}$.
\\\\
($\supset$) Consider $M \in \bigcap_{\mathcal{E} \in \Omega} \mathbb{H}^n_{+1} + \text{Ker}(\phi^\mathcal{E})$. For all $\mathcal{E} \in \Omega$, there exists $\textbf{v}^\mathcal{E} \in \mathbb{C}^n$ and $N^\mathcal{E} \in \text{Ker}(\phi^\mathcal{E})$ such that :
$$ M = \textbf{v}^\mathcal{E} (\textbf{v}^\mathcal{E})^H + N^\mathcal{E} $$
Thus, for all $i,j \in \mathcal{N}$ :
$$ M_{ij} = v_i^\mathcal{E} (v_j^\mathcal{E})^H + N^\mathcal{E}_{ij} $$
Moreover, $\phi^\mathcal{E}(N^\mathcal{E}) = 0$ so $N^\mathcal{E}_{ij} = 0$ if $(i,j) \in \mathcal{E}$ or $i = j$. It follows that for all $\mathcal{E} \in \Omega$ :
\begin{subnumcases}{}
\label{subcase:Ma}
M_{ii} = v_i^\mathcal{E} (v_i^\mathcal{E})^H ~~,~~ \text{for all vertices $i$ of $\mathcal{E}$} \\
\label{subcase:Mb}
M_{ij} = v_i^\mathcal{E} (v_j^\mathcal{E})^H ~~,~~ \forall (i,j) \in \mathcal{E}
\end{subnumcases}
$\mathcal{L}$ is an undirected connected graph so there exists a spanning tree $\mathcal{T}$ of $\mathcal{L}$ according to lemma $\ref{lemma:spanning tree}$. The real numbers $(\text{arg}(M_{ij}))_{(i,j)\in \mathcal{T}}$ satisfy $\text{arg}(M_{ij}) + \text{arg}(M_{ji}) \equiv 0 ~[2\pi]$ due to $(\ref{subcase:Mb})$. Indeed, given $(i,j)\in \mathcal{T}$, there exists some $\mathcal{E} \in \Omega$ such that $(i,j) \in \mathcal{E}$ and $(j,i) \in \mathcal{E}$ so that $M_{ij} = v_i^\mathcal{E} (v_j^\mathcal{E})^H$ and $M_{ji} = v_j^\mathcal{E} (v_i^\mathcal{E})^H$. Moreover, the set of vertices $\mathcal{N}$ is finite. Lemma $\ref{lemma:equation over spanning tree}$ can thereby by applied to prove that there exists a vector $\textbf{v} \in \mathbb{C}^n$ such that : 
\begin{subnumcases}{}
\label{subcase:va}
|v_i| = \sqrt{M_{ii}} ~~,~~ \forall i \in \mathcal{N}  \\
\label{subcase:vb}
\text{arg}(v_i) - \text{arg}(v_j) \equiv \text{arg}(M_{ij})~[2\pi] ~~,~~ \forall (i,j) \in \mathcal{T}
\end{subnumcases}
Notice that for all $(i,j) \in \mathcal{L}$,
$$ \begin{array}{rcll}
|v_i v_j^H| & = & |v_i|~|v_j|                          & \\
            & = & \sqrt{M_{ii}} \sqrt{M_{jj}}          & \text{due to ($\ref{subcase:va}$)} \\
            & = & |v_i^\mathcal{E}|~|v_j^\mathcal{E}|  & \text{due to ($\ref{subcase:Ma}$) where}~(i,j) \in \mathcal{E}~\text{for some}~\mathcal{E} \in \Omega \\
            & = & |v_i^\mathcal{E}(v_j^\mathcal{E})^H| & \\
            & = & |M_{ij}|                             & \text{due to ($\ref{subcase:Mb}$)} \\
\end{array}
$$
Moreover, for all $(i,j) \in \mathcal{L}$,
$$ \begin{array}{rcll}
\text{arg}(v_i v_j^H) & \equiv & \text{arg}(v_i) - \text{arg}(v_j)~[2\pi] & \\
            		  & \equiv & \sum_{l,m} \text{arg}(v_l) - \text{arg}(v_m)~[2\pi] & \text{telescoping sum along a path in $\mathcal{T}$ from $i$ to $j$} \\
            		  & \equiv & \sum_{l,m} \text{arg}(M_{lm}) ~[2\pi] & \text{due to ($\ref{subcase:vb}$) and $(l,m) \in \mathcal{T}$} \\
            		  & \equiv & \sum_{\mathcal{C}} \sum_{(l,m)\in \vec{\mathcal{C}}} ~ \text{arg}(M_{lm}) ~[2\pi] & \text{sum over cycles $\mathcal{C} \in \Omega$ for which cycle defined}\\
            		  & & & \text{by path in $\mathcal{T}$ has nonzero coordinates}\\
            		  & \equiv & \sum_{\mathcal{C}} \sum_{(l,m)\in \vec{\mathcal{C}}} ~ \text{arg}(v_l^\mathcal{C} (v_m^\mathcal{C})^H) ~[2\pi] & \text{due to ($\ref{subcase:Mb}$)} \\
            		  & \equiv & \sum_{\mathcal{C}} \sum_{(l,m)\in \vec{\mathcal{C}}} ~ \text{arg}(v_l^\mathcal{C}) - \text{arg}(v_m^\mathcal{C}) ~[2\pi] & \text{telescoping sums over cycles are equal to zero} \\
            		  & \equiv & \sum_{l,m} \text{arg}(v_l^\mathcal{C}) - \text{arg}(v_m^\mathcal{C})~[2\pi] & \text{telescoping sum along path in $\mathcal{T}$ from $i$ to $j$} \\
            		  & \equiv & \text{arg}(v_i^\mathcal{C}) - \text{arg}(v_j^\mathcal{C})~[2\pi] & \\
            		  & \equiv & \text{arg}\{v_i^\mathcal{C} (v_j^\mathcal{C})^H\}~[2\pi] & \\
            		  & \equiv & \text{arg}(M_{ij})~[2\pi] & \text{due to ($\ref{subcase:Mb}$)}
\end{array}
$$
To sum up :
$$ \left\{
\begin{array}{rcll}
M_{ii} & = & v_i v_i^H ~,& \forall i \in \mathcal{N} \\
M_{ij} & = & v_i v_j^H ~,& \forall (i,j) \in \mathcal{L}
\end{array}
\right.
$$
Therefore there exists $N^\mathcal{L} \in \text{Ker}(\phi^\mathcal{L})$ such that :
$$ M = \textbf{v} \textbf{v}^H + N^\mathcal{L} \in \mathbb{H}^n_{+1} + \text{Ker}(\phi^\mathcal{L}) $$
\end{proof}
\begin{lemma}
\label{lemma:equation over cycle}
\normalfont
\textit{Let $\mathcal{C}$ denote a cycle of a finite and undirected graph $(\mathcal{N},\mathcal{L})$ and let $(\theta_{ij})_{(i,j)\in \mathcal{C}}$ denote some real numbers. Assume that $\theta_{ij} + \theta_{ji} \equiv 0 ~[2\pi]$ for all $(i,j)\in \mathcal{C}$ and that $\sum_{(i,j)\in \vec{\mathcal{C}}} \theta_{ij} \equiv 0~[2\pi]$. Then for each vertex $i$ of $\mathcal{C}$, there exists a real number $\theta_i$ such that:}
\begin{equation*}
\theta_i - \theta_j \equiv \theta_{ij}~[2\pi] ~~,~~ \forall (i,j) \in \mathcal{C}
\end{equation*}
\end{lemma}
\begin{proof}[Proof]
Simply define the set of all $\theta_i$'s for each vertex $i$ of $\mathcal{C}$ by choosing a random real number $\theta_{i_0}$ for some vertex $i_0$ of $\mathcal{C}$ and for all other vertices $i \in \mathcal{N}\setminus \{i_0\}$, choose $\theta_i \equiv \theta_{i_0} + \sum_{l,m} \theta_{lm} ~[2\pi]$ where the sum is taken over the path in $\vec{\mathcal{C}}$ linking $i$ to $i_0$ for some orientation of $\mathcal{C}$. This is possible because $(\mathcal{N},\mathcal{L})$ is a finite graph. Thus defined, the real numbers $\theta_i$ satisfy for all $(i,j) \in \mathcal{C}$:
$$
\begin{array}{rcll}
\theta_i - \theta_j & \equiv & \theta_{ij}~[2\pi] & \\
\end{array}
$$
Indeed, for all $(i,j) \in  \vec{\mathcal{C}}$ but the edge ending in $i_0$, all terms in the sums associated respectively to $\theta_i$ and $\theta_j$ cancel each other out except for $\theta_{ij}$. As for the edge $(i,i_0)$ ending in $i_0$, notice that $\theta_i \equiv \theta_{i_0} + \sum_{l,m} \theta_{lm} ~[2\pi]$ where the sum is taken over the path in $\vec{\mathcal{C}}$ linking $i$ to $i_0$ so that $\sum_{(l,m) \in \vec{\mathcal{C}}} \theta_{lm} \equiv \theta_{i_0i} + \sum_{l,m} \theta_{lm} ~[2\pi]$. It is assumed that $\sum_{(l,m) \in \vec{\mathcal{C}}} \theta_{lm} \equiv 0~[2\pi]$ thus $\sum_{l,m} \theta_{lm} \equiv -\theta_{i_0i}~[2\pi]$. It follows that $\theta_i \equiv \theta_{i_0} -\theta_{i_0i} ~[2\pi]$, ie $\theta_i - \theta_{i_0} \equiv \theta_{ii_0} ~[2\pi]$.
\end{proof}
\begin{proposition}
\label{prop:matrix completion over all edges}
\normalfont
\textit{The following statement holds:}
$$
M \in \mathbb{H}^n_{+1} + \text{Ker}(\phi^\mathcal{L}) ~~\Longleftrightarrow~~ M \in  \bigcap_{{(i,j)} \in \mathcal{L}} \mathbb{H}^n_{+1} + \text{Ker}(\phi^{(i,j)}) ~~\&~~  \sum_{(i,j)\in \vec{\mathcal{C}}} \text{arg}(M_{ij}) \equiv 0~[2\pi] ~\text{for all cycles}~ \mathcal{C} \in \Omega
$$
\end{proposition}
\begin{proof}[Proof]
($\subset$) Consider $M \in \mathbb{H}^n_{+1} + \text{Ker}(\phi^\mathcal{L})$. For all ${(i,j)} \in \mathcal{L}$, $\text{Ker}(\phi^\mathcal{L}) \subset \text{Ker}(\phi^{(i,j)})$ thus $M \in  \bigcap_{{(i,j)} \in \mathcal{L}} \mathbb{H}^n_{+1} + \text{Ker}(\phi^{(i,j)})$.\\\\ Consider a cycle $\mathcal{C} \in \Omega$. There exists $\textbf{v} \in \mathbb{C}^n$ such that for all $(i,j)\in \vec{\mathcal{C}}$, $M_{ij} = v_i v_j^H$. Therefore: 
$$\begin{array}{rcll}
\sum_{(i,j)\in \vec{\mathcal{C}}} \text{arg}(M_{ij}) & \equiv & \sum_{(i,j)\in \vec{\mathcal{C}}} \text{arg}(v_i v_j^H)~[2\pi] & \\
 & \equiv & \sum_{(i,j)\in \vec{\mathcal{C}}} \text{arg}(v_i) - \text{arg}(v_j)~[2\pi] & \text{telescoping sum} \\
 & \equiv & 0 ~[2\pi] & 
\end{array}
$$
($\supset$) Consider $M \in  \bigcap_{{(i,j)} \in \mathcal{L}} \mathbb{H}^n_{+1} + \text{Ker}(\phi^{(i,j)})$ such that $\sum_{(i,j)\in \vec{\mathcal{C}}} \text{arg}(M_{ij}) \equiv 0~[2\pi] ~\text{for all cycles}~ \mathcal{C} \in \Omega$. The following is true for all $(i,j) \in \mathcal{L}$ :
\begin{subnumcases}{}
\label{subcase:Ma cycle}
M_{kk} = v_k^{(i,j)} (v_k^{(i,j)})^H ~~,~~ \text{for $k = i,j$}\\
\label{subcase:Mb cycle}
M_{lm} = v_l^{(i,j)} (v_m^{(i,j)})^H ~~,~~ \text{for $(l,m) = (i,j),(j,i)$}
\end{subnumcases}
Consider a cycle $\mathcal{C} \in \Omega$. The real numbers $(\text{arg}(M_{ij}))_{(i,j)\in \mathcal{C}}$ satisfy $\text{arg}(M_{ij}) + \text{arg}(M_{ji}) \equiv 0 ~[2\pi]$ due to $(\ref{subcase:Mb cycle})$. Moreover, it is assumed that $\sum_{(i,j)\in \vec{\mathcal{C}}} \text{arg}(M_{ij}) \equiv 0~[2\pi]$. Lemma $\ref{lemma:equation over cycle}$ can therefore be used to show that there exists $\textbf{v} \in \mathbb{C}^n$ such that:
\begin{subnumcases}{}
\label{subcase:va cycle}
|v_i| = \sqrt{M_{ii}} ~~,~~ \text{for all vertices $i$ of $\mathcal{C}$}  \\
\label{subcase:vb cycle}
\text{arg}(v_i) - \text{arg}(v_j) \equiv \text{arg}(M_{ij})~[2\pi] ~~,~~ \forall (i,j) \in \mathcal{C}
\end{subnumcases}
Notice that for all $(i,j) \in \mathcal{C}$ :
$$ \begin{array}{rcll}
|v_i v_j^H| & = & |v_i|~|v_j|                          & \\
            & = & \sqrt{M_{ii}} \sqrt{M_{jj}}          & \text{due to ($\ref{subcase:va cycle}$)} \\
            & = & |v_i^{(i,j)}|~|v_j^{(i,j)}|  & \text{due to $(\ref{subcase:Ma cycle})$} \\
            & = & |v_i^{(i,j)}(v_j^{(i,j)})^H| & \\
            & = & |M_{ij}|                             & \text{due to ($\ref{subcase:Mb cycle}$)} \\
\end{array}
$$
Thus for all $\mathcal{C} \in \Omega$ : $$ M \in \mathbb{H}^n_{+1} + \text{Ker}(\phi^\mathcal{C}) $$
Moreover, it is assumed that :
$$M \in  \bigcap_{{(i,j)} \in \mathcal{L}} \mathbb{H}^n_{+1} + \text{Ker}(\phi^{(i,j)})$$
Therefore :
$$ M \in \bigcap_{\mathcal{E} \in \Omega} \mathbb{H}^n_{+1} + \text{Ker}(\phi^\mathcal{E}) ~~=~~ \mathbb{H}^n_{+1} + \text{Ker}(\phi^\mathcal{L}) $$
where the equality follows from proposition $\ref{prop:matrix completion over cycle basis}$.
\end{proof}

\section{Definitions of complex numbers}
\label{subsec:Definitions of complex numbers}


Complex numbers are a central aspect of the thesis. They are used to model an oscillatory phenomenon, namely alternating current. We now consider several definitions of complex numbers.


A complex number $x + \textbf{i}y$ can be thought of as the matrix:
$$ 
\left(
\begin{array}{cr}
x & -y \\
y &  x
\end{array}
\right).
$$
The additions and multiplications of complex numbers translate into additions and multiplications of real matrices. 
Concerning addition, we have
$$ 
\begin{array}{ccccc}
(a+ \textbf{i} b) & + & (c+ \textbf{i} d) & = & (a + c) + \textbf{i}(b+d) \\[.5em]
\left(
\begin{array}{cr}
a & -b \\
b &  a
\end{array}
\right)
&
+
&
\left(
\begin{array}{cr}
c & -d \\
d &  c
\end{array}
\right)
&
=
&
\left(
\begin{array}{cr}
a+c & -b-d \\
b+d &  a+c
\end{array}
\right)
\end{array}
$$
Concerning multiplication, we have:
$$ 
\begin{array}{ccccc}
(a+ \textbf{i} b) & \times & (c+ \textbf{i} d) & =  & ac - bd + \textbf{i}(ad + bc)\\[.5em]
\left(
\begin{array}{cr}
a & -b \\
b &  a
\end{array}
\right)
& 
\times
&
\left(
\begin{array}{cr}
c & -d \\
d &  c
\end{array}
\right)
&
=
&
\left(
\begin{array}{cr}
ac - bd & -ad - bc \\
ad + bc &  ac - bd
\end{array}
\right)
\end{array}
$$
In particular, we have
$$ 
\begin{array}{ccccc}
\textbf{i} & \times & \textbf{i} & =  & -1\\[.5em]
\left(
\begin{array}{cr}
0 & -1 \\
1 &  0
\end{array}
\right)
& 
\times
&
\left(
\begin{array}{cr}
0 & -1 \\
1 &  0
\end{array}
\right)
&
=
&
\left(
\begin{array}{rr}
-1 & 0 \\
0 &  -1
\end{array}
\right)
\end{array}
$$

As stated above, a complex number can be viewed as a real matrix of size $2 \times 2$. More generally, a complex square matrix can be viewed as a real matrix of double its size. This will become very helpful when we consider optimization over complex matrix variables in Chapter \ref{sec:Complex hierarchy for enhanced tractability}.

We note that complex numbers can also be defined using the Euclidian division of polynomials. 
The idea is to build a solution to the equation $x^2 = -1$ though it has no real solution. To do so, consider the ring of polynomials $\mathbb{R}[X]$ with one real indeterminate $X$. The remainder of $X^2$ when divided by $1+X^2$ is equal to $-1$, which is written $X^2 \equiv -1 ~ [1+X^2]$. More generally, given a polynomial $P \in \mathbb{R}[X]$, its division by $1+X^2$ has a remainder of the form $a+bX$ where $a$ and $b$ are some real numbers. Indeed, the degree of the remainder must be strictly less than the degree of $1+X^2$. 
Let $\text{cl}(P)$ denote the equivalence class modulo $1+X^2$ represented by $P$. The set of equivalence classes is thus equal to $\{ ~  \text{cl}(a+bX) ~|~ a,b \in \mathbb{R}~\}$. This set may be identified with the set of complex numbers because it is isomorphic to it. Indeed, consider two classes $\text{cl}(a+bX)$ and $\text{cl}(c+dX)$ with $a,b,c,d \in \mathbb{R}$. Concerning addition, we have
$$\text{cl}(a+bX) + \text{cl}(c+dX) = \text{cl} \boldsymbol{\left(\right.} ~ (a+c) + (b+d)X ~ \boldsymbol{\left.\right)}$$
since
$$ (a+bX) + (c+dX) \equiv (a+c) + (b+d)X ~ [1+X^2]. $$
Concerning multiplication, we have
$$\text{cl}(a+bX) \times \text{cl}(c+dX) = \text{cl} \boldsymbol{\left(\right.} ~ ac - bd + (ad + bc)X ~ \boldsymbol{\left.\right)}$$
since
$$
\begin{array}{cccl}
(a+bX)(c+dX) & \equiv & ac + bdX^2 + (ad + bc)X & [1+X^2] \\
 & \equiv & ac - bd\hphantom{X^2} + (ad + bc)X & [1+X^2] \\[.5em]
 & & (\text{because}~ X^2 \equiv -1 ~ [1+X^2]). &
\end{array}
$$


\chapter{Lasserre hierarchy for small-scale networks}
\label{sec:Lasserre hierarchy for small-scale networks}

Finding a global solution to the optimal power flow (OPF) problem is
difficult due to its nonconvexity. A convex relaxation in the form of
semidefinite optimization (SDP) had attracted much attention when I started my Ph.D. Indeed, it
yielded a global solution in several practical cases. However, it did
not in all cases, and such cases had been documented in several
publications. Here we present another SDP method known as the
moment-sos (sum of squares) approach, which generates a sequence that
converges towards a global solution to the OPF problem at the cost of
higher runtime. Our finding is that in the small examples where the
previously studied SDP method fails, this approach finds the global
solution. The higher cost in runtime is due to an increase in the matrix
size of the SDP problem, which can vary from one instance to another.
Numerical experiment shows that the size is very often a quadratic
function of the number of buses in the network, whereas it is a linear
function of the number of buses in the case of the previously studied
SDP method. The material in this chapter is based on the publication: 
\\\\
\noindent {\scshape C. Josz, J. Maeght, P. Panciatici, and J.C. Gilbert}, \textit{Application of the Moment-SOS
Approach to Global Optimization of the OPF Problem}, Institute of Electrical and Electronics Engineers, Transactions on Power Systems, 30, pp. 463–470, May 2014.  \href{http://dx.doi.org/10.1109/TPWRS.2014.2320819}{[doi]} \href{http://arxiv.org/pdf/1311.6370v1.pdf}{[preprint]}

\section{Introduction}

The optimal power flow can be cast as a nonlinear optimization problem which is NP-hard, as was shown
in \cite{lavaei-low-2012}. So far, the various methods
\cite{huneault-1991, pandya-2008} that have been
investigated to solve the OPF can only guarantee local optimality, due
to the nonconvexity of the problem. Recent progress suggests that it may
be possible to design a method, based on semidefinite optimization (SDP),
that yields global optimality rapidly.

The first attempt to use SDP to solve the OPF problem was made by Bai et
al. \cite{bai-fujisawa-wang-wei-2008} in 2008. In
\cite{lavaei-low-2012}, Lavaei and Low show that the OPF can be written
as an SDP problem, with an additional constraint imposing that the rank
of the matrix variable must not exceed~1.
They discard the rank constraint, as it is done in Shor's relaxation
\cite{shor-1987b}, a procedure which applies to quadratically
constrained quadratic problems (see~\cite{sturm-zhang.s-2003,
luo-2010} and the references therein).
They also accept quartic terms that appear in some formulations of the
OPF, transforming them by Schur's complement.
Their finding is that
for all IEEE benchmark networks, namely the 9, 14, 30, 57, 118, and
300-bus systems, the rank constraint is satisfied if a small resistance
is added in the lines of the network that have zero resistance. Such a
modification to the network is acceptable because in reality, resistance
is never equal to zero.

There are cases when the rank constraint is not satisfied and a global
solution can thus not be found. Lesieutre et al.
\cite{lesieutre_molzahn_borden_demarco-allerton2011} illustrate this with a
practical 3-bus cyclic network. Gopalakrishnan et al.
\cite{biegler-gopalakrishnan-nikovski-raghunathan-2012} find yet more
examples by modifying the IEEE benchmark networks. Bukhsh et al.
\cite{bukhsh2013} provide a 2-bus and a 5-bus
example. In addition, they document the local solutions to the OPF in
many of the above-mentioned examples where the rank constraint is not
satisfied \cite{bukhsh-grothey-mckinnon-trodden-2013b}.

Several papers propose ways of handling cases when the rank constraint
is not satisfied. Gopalakrishnan et al.
\cite{biegler-gopalakrishnan-nikovski-raghunathan-2012} propose a branch
and reduce algorithm. It is based on the fact that the rank relaxation
gives a lower bound of the optimal value of the OPF. But according to
the authors, using the classical Lagrangian
dual to evaluate a lower bound is about as
efficient. Sojoudi and Lavaei \cite{lavaei-sojoudi-2012} prove that if
one could add controllable phase-shifting transformers to every loop in
the network and if the objective is an increasing function of generated
active power, then the rank constraint is satisfied. Though numerical
experiments confirm this \cite{farivar-low-2013}, such a modification to
the network is not realistic, as opposed to the one mentioned earlier.

Cases where the rank constraint holds have been identified. Authors of
\cite{bose-chandy-gayme-low-2011, tse-zhang-2011, lavaei-sojoudi-2012b}
prove that the rank constraint is satisfied if the graph of the network
is acyclic and if load over-satisfaction is allowed. This is typical of
distribution networks but it is not true of transmission networks.

This paper examines the applicability of the moment-sos (sum of
squares) approach to the OPF. This approach~\cite{lasserre-2000,
parrilo-2000b, lasserre-2001} aims at finding global solutions to
polynomial optimization problems, of which the OPF is a particular
instance. The approach can be viewed as an extension of the SDP method
of \cite{lavaei-low-2012}. Indeed, it proposes a sequence of SDP
relaxations whose first element is the rank relaxation in many cases.
The subsequent relaxations of the sequence become more and more
accurate. When the rank relaxation fails, it is therefore natural to see
whether the second order relaxation provides the global minimum, then
the third, and so on.

The limit to this approach is that the complexity of the relaxations
rapidly increases. The matrix size of the SDP relaxation of order $d$ is
roughly equal to the number of buses in the network to the power $d$.
Surprisingly, in the 2, 3, and 5-bus systems found in
\cite{lesieutre_molzahn_borden_demarco-allerton2011,
bukhsh2013} where the rank relaxation fails,
the second order relaxation nearly always finds the global solution. 

Below, section \ref{subsec:Polynomial optimization formulation} shows that the OPF can be viewed as a
polynomial optimization problem. The moment-sos approach which aims at solving such problems is described in section
\ref{subsec:moment-sos-approach}. In section \ref{subsec:Numerical results2},
numerical results show that this approach successfully
finds the global solution to the 2, 3, and 5-bus systems mentioned
earlier. Conclusions are given in section \ref{subsec:Conclusion2}.

\section{Polynomial optimization formulation}
\label{subsec:Polynomial optimization formulation}

In order to obtain a polynomial formulation of the OPF, we proceed in 3
steps. First, we write a formulation in complex numbers. Second, we use
it to write a formulation in real numbers. Third, we use the real
formulation to write a polynomial formulation.

Let $\textbf{a}^H$ and $A^H$ denote the conjugate transpose of a complex
vector $\textbf{a}$ and of a complex matrix $A$ respectively. It can be
deduced from~\cite{lavaei-sojoudi-2012} that there exist finite sets
$\mathcal{I}$ and $\mathcal{J}$, Hermitian matrices
$(A_k)_{k\in\mathcal{G}}$ of size $n$, complex matrices
$(B_i)_{i\in\mathcal{I}}$ and $(C_i)_{i\in\mathcal{J}}$ of size $n$, and
complex numbers $(b_i)_{i\in\mathcal{I}}$ and $(c_i)_{i\in\mathcal{J}}$
such that the OPF can be written as 
\begin{equation}
\min_{\textbf{v}\in \mathbb{C}^n} ~ \sum_{k \in \mathcal{G}} c_{k2} (\textbf{v}^H A_k \textbf{v})^2 + c_{k1}\textbf{v}^H A_k \textbf{v}+ c_{k0},
\label{eq:complexobj}
\end{equation}
subject to
\begin{gather}
\forall\, i \in \mathcal{I},~~~ \textbf{v}^H B_i \textbf{v}~ \leq ~ b_i,
\label{eq:complexcon1}
\\
\forall\, i \in \mathcal{J},~~~ |\textbf{v}^H C_i \textbf{v}| ~ \leq ~ c_i.
\label{eq:complexcon2}
\end{gather}
Constraints
\eqref{eq:complexcon2} correspond to bounds on apparent power flow. Constraints
\eqref{eq:complexcon1} correspond to all other constraints. 

Let $\textbf{x} \in \mathbb{R}^{2n}$ denote $[\text{Re}(\textbf{v})^T ~ \text{Im}(\textbf{v})^T ]^T$
as is done in \cite{lavaei-low-2012}. In order to transform the complex
formulation of the OPF \eqref{eq:complexobj}-\eqref{eq:complexcon2} into
a real number formulation, observe that $\textbf{v}^H M \textbf{v}=(\textbf{x}^T M^\text{re}
\textbf{x})+\text{j}(\textbf{x}^T M^\text{im}\textbf{x})$, where the superscript $^T$ denotes
transposition,
\begin{align*}
M^\text{re}&:=
\begin{bmatrix}
\text{Re}(M) &            -\text{Im}(M) \\
\text{Im}(M) & \hphantom{-}\text{Re}(M)
\end{bmatrix},
\quad\text{and}
\\
M^\text{im}&:=
\begin{bmatrix}
\hphantom{-}\text{Im}(M) & \text{Re}(M) \\
           -\text{Re}(M) & \text{Im}(M)
\end{bmatrix}.
\end{align*}

Then \eqref{eq:complexobj}-\eqref{eq:complexcon2} becomes
\begin{equation}
\label{eq:real objective}
\min_{\textbf{x} \in \mathbb{R}^{2n}} ~ \sum_{k \in \mathcal{G}} c_{k2} (\textbf{x}^T A_k^\text{re} \textbf{x})^2 + c_{k1}\textbf{x}^T A_k^\text{re} \textbf{x} + c_{k0},
\end{equation}
subject to
\begin{gather}
\forall\, i \in \mathcal{I},~~~ \textbf{x}^T B_i^\text{re} \textbf{x} ~ \leq ~ \text{Re}(b_i),
\label{eq:real constraint1}
\\
\forall\, i \in \mathcal{I},~~~ \textbf{x}^T B_i^\text{im} \textbf{x} ~ \leq ~ \text{Im}(b_i),
\label{eq:real constraint2}
\\
\forall\, i \in \mathcal{J},~~~ (\textbf{x}^T C_i^\text{re} \textbf{x})^2 + (\textbf{x}^T C_i^\text{im} \textbf{x})^2 ~ \leq ~ c_i^2.
\label{eq:real constraint3}
\end{gather}

We recall that a polynomial is a function $p: \textbf{x}\in\mathbb{R}^n\mapsto
\sum_{\alpha\in \mathcal{A}}p_\alpha \textbf{x}^\alpha$, where $\mathcal{A}\subset\mathbb{N}^n$ is a finite
set of integer multi-indices, the coefficients $p_\alpha$ are real
numbers, and $\textbf{x}^\alpha$ is the monomial $x_1^{\alpha_1}\cdots
x_n^{\alpha_n}$. Its degree, denoted $\deg p$, is the largest
$|\alpha|=\sum_{i=1}^n\alpha_i$ associated with a nonzero~$p_\alpha$.

The formulation of the OPF in real numbers \eqref{eq:real
objective}-\eqref{eq:real constraint3} is said to be a polynomial
optimization problem since the functions that define it are
polynomials.
Indeed, the objective \eqref{eq:real objective} is a polynomial of
$\textbf{x}\in\mathbb{R}^n$ of degree 4, the constraints \eqref{eq:real
constraint1}-(\ref{eq:real constraint2}) are polynomials of $\textbf{x}$ of
degree~2, and the constraints \eqref{eq:real constraint3} are
polynomials of $\textbf{x}$ of degree~4.

Formulation \eqref{eq:real
objective}-\eqref{eq:real constraint3} will however not be used below because it has
infinitely many global solutions. Indeed, formulation
\eqref{eq:complexobj}-\eqref{eq:complexcon2} from which it derives is
invariant under the change of variables $\textbf{v}\to\textbf{v} e^{\text{j}\theta}$
where $\theta \in \mathbb{R}$. This invariance property transfers to \eqref{eq:real
objective}-\eqref{eq:real constraint3}. An optimization
problem with non isolated solutions is generally more difficult to solve
than one with a unique
solution~\cite{bonnans-gilbert-lemarechal-sagastizabal-2006}. This
feature manifests itself in some properties of the moment-sos approach
described in section~\ref{subsec:moment-sos-approach}. For this reason, we
choose to arbitrarily set the voltage phase at bus $n$ to zero. Bearing
in mind that $v_n^\text{min} \geq 0$, this can be done by replacing
voltage constraint \eqref{eq:volt con} at bus~$n$ by \eqref{eq:phase
zero}:
\begin{gather}
(v_n^\text{min})^2 \leq x_{n}^2 + x_{2n}^2 \leq (v_n^\text{max})^2,
\label{eq:volt con} \\
x_{2n} = 0 ~~\text{and}~~ v_n^\text{min} \leq x_{n} \leq v_n^\text{max}.
\label{eq:phase zero}
\end{gather}

In light of \eqref{eq:phase zero}, a polynomial optimization problem
where there are $2n-1$ variables instead of $2n$ variables can be
formulated. More precisely, the OPF can be cast as the following polynomial
optimization problem

\textbf{PolyOPF:}
\begin{equation}
\min_{\textbf{x} \in \mathbb{R}^{2n-1}} ~ f_0(\textbf{x}) := \sum_{\alpha} f_{0,\alpha} \textbf{x}^\alpha,
\label{eq:obj}
\end{equation}
subject to
\begin{equation}
%
\forall\, i=1, \ldots, m , ~~~ f_i(\textbf{x}) := \sum_{\alpha} f_{i,\alpha} \textbf{x}^\alpha \geq 0,
\label{eq:con}
\end{equation}
where $m$ is an integer, $f_{i,\alpha}$ denotes the real coefficients of
the polynomial functions $f_i$, and summations take place
over~$\mathbb{N}^{2n-1}$. The summations are nevertheless finite because
only a finite number of coefficients are nonzero.

\section{Moment-sos approach}
\label{subsec:moment-sos-approach}

We first review some theoretical aspects of the moment-sos approach (a
nice short account can be found in~\cite{anjos-lasserre-2012b}, and more
in~\cite{lasserre_book, blekherman-parrilo-thomas-2013}). Next, we
present a set of relaxations of PolyOPF obtained by this method and
illustrate it on a simple example. Finally, we emphasize the
relationship between the moment-sos approach and the rank relaxation of
\cite{lavaei-low-2012}.

The moment-sos approach has been designed to find global solutions to
polynomial optimization problems. It is grounded on deep
results from real algebraic geometry. The term \textit{moment-sos}
derives from the fact that the approach has two dual aspects: the moment
and the sum of squares approaches. 
Both approaches are dual of one another in the
sense of Lagrangian duality~\cite{rockafellar-1974b}. Below, we focus on
the moment approach because it leads to SDP problems that have a close
link with the previously studied SDP method
in~\cite{lavaei-low-2012}.

Let $\textbf{K}$ be a subset of $\mathbb{R}^{2n-1}$. The
moment approach rests on the surprising (though easy to prove) fact that
the problem $\min \{ f_0(\textbf{x})$: $\textbf{x}\in \textbf{K} \}$ is equivalent to the
convex\ optimization problem in $\mu$:
\begin{equation}
\label{lasserre-measure-pbl}
\min_{ \substack{\mu ~\text{positive measure on \textbf{K}} \\ \int d \mu=1} } \int f_0 d \mu.
\end{equation}
Although the latter problem has a simple structure, it cannot be solved
directly, since its unknown $\mu$ is an infinite dimensional object.
Nevertheless, the realized transformation suggests that the initial
difficult global optimization problem can be structurally simplified by
judiciously expressing it on a space of larger dimension. 
The moment-sos approach goes along this way by introducing a hierarchy
of more and more accurate approximations of problem
\eqref{lasserre-measure-pbl}, hence \eqref{eq:obj}-\eqref{eq:con},
defined on spaces of larger and larger dimension.

When~$f_0$ is a polynomial and $\textbf{K}:=\{x\in\mathbb{R}^{2n-1}$: $f_i(\textbf{x})\geq0$, for
$i=1,\ldots, m\}$ is defined by polynomials $f_i$ like in PolyOPF, it
becomes natural to approximate the measure~$\mu$ by a finite number of
its moments. The \textit{moment} of $\mu$,
associated with  
$\alpha \in \mathbb{N}^{2n-1}$, is the real number $y_\alpha:=
\int \textbf{x}^\alpha\,d\mu$. Then, when~$f_0$ is the polynomial in \eqref{eq:obj},
the objective of \eqref{lasserre-measure-pbl} becomes 
$\int f_0 d\mu =\int (\sum_{\alpha}f_{0,\alpha} \textbf{x}^\alpha) d\mu
=\sum_{\alpha}f_{0,\alpha} \int \textbf{x}^\alpha d\mu
=\sum_{\alpha}f_{0,\alpha} y_\alpha$, 
whose linearity in the new
unknown $y$ is transparent.
The constraint $\int d\mu=1$ is also readily transformed into $y_0=1$. 
In contrast, expressing which are the vectors $y$ that are
moments of a positive measure $\mu$ on~$\textbf{K}$ (the other constraint in
\eqref{lasserre-measure-pbl}) is a much more difficult task
known as the \textit{moment problem}, which has been studied for over a
century~\cite{putinar-schmudgen-2008}.
It is that constraint that is approximated in the moment-sos approach,
with more and more accuracy in spaces of higher and higher
dimension.

The sum of squares approach is dual to the moment approach in the
sense of Lagrangian duality~\cite{rockafellar-1974b}. It relies on the
fact that minimizing a function $f_0$ over a set $\textbf{K}$ is equivalent to
maximizing a real number $\lambda$ under the constraints
$f_0(\textbf{x})-\lambda \geq 0$ for all $\textbf{x} \in \textbf{K}$. These trivial \textit{linear}
constraints are intractable because there is an \textit{infinite} number
of them. In the case of polynomial optimization, one recovers the
problem of finding certificates ensuring the positivity of the
polynomial $f_0-\lambda$ on the semi-algebraic set~$\textbf{K}$, which involves
sums of squares of polynomials~\cite{marshall-2008}. Relaxations consist
in imposing degree bounds on these sos polynomials.

Lasserre \cite{lasserre-2001} proposes a sequence of relaxations for any
polynomial optimization problem like PolyOPF that grow better in
accuracy and bigger in size when the order~$d$ of the relaxation
increases. Here and below, $d$ is an integer larger than or equal to
each $v_i:=\lceil(\deg f_i)/2\rceil$ for all $i=0,\ldots,m$ (we
have denote by $\lceil\cdot\rceil$ the ceiling operator).

Let $Z \succcurlyeq 0$ denote that $Z$ is a symmetric positive semidefinite
matrix. Define $\mathbb{N}^p_q:=\{ \alpha \in \mathbb{N}^p: |\alpha| \leq q \}$, whose
cardinality is $|\mathbb{N}^p_q|=\left( \begin{array}{c} p+q \\ q \end{array} \right):=(p+q)!/(p!\,q!)$, and denote
by $(z_{\alpha,\beta})_{\alpha,\beta \in \mathbb{N}^p_q}$ a matrix indexed by
the elements of~$\mathbb{N}^p_q$.

\textbf{Relaxation of order d:}
\begin{equation}
\min_{(y_\alpha)_{\alpha \in \mathbb{N}_{2d}^{2n-1}}} ~
\sum_{\alpha} f_{0,\alpha} y_{\alpha},
\label{eq:objd}
\end{equation}
subject to
\begin{gather}
y_0 = 1,
\label{eq:con1d}
\displaybreak[0]\\
( y_{\alpha+\beta} )_{\alpha,\beta \in \mathbb{N}_{d}^{2n-1}} \succcurlyeq 0,
\label{eq:con2d}
\displaybreak[0]\\
\forall\, i=1, \ldots, m , ~~ \sum_{\gamma} f_{i,\gamma} \left(
y_{\alpha+\beta+\gamma} \right)_{\alpha,\beta \in
\mathbb{N}^{2n-1}_{d-v_i}} \succcurlyeq 0.
\label{eq:con3d}
\end{gather}
We have already discussed the origin of \eqref{eq:objd}-\eqref{eq:con1d}
in the above SDP problem, while \eqref{eq:con2d}-\eqref{eq:con3d} are
necessary conditions to ensure that $y$ is formed of moments of some
positive measure on $\textbf{K}$. When $d$
increases, these problems form a \textit{hierarchy of semidefinite
relaxations}, called that way because the objective \eqref{eq:objd} is
not affected and the feasible set is reduced as the size of the matrices in \eqref{eq:con2d}-\eqref{eq:con3d} increases. These properties show that
the optimal value of problem \eqref{eq:objd}-\eqref{eq:con3d} increases
with $d$ and remains bounded by the optimal value of
\eqref{eq:obj}-\eqref{eq:con}.

For the method to give better results, a ball constraint $\|\textbf{x}\|^2\leq M$ must be added
according to the technical assumption~1.1 in \cite{anjos-lasserre-2012b}.
For the OPF problem, this can be done easily by setting~$M$
to $\sum_{k \in \mathcal{N}} (v_k^\text{max})^2$ without modifying the
problem. The following two properties hold in this case~\cite[theorem
1.12]{anjos-lasserre-2012b}:
 \begin{enumerate}
 \item
 the optimal values of the hierarchy of semidefinite relaxations
 increasingly converge toward the optimal value of
 PolyOPF,
 \item
 let $\textbf{y}^d$ denote a global solution to the relaxation of order~$d$
 and $(\textbf{e}^i)_{1\leq i\leq 2n-1}$ denotes the canonical basis of~$\mathbb{N}^{2n-1}$; if PolyOPF has a unique global
 solution, then $(\textbf{y}^d_{\textbf{e}^i})_{1\leq i \leq 2n-1}$
 converges towards the global solution to PolyOPF as $d$ tends to
 $+\infty$.
 \end{enumerate}

The largest matrix size of the moment relaxation
appears in \eqref{eq:con1d} and has the value $|\mathbb{N}_{d}^{2n-1}| =
\left( \begin{array}{c} 2n-1+d \\ d \end{array} \right)$,
where $n$ is the number of buses. For a fixed $d$, matrix size is
therefore equal to $O(n^d)$. This makes high order relaxations too large
to compute with currently available SDP software packages. Consequently,
the success of the moment-sos approach relies wholly upon its ability to
find a global solution with a low order relaxation, for which there is
no guarantee. Note that the global solution is found by a finite order
relaxation under conditions that include the convexity of the
problem~\cite{lasserre-2008b} (not the case of PolyOPF though) or the
positive definiteness of the Hessian of the Lagrangian at the saddle
points of the Lagrangian~\cite{deklerk-laurent-2011b} (open question in
the case of PolyOPF).

\subsection*{Moment-sos relaxations and rank relaxation}

When the polynomials $f_i$ defining PolyOPF are quadratic, the first
order ($d=1$) relaxation \eqref{eq:objd}-\eqref{eq:con3d} is equivalent
to Shor's relaxation~\cite{lasserre-2001b}. To make the link with the
rank relaxation of~\cite{lavaei-low-2012}, consider now the case when
the varying part of the $f_i$'s are quadratic and \textit{homogeneous} like
in~\cite{lavaei-low-2012}, that is $f_i(\textbf{x})=\textbf{x}^T A_i \textbf{x}-a_i $ for
all~$i=0,\ldots,m$, with symmetric matrices~$A_i$ and
scalars~$a_i$. Then introducing the
vector $\textbf{s}$ and the matrix $Y$ defined by $\textbf{s}_i=y_{\textbf{e}^i}$ and
$Y_{kl}=y_{\textbf{e}^k+\textbf{e}^l}$,
reads
\begin{equation}
\label{qp-order1relation-obj}
\textstyle
\min_{(\textbf{s},Y)}\;\textbf{trace}(A_0Y) - a_0,
\end{equation}
subject to
\begin{equation}
\label{qp-order1relation-con}
\begin{bmatrix}
1   & \textbf{s}^T \\
\textbf{s} & Y
\end{bmatrix}
\succcurlyeq 0
\quad\mbox{and}\quad
\textbf{trace}(A_iY)\geq a_i ~~(\forall\,i=1,\ldots,m).
\end{equation}
Using Schur's complement, the positive semidefiniteness condition
in \eqref{qp-order1relation-con} is equivalent to $Y-\textbf{s} \textbf{s}^T \succcurlyeq 0$.
Since $\textbf{s}$ does not intervene elsewhere in
\eqref{qp-order1relation-obj}-\eqref{qp-order1relation-con}, it can be
eliminated and the constraints of the problem can be replaced by
\begin{equation}
\label{qp-order1relation-con'}
Y \succcurlyeq 0
\quad\mbox{and}\quad
\text{trace}(A_iY)\geq a_i ~~(\forall\,i=1,\ldots,m).
\end{equation}
The pair made of \eqref{qp-order1relation-obj} and
\eqref{qp-order1relation-con'} is the rank relaxation of
\cite{lavaei-low-2012}.
We have just shown that the equivalence between that the SDP
relaxation of \cite{lavaei-low-2012} and to the first-order moment
relaxation holds when the varying part of the $f_i$'s are quadratic and
homogeneous. For the OPF problem, this certainly occurs when
\begin{enumerate}
\item the objective of the OPF is an affine function of active power,
\item there are no constraints on apparent power flow,
\item \eqref{eq:volt con} is not replaced by \eqref{eq:phase zero}.
%
\end{enumerate}
Point~1 ensures that the objective is quadratic and has a
homogeneous varying part. Points~2 and~3 guarantee the same property for
the constraint functions.

\section{Numerical results}
\label{subsec:Numerical results2}

We present numerical results for the moment-sos approach applied to
instances of the OPF for which the rank relaxation method of
\cite{lavaei-low-2012} fails to find the global solution. We focus on
the WB2 2-bus system, LMBM3 3-bus system, and the WB5 5-bus system that
are described in \cite{bukhsh2013}. Note that
LMBM3 is also found in \cite{lesieutre_molzahn_borden_demarco-allerton2011}. For
each of the three systems, the authors of
\cite{bukhsh2013} modify a bound in the data
and specify a range for which the rank relaxation fails. We consider 10
values uniformly distributed in the range in order to verify that the
rank relaxation fails and to assess the moment-sos approach.
We proceed in accordance with the discussion of section
\ref{subsec:moment-sos-approach} by adding the
redundant ball constraint. Surprisingly, the second order relaxation
whose greatest matrix size is equal to $(2n+1)n$ nearly always finds the
global solution. 

The materials used are:
\begin{itemize}
\item
Data of WB2, LMBM3, WB5 systems available online
\cite{bukhsh-grothey-mckinnon-trodden-2013b},
\item
Intel\textregistered\ Xeon\texttrademark\ MP CPU 2.70 GHz 7.00 Go RAM,
\item
MATLAB version 7.7 2008b,
\item
MATLAB-package MATPOWER version 3.2
\cite{matpower},
\item
SeDuMi 1.02 \cite{sturm-1999} with tolerance parameter \texttt{pars.eps}
set to $10^{-12}$ for all computations,
\item
MATLAB-based toolbox YALMIP~\cite{yalmip} to compute Optimization 4
(Dual OPF) in \cite{lavaei-low-2012} that yields the solution to the
rank relaxation,
\item
MATLAB-package GloptiPoly version 3.6.1
\cite{henrion-lasserre-loefberg-2009} to compute solutions to a
hierarchy of SDP relaxations \eqref{eq:objd}-\eqref{eq:con3d}.
\end{itemize}

The same precision is used as in the solutions of the test archives
\cite{bukhsh-grothey-mckinnon-trodden-2013b}. In other words, results
are precise up to $10^{-2}$ p.u. for voltage phase, $10^{-2}$ degree for
angles, $10^{-2}$ MW for active power, $10^{-2}$ MVA for reactive power,
and cent per hour for costs. Computation time is several seconds. 

GloptiPoly can guarantee that it has found a global solution to a
polynomial optimization problem, up to a given precision. This is
certainly the case when it finds a feasible point $\textbf{x}$ giving to the
objective a value sufficiently close to the optimal value of the
relaxation.

\subsection*{2-bus network: WB2}

Authors of \cite{bukhsh2013} observe that in
the WB2 2-bus system of figure \ref{fig:WB2 2-bus system}, the rank
constraint is not satisfied in the rank relaxation method of
\cite{lavaei-low-2012} when $0.976~\text{p.u.} < v_2^\text{max} <
1.035~\text{p.u.}$ In table \ref{tab:WB2}, the first column is made up
of 10 points in that range that are uniformly distributed. The second
column contains the lowest order of the relaxations that yield a global
solution. The optimal value of the relaxation of that order is written
in the third column. The fourth column contains the optimal value of the
rank relaxation (it is put between parentheses when the relaxation is
inexact).

\begin{figure}[H]
  \centering
    \includegraphics[width=.6\textwidth]{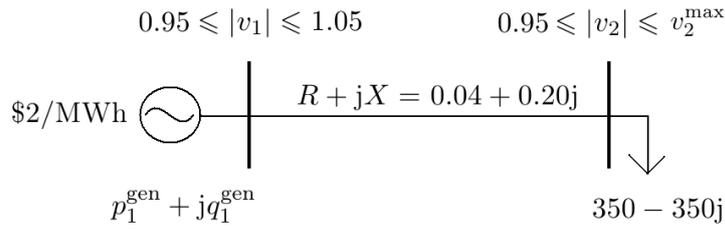}
  \caption{WB2 2-bus system}
  \label{fig:WB2 2-bus system}
\end{figure}

\begin{table}[H]
\centering
\caption{Order of hierarchy needed to reach global solution to WB2 when rank relaxation fails}
\begin{tabular}{c|c|c|c}
$v_2^\text{max}$ & relax. & optimal & rank relax. \\
(p.u.) & order & value (\$/h) &  value (\$/h)\\
\hline
0.976 & 2 & 905.76 &  905.76  \\
0.983 & 2 & 905.73 & (903.12) \\
0.989 & 2 & 905.73 & (900.84) \\
0.996 & 2 & 905.73 & (898.17) \\
1.002 & 2 & 905.73 & (895.86) \\
1.009 & 2 & 905.73 & (893.16) \\
1.015 & 2 & 905.73 & (890.82) \\
1.022 & 3 & 905.73 & (888.08) \\
1.028 & 3 & 905.73 & (885.71) \\
1.035 & 2 & 882.97 &  882.97  \\
\end{tabular}
\label{tab:WB2}
\end{table}

The hierarchy of SDP relaxations is defined for $d\geq 1$ because the
objective is an affine function and there are no apparent flow
constraints. Let's explain how it works in the case where
$v_2^\text{max} = 1.022$~p.u. The optimal value of the first order
relaxation is 861.51~\$/h, that of the second order relaxation is
901.38~\$/h, and that of the third is 905.73~\$/h. This is coherent with
point~1 of the discussion of section \ref{subsec:moment-sos-approach} that claims that the optimal values increase
with~$d$. Computing higher orders is not necessary because GloptiPoly
numerically proves global optimality for the third order.

Notice that for $v_2^\text{max} = 1.022$~p.u. the value of the rank
relaxation found in table \ref{tab:WB2} (888.08~\$/h) is different from
the value of the first order relaxation (861.51~\$/h). If we run
GloptiPoly with \eqref{eq:volt con} instead of \eqref{eq:phase zero},
the optimal value of the first order relaxation is equal 888.08~\$/h as
expected according to section \ref{subsec:moment-sos-approach}.

For $v_2^\text{max} = 0.976~\text{p.u.}$ and $v_2^\text{max} =
1.035~\text{p.u.}$ (see
the first and last rows of table~\ref{tab:WB2}), the rank constraint is satisfied in the rank
relaxation method so its optimal value is equal to the one of the successful
moment-sos method. In between those
values, the rank constraint is not satisfied since the optimal value is
less than the optimal value of the OPF. Notice the correlation
between the results of table \ref{tab:WB2} and the upper half of figure
8 in \cite{bukhsh2013}. Indeed, the figure
shows the optimal value of the OPF is constant whereas the optimal value
of the rank relaxation decreases in a linear fashion when
$0.976~\text{p.u.} < v_2^\text{max} < 1.035~\text{p.u.}$

Surprisingly and encouragingly, according to the second column of table
\ref{tab:WB2}, the second order moment-sos relaxation finds the global
solution in 8 out of 10 times, and the third order relaxation always
find the global solution.

\textit{Remark:} The fact that the rank constraint is not satisfied for
the WB2 2-bus system of \cite{bukhsh2013}
seems in contradiction with the results of papers
\cite{bose-chandy-gayme-low-2011, tse-zhang-2011, lavaei-sojoudi-2012b}.
Indeed, the authors of the papers state that the rank is less than or
equal to 1 if the graph of the network is acyclic and if load
over-satisfaction is allowed. However, load over-satisfaction is not
allowed in this network. For example, for $v_2^\text{max} = 1.022$~p.u.,
adding 1 MW of load induces the optimal value to go down from
905.73~\$/h to 890.19~\$/h. One of the sufficient conditions in
\cite{bose-chandy-gayme-low-2015} for the rank is less than or equal to
1 relies on the existence of a strictly feasible point. It is not the
case here because equality constraints must be enforced in the power
balance equation.

\subsection*{3-bus network: LMBM3}

We observe that in the LMBM3 3-bus system of figure \ref{fig:LMBM3 3-bus
system}, the rank constraint is not satisfied in the rank relaxation
method of \cite{lavaei-low-2012} when $28.35~\text{MVA} \leq
s_{23}^\text{max} = s_{32}^\text{max} < 53.60~\text{MVA}$. Below
$28.35~\text{MVA}$, no solutions can be found by the OPF solver
\texttt{runopf} in MATPOWER nor by the hierarchy of SDP relaxations. At
$53.60~\text{MVA}$, the rank constraint is satisfied in the rank
relaxation method so its optimal value is equal to the optimal value of
the OPF found by the second order relaxation; see to the last row of
table \ref{tab:LMBM3}.

\begin{figure}[H]
  \centering
    \includegraphics[width=.5\textwidth]{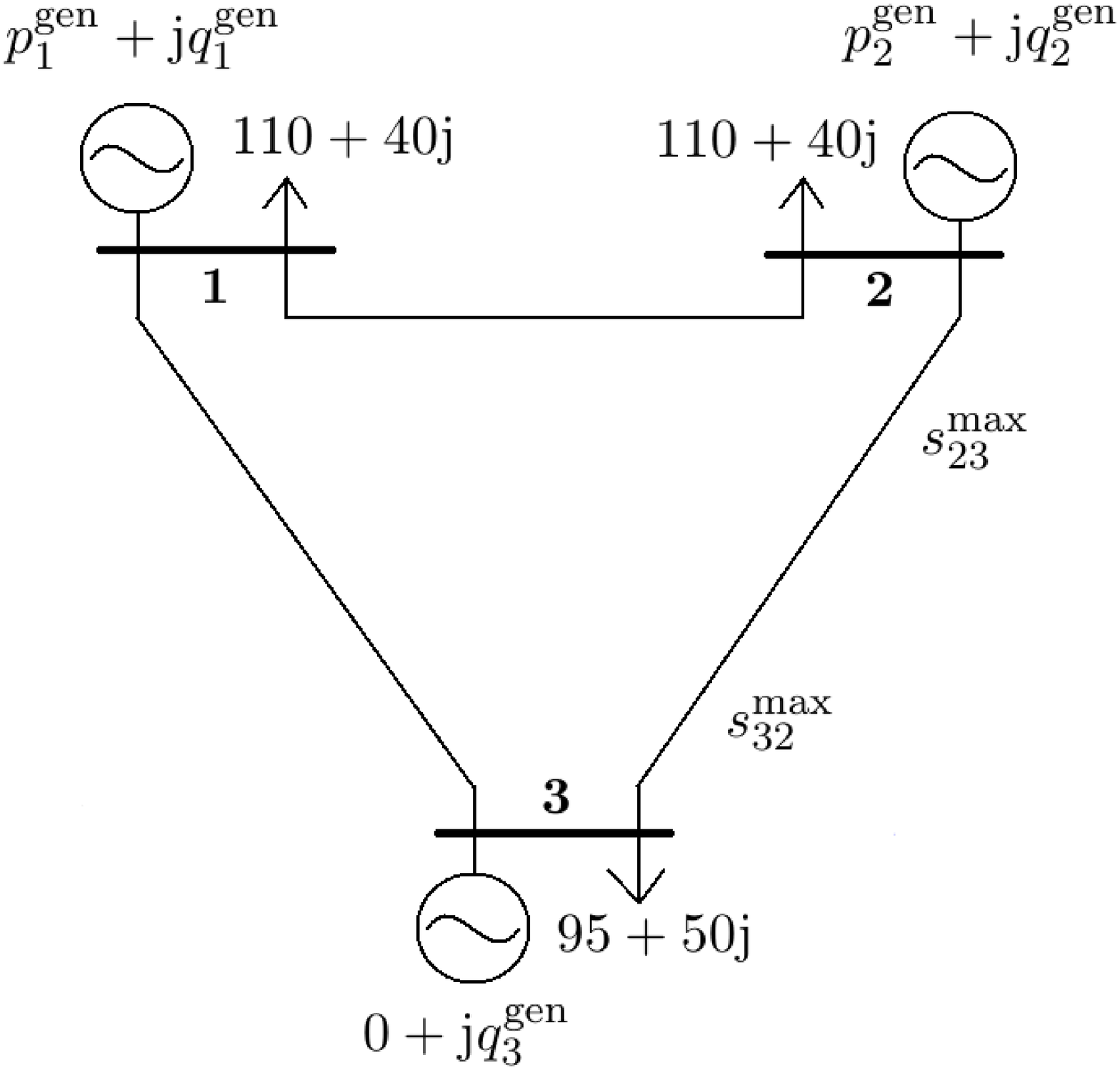}
  \caption{LMBM3 3-bus system}
  \label{fig:LMBM3 3-bus system}
\end{figure}

\begin{table}[H]
\centering
\caption{Order of hierarchy needed to reach global solution to LMBM3 when rank relaxation fails}
\begin{tabular}{c|c|c|c}
$s_{23}^\text{max} = s_{32}^\text{max}$ &  relax. & optimal & rank relax. \\
(MVA) & order & value (\$/h) & value (\$/h) \\
\hline
28.35 & 2 &            10294.88 & (6307.97) \\
31.16 & 2 & \hphantom{1}8179.99 & (6206.78) \\
33.96 & 2 & \hphantom{1}7414.94 & (6119.71) \\
36.77 & 2 & \hphantom{1}6895.19 & (6045.33) \\
39.57 & 2 & \hphantom{1}6516.17 & (5979.38) \\
42.38 & 2 & \hphantom{1}6233.31 & (5919.12) \\
45.18 & 2 & \hphantom{1}6027.07 & (5866.68) \\
47.99 & 2 & \hphantom{1}5882.67 & (5819.02) \\
50.79 & 2 & \hphantom{1}5792.02 & (5779.34) \\  
53.60 & 2 & \hphantom{1}5745.04 &  5745.04  \\ 
\end{tabular}
\label{tab:LMBM3}
\end{table}

The objective of the OPF is a quadratic function of active power so the hierarchy of SDP relaxations is defined for $d~\geqslant~2$. Again, it is surprising that the second order moment-sos relaxation always finds
the global solution to the LMBM3 system, as can be seen in the second
column of table \ref{tab:LMBM3}.

Authors of \cite{lavaei-low-2012} make the assumption that the
objective of the OPF is an increasing function of generated active
power. The moment-sos approach does not require such an assumption. For
example, when $s_{23}^\text{max} = s_{32}^\text{max} = 50~\text{MVA}$,
active generation at bus~1 is equal to 148.07~MW and active generation
at bus~2 is equal to 170.01~MW using the increasing cost function of
\cite{lesieutre_molzahn_borden_demarco-allerton2011,
bukhsh-grothey-mckinnon-trodden-2013b}. Suppose we choose a different
objective which aims at reducing deviation from a given active generation
plan at each generator. Say that this plan is $p_1^\text{plan} =
170~\text{MW}$ at bus~1 and $p_2^\text{plan} = 150~\text{MW}$ at bus 2.
The objective function is equal to $(p_1^\text{gen}-p_1^\text{plan})^2 +
(p_2^\text{gen}-p_2^\text{plan})^2$. It is not an increasing function of
$p_1^\text{gen}$ and $p_2^\text{gen}$. The second order relaxation
yields a global solution in which active generation at bus~1 is equal to
169.21~MW and active generation at bus 2 is equal to 149.19~MW.

\subsection*{5-bus network: WB5}

Authors of \cite{bukhsh2013} observe that in
the WB5 5-bus system of figure \ref{fig:WB5 5-bus system}, the rank
constraint is not satisfied in the rank relaxation method of
\cite{lavaei-low-2012} when $q_{5}^\text{min}
> -30.80~\text{MVAR}$. Above $61.81~\text{MVAR}$, no solutions can be
found by the OPF solver \texttt{runopf} in MATPOWER. At
$-30.80~\text{MVAR}$, the rank constraint is satisfied in the rank
relaxation method so its optimal value is equal to the optimal value
of the OPF found by the second order moment-sos relaxation; see the
first row of table \ref{tab:WB5}. As for the 9 values considered greater
than $-30.80~\text{MVAR}$, the rank constraint is not satisfied since
the optimal value is not equal to the optimal value of the OPF. Notice that the objective of the OPF is a linear function of active power and there are bounds on apparent flow so the hierarchy of SDP relaxations is defined for $d\geqslant 1$.
\begin{figure}[H]
  \centering
    \includegraphics[width=.6\textwidth]{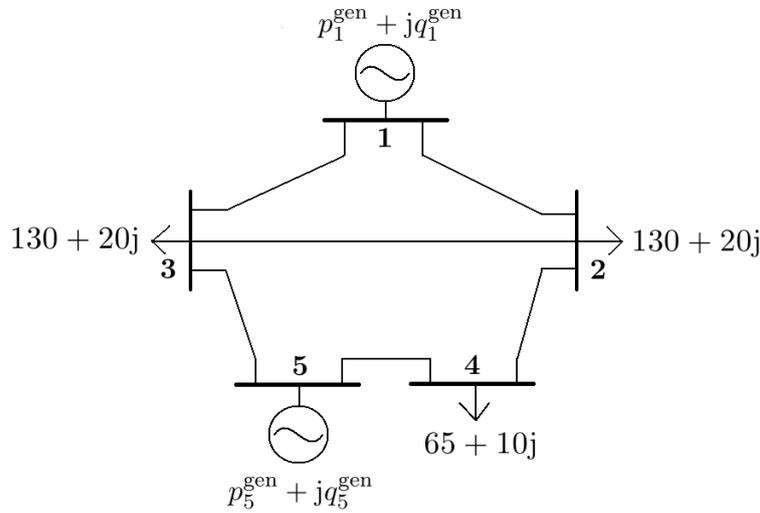}
  \caption{WB5 5-bus system}
  \label{fig:WB5 5-bus system}
\end{figure}
\begin{table}[H]
\centering
\caption{Order of hierarchy needed to reach global solution to WB5 when rank relaxation fails}
\begin{tabular}{c|c|c|c}
$q_{5}^\text{min}$ & relax. & optimal & rank relax. \\
(MVAR) & order & value (\$/h) & value (\$/h) \\
\hline
           -30.80 & 2 & \hphantom{1}945.83 & \hphantom{1(}945.83\hphantom{)} \\
           -20.51 & 2 &            1146.48 & \hphantom{1}(954.82) \\
           -10.22 & 2 &            1209.11 & \hphantom{1}(963.83) \\
\hphantom{-}00.07 & 2 &            1267.79 & \hphantom{1}(972.85) \\
\hphantom{-}10.36 & 2 &            1323.86 & \hphantom{1}(981.89) \\
\hphantom{-}20.65 & 2 &            1377.97 & \hphantom{1}(990.95) \\
\hphantom{-}30.94 & 2 &            1430.54 &            (1005.13) \\
\hphantom{-}41.23 & 2 &            1481.81 &            (1033.07) \\
\hphantom{-}51.52 & 2 &            1531.97 &            (1070.39) \\
\hphantom{-}61.81 & - &              -     &            (1114.90) \\
\end{tabular}
\label{tab:WB5}
\end{table}
When $q_{5}^\text{min} = 61.81~\text{MVAR}$, the hierarchy of SDP
relaxations is unable to find a feasible point, hence the empty slots in
the last row of table \ref{tab:WB5}. Apart from that value, the second
order moment-sos relaxation again always finds the global solution
according to the second column of table \ref{tab:WB5}.

Waki \textit{et al.}\ \cite{waki2006}
have produced a piece of software called
SparsePOP~\cite{kim-kojima-muramatsu-waki-2008}
similar to GloptiPoly only that it seeks to reduce problem size in
Lasserre's relaxations using matrix completion theory in semidefinite
programming. SparsePOP successfully solves the systems studied in this
paper to global optimality but fails to reduce the size of the
moment-sos relaxations and to solve problems with a larger number of
buses.

\section{Conclusion}
\label{subsec:Conclusion2}

This chapter examined the application of the moment-sos (sum of squares)
approach to the global optimization of the optimal power flow (OPF)
problem. The result is that the OPF can be successfully
convexified in the case of several small networks where a previously known
SDP method fails. The SDP problems considered in this paper can be
viewed as extensions of the previously used rank relaxation. It is
guaranteed to be more accurate than the previous one but requires more
runtime. 

Interestingly, Daniel K. Molzahn and Ian A. Hiskens independently made very similar findings~\cite{pscc2014} as presented in this chapter. They successfully solved networks with up to 10 buses. A group at IBM research Ireland was also working on the same ideas~\cite{ibm_paper}. They managed to solve networks with up to 40 buses. To do so, they formulated the OPF problem as a quadratically-constrained quadratic problem and used SparsePOP. This works better than what we had tried, namely using SparsePOP with an OPF formulation with monomials of order 4.

We next focus on a property of the moment-sos approach to further prove its applicability in practice. In the small examples considered in this chapter, there is no duality gap at each order of the moment-sos hierarchy according to numerical results. This property is necessary for efficient solvers to work such as interior-point solvers. However, in the existing literature, there were no results guaranteeing this property.  In the next chapter, we prove there is no duality gap in the moment-sos hierarchy in the presence of a ball constraint. We also explain why there is no duality gap when applying the moment-sos hierarchy to the OPF without a ball constraint.

\chapter{Zero duality gap in the Lasserre hierarchy}
\label{sec:Zero duality gap in the Lasserre hierarchy}

A polynomial optimization problem (POP) consists of minimizing a multivariate real polynomial on a semi-algebraic
set $K$ described by polynomial inequalities and equations. In its full generality it is a nonconvex,
multi-extremal, difficult global optimization problem.  More than an decade ago, J.~B.~Lasserre
proposed to solve POPs by a hierarchy of convex semidefinite optimization (SDP) relaxations
of increasing size. Each problem in the hierarchy has a primal SDP formulation (a relaxation of
a moment problem) and a dual SDP formulation (a sum-of-squares representation
of a polynomial Lagrangian of the POP). In this chapter,
we show that there is no duality gap between each primal and dual SDP problem
in Lasserre's hierarchy, provided one of the constraints in the description of set $K$ is a ball constraint. Our proof uses elementary results on SDP duality,
and it does not assume that $K$ has a strictly feasible point. The material in this chapter is based on the publication:
\\\\
\noindent {\scshape C. Josz and D. Henrion}, \textit{Strong Duality in Lasserre’s Hierarchy for Polynomial Optimization},
Optimization Letters, February 2015. \href{http://dx.doi.org/10.1007/s11590-015-0868-5}{[doi]} \href{https://docs.google.com/viewer?a=v&pid=sites&srcid=ZGVmYXVsdGRvbWFpbnxjZWRyaWNqb3N6fGd4OjY1M2E5NDAyMjg2M2U2Y2Q}{[preprint]}

\section{Introduction}
\label{subsec:Introduction3}

Consider the following polynomial optimization problem (POP)
\begin{equation}
\label{eq:pop}
\begin{array}{ll}
\inf_x & f(x) := \sum_\alpha f_{\alpha} x^\alpha \\
\mathrm{s.t.} & g_i(x) := \sum_\alpha g_{i,\alpha} x^\alpha \geq 0, \quad i=1,\ldots,m
\end{array}
\end{equation}
where we use the multi-index notation $x^\alpha := x_1^{\alpha_1} \cdots x_n^{\alpha_n}$ for $x \in {\mathbb R}^n$,
$\alpha \in {\mathbb N}^n$ and
where the data are polynomials $f, g_1, \ldots, g_m \in {\mathbb R}[x]$
so that in the above sums only a finite number of coefficients
$f_{\alpha}$ and $g_{i,\alpha}$ are nonzero. Let $K$ denote its feasible set:
\[
K:=\{x \in {\mathbb R}^n \: :\: g_i(x) \geq 0, \: i=1,\ldots,m\}
\]
To solve POP (\ref{eq:pop}), 
Lasserre \cite{lasserre-2000,lasserre-2001} proposed a semidefinite optimization (SDP) relaxation hierarchy
with guaranteed asymptotic global convergence provided an algebraic assumption holds:

\begin{assumption}\label{archimedean}
There exists a polynomial $u \in \mathbb{R}[x]$ such that $\{x\in \mathbb{R}^n  \: :\: u(x)\geq 0\}$ is bounded
and $u=u_0 + \sum_{i=1}^m u_i g_i $ where polynomials $u_i \in {\mathbb R}[x]$, $i=0,1,\ldots,m$
are sums of squares (SOS) of other polynomials.
\end{assumption}

Nie \textit{et al.} \cite{nie-2014} have proven that Assumption \ref{archimedean} also implies generically finite convergence, that is to say that for almost every instance of POP, there exists a finite-dimensional SDP relaxation in the hierarchy whose optimal value is equal to the optimal value of the POP. Assumption \ref{archimedean} can be difficult to check computationally (as the degrees of the SOS
multipliers can be arbitrarily large), and it is often replaced
by the following slightly stronger assumption:

\begin{assumption}\label{ball}
The description of $K$ contains a ball constraint, say $g_m(x) = R^2 - \sum_{i=1}^n x_i^2$ for some real number $R$.
\end{assumption}

Indeed, under Assumption \ref{ball}, simply choose $u=g_{m}$, $u_1=\cdots=u_{m-1}=0$, and $u_{m}=1$
to conclude that Assumption \ref{archimedean} holds as well. 
In practice, it is often easy to see to it that Assumption \ref{ball} holds. In the case of a POP with a bounded feasible set, a redundant ball constraint can be added.

More generally, if the intersection of the sublevel set $\{x \in \mathbb{R}^n \: :\: f(x) \leq f(x_0)\}$ with the feasible set
of the POP is bounded for some feasible point $x_0$, then a redundant ball constraint can also be added.
As an illustration, a reviewer suggested the example of the minimization of
$f(x)=x^2_1+x^2_2-3x_1x_2$ on the unbounded set defined on ${\mathbb R}^2$ by the constraint $g_1(x)=1-3x_1x_2\geq 0$.
The intersection of the feasible set with the set defined by the constraint $f(x) \leq f(0)=0$
is included in the ball defined by $g_2(x)=1-x^2_1-x^2_2\geq 0$ so that the POP can be equivalently defined
on the bounded set $K=\{x \in {\mathbb R}^2 \: :\: g_1(x) \geq 0, \: g_2(x) \geq 0\}$.

Each problem in Lasserre's hierarchy consists of a primal-dual SDP pair, called SDP relaxation,
where the primal  corresponds
to a convex moment relaxation of the original (typically nonconvex) POP, and the dual  corresponds
to a SOS representation of a polynomial Lagrangian of the POP. The question arises of whether the
duality gap vanishes in each SDP relaxation. This is of practical importance because numerical algorithms
to solve SDP problems are guaranteed to converge only where there is a zero duality gap,
and sometimes under the stronger assumption that there is a primal or/and dual SDP interior point.

 In \cite[Example 4.9]{schweighofer-2005},
Schweighofer provides a two-dimensional POP with no interior point
for which Assumption \ref{archimedean} holds, yet a duality gap exists at the first
SDP relaxation: $\inf x_1x_2 \:\mathrm{s.t.}\: x \in K=\{x \in {\mathbb R}^2 \: :\:
-1 \leq x_1 \leq 1, \: x^2_2 \leq 0\}$, with
primal SDP value equal to zero and dual SDP value equal to minus infinity.
This shows that a stronger assumption is required to ensure a zero SDP duality gap. 
A sufficient condition for strong duality has been given in \cite{lasserre-2001}:
set $K$ should contain an interior point. However, this may be too restrictive:
in the proof of Lemma 1 in \cite{henrion-2012} the authors use notationally awkward arguments
involving truncated moment matrices
to prove the absence of SDP duality gap for a certain set $K$ that contains no interior point.
This shows that the existence of an interior point is not necessary for a zero SDP duality gap.
More generally, it is not possible to assume the existence
of an interior point for POPs with explicit equality constraints,
and a weaker assumption for zero SDP duality gap is welcome.

Motivated by these observations, in this note we prove that under the basic Assumption \ref{ball}
on the description of set $K$, there is no duality gap in the SDP hierarchy.
Our interpretation of this result,
and the main message of this contribution, is that in the context of Lasserre's hierarchy
for POP, a practically relevant description of a bounded semialgebraic feasibility set
must include a redundant ball constraint.

\section{Proof}
\label{subsec:Proof}

For notational convenience, let $g_0(x)=1 \in {\mathbb R}[x]$ denote the unit polynomial.
Define the localizing matrix
\[
M_{d-d_i}(g_iy) := \left(\sum_\gamma  g_{i,\gamma} y_{\alpha+\beta+\gamma}\right)_{|\alpha|,|\beta|\leq d-d_i}
 = \sum_{|\alpha|\leq 2d} A_{i,\alpha} y_{\alpha}
\]
where $d_i$ is the smallest integer greater than or equal to half the degree of $g_i$,
for $i=0,1,\ldots,m$, and
$|\alpha| =\sum_{i=1}^n \alpha_i$.
The Lasserre hierarchy for POP (\ref{eq:pop})
consists of a  primal moment SDP problem
\[
(P_d) ~:~ 
\begin{array}{ll}
\inf_y & \sum_{\alpha}  f_{\alpha} y_\alpha \\
\mathrm{s.t.} & y_0 = 1 \\
& M_{d-d_i}(g_i y) \succeq 0, \:i=0,1,\ldots,m
\end{array}
\]
and a dual SOS SDP problem
\[
(D_d) ~:~ 
\begin{array}{ll}
\sup_{z,Z} & z \\
\mathrm{s.t.} & f_0 - z =  \sum_{i=0}^m \langle A_{i,0}, Z_i \rangle\\
& f_\alpha = \sum_{i=0}^m \langle A_{i,\alpha}, Z_i \rangle, \quad 0<|\alpha|\leq 2d \\
& Z_i \succeq 0, \:i=0,1,\ldots,m, ~~ z \in \mathbb{R}
\end{array}
\]
where $A\succeq 0$ stands for matrix $A$ positive semidefinite, $\langle A,B \rangle = \mathrm{trace}\:AB$ is
the inner product between two matrices.
The Lasserre hierarchy is indexed by an integer $d \geq d_{\min}:=\max_{i=0,1,\ldots,m} d_i$.
The primal-dual pair $(P_d,D_d)$ is called the SDP relaxation of order $d$ for POP (\ref{eq:pop}). The size of the primal variable $(y_\alpha)_{|\alpha| \leq 2d}$ is $\left( \begin{array}{c} n + 2d \\ n \end{array} \right)$ and the size of the dual variable $Z_i$ is $\left( \begin{array}{c} n + d-d_i \\ n \end{array} \right).$

Let us define the following sets:
\begin{itemize}
\item $\mathcal{P}_d$: feasible points for $P_d$;
\item $\mathcal{D}_d$: feasible points for $D_d$;
\item $\mathrm{int}\:\mathcal{P}_d$: strictly feasible points for  $P_d$;
\item $\mathrm{int}\:\mathcal{D}_d$: strictly feasible points for $D_d$;
\item $\mathcal{P}^*_d$:  optimal solutions for  $P_d$;
\item $\mathcal{D}^*_d$:  optimal solutions for $D_d$.
\end{itemize}
Finally, let us denote by $\mathrm{val}\:P_d$ the infimum in problem $P_d$
and by $\mathrm{val}\:D_d$ the supremum in problem $D_d$.

\begin{lemma}
\label{lemma:Slater}
$\mathrm{int}\:\mathcal{P}_d$ nonempty or $\mathrm{int}\:\mathcal{D}_d$ nonempty
implies $\mathrm{val}\:P_d=\mathrm{val}\:D_d$.
\end{lemma}

Lemma \ref{lemma:Slater} is classical in convex optimization, and it is generally
called Slater's condition, see e.g. \cite[Theorem 4.1.3]{shapiro-2000}.

\begin{lemma}
\label{lemma:Trnovska}
The two following statements are equivalent :
\begin{enumerate}
\item $\mathcal{P}_d$ is nonempty and $\mathrm{int}\:\mathcal{D}_d$ is nonempty;
\item $\mathcal{P}^*_d$ is nonempty and bounded.
\end{enumerate}
\end{lemma}

A proof of Lemma \ref{lemma:Trnovska} can be found in \cite{maria-2005}.
According to Lemmas \ref{lemma:Slater} and \ref{lemma:Trnovska},
$\mathcal{P}_d^*$ nonempty and bounded implies strong duality.
This result is also mentioned without proof at the end of \cite[Section 4.1.2]{shapiro-2000}.

\begin{lemma}
\label{lemma:bound}
Under Assumption \ref{ball}, set $\mathcal{P}_d$ is included in the Euclidean ball of radius\\ $\sqrt{\left( \begin{array}{c}
n+d \\
n
\end{array}
\right)}\sum_{k=0}^{d} R^{2k}$ centered at the origin.
\end{lemma}

\begin{proof}
If $\mathcal{P}_d = \emptyset$, the result is obvious. If not, consider a feasible point $(y_\alpha)_{|\alpha| \leq 2d} \in \mathcal{P}_d$.
 Let $k \in \{1,\ldots,d\}$. In the SDP problem $P_k$,
the localizing matrix associated to the ball constraint $g_{m}(x) = R^2 - \sum_{i=1}^n x^2_i \geq 0$ reads
\[
M_{k-1}(g_{m} y) = \left(\sum_\gamma g_{m,\gamma} ~ y_{\alpha+\beta+\gamma}\right)_{|\alpha|,|\beta|\leq k-1}
\]
with trace equal to
$$ 
\begin{array}{rll}
\mathrm{trace}\:M_{k-1}(g_{m} y) & = & \sum_{|\alpha|\leq k-1} \sum_\gamma g_{m,\gamma} ~ y_{2\alpha+\gamma} \\\\
  & = & \sum_{|\alpha|\leq k-1} \left( g_{m,0} ~ y_{2\alpha} + \sum_{|\gamma|=1} g_{m,2\gamma} ~ y_{2\alpha+2\gamma} \right) \\\\
  & = & \sum_{|\alpha|\leq k-1} \left( R^2 y_{2\alpha} - \sum_{|\gamma|=1} y_{2(\alpha+\gamma)} \right) \\\\
  & = & \sum_{|\alpha|\leq k-1} R^2 y_{2\alpha} - \sum_{|\alpha|\leq k-1,|\gamma|=1} y_{2(\alpha+\gamma)} \\\\
  & \leq\footnotemark & R^2 (\sum_{|\alpha|\leq k-1} y_{2\alpha}) + y_0 - \sum_{|\alpha|\leq k} y_{2\alpha} \\\\
  & \leq & R^2 \:\mathrm{trace}\:M_{k-1}(y) + 1 -  \mathrm{trace}\:M_{k}(y).\\
\end{array}
$$
\footnotetext{In the associated publication \cite{josz-2015}, there is an equality instead of an inequality, which is an error. Indeed, $\sum_{|\alpha|\leq k-1,|\gamma|=1} y_{2(\alpha+\gamma)} = \sum_{0<|\alpha|\leqslant k} y_{2\alpha}$ is false whereas $\sum_{|\alpha|\leq k-1,|\gamma|=1} y_{2(\alpha+\gamma)} \geqslant \sum_{0<|\alpha| \leqslant k} y_{2\alpha}$ is true. The reason for this is that each term $y_{2\alpha}$ with $0<|\alpha| \leqslant k$ appears at least once in $\sum_{|\alpha|\leq k-1,|\gamma|=1} y_{2(\alpha+\gamma)}$, but can potentially appear more than once.  More precisely, we have $\sum_{|\alpha|\leq k-1,|\gamma|=1} y_{2(\alpha+\gamma)} = \sum_{|\alpha|\leqslant k} |\alpha| y_{2\alpha} \geqslant \sum_{0<|\alpha|\leqslant k} y_{2\alpha}$. This error thankfully has no impact on the rest of the proof other than the inequality right below it.}
From the structure of the localizing matrix, it holds $M_{k-1}(g_{m} y) \succeq 0$ hence \\
$\mathrm{trace}\:M_{k-1}(g_{m} y)~\geq~0$ and
\[
\mathrm{trace}\:M_k(y) \leq 1 + R^2\:\mathrm{trace}\:M_{k-1}(y)
\]
from which we derive
\[
\mathrm{trace}\:M_d(y) \leq  \sum_{k=1}^{d} R^{2(k-1)} + R^{2d}\:\mathrm{trace}\:M_0(y) =  \sum_{k=0}^{d} R^{2k} 
\]
since $\mathrm{trace}\:M_0(y) = y_0 = 1$.
The operator norm
$\|M_{d}(y)\|$, equal to the maximum eigenvalue of $M_d(y)$, is upper bounded
by $\mathrm{trace}\:M_d(y)$, the sum of the eigenvalues of $M_d(y)$, which are all nonnegative.
Moreover the Frobenius norm satisfies
$$\begin{array}{rcl}
\|M_d(y)\|_F^2 & := & \langle ~ M_d(y)~,~M_d(y) ~ \rangle  \\\\
& = & \langle ~ \sum_{|\delta|\leq 2d} A_{0,\delta}~y_\delta ~,~ \sum_{|\delta|\leq 2d} A_{0,\delta}~y_\delta ~ \rangle \\\\
 & = & \sum_{|\delta|\leq  2d} ~ \langle A_{0,\delta},A_{0,\delta} \rangle ~ y_\delta^2 ~~~ \text{by orthogonality of matrices
$(A_{0,\delta})_{|\delta|\leq  2d}$}\\\\
 & \geq  & \sum_{|\delta|\leq  2d} ~ y_\delta^2 ~~~ \text{because $\langle A_{0,\delta},A_{0,\delta} \rangle \geq  1$}
\end{array} 
$$
where matrices $(A_{0,\delta})_{|\delta| \leq 2d}$ can be written using column vectors $(e_\alpha)_{|\alpha|\leq d}$~, containing only zeros apart from the value 1 at index $\alpha$, via the explicit formula
$$A_{0,\delta} = \sum_{ \begin{array}{c} \alpha + \beta = \delta  \\ |\alpha|,|\beta|\leq d \end{array} }     e_{\alpha} e_{\beta}^T.$$
The proof follows then from
\[
\sqrt{\sum_{|\delta|\leq  2d} y^2_\delta} \leq \|M_d(y)\|_F \leq 
\sqrt{
\left( \begin{array}{c}
n+d \\
n
\end{array}
\right)}
\|M_d(y)\| \leq 
\sqrt{
\left( \begin{array}{c}
n+d \\
n
\end{array}
\right)
}
\sum_{k=0}^{d} R^{2k}.
\]
\end{proof}

\begin{theorem}
\label{theorem:strong duality}
Assumption \ref{ball} implies that $-\infty < \mathrm{val}\:P_d=\mathrm{val}\:D_d $ for all
$d \geq d_{\min}$.
\end{theorem}

\begin{proof}
Let $d \geq d_{\min}$.
Firstly, let us consider the case when $\mathcal{P}_d$ is nonempty.
According to Lemma \ref{lemma:bound}, $\mathcal{P}_d$ is bounded and closed, and the objective function in $P_d$
is linear, so we conclude that
${\mathcal P}^*_d$ is nonempty and bounded. According to Lemma \ref{lemma:Trnovska}, $\mathrm{int}\:\mathcal{D}_d$
is nonempty, and from Lemma \ref{lemma:Slater}, $\mathrm{val}\:P_d=\mathrm{val}\:D_d $.\\\\
Secondly, let us consider the case when $\mathcal{P}_d$ is empty. An infeasible SDP problem can be either weakly infeasible or strongly infeasible, see \cite[Section 5.2]{klerk-2000} for definitions. Let us prove by contradiction that $P_d$ cannot be weakly infeasible. If $P_d$ is weakly infeasible,
there must exist a sequence $(y^p)_{p \in \mathbb{N}}$ such that
$$\forall p \in \mathbb{N}~,~~~
\left\{
\begin{array}{l}
1- \frac{1}{p+1} \leq y^p_0 \leq 1 + \frac{1}{p+1} \\
\lambda_{\min} ( M_{d-d_i}(g_i y^p) ) \geq -\frac{1}{p+1}, \:i=0,1,\ldots,m
\end{array}
\right.
$$
where $\lambda_{\min}$ denotes the minimum eigenvalue of a symmetric matrix.
According to the proof of Lemma \ref{lemma:bound}, for all $1 \leq k \leq d$ and all real numbers $(y_\alpha)_{|\alpha|\leq 2d}$, one has
$$ \mathrm{trace}\:M_{k-1}(g_{m} y) = R^2 \:\mathrm{trace}\:M_{k-1}(y) + y_0 -  \mathrm{trace}\:M_{k}(y). $$
Clearly, $\mathrm{trace}\:M_{k-1}(g_{m} y) \geq - \frac{c}{1+p}$ where $c:=\left( \begin{array}{c} n + d \\ n \end{array} \right)$ denotes the size of the moment matrix $M_{d}(y)$. The following holds
$$ \mathrm{trace}\:M_{k}(y^p) \leq R^2 \:\mathrm{trace}\:M_{k-1}(y^p) + 1 + \frac{1+c}{1+p} $$
from which we derive
$$ \mathrm{trace}\:M_{d}(y^p) \leq (1 + \frac{1+c}{1+p})\sum_{k=0}^{d} R^{2k}.$$
Together with $\lambda_{\min}(M_{d}(y^p))\geq -\frac{1}{1+p}$, this yields
$$ \lambda_{\max} (M_{d}(y^p)) \leq (1 + \frac{1+c}{1+p})\sum_{k=0}^{d} R^{2k} + \frac{c-1}{1+p}$$
where $\lambda_{\max}$ denotes the minimum eigenvalue of a symmetric matrix.
Hence for all $p\in \mathbb{N}$, the spectrum of the moment matrix $M_d(y^p)$ is lower bounded by $l:=-1$ and upper bounded by $u:=(2+c)\sum_{k=0}^{d} R^{2k} + c-1$. Therefore:
\[
\sqrt{\sum_{|\delta|\leq  2d} (y^p_\delta)^2} \leq \|M_d(y^p)\|_F \leq 
\sqrt{c} \: \max ( |l| , |u| ) \]
The sequence $(y^p)_{p \in \mathbb{N}}$ is hence included in a compact set. Thus there exists a subsequence which converges towards $y^{\text{lim}}$ such that $y^{\text{lim}}_0 = 1$ and $\lambda_{\min} ( M_{d-d_i}(g_i y^{\text{lim}}) )\geq~0$, $\:i=0,1,\ldots,m$. The limit $y^{\text{lim}}$ is thus included is $\mathcal{P}_d$, which is a contradiction.\\\\
SDP problem $P_d$ is strongly infeasible which means that its dual problem $D_d$ has an improving ray \cite[Definition 5.2.2]{klerk-2000}. To conclude that $\text{val}\:D_d = + \infty$, all that is left to prove is that $\mathcal{D}_d \neq \emptyset$. Consider the 
primal problem $P_d$ discarding all constraints but $y_0=1$, $M_d(y) \succcurlyeq 0$, and $M_{d-1}(g_my) \succcurlyeq 0$. It is a feasible and bounded SDP problem owing to Lemma \ref{lemma:bound}. According to Lemma \ref{lemma:Trnovska}, its dual problem must contain a feasible point $(z,Z_0,Z_m)$ and
hence $(z,Z_0,0,\hdots,0,Z_m) \in \mathcal{D}_d$.
  
\end{proof}

\section{Conclusion}
\label{subsec:Conclusion3}

We prove that there is no duality gap in Lasserre's SDP hierarchy for POPs
whose description of the feasible set contains a ball constraint. Prior results ensuring zero duality gap required the existence of a strictly feasible point, which excludes POPs with equality constraints. A zero duality gap is an important property for interior-point solvers to successfully find solutions. A slight adaption of the proof we propose shows that in the case of the optimal power flow problem, upper bounds on voltage imply that there is no duality gap at each order of the Lasserre hierarchy. The adaption consists of summing the traces of each localizing matrix associated to an upper voltage constraint. The sum is equal to the trace of the localizing matrix of a sphere constraint that would be obtained by adding all upper voltage constraints. This means that the computation in Lemma \ref{lemma:bound} is still valid and that the overall proof still holds.

In Chapter \ref{sec:Lasserre hierarchy for small-scale networks}, we've proven the applicability of Lasserre's hierarchy from numerical perspective. In this chapter, its applicability was enforced from a theoretical perspective. To prove its applicability to large-scale networks, new test cases needed to be made publicly available. Indeed, the only large-scale networks available so far were Polish networks, each corresponding to a different period in the year, and the Great Britain network. These networks contain roughly two to three thousand buses each. In the next chapter, we present new data of European networks from various countries and with up to nine thousand buses. Figure \ref{fig:timeline} gives a brief history of the computations prior to the introduction of these networks. Only computations leading to physically meaningful results are presented, in other words those that lead to feasible solutions. These are either global solutions or nearly global solutions (in the case of penalization). Working in collaboration with Daniel K. Molzahn and Ian Hiskens, we were to find global solutions in the case of active power minimization for well-conditioned networks with up to two thousand variables (Chapters  \ref{sec:Penalized Lasserre hierarchy for large-scale networks} and \ref{sec:Complex hierarchy for enhanced tractability}).

\begin{figure}[H]
	\centering
	\includegraphics[width=.8\textwidth]{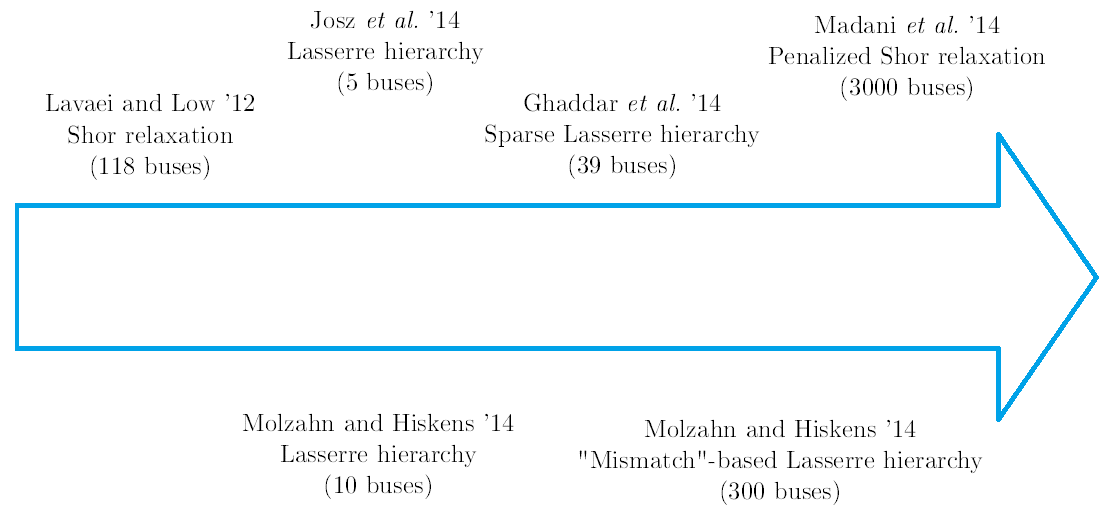}
	\caption{Timeline of computional advances \cite{lavaei-low-2012,cedric_tps,pscc2014,ibm_paper,mh_sparse_msdp,lavaei_allerton2014} }
	 \label{fig:timeline}
\end{figure}

\chapter{Data of European transmission network}
\label{sec:Data of European high-voltage transmission network}

We contributed four new test cases~\cite{josz-2016} to \matpower{}~\cite{matpower}, a package of MATLAB. They represent parts and all of the European high voltage AC transmission network. The data stems from the Pan European Grid Advanced Simulation and State Estimation (PEGASE) project, part of the 7\textsuperscript{th} Framework Program of the European Union (\url{http://www.fp7-pegase.com/}). Its goal was to develop new tools for the real-time control and operational planning of the pan-Euporean transmission network. Specifically, new approaches were implemented for state estimation, dynamic security analysis, steady state optimization. A dispatcher training simulator was also created. The associated public data we provided is:
\\\\
\noindent C. Josz, S. Fliscounakis, J. Maeght, and P. Panciatici, \textit{Power Flow Data for European High-Voltage Transmission Network: 89, 1354, 2869, and 9241-bus PEGASE Systems}, \matpower{} 5.1, March 2015. \href{http://www.pserc.cornell.edu//matpower/}{[link]} 
\\\\
PEGASE data contains asymmetric shunt conductance and susceptance in the PI transmission line model of branches. However, \matpower{} format do not allow for asymmetry. As a result, we set the total line charging susceptance of branches to 0 per unit in the \matpower{} files. We used the nodal representation of shunt conductance and susceptance. This procedure leaves the power flow equations unchanged compared with the original PEGASE data. However, line flow constraints in the optimal power flow problem are modified.
\\\\
The data includes negative resistances and negative lower bounds on active power generation. This is due to sections of the network that are not represented. These sections may include generators, which accounts for the negative resistances. Imports and exports with these sections account for negative bounds on generation. These sections may be a country connected to the European network that is represented by the data. It may also be a section within a country for which data was not provided. A big challenge when making realistic test cases is to account for missing data in a sensible manner. See~\cite{monticelli,alvarado} for work on the subject. Note that the non-represented sections also account for buses where voltage magnitudes have very large lower and upper bounds. All others buses have tight constraints, plus or minus $10\%$ of the nominal value. This is visible is figure \ref{fig:plane}. It represents the voltages at each bus along with the lower and upper constraints resulting in an annulus.

\section{case89pegase}

This case accurately represents the size and complexity of a small part of the European high voltage transmission network. The network contains 89 buses, 12 generators, and 210 branches. It operates at 380, 220, and 150 kV. 

\section{case1354pegase}

This case accurately represents the size and complexity of a medium part of the European high voltage transmission network. The network contains 1,354 buses, 260 generators, and 1,991 branches. It operates at 380, and 220 kV.

\section{case2869pegase}

This case accurately represents the size and complexity of a large part of the European high voltage transmission network. The network contains 2,869 buses, 510 generators, and 4,582 branches. It operates at 380, 220, 150, and 110 kV.

\section{case9241pegase}

This case accurately represents the size and complexity of the European high voltage transmission network. The network contains 9,241 buses, 1,445 generators, and 16,049 branches. It operates at 750, 400, 380, 330, 220, 154, 150, 120 and 110 kV. It represents 23 countries as can be seen in figure \ref{fig:pegase}. The numbers between the countries correspond the sum of the active power flows traded at the interconnections for the voltage profile of figure \ref{fig:plane}. This voltage profile was found with the nonlinear solver KNITRO. It is contained in the case9241 \matpower{} file.
\begin{figure}[H]
	\centering
	\includegraphics[width=1\textwidth]{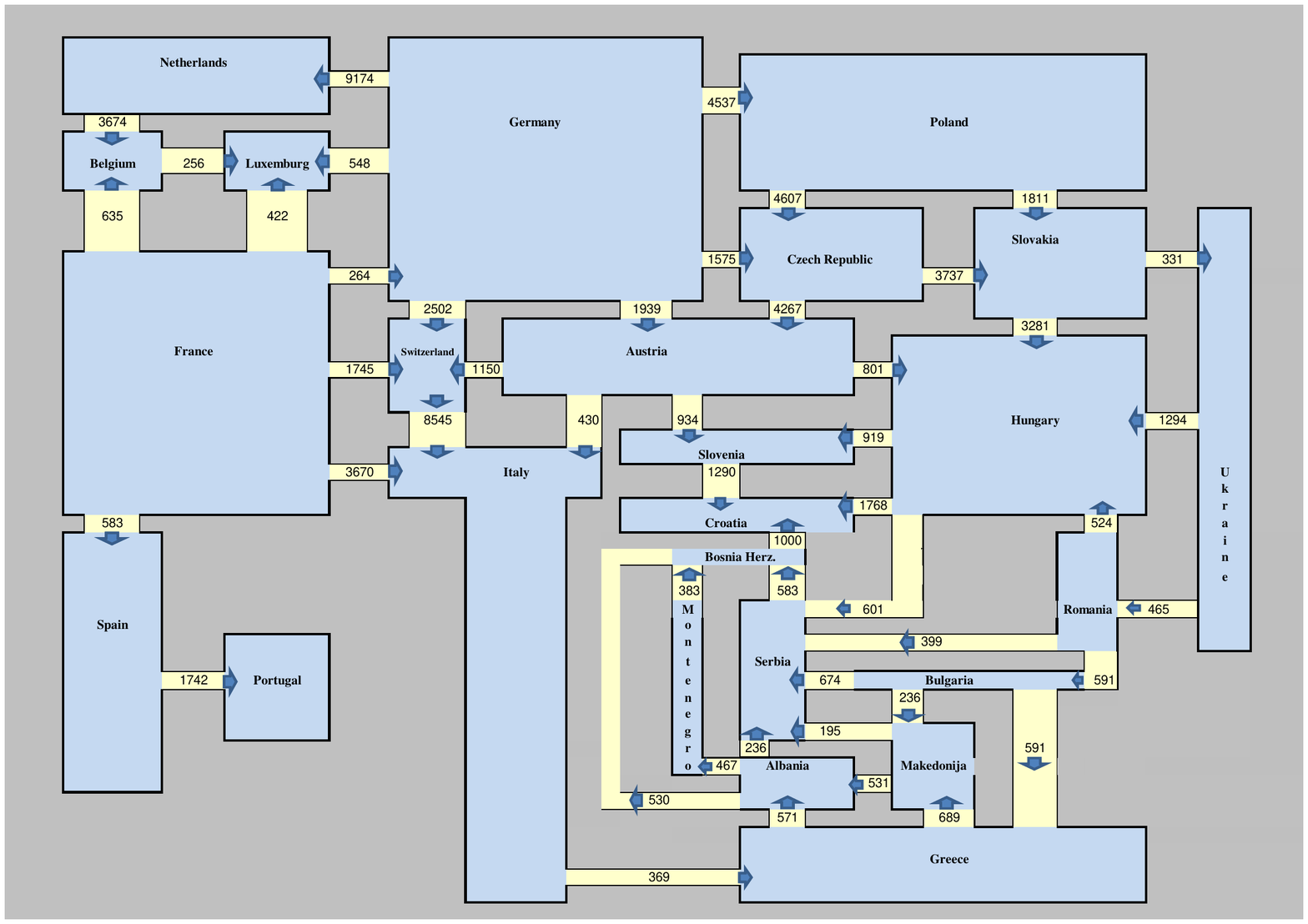}
	\caption{PEGASE network with active power flows in MW (courtesy of St\'ephane Fliscounakis)}
	 \label{fig:pegase}
\end{figure}
\begin{figure}[H]
	\centering
	\includegraphics[width=1\textwidth]{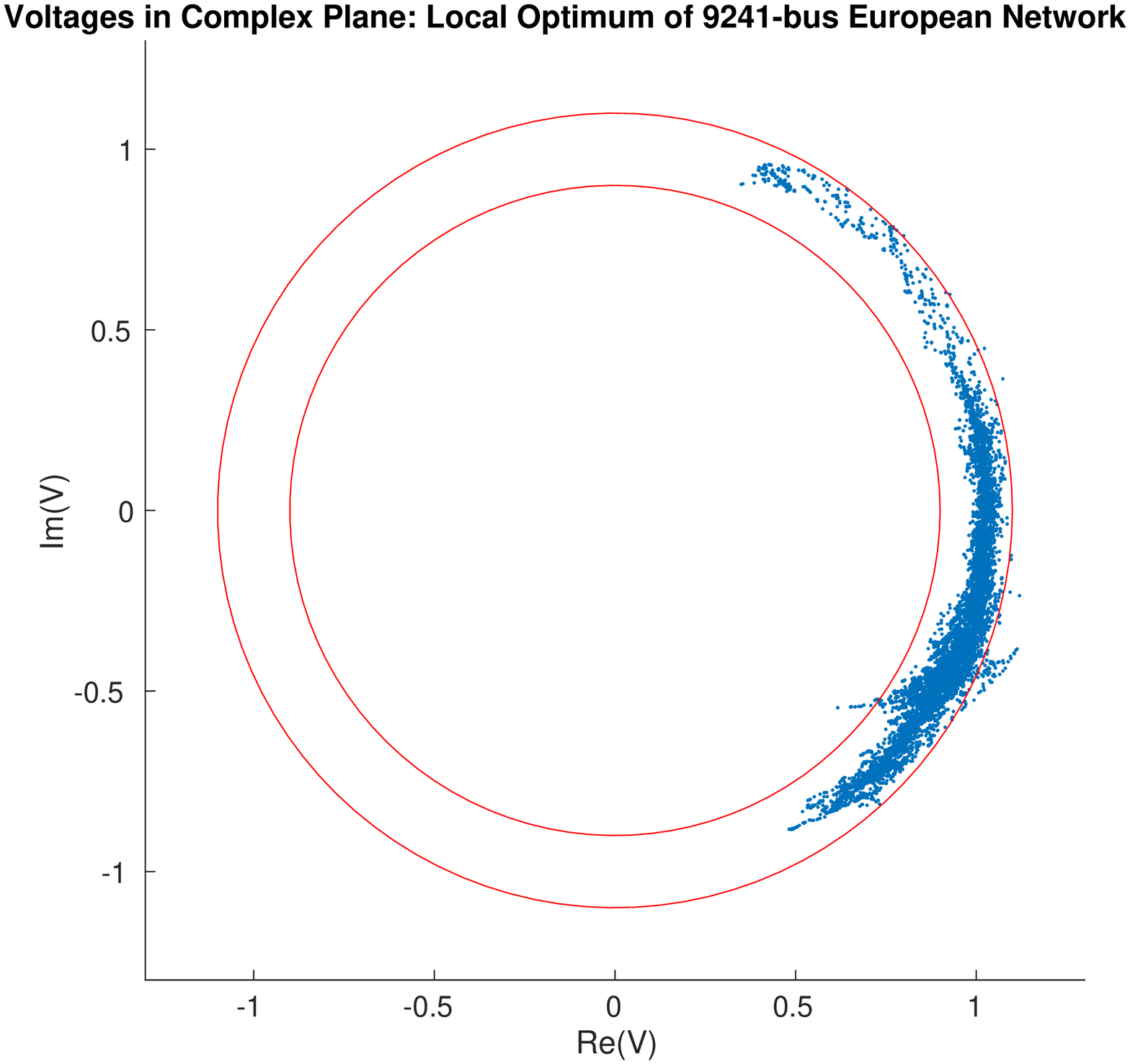}
	\caption{Local optimum}
	 \label{fig:plane}
\end{figure}
The PEGASE data provides large-scale test cases on which to test new methods and tools. After having shown the applicability of moment/sum-of-squares approach to small-scale systems, the main challenge was to make the approach tractable for large-scale systems. A significant turn took place when Daniel K. Molzahn and Ian A. Hiskens proposed a novel way to exploit sparsity when applying the moment-sos hierarchy to the OPF problem~\cite{mh_sparse_msdp}. The main idea is that a global solution can be obtained by only applying high-order moments to some constraints. These constraints are deduced by looking at the moment matrix when the relaxation fails. The authors of~\cite{mh_sparse_msdp} were thus able to solve networks to global optimality with up to three hundred buses. In the next chapter, their method is combined with a penalization approach due to Madani \textit{et al}~\cite{lavaei_allerton2014}. Tests are made with the PEGASE data in the next three chapters.

\chapter{Penalized Lasserre hierarchy}
\label{sec:Penalized Lasserre hierarchy for large-scale networks}

Applications of convex relaxation techniques to the nonconvex OPF problem have been of recent interest, including work using the Lasserre hierarchy of ``moment'' relaxations to globally solve many OPF problems. By preprocessing the network model to eliminate low-impedance lines, this chapter demonstrates the capability of the moment relaxations to globally solve large OPF problems that minimize active power losses for portions of several European power systems. Large problems with more general objective functions have thus far been computationally intractable for current formulations of the moment relaxations. To overcome this limitation, this chapter proposes the combination of an objective function penalization with the moment relaxations. This combination yields feasible points with objective function values that are close to the global optimum of several large OPF problems. Compared to an existing penalization method, the combination of penalization and the moment relaxations eliminates the need to specify one of the penalty parameters and solves a broader class of problems.
The material presented in this chapter is based on the publication: 
\\\\
\noindent{\scshape D.K. Molzahn, C. Josz, I.A. Hiskens, and P. Panciatici}, \textit{Solution of Optimal Power Flow Problems using Moment Relaxations Augmented with Objective Function Penalization}, 54th Annual Conference on Decision and Control, Osaka, December 2015. \href{http://arxiv.org/pdf/1508.05037v1.pdf}{[preprint]}

\section{Introduction}

The SDP relaxation of~\cite{lavaei-low-2012} has been generalized to a family of ``moment relaxations'' using the Lasserre hierarchy~\cite{lasserre_book} for polynomial optimization~\cite{pscc2014,cedric_tps,ibm_paper}. The moment relaxations take the form of SDPs, and the first-order relaxation in this hierarchy is equivalent to the SDP relaxation of~\cite{lavaei-low-2012}. Increasing the relaxation order in this hierarchy enables global solution of a broader class of OPF problems.

The ability to globally solve a broader class of OPF problems has a computational cost; the moment relaxations quickly become intractable with increasing order. Fortunately, second- and third-order moment relaxations globally solve many small problems for which the first-order relaxation fails to yield the globally optimal decision variables.

However, increasing system size results in computational challenges even for low-order moment relaxations. The second-order relaxation is computationally intractable for OPF problems with more than ten buses. Exploiting network sparsity enables solution of the first-order relaxation for systems with thousands of buses~\cite{jabr11,dan2013} and the second-order relaxation for systems with about forty buses~\cite{mh_sparse_msdp,ibm_paper}. Recent work~\cite{mh_sparse_msdp} solves larger problems (up to 300 buses) by both exploiting sparsity and only applying the computationally intensive higher-order moment relaxations to specific buses in the network. Other recent work improves tractability using a second-order cone programming (SOCP) relaxation of the higher-order moment constraints~\cite{dan2015}.

Solving larger problems using moment relaxations is often limited by numerical convergence issues rather than computational performance. We present a preprocessing method that improves numerical convergence by removing low-impedance lines from the network model. Similar methods are commonly employed (e.g., PSS/E~\cite{PSSEManual}), but more extensive modifications are needed for adequate convergence due to the limited numerical capabilities of current SDP solvers.

After this preprocessing, the moment relaxations globally solve several large OPF problems which minimize active power losses for European power systems. Directly using the moment relaxations to globally solve more general large OPF problems with objective functions that minimize generation cost has been computationally intractable thus far.

To solve these OPF problems, we form moment relaxations using a penalized objective function. Previous literature~\cite{lavaei_mesh,lavaei_allerton2014} augments the SDP relaxation~\cite{lavaei-low-2012} with penalization terms for the total reactive power generation and the apparent power loss of certain lines. For many problems, this penalization finds feasible points with objective function values that are very close to the lower bounds obtained from the SDP relaxation. Related work~\cite{laplacian_obj} uses a Laplacian-based objective function with a constraint on generation cost to find feasible points are very near the global optima. This section analyzes the physical and convexity properties of the reactive power penalization.

There are several disadvantages of the penalization method of~\cite{lavaei_allerton2014}. This penalization often requires choosing multiple parameters. (See~\cite{laplacian_obj} for a related approach that does not require choosing penalty parameters.) Also, there are OPF problems that are globally solved by the moment relaxations, but no known penalty parameters yield a feasible solution.

We propose a ``moment+penalization'' approach that augments the moment relaxations with a reactive power penalty. Typical penalized OPF problems only require higher-order moment constraints at a few buses. Thus, for a variety of large test cases, augmenting the moment relaxation with the proposed single-parameter penalization achieves feasible solutions that are at least very near the global optima (within at least 1\% for a variety of example problems). The moment+penalization approach enables solution of a broader class of problems than either method individually.

Below, Section~\ref{l:preprocess} describes the low-impedance line preprocessing. Section~\ref{l:penalization} discusses the existing penalization and the proposed moment+penalization approaches. Section~\ref{results5} demonstrates the moment+penalization approach using several large test cases, and Section~\ref{l:conclusion} concludes.

\section{Preprocessing low-impedance lines}
\label{l:preprocess}

By exploiting sparsity and selectively applying the higher-order constraints, the moment relaxations globally solve many OPF problems with up to 300 buses. Solution of larger problems with higher-order relaxations is typically limited by numerical convergence issues rather than computational concerns. This section describes a preprocessing method for improving numerical properties of the moment relaxations.

Low-impedance lines, which often represent ``jumpers'' between buses in the same physical location, cause numerical problems for many algorithms. Low line impedances result in a large range of values in the bus admittance matrix $\mathbf{Y}$, which causes numerical problems in the constraint equations.

To address these numerical problems, many software packages remove lines with impedances below a threshold. For instance, lines with impedance below a parameter \texttt{thrshz} are removed prior to applying other algorithms in PSS/E~\cite{PSSEManual}.

We use a slightly modified version of the low-impedance line removal procedure in PSS/E~\cite{PSSEManual}.\footnote{Lines with non-zero resistances are not considered to be ``low impedance'' by PSS/E. We consider both the resistance and the reactance.} Low-impedance lines are eliminated by grouping buses that are connected by lines with impedances below a specified threshold \texttt{thrshz}. Each group of buses is replaced by one bus that is connected to all lines terminating outside the group. Generators, loads, and shunts (including the shunt susceptanes of lines connecting buses within a group) are aggregated. The series parameters of lines connecting buses within a group are eliminated.

Removing low-impedance lines typically has a small impact on the solution. To recover an approximate solution to the original power system model, assign identical voltage phasors to all buses in each group and distribute flows on lines connecting buses within a group under the approximation that all power flows through the low-impedance lines.

A typical low-impedance line threshold \texttt{thrshz} is \mbox{$1\times 10^{-4}$}~per unit. However, the numerical capabilities of SDP solvers are not as mature as other optimization tools. Therefore, we require a larger $\mathtt{thrshz} = 1\times 10^{-3}$ per unit to obtain adequate convergence of the moment relaxations. This larger threshold typically introduces only small errors in the results, although non-negligible errors are possible.

{\scshape Matpower} solutions obtained for the Polish~\cite{matpower} and most \mbox{PEGASE} systems~\cite{pegase,josz-2016} were the the same before and after low-impedance line preprocessing to within 0.0095 per unit voltage magnitude and $0.67^\circ$ voltage angle difference across each line. Operating costs for all test problems were the same to within 0.04\%. The 2869-bus \mbox{PEGASE} system had larger differences: 0.0287 per unit voltage magnitude and $1.37^\circ$ angle difference. A power flow solution for the full network using the solution to the OPF problem after low-impedance line preprocessing yields smaller differences: 0.0059 per unit voltage magnitude and $1.17^\circ$ angle difference. Thus, the differences from the preprocessing for the 2869-bus \mbox{PEGASE} system can be largely attributed to the sensitivity of the OPF problem itself to small changes in the low-impedance line parameters. Preprocessing reduced the number of buses by between 21\% and 27\% for the PEGASE and between 9\% and 26\% for the Polish systems.

These results show the need for further study of the sensitivity of OPF problems to low-impedance line parameters as well as additional numerical improvements of the moment relaxations and SDP solvers to reduce \texttt{thrshz}. 

\section{Moment relaxations and penalization}
\label{l:penalization}

As will be shown in Section~\ref{l:results}, the moment relaxations globally solve many large OPF problems with active power loss minimization objectives after removing low-impedance lines as described in Section~\ref{l:preprocess}. Directly applying the moment relaxations to many large OPF problems with more general cost functions has so far been computationally intractable. This section describes the nonconvexity associated with more general cost functions and proposes a method to obtain feasible solutions that are at least near the global optimum, if not, in fact, globally optimal for many problems.

Specifically, we propose augmenting the moment relaxations with a penalization in the objective function. Previous literature~\cite{lavaei_mesh,lavaei_allerton2014} adds terms to the first-order moment relaxation that penalize the total reactive power injection and the apparent power line loss (i.e., $\sqrt{\left(f_{Plm} + f_{Pml}\right)^2 + \left(f_{Qlm} + f_{Qml}\right)^2}$) for ``problematic'' lines identified by a heuristic. This penalization often finds feasible points that are at least nearly globally optimal.

However, the penalization in~\cite{lavaei_allerton2014} requires choosing two penalty parameters and fails to yield a feasible solution to some problems. This section describes progress in addressing these limitations by augmenting the moment relaxations with a reactive power penalization. The proposed ``moment+penalization'' approach only requires a single penalty parameter and finds feasible points that are at least nearly globally optimal for a broader class of OPF problems. This section also analyzes the convexity properties and provides a physical intuition for reactive power penalization.

\section{Penalization of reactive power generation}
\label{l:penalty_explanation}

The penalization method proposed in~\cite{lavaei_allerton2014} perturbs the objective function to include terms that minimize the total reactive power loss and the apparent power loss on specific lines determined by a heuristic method. These terms enter the objective function with two scalar parameter coefficients. Obtaining a feasible point near the global solution requires appropriate choice of these parameters.

For typical operating conditions, reactive power is strongly associated with voltage magnitude. Penalizing reactive power injections tends to reduce voltage magnitudes, which also tends to increase active power losses since a larger current flow, with higher associated loss, is required to deliver a given quantity of power at a lower voltage magnitude.

For many problems for which the first-order moment relaxation fails to yield the global optimum, the relaxation ``artificially'' increases the voltage magnitudes to reduce active power losses. This results in voltage magnitudes and power injections that are feasible for the SDP relaxation but infeasible for the OPF problem. 

By choosing a reactive power penalty parameter that balances these competing tendencies (increasing voltage magnitudes to reduce active power losses vs. decreasing voltage magnitudes to reduce the penalty), the penalized relaxation finds a feasible solution to many OPF problems. Since losses typically account for a small percentage of active power generation and active and reactive power are typically loosely coupled, the reactive power penalization often results in a feasible point that is near the global optimum.

We next study the convexity properties of the cost function and the reactive power penalization. The cost function is convex in terms of active power generation but not necessarily in terms of the real and imaginary voltage components.\footnote{The cost function of the \emph{moment relaxation} is always convex. This section studies the convexity of the penalized objective function for the \emph{original} nonconvex OPF problem.} Thus, the objective function is a potential source of nonconvexity which may result in the relaxation's failure to globally solve the OPF problem.

Consider the eigenvalues of the symmetric matrices $\mathbf{C}$ and $\mathbf{D}$, where, for the vector $\hat{x}$ containing the voltage components, $\hat{x}^\intercal \mathbf{C} \hat{x}$ is a linear cost of active power generation and $\hat{x}^\intercal \mathbf{D} \hat{x}$ calculates the reactive power losses. For the 2383-bus Polish system~\cite{matpower}, which has linear generation costs, the most negative eigenvalue of $\mathbf{C}$ is $-8.53\times 10^7$. Thus, the objective function of the original OPF problem is nonconvex in terms of the voltage components, which can cause the SDP relaxation to fail to yield the global optimum. Conversely, active power loss minimization is convex in terms of the voltage components due to the absence of negative resistances.

As indicated by the potential for negative eigenvalues of $\mathbf{D}$ (e.g., the matrix $\mathbf{D}$ for the 2383-bus Polish system has a pair of negative eigenvalues at $-0.0175$), penalizing reactive power losses is generally nonconvex due to capacitive network elements (i.e., increasing voltage magnitudes may decrease the reactive power loss). See~\cite{laplacian_obj} for related work that uses a convex objective based on a Laplacian matrix.


Further work is needed to investigate the effects of reactive power penalization on OPF problems with more realistic generator models that explicitly consider the trade-off between active and reactive power outputs (i.e., generator ``D-curves''). A tighter coupling between active and reactive power generation may cause the reactive power penalization to yield solutions that are far from the global optimum.

The apparent power line loss penalty's effects are not as easy to interpret as the reactive power penalty. Ongoing work includes understanding the effects of the line loss penalty.

\section{Moment+penalization approach}

Although the reactive power penalization often yields a near rank-one solution, this penalization alone is not sufficient to obtain a feasible point for many problems. Reference~\cite{lavaei_allerton2014} penalizes the apparent power line loss associated with certain lines to address the few remaining non-rank-one ``problematic'' submatrices. However, this approach has several disadvantages.

First, penalizing apparent power line losses introduces another parameter.\footnote{Reference~\cite{lavaei_allerton2014} uses the same penalization parameter for each ``problematic'' line. Generally, each line could have a different penalty parameter.} Introducing parameters is problematic, especially when lacking an intuition for appropriate values.

Second, the combination of reactive power and line loss penalization may not yield a feasible solution to some problems. For instance, the OPF problems case9mod and case39mod1 from~\cite{bukhsh2013} are globally solved with low-order moment relaxations, but there is no known penalization of reactive power and/or apparent power line loss that yields a feasible solution for these problems. Also, the penalization approach is not guaranteed to yield a feasible solution that is close to the global optimum.

Unlike the penalization approach, the moment relaxation approach does not require the choice of penalty parameters, globally solves a broader class of OPF problems, and is guaranteed to yield the global optimum when the rank-one condition is satisfied. However, direct application of the moment relaxations to large problems has so far been limited to active power loss minimization objective functions. We conjecture that the nonconvexity associated with more general cost functions requires higher-order moment constraints at too many buses for computational tractability.

To apply the moment relaxations to large OPF problems with active power generation cost objective functions, we augment the moment relaxations with a reactive power penalty. Specifically, we add to the objective the total reactive power produced by all generators multiplied by a penalization parameter $\epsilon_b$ (which is a positive scalar). That is, rather than apply an apparent power loss penalization to the objective function, we apply higher-order moment constraints to specific buses~\cite{mh_sparse_msdp}. As will be demonstrated in Section~\ref{l:results}, higher-order moment constraints are only needed at a few buses in typical OPF problems after augmenting the objective function with a reactive power penalization term.

Similar to the existing penalization, when the rank condition is satisfied, the proposed ``moment+penalization'' approach yields the global solution to the \emph{modified} OPF problem, but not necessarily to the original OPF problem. However, since the penalization does not change the constraint equations, the solution to the moment+penalization approach is \emph{feasible} for the original OPF problem. The first-order moment relaxation without penalization (i.e., $\epsilon_b = 0$) gives a lower bound on the globally optimal objective value for the original OPF problem. This lower bound provides an optimality metric for the feasible solution obtained from the moment+penalization approach. As will be shown in Section~\ref{l:results}, the feasible solutions for a variety of problems are within at least 1\% of the global optimum.

The moment+penalization approach inherits a mix of the advantages and disadvantages of the moment relaxation and penalization methods. First, the moment+penalization approach requires selection of a single scalar parameter (one more than needed for the moment relaxations, but one less than generally needed for the penalization in~\cite{lavaei_allerton2014}). This parameter must be large enough to result in a near rank-one solution, but small enough to avoid large changes to the OPF problem.

Second, the penalization eliminates the moment relaxations' guarantees: the \\ moment+penalization approach may yield a feasible solution that is far from the global optimum or not give any solution. However, the moment+penalization approach finds global or near-global solutions to a broader class of small OPF problems than penalization approach of~\cite{lavaei_allerton2014} (e.g., case9mod and case39mod1 with $\epsilon_b = 0$, and case39mod3 with $\epsilon_b = \$0.10/\mathrm{MVAr}$~\cite{bukhsh2013}). This suggests that the moment+penalization approach inherits the ability of the moment relaxations to solve a broad class of OPF problems.

Finally, the penalization in the moment+penalization approach enables calculation of feasible solutions that are at least nearly globally optimal for a variety of large OPF problems with objective functions that minimize active power generation cost rather than just active power losses.

Note that it is not straightforward to compare the computational costs of the \\ moment+penalization approach and the penalization approach in~\cite{lavaei_allerton2014}. A single solution of a penalized first-order moment relaxation, as in~\cite{lavaei_allerton2014}, is faster than a relaxation with higher-order moment constraints. Thus, if one knows appropriate penalty parameters, the method in~\cite{lavaei_allerton2014} is faster. Although a relatively wide range of penalty parameters tends to work well for typical OPF problems, there are problems for which no known penalty parameters yield feasible solutions. For these problems, the moment+penalization approach has a clear advantage.

The moment+penalization approach has the advantage of systematically tightening the relaxation rather than requiring the choice of penalty parameters. However, the higher-order constraints can significantly increase solver times. Thus, there is a potential trade-off between finding appropriate penalization parameters for the approach in~\cite{lavaei_allerton2014} and increased solver time from the moment+penalization approach. The speed of the moment+penalization approach may be improved using the mixed SDP/SOCP relaxation from~\cite{dan2015}.

\section{Numerical results}
\label{results5}

This section first globally solves several large, active-power-loss minimizing OPF problems using moment relaxations without penalization (\mbox{$\epsilon_b = 0$}).  Next, this section applies the moment+penalization approach to find feasible points that are at least nearly globally optimal for several test cases which minimize active power generation cost. Unless otherwise stated, the preprocessing method from Section~\ref{l:preprocess} with \texttt{thrshz} set to $1\times 10^{-3}$ per unit is applied to all examples. No example enforces a minimum line resistance.

The results are generated using the iterative algorithm from~\cite{mh_sparse_msdp} which selectively applies the higher-order moment relaxation constraints. The algorithm terminates when all power injection mismatches are less than 1~MVA.

The implementation uses \mbox{MATLAB 2013a}, YALMIP 2015.06.26~\cite{yalmip}, and Mosek 7.1.0.28, and was solved using a computer with a quad-core 2.70~GHz processor and 16~GB of RAM. The test cases are the Polish system models in {\scshape Matpower} \cite{matpower} and several \mbox{PEGASE} systems~\cite{pegase,josz-2016} representing portions of the European power system.


\subsubsection{Active Power Loss Minimization Results}

Table~\ref{t:lossresults} shows the results of applying the moment relaxations to several large OPF problems that minimize active power losses. The solutions to the preprocessed problems are guaranteed to be globally optimal since there is no penalization. The columns in Table~\ref{t:lossresults} list the case name, the number of iterations of the algorithm from~\cite{mh_sparse_msdp}, the maximum power injection mismatch, the globally optimal objective value, and the solver time summed over all iterations. The abbreviation ``PL'' stands for ``Poland''. Table~\ref{t:lossresults} excludes several cases (the 89-bus \mbox{PEGASE} system and the Polish 2736sp, 2737sop, and 2746wp systems) which only require the first-order relaxation and thus do not illustrate the capabilities of the higher-order relaxations. PEGASE-1354 and PEGASE-2869 use a \texttt{thrshz} of \mbox{$3\times 10^{-3}$}~per unit. All other systems use $1\times 10^{-3}$ per unit.

\begin{table}[thb]
\caption{Active Power Loss Minimization Results}
\label{t:lossresults}
\footnotesize
\centering
\begin{tabular}{|c|c|c|c|c|}
\hline 
\textbf{Case} & \textbf{Num.}  & \textbf{Global Obj.} & \textbf{Max $\mathrm{S}^{\mathrm{mis}}$}& \textbf{Solver} \\
\textbf{Name} & \textbf{Iter.} & \textbf{Val. (\$/hr)} & \textbf{(MVA)} & \textbf{\!\!Time (sec)\!\!} \\ \hline
PL-2383wp & 3 & \hphantom{1}$24990$ & 0.25 & \hphantom{1}583\\ 
PL-2746wop & 2 & \hphantom{1}$19210$ & 0.39 & 2662\\ 
PL-3012wp & 5 & \hphantom{1}$27642$ & 1.00 & \hphantom{1}319\\ 
PL-3120sp & 7 & \hphantom{1}$21512$ & 0.77 & \hphantom{1}387\\ 
PEGASE-1354 & 5 & \hphantom{1}$74043$ & 0.85 & \hphantom{1}407 \\ 
PEGASE-2869 & 6 & $133944$ & 0.63 & \hphantom{1}921 \\\hline 
\end{tabular}
\end{table}

Each iteration of the algorithm in~\cite{mh_sparse_msdp} after the first adds second-order constraints at two buses. Thus, a small number of second-order buses (between 0.1\% and 0.7\% of the number of buses in the systems in Table~\ref{t:lossresults} after the low-impedance line preprocessing) are applied to all examples in Table~\ref{t:lossresults}. This results in computational tractability for the moment relaxations.

Note that PL-2746wop has a much greater solver time than the other systems even though it only has second-order constraints at two buses. This slow solution time is due to the fact that the two second-order buses are contained in submatrices corresponding to cliques with 10 and 11 buses. The second-order constraints for these large submatrices dominate the solver time. The mixed SDP/SOCP relaxation in~\cite{dan2015} may be particularly useful beneficial for such cases.

Since the low-impedance line preprocessing has been applied to these systems, the solutions do not exactly match the original OPF problems. {\scshape Matpower} \cite{matpower} solutions of the original problems have objective values that are slightly larger than the values in Table~\ref{t:lossresults} due to losses associated with the line resistances removed by the preprocessing.

After the low-impedance line preprocessing, local solutions from {\scshape Matpower} match the solutions from the moment relaxations and are therefore, in fact, globally optimal. This is not the case for all OPF problems~\cite{bukhsh2013,mh_sparse_msdp}.

\subsubsection{Moment+Penalization for More General Cost Functions}

As discussed in Section~\ref{l:penalty_explanation}, minimization of active power generation cost often yields a nonconvex objective function in terms of the voltage components. Despite this nonconvexity, low-order moment relaxations typically yield global solutions to small problems, including problems without known penalty parameters for obtaining a feasible points (e.g., case9mod and case39mod1 from~\cite{bukhsh2013}).

\begin{figure}[t]
\centering
\vspace{1.5pt}
\includegraphics[totalheight=0.179\textheight]{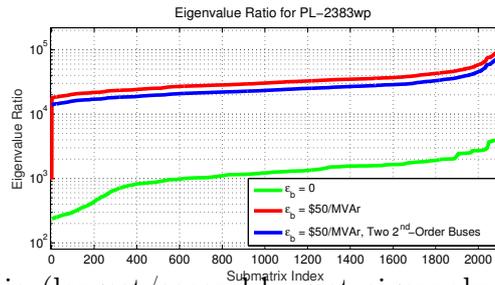}
\vspace{-20pt}
\caption{\hspace{-8pt} Eigenvalue ratio (largest/second-largest eigenvalue) for each submatrix for PL-2383wp. Large values ($>\!\!10^4$) indicate satisfaction of the rank-1 condition. For $\epsilon_b = 0$ (green), most of the submatrices are not rank one. For $\epsilon_b = \$50/\mathrm{MVAr}$ (red), most submatrices satisfy the rank condition with the exception of those on the far left of the figure. Applying second-order moment constraints to two of the buses that are in these submatrices (blue) results in all submatrices satisfying the rank condition.}
\label{f:case2383wp_eigWratio}
\vspace{-13pt}
\end{figure}

However, the moment relaxations are thus far intractable for some large OPF problems with nonconvex objective functions. A reactive power penalty often results in the first-order moment relaxation yielding a solution that is nearly globally optimal (i.e., most of the submatrices in the clique decomposition satisfy the rank-one condition). Enforcing higher-order constraints at buses in the remaining submatrices yields a feasible solution to the OPF problem. This is illustrated in Fig.~\ref{f:case2383wp_eigWratio}, which shows the ratio between the largest and second-largest eigenvalues of the submatrices of the moment matrix, arranged in increasing order, for the 2383-bus Polish system. If the submatrices were all rank one, then this eigenvalue ratio would be infinite. Thus, large numeric values (i.e., greater than $1\times 10^4$) indicate satisfaction of the rank condition within numerical precision. Without the reactive power penalty, the rank condition is not satisfied for most submatrices. With the reactive power penalty, the rank condition is satisfied for many but not all submatrices. Enforcing higher-order moment constraints at two buses which are in the high-rank submatrices results in a feasible (rank-one) operating point for the OPF problem which is within at least 0.74\% of the global optimum.

To further illustrate the effectiveness of the moment+penalization approach, Table~\ref{t:costresults} shows the results of applying the moment+penalization approach to several large OPF problems with active power generation cost functions. The optimality gap column gives the percent difference between a lower bound on the optimal objective value from the first-order moment relaxation and the feasible solution obtained from the moment+penalization approach for the system after low-impedance line preprocessing.

\begin{table}[thb]
\caption{Generation Cost Minimization Results}
\label{t:costresults}
\footnotesize
\centering
\begin{tabular}{|c|c|c|c|c|c|}
\hline 
\textbf{Case} & \textbf{$\epsilon_b$} & \textbf{Num.}  & \textbf{Opt.} & \textbf{Max $\mathrm{S}^{\mathrm{mis}}$}& \textbf{Solver} \\
\textbf{Name} & \textbf{(\$/MVAr)} & \textbf{Iter.} & \textbf{Gap} & \textbf{(MVA)} & \textbf{Time (sec)} \\ \hline
PL-2383wp & 50 & 2 & 0.74\% & 0.13 & 152.2\\ 
PL-3012wp & 50 & 7 & 0.49\% & 0.20 & 1056.3\\ 
PL-3120sp & 100 & 6 & 0.92\% & 0.08 & 1164.4\\\hline
\end{tabular}
\end{table}

The penalized first-order relaxation requires 74.6, 88.9, and 97.0 seconds for PL-2383wp, PL-3012wp, and PL-3120sp, respectively. Attributing the rest of the solver time to the higher-order constraints implies that these constraints accounted for 3.1, 433.7, and 582.4 seconds beyond the time required to repeatedly solve the first-order relaxations.

The moment+penalization approach can yield feasible points that are at least nearly globally optimal for cases where both the penalization method of~\cite{lavaei_allerton2014} and low-order moment relaxations fail individually. For instance, the moment+penalization approach with a reactive power penalty of $\epsilon_b = \$0.10/\mathrm{MVAr}$ gives a feasible point within 0.28\% of the global optimum for case39mod3 from~\cite{bukhsh2013}, but both second- and third-order moment relaxations and the penalization method in~\cite{lavaei_allerton2014} fail to yield global solutions.

\section{Conclusion}
\label{l:conclusion}

``Moment'' relaxations from the Lasserre hierarchy for polynomial optimization globally solve a broad class of OPF problems. By exploiting sparsity and selectively applying the computationally intensive higher-order moment relaxations, previous literature demonstrated the moment relaxations' capability to globally solve moderate-size OPF problems. This chapter presented a preprocessing method that removes low-impedance lines to improve the numerical conditioning of the moment relaxations. After applying the preprocessing method, the moment relaxations globally solve a variety of OPF problems that minimize active power losses for systems with thousands of buses. A proposed ``moment+penalization'' method is capable of finding feasible points that are at least nearly globally optimal for large OPF problems with more general cost functions. This method has several advantages over previous penalization approaches, including requiring fewer parameter choices and solving a broader class of OPF problems. In the next chapter, we devise a method that provides nearly global solutions to large-scale OPF problems without the need to specify any parameter.

\chapter{Laplacian matrix approach}
\label{sec:Laplacian matrix gets rid of penalization parameter}

A semidefinite optimization (SDP) relaxation globally solves many optimal power flow (OPF) problems. For other OPF problems where the SDP relaxation only provides a lower bound on the objective value rather than the globally optimal decision variables, recent literature has proposed a penalization approach to find feasible points that are often nearly globally optimal. A disadvantage of this penalization approach is the need to specify penalty parameters. This paper presents an alternative approach that algorithmically determines a penalization appropriate for many OPF problems. The proposed approach constrains the generation cost to be close to the lower bound from the SDP relaxation. The objective function is specified using iteratively determined weights for a Laplacian matrix. This approach yields feasible points to the OPF problem that are guaranteed to be near the global optimum due to the constraint on generation cost. The proposed approach is demonstrated on both small OPF problems and a variety of large test cases representing portions of European power systems.
The material presented in this chapter is based on the submitted manuscript:
\\\\
\noindent {\scshape D.K. Molzahn, C. Josz, I.A. Hiskens, and P. Panciatici}, \textit{A Laplacian-Based Approach for Finding Near Globally Optimal Solutions to OPF Problems}, submitted to Institute of Electrical and Electronics Engineers, Transactions on Power Systems. \href{http://arxiv.org/pdf/1507.07212v1.pdf}{[preprint]}

\section{Introduction}
\label{subsec:Introduction6}

Literature has
proposed an objective function penalization approach for finding
feasible points that are near the global optimum for the OPF
problem~\cite{lavaei_mesh,lavaei_allerton2014}. The penalization
approach has the advantage of not using potentially
computationally expensive higher-order moment constraints, but has the
disadvantage of requiring the choice of appropriate penalization
parameters. This choice involves a compromise, as the
  parameters must induce a feasible solution to the original problem
  while avoiding large modifications to the problem that would cause
  unacceptable deviation from the global optimum.


The penalization formulation in the existing
literature~\cite{lavaei_allerton2014} generally requires specifying
penalty parameters for both the total reactive power injection and
apparent power flows on certain lines. Penalty parameters
  in the literature range over several orders of magnitude for various
  test cases, and existing literature largely lacks  systematic algorithms for determining appropriate parameter values. Recent work~\cite{cdc2015} proposes a ``moment+penalization'' approach that eliminates the need to choose apparent power flow penalization parameters, but still requires selection of a penalty parameter associated with the total reactive power injection.


This chapter presents an iterative algorithm that builds an objective function intended to yield near-globally-optimal solutions to OPF problems for which the SDP relaxation is not exact. The proposed algorithm first solves the SDP relaxation to obtain a lower bound on the optimal objective value. This lower bound is often very close to the global optimum of many practical OPF problems. The proposed approach modifies the SDP relaxation by adding a constraint that the generation cost must be within a small percentage (e.g., 0.5\%) of this lower bound. This percentage is the single externally specified parameter in the proposed approach.

This constraint on the generation cost provides freedom to specify an
objective function that aims to obtain a \emph{feasible} rather than
\emph{minimum-cost} solution for the OPF problem. In other words, we
desire an objective function such that the SDP relaxation yields a
feasible solution to the original nonconvex OPF problem, with
near-global optimality ensured by the constraint on generation
cost.

This chapter proposes an algorithm for calculating an appropriate
objective function defined using a weighted Laplacian matrix. The
weights are determined iteratively based on the mismatch between
 the solution to the relaxation and the power flows resulting from a related set of voltages. The paper will formalize these concepts and
  demonstrate that this approach results in near global solutions to
many OPF problems, including large test cases. Like many penalization/regularization techniques, the proposed approach is not guaranteed to yield a feasible solution. As supported by the results for several large-scale, realistic test cases, the proposed algorithm broadens the applicability of the SDP relaxation to achieve operating points for many OPF problems that are within specified tolerances for both constraint feasibility and global optimality.

There is related work that chooses the objective function of a
relaxation for the purpose of obtaining a feasible solution for the
original nonconvex problem. For instance,~\cite{lesieutre_molzahn_borden_demarco-allerton2011} specifies objective functions that are
linear combinations of squared voltage magnitudes in order to find
multiple solutions to the power flow equations. Additionally,~\cite{lavaei_pf} proposes a method for determining an
objective function that yields solutions to the power flow equations
for a variety of parameter choices. The objective function
in~\cite{lavaei_pf} is defined by a matrix with three properties:
positive semidefiniteness, a simple eigenvalue of 0, and null space
containing the all-ones vector. We note that the weighted Laplacian
objective function developed in this paper is a special case of an objective function that also has these three properties.

This chapter is organized as follows. Section~\ref{subsec:Laplacian objective function} describes the Laplacian objective function approach that is the main contribution of this chapter. Section~\ref{subsec:Algorithm for determining the Laplacian weights} describes an algorithm for determining the Laplacian weights. Section~\ref{l:results} demonstrates the effectiveness of the proposed approach through application to a variety of small OPF problems as well as several large test cases representing portions of European power systems. Section~\ref{subsec:Conclusion6} concludes the chapter.

\section{Laplacian objective function}
\label{subsec:Laplacian objective function}
The proposed approach exploits the empirical observation that the SDP relaxation provides a very close lower bound on the optimal objective value of many typical OPF problems (i.e., there is a very small \emph{relaxation gap}). For instance, the SDP relaxation gaps for the large-scale Polish~\cite{matpower} and PEGASE~\cite{pegase,josz-2016} systems, which represent portions of European power systems, are all less than 0.3\%.\footnote{To obtain satisfactory convergence of the SDP solver, these systems are pre-processed to remove low-impedance lines (i.e., lines whose impedance values have magnitudes less than $1\times 10^{-3}$ per unit) as in~\cite{cdc2015}.}\footnote{These relaxation gaps are calculated using the objective values from the Shor SDP relaxation and solutions obtained either from the second-order moment relaxation~\cite{mh_sparse_msdp} (where possible) or from \matpower{}~\cite{matpower}.} Further, the SDP relaxation is \emph{exact} (i.e., zero relaxation gap) for the IEEE \mbox{14-,} \mbox{30-,} \mbox{39-,} \mbox{57-bus} systems, the 118-bus system modified to enforce a small minimum line resistance, and several of the large-scale Polish test cases.

We assume that the lower bound $c^*$ provided by the SDP relaxation is within a given percentage $\delta$ of the global optimum to the OPF problem. (Most of the examples in Section~\ref{l:results} specify $\delta = 0.5\%$.) We constrain the generation cost using this assumption by requiring it to be less than or equal to $c^*(1+\delta)$.
Note that the feasible set thus defined is non-empty (i.e., the corresponding SDP problem is feasible) for any choice of $\delta \geq 0$. However, if $\delta$ is too small, there may not exist a rank-one matrix $\mathbf{W}$ (i.e., a feasible point for the original OPF problem) in the feasible space.

The lack of a priori guarantees on the size of the relaxation gap is a challenge that the proposed approach shares with many related approaches for convex relaxations of the optimal power flow problem. Existing sufficient conditions that guarantee zero relaxation gap generally require satisfaction of non-trivial technical conditions and a limited set of network topologies~\cite{low_tutorial,lavaei_mesh}. The SDP relaxation is exact for a significantly broader class of OPF problems than those that have a priori exactness guarantees, and has a small relaxation gap for an even broader class of OPF problems.

There are test cases that are specifically constructed to exhibit somewhat anomalous behavior in order to test the limits of the convex relaxations. The SDP relaxation gap is not small for some of these test cases. For instance, the \mbox{3-bus} system in~\cite{iscas2015}, the \mbox{5-bus} system in~\cite{hicss2014}, and the \mbox{9-bus} system in~\cite{bukhsh2013} have relaxation gaps of 20.6\%, 8.9\%, and 10.8\%, respectively, and the test cases in~\cite{kocuk15} have relaxation gaps as large as 52.7\%. The approach proposed in this paper is not appropriate for such problems. Future progress in convex relaxation theory is required to develop broader conditions that provide a priori certification that the SDP relaxation is exact or has a small relaxation gap. We also await the development of more extensive sets of OPF test cases to further explore the observation that many typical existing practical test cases have small SDP relaxation gaps.

By inserting the original cost function as a constraint, this effectively frees the choice of the objection function (which we will denote $f\left(\mathbf{W} \right)$) to obtain a \emph{feasible} rather than \emph{minimum-cost} solution to the OPF. Here, $\mathbf{W}$ denotes the real SDP matrix variable of $2n \times 2n$ where $n$ is the number of buses.
Any solution with $\mathrm{rank}\left(\mathbf{W}\right) = 1$ yields a feasible solution to the OPF problem within $\delta$ of the globally optimal objective value.
We therefore seek an objective function $f\left(\mathbf{W}\right)$
which maximizes the likelihood of obtaining
$\mathrm{rank}\left( \mathbf{W}\right) = 1$. This section describes a
Laplacian form for the function
$f\left(\mathbf{W}\right)$. Specifically, we consider a $n_l \times
n_l$ diagonal matrix $\mathbf{D}$ containing weights for the network
Laplacian matrix $\mathbf{L} = \mathbf{A}_{inc}^\intercal \mathbf{D}
\mathbf{A}_{inc}$, where $\mathbf{A}_{inc}$ is the $n_l\times n$
incidence matrix for the network. The off-diagonal term
$\mathbf{L}_{ij}$ is equal to the negative of the sum of the weights
for the lines connecting buses~$i$ and~$j$, and the diagonal term
$\mathbf{L}_{ii}$ is equal to the sum of the weights of the lines
connected to bus~$i$. The objective function is
\begin{equation}\label{lap_obj}
f\left(\mathbf{W}\right) = \mathrm{tr}\left(\begin{bmatrix}\mathbf{L} & \mathbf{0}_{n\times n} \\ \mathbf{0}_{n\times n} & \mathbf{L}\end{bmatrix}\mathbf{W} \right).
\end{equation}

The choice of an objective function based on a Laplacian matrix is
motivated by previous literature. An existing penalization
approach~\cite{lavaei_mesh} augments the objective
  function by adding a term that minimizes the total reactive power
injection. This reactive power penalty can
be implemented by adding the term
\begin{equation}\label{Qpen}
\epsilon_b\, \mathrm{tr}\left(
\begin{bmatrix}\hphantom{-}\Re\left(\frac{\mathbf{Y}^H - \mathbf{Y}}{2\mathbf{j}}\right) & \Im\left(\frac{\mathbf{Y}^H - \mathbf{Y}}{2\mathbf{j}}\right) \\
-\Im\left(\frac{\mathbf{Y}^H - \mathbf{Y}}{2\mathbf{j}}\right) & \Re\left(\frac{\mathbf{Y}^H - \mathbf{Y}}{2\mathbf{j}}\right)
\end{bmatrix} \mathbf{W}\right)
\end{equation}
to the objective function of the Shor relaxation of the OPF, where
$\epsilon_b$ is a specified penalty parameter,
$\left(\cdot\right)^H$ indicates conjugate transpose, and $ \mathbf{Y}$ is the admittance matrix. In the absence of phase-shifting transformers (i.e.,
$\theta_{lm} = 0 \quad \forall \left(l,m\right) \in \mathcal{L}$), the
matrix $ \frac{\mathbf{Y}^H - \mathbf{Y}}{2\mathbf{j}}$ is equivalent
to $-\Im\left(\mathbf{Y}\right) = -\mathbf{B}$, which is a
\emph{weighted Laplacian matrix} (with weights determined by the
branch susceptance parameters $b_{lm} = \frac{-X_{lm}}{R_{lm}^2 +
  X_{lm}^2}$) plus a diagonal matrix composed of shunt susceptances.

Early work on SDP relaxations of OPF problems~\cite{lavaei-low-2012} advocates enforcing a minimum resistance of $\epsilon_r$ for all lines in the network. For instance, the SDP relaxation fails to be exact for the IEEE 118-bus system~\cite{ieee_test_cases}, but the relaxation is exact after enforcing a minimum line resistance of $\epsilon_r = 1\times 10^{-4}$~per~unit. After enforcing a minimum line resistance, the active power losses are given by
\begin{equation}
\mathrm{tr}\left(\begin{bmatrix}
\hphantom{-}\Re\left(\frac{\mathbf{Y}_r + \mathbf{Y}_r^H}{2}\right) & \Im\left(\frac{\mathbf{Y}_r + \mathbf{Y}_r^H}{2}\right) \\
-\Im\left(\frac{\mathbf{Y}_r + \mathbf{Y}_r^H}{2}\right) & \Re\left(\frac{\mathbf{Y}_r + \mathbf{Y}_r^H}{2}\right) \end{bmatrix}\mathbf{W}\right)
\end{equation}
where $\mathbf{Y}_r$ is the network admittance matrix after enforcing a minimum branch resistance of $\epsilon_r$. In the absence of phase-shifting transformers, $\frac{\mathbf{Y}_r + \mathbf{Y}_r^H}{2}$ is equivalent to $\Re\left(\mathbf{Y}_r\right)$, which is a \emph{weighted Laplacian matrix} (with weights determined by the branch conductance parameters $g_{lm} = \frac{R_{lm}}{R_{lm}^2 + X_{lm}^2}$) plus a diagonal matrix composed of shunt conductances. Since typical OPF problems have objective functions that increase with active power losses, enforcing minimum line resistances is similar to a weighted Laplacian penalization.\footnote{Note that since enforcing minimum line resistances also affects the power injections and the line flows, the minimum line resistance cannot be solely represented as a Laplacian penalization of the objective function.}

The proposed objective function~\eqref{lap_obj} is equivalent to a linear combination of certain components of $\mathbf{W}$:
\begin{align}\nonumber
f\left(\mathbf{W}\right) = \sum_{\left(l,m\right) \in \mathcal{L}}&  \mathbf{D}_{\left(l,m\right)} \left(\mathbf{W}_{ll} - 2\mathbf{W}_{lm} + \mathbf{W}_{mm} \right. \\
&  \left. + \mathbf{W}_{l+n,l+n} - 2\mathbf{W}_{l+n,m+n} + \mathbf{W}_{m+n,m+n} \right)
\end{align}
where $\mathbf{D}_{\left(l,m\right)}$ is the diagonal element of $\mathbf{D}$ corresponding to the line from bus~$l$ to bus~$m$. 

From a physical perspective, the Laplacian objective's tendency to
reduce voltage differences is similar to both the reactive power
penalization proposed in~\cite{lavaei_mesh} and the minimum branch
resistance advocated in~\cite{lavaei-low-2012}. For typical operating
conditions, reactive power injections are closely related to voltage
magnitude differences, so penalizing reactive power injections tends
to result in solutions with similar voltages. Likewise,
the active power losses associated with line resistances increase with
the square of the current flow through the line, which is determined
by the voltage difference across the line. Thus, enforcing a minimum
line resistance tends to result in solutions with smaller voltage
differences in order to reduce losses.

In addition, a Laplacian regularizing term has been used to obtain desirable solution characteristics for a variety of other optimization problems (e.g., machine learning problems~\cite{smola2003,melacci2011}, sensor network localization problems~\cite{weinberger2006}, and analysis of flow networks~\cite{taylor2011}). The Laplacian matrix is also used for image reconstruction problems~\cite{coifman-2008}.


\section{Determining Laplacian weights}
\label{subsec:Algorithm for determining the Laplacian weights}

Having established a weighted Laplacian form for the objective function, we introduce an iterative algorithm for determining appropriate weights $\mathbf{D}$ for obtaining a solution with $\mathrm{rank}\left(\mathbf{W} \right) = 1$. We note that the proposed algorithm is similar in spirit to the method in~\cite[Section~2.4]{candes-2013}, which iteratively updates weighting parameters to promote low-rank solutions of SDPs related to image reconstruction problems.

The proposed algorithm is inspired by the apparent power line flow
penalty used in~\cite{lavaei_allerton2014} and the iterative approach
to determining appropriate buses for enforcing higher-order moment
constraints in~\cite{mh_sparse_msdp}. The approach
in~\cite{lavaei_allerton2014} penalizes the apparent power flows on
lines associated with certain submatrices of $\mathbf{W}$ that are not
rank one.\footnote{The submatrices are determined by the maximal
  cliques of a chordal supergraph of the network;
  see~\cite{jabr11,dan2013,lavaei_allerton2014}
  for further details.} Similar to the approach
in~\cite{lavaei_allerton2014}, the proposed algorithm adds terms to
the objective function that are associated with certain
``problematic lines.''

The heuristic for identifying problematic lines is inspired by the approach used in~\cite{mh_sparse_msdp} to detect ``problematic buses'' for application of higher-order moment constraints. Denote the SDP solution as $\mathbf{W}^\star$ and the closest rank-one matrix to $\mathbf{W}^\star$ as $\mathbf{W}^{\left(1\right)}$. Previous work~\cite{mh_sparse_msdp} compares the power injections associated with $\mathbf{W}^\star$ and $\mathbf{W}^{\left(1\right)}$ to calculate power injection mismatches $S_k^{inj\,mis}$ for each bus~$k \in \mathcal{N}$:
\begin{align}\nonumber
& S_{k}^{inj\,mis} = \\ \label{sdpmis} & \;\; \left| 
\mathrm{tr}\left\lbrace\mathbf{Y}_k \left(\mathbf{W}^\star-\mathbf{W}^{\left(1\right)}\right)\right\rbrace  + \mathbf{j}
\mathrm{tr}\left\lbrace\mathbf{\bar{Y}}_k \left(\mathbf{W}^\star-\mathbf{W}^{\left(1\right)}\right)\right\rbrace\right|
\end{align}
where $\left|\,\cdot\, \right|$ denotes the magnitude of the complex argument. In the parlance of~\cite{mh_sparse_msdp}, problematic buses are those with large power injection mismatches $S_{k}^{inj\,mis}$.

To identify problematic lines rather than buses, we modify~\eqref{sdpmis} to calculate apparent power flow mismatches $S_{\left(l,m\right)}^{flow\,mis}$ for each line~$\left(l,m\right) \in \mathcal{L}$:
\begin{align}\nonumber
& S_{\left(l,m\right)}^{flow\,mis} = \\ \nonumber &\;\; \left| 
\mathrm{tr}\left\lbrace\mathbf{Z}_{lm} \left(\mathbf{W}^\star-\mathbf{W}^{\left(1\right)}\right)\right\rbrace + \mathbf{j}
\mathrm{tr}\left\lbrace\mathbf{\bar{Z}}_{lm} \left(\mathbf{W}^\star-\mathbf{W}^{\left(1\right)}\right)\right\rbrace\right| \\ \label{flowmis}
&\;\; + \left| 
\mathrm{tr}\left\lbrace\mathbf{Z}_{ml} \left(\mathbf{W}^\star-\mathbf{W}^{\left(1\right)}\right)\right\rbrace + \mathbf{j}
\mathrm{tr}\left\lbrace\mathbf{\bar{Z}}_{ml} \left(\mathbf{W}^\star-\mathbf{W}^{\left(1\right)}\right)\right\rbrace\right|.
\end{align}
Observe that $S_{\left(l,m\right)}^{flow\,mis}$ sums the magnitude of the apparent power flow mismatches at both ends of each line.

The condition $\mathrm{rank}\left(\mathbf{W}^\star\right) = 1$ (i.e., ``feasibility'' in this context) is considered satisfied for practical purposes using the criterion that the maximum line flow and power injection mismatches (i.e., $\max_{\left(l,m\right)\in\mathcal{L}} S_{\left(l,m\right)}^{flow\,mis}$ and $\max_{k\in\mathcal{N}}S_{k}^{inj\,mis}$) are less than specified tolerances~$\epsilon_{flow}$ and~$\epsilon_{inj}$, respectively, and the voltage magnitude limits are satisfied to within a specified tolerance~$\epsilon_{V}$.\footnote{For all test cases, the voltage magnitude limits were satisfied whenever the power injection and line flow mismatch tolerances were achieved.}

\newfloat{algorithm}{bt}{lop}
\begin{algorithm}
\caption{Iterative Algorithm for Determining Weights}\label{a:weights}
\begin{algorithmic}[1]
\State \textbf{Input}: tolerances $\epsilon_{flow}$ and $\epsilon_{inj}$, max relaxation gap $\delta$
\State Set $\mathbf{D} = \mathbf{0}_{n_l\times n_l}$
\State Solve the Shor SDP relaxation to obtain $c^*$
\State Calculate $S^{flow\,mis}$ and $S^{inj\,mis}$ using~\eqref{flowmis} and~\eqref{sdpmis}
\NoDo
\While{termination criteria not satisfied} 
	\State Update weights: $\mathbf{D} \leftarrow \mathbf{D} + \mathrm{diag}\left(S^{flow\,mis}\right)$
	\State Solve the generation-cost-constrained relaxation
	\State Calculate $S^{flow\,mis}$ and $S^{inj\,mis}$ using~\eqref{flowmis} and~\eqref{sdpmis}
\EndWhile
\State Calculate the voltage phasors and terminate
\end{algorithmic}
\end{algorithm}

As described in Algorithm~\ref{a:weights}, the weights on the diagonal of $\mathbf{D}$ are determined from the line flow mismatches $S_{\left(l,m\right)}^{flow\,mis}$. Specifically, the proposed algorithm first solves the Shor relaxation to obtain both the lower bound $c^*$ on the optimal objective value and the initial line flow and power injection mismatches $S_{\left(l,m\right)}^{flow\,mis},\, \forall \left(l,m\right) \in\mathcal{L}$ and $S_{k}^{inj\,mis},\, \forall k\in\mathcal{N}$. 

While the termination criteria ($\max_{\left(l,m\right)\in\mathcal{L}} \big\{ S_{\left(l,m\right)}^{flow\,mis} \big\} <
\epsilon_{flow}$, $\max_{k\in\mathcal{N}}  \big\{
  S_{k}^{inj\,mis} \big\} < \epsilon_{inj}$, and no
  voltage limits violated by more than $\epsilon_V$) are not
satisfied, the algorithm solves the SDP
relaxation with the constraint ensuring that the generation cost is
within $\delta$ of the lower bound. The objective function is defined
using the weighting matrix $\mathbf{D} =
\mathrm{diag}\left(S^{flow\,mis}\right)$, where
$\mathrm{diag}\left(\cdot\right)$ denotes the matrix with the vector
argument on the diagonal and other entries equal to zero. Each
iteration adds the line flow mismatch vector $S^{flow\,mis}$ to the previous weights (i.e.,
$\mathbf{D} \leftarrow \mathbf{D} +
\mathrm{diag}\left(S^{flow\,mis}\right)$).

Note that Algorithm~\ref{a:weights} is not guaranteed to converge. Non-convergence may be due to the value of $\delta$ being too small (i.e., there does not exist a rank-one solution) or failure to find a rank-one solution that does exist. To address the former case, Algorithm~\ref{a:weights} could be modified to include an ``outer loop'' that increments $\delta$ by a specified amount (e.g., $0.5$\%) if convergence is not achieved in a certain  number of iterations. We note that, like other convex relaxation methods, the proposed approach would benefit from further theoretical work regarding the development of a priori guarantees on the size of the relaxation gap for various classes of OPF problems.

For some problems with large relaxation gaps (e.g., the
  \mbox{3-bus} system in~\cite{iscas2015}, the 5-bus system
  in~\cite{hicss2014}, and the 9-bus system in~\cite{bukhsh2013}), no
  purely penalization-based methods have so far successfully addressed
  the latter case where the proposed algorithm fails to
  find a rank-one solution that satisfies the generation cost
  constraint with 
    sufficiently large $\delta$ (i.e., no known penalty parameters
  yield feasible solutions using the methods
  in~\cite{lavaei_mesh,lavaei_allerton2014} for these test cases). One
  possible approach for addressing this latter case is the combination
  of penalization techniques with Lasserre's moment relaxation
  hierarchy~\cite{pscc2014,cedric_tps,ibm_paper,mh_sparse_msdp}. The
  combination of the moment relaxations with the penalization methods
  enables the computation of near-globally-optimal solutions for a
  broader class of OPF problems than either method achieves
  individually. See~\cite{cdc2015} for further details on this
  approach.

We note that despite the lack of a convergence guarantee, the examples in Section~\ref{l:results} demonstrate that Algorithm~\ref{a:weights} is capable of finding feasible points that are near the global optimum for many OPF problems, including large test cases.

\section{Numerical results}
\label{l:results}

This section demonstrates the effectiveness of the
proposed approach using several small example problems as well as
large test cases representing portions of European power
systems. The SDP relaxation yields a small but non-zero
  relaxation gap for the test cases selected in this section, and
  Algorithm~\ref{a:weights} yields points that are feasible
  for the OPF (to within the specified termination
    criteria) and that are near the global optimum for these test
  cases. For other test cases with a large SDP relaxation gap, such as those mentioned earlier in~\cite{bukhsh2013,hicss2014,iscas2015,kocuk15}, the proposed algorithm does
  not converge when tested with a variety of values for $\delta$.

The results in this section use line flow and power injection mismatch tolerances $\epsilon_{flow}$ and $\epsilon_{inj}$ that are both equal to $1$~MVA and $\epsilon_V = 5\times 10^{-4}$~per~unit. The implementation of Algorithm~\ref{a:weights} uses \mbox{MATLAB 2013a}, YALMIP \mbox{2015.02.04}~\cite{yalmip}, and Mosek \mbox{7.1.0.28}~\cite{mosek}, and was solved using a computer with a quad-core 2.70~GHz processor and 16~GB of RAM.

Applying Algorithm~\ref{a:weights} to several small- to medium-size
test cases
from~\cite{lesieutre_molzahn_borden_demarco-allerton2011,hicss2014,mh_sparse_msdp,ieee_test_cases,kocuk15}
yields the results shown in
Table~\ref{t:small_results}. Tables~\ref{t:large_gencost_results}
and~\ref{t:large_loss_results} show the results from applying
Algorithm~\ref{a:weights} to large test cases which minimize
generation cost and active power losses, respectively. These test
cases, which are from~\cite{matpower} and~\cite{pegase,josz-2016}, represent
portions of European power systems. The Shor
relaxation has a small but non-zero relaxation gap
for all test cases considered in this section. The columns of
Tables~\ref{t:small_results}--\ref{t:large_loss_results} show the case
name and reference, the number of iterations of
Algorithm~\ref{a:weights}, the final maximum apparent power flow
mismatch $\max_{\left(l,m\right)\in\mathcal{L}} \big\{
  S_{\left(l,m\right)}^{flow\,mis} \big\}$ in MVA, the final maximum
power injection mismatch $\max_{k\in\mathcal{N}}  \big\{
  S_{k}^{inj\,mis} \big\}$ in MVA, the specified value of $\delta$, an
  upper bound on the relaxation gap from the solution to the Shor
  relaxation, and the total solver time in seconds.

Note that the large test cases in Tables~\ref{t:large_gencost_results}
and~\ref{t:large_loss_results} were preprocessed to remove
low-impedance lines as described in~\cite{cdc2015} in order to improve
the numerical convergence of the SDP relaxation. Lines
  which have impedance magnitudes less than a threshold
  (\texttt{thrshz} in~\cite{cdc2015}) of $1\times 10^{-3}$~per~unit
  are eliminated by merging the terminal buses. Table~\ref{tab:size}
  describes the number of buses and lines before and after this
  preprocessing. Low-impedance line preprocessing was not needed for
  the test cases in Table~\ref{t:small_results}. After preprocessing,
  MOSEK's SDP solver converged with sufficient accuracy to yield
  solutions that satisfied the voltage magnitude limits to within
  \mbox{$\epsilon_V = 1\times 10^{-4}$}~per~unit and the 
    power injection and line flow constraints to within the
  corresponding mismatches shown in
  Tables~\ref{t:small_results}--\ref{t:large_loss_results}.

These results show that Algorithm~\ref{a:weights} finds feasible points (within the specified tolerances) that have objective values near the global optimum for a variety of test cases.  Further, Algorithm~\ref{a:weights} globally solves all OPF problems for which the Shor relaxation is exact (e.g., many of the IEEE test cases~\cite{lavaei-low-2012}, several of the Polish test systems~\cite{dan2013}, and the 89-bus PEGASE system~\cite{pegase,josz-2016}). Thus, the algorithm is a practical approach for addressing a broad class of OPF problems. 

We note, however, that Algorithm~\ref{a:weights} does not yield a feasible point for all OPF problems. For instance, the test case WB39mod from~\cite{bukhsh2013} has line flow and power injection mismatches of 18.22~MVA and 12.99~MVA, respectively, after 1000 iterations of Algorithm~\ref{a:weights}. The challenge associated with this case seems to result from light loading with limited ability to absorb a surplus of reactive power injections, yielding at least two local solutions. In addition to challenging the method proposed in this paper, no known penalty parameters yield feasible solutions to this problem. Generalizations of the SDP relaxation using the Lasserre hierarchy have successfully calculated the global solution to this case~\cite{mh_sparse_msdp,cdc2015}. Further, while Algorithm~\ref{a:weights} converges for five of the seven test cases in~\cite{kocuk15} which have small relaxation gaps (less than 2.5\%), the algorithm fails for two other such test cases as well as several other test cases in~\cite{kocuk15} which have large relaxation gaps. We note that the tree topologies used in the test cases in~\cite{kocuk15} are a significant departure from the mesh networks used in the standard test cases from which they were derived; the proposed algorithm succeeds for several test cases that share the original network topologies.

\begin{table}[t]
\centering
\caption{Results for Small and Medium Size Test Cases}
\label{t:small_results}
\resizebox{\columnwidth}{!}{%
\begin{tabular}{|l|c|c|c|c|c|c|}
\hline 
\multicolumn{1}{|c|}{\textbf{Case}} & \!\!\textbf{Num.}\!\! & \textbf{Max} & \textbf{Max} & \textbf{$\delta$} & \textbf{Max} & \!\!\textbf{Solver}\!\!\\
\multicolumn{1}{|c|}{\textbf{Name}} & \!\!\textbf{Iter.}\!\! & \!\!\textbf{Flow Mis.}\!\! & \!\!\textbf{Inj. Mis.}\!\! & \textbf{(\%)} & \!\!\textbf{Relax.}\!\! & \textbf{Time} \\
& & \textbf{(MVA)} & \textbf{(MVA)} & & \!\!\textbf{Gap (\%)}\!\! & \textbf{(sec)} \\ \hline\hline
\!\!\! LMBD3~\cite{lesieutre_molzahn_borden_demarco-allerton2011}\!\!\! & 1 & $1.3${\textsc e}${-5}$ & $1.6${\textsc e}${-5}$ & 0.5 & 0.50 & 0.7 \\ \hline
\!\!\! MLD3~\cite{hicss2014}\!\!\! & 1 & $7.3${\textsc e}${-6}$ & $7.2${\textsc e}${-5}$ & 0.5 & 0.50 & 0.5 \\ \hline
\!\!\! MH14Q~\cite{mh_sparse_msdp}\!\!\! & 2 & $1.8${\textsc e}${-5}$ & $9.9${\textsc e}${-6}$ & 0.5 & 0.02 & 1.2 \\ \hline
\!\!\! MH14L~\cite{mh_sparse_msdp}\!\!\! & 2 & $8.1${\textsc e}${-5}$ & $7.8${\textsc e}${-5}$ & 0.5 & 0.33 & 1.2 \\ \hline
\!\!\! KDS14Lin~\cite{kocuk15}\!\!\! & 1 & $1.2${\textsc e}${-3}$ & $9.2${\textsc e}${-4}$ & 1.0 & 1.00 & 0.7 \\ \hline
\!\!\! KDS14Quad~\cite{kocuk15}\!\!\! & 1 & $1.4${\textsc e}${-4}$ & $8.4${\textsc e}${-5}$ & 1.0 & 1.00 & 0.6 \\ \hline
\!\!\! KDS30Lin~\cite{kocuk15}\!\!\! & 7 & $9.3${\textsc e}${-1}$ & $9.2${\textsc e}${-1}$ & 2.5 & 2.50 & 4.6 \\ \hline
\!\!\! KDS30Quad~\cite{kocuk15}\!\!\! & 6 & $8.1${\textsc e}${-1}$ & $8.0${\textsc e}${-1}$ & 2.0 & 2.00 & 3.6 \\ \hline
\!\!\! KDS30IEEEQuad~\cite{kocuk15}\!\!\! & 100 & $9.5${\textsc e}${-1}$ & $7.2${\textsc e}${-1}$ & 2.5 & 2.50 & 129.8 \\ \hline
\!\!\! MH39L~\cite{mh_sparse_msdp}\!\!\! & 1 & $1.3${\textsc e}${-2}$ & $9.8${\textsc e}${-3}$ & 0.5 & 0.27 & 0.7 \\ \hline
\!\!\! MH57Q~\cite{mh_sparse_msdp}\!\!\! & 1 & $1.2${\textsc e}${-3}$ & $6.9${\textsc e}${-4}$ & 0.5 & 0.03 & 0.7 \\ \hline
\!\!\! MH57L~\cite{mh_sparse_msdp}\!\!\! & 1 & $3.2${\textsc e}${-4}$ & $5.2${\textsc e}${-4}$ & 0.5 & 0.16 & 0.9 \\ \hline
\!\!\! MH118Q~\cite{mh_sparse_msdp}\!\!\! & 2 & $3.3${\textsc e}${-3}$ & $2.7${\textsc e}${-3}$ & 0.5 & 0.50 & 2.6 \\ \hline
\!\!\! MH118L~\cite{mh_sparse_msdp}\!\!\! & 2 & $3.1${\textsc e}${-3}$ & $3.1${\textsc e}${-3}$ & 1.0 & 1.00 & 3.3\\ \hline 
\!\!\! IEEE 300~\cite{ieee_test_cases}\!\!\! & 1 & $1.3${\textsc e}${-1}$ & $1.2${\textsc e}${-1}$ & 0.5 & 0.01 & 3.0\\ \hline
\end{tabular}
}
\end{table}


\begin{table}[t]
\centering
\caption{Results for Large Test Cases that Minimize Generation Cost}
\label{t:large_gencost_results}
\resizebox{\columnwidth}{!}{%
\begin{tabular}{|l|c|c|c|c|c|c|}
\hline 
\multicolumn{1}{|c|}{\textbf{Case}} & \!\!\textbf{Num.}\!\! & \textbf{Max} & \textbf{Max} & \textbf{$\delta$} & \textbf{Max} & \!\!\textbf{Solver}\!\!\\
\multicolumn{1}{|c|}{\textbf{Name}} & \!\!\textbf{Iter.}\!\! & \!\!\textbf{Flow Mis.}\!\! & \!\!\textbf{Inj. Mis.}\!\! & \textbf{(\%)} & \!\!\textbf{Relax.}\!\! & \textbf{Time} \\
& & \textbf{(MVA)} & \textbf{(MVA)} & & \!\!\textbf{Gap (\%)}\!\! & \textbf{(sec)} \\ \hline\hline
\!\!\! PL-2383wp~\cite{matpower}\!\!\! & 2 & 0.54 & 0.50 & 0.5 & 0.50 & 78.6 \\ \hline
\!\!\! PL-3012wp~\cite{matpower}\!\!\! & 2 & 0.36 & 0.27 & 0.5 & 0.50 & 107.6 \\ \hline
\!\!\! PL-3120sp~\cite{matpower}\!\!\! & 2 & 0.56 & 0.33 & 0.5 & 0.50 & 84.2 \\ \hline
\end{tabular}
}
\end{table}

\begin{table}[t!]
\centering
\caption{Results for Large Test Cases that Minimize Active Power Loss}
\label{t:large_loss_results}
\resizebox{\columnwidth}{!}{%
\begin{tabular}{|l|c|c|c|c|c|c|}
\hline 
\multicolumn{1}{|c|}{\textbf{Case}} & \!\!\textbf{Num.}\!\! & \textbf{Max} & \textbf{Max} & \textbf{$\delta$} & \textbf{Max} & \!\!\textbf{Solver}\!\!\\
\multicolumn{1}{|c|}{\textbf{Name}} & \!\!\textbf{Iter.}\!\! & \!\!\textbf{Flow Mis.}\!\! & \!\!\textbf{Inj. Mis.}\!\! & \textbf{(\%)} & \!\!\textbf{Relax.}\!\! & \textbf{Time} \\
& & \textbf{(MVA)} & \textbf{(MVA)} & & \!\!\textbf{Gap (\%)}\!\! & \textbf{(sec)} \\ \hline\hline
\!\!\! PL-2383wp~\cite{matpower} & 5 & 0.21 & 0.16 & 0.5 & 0.26 & 154.0 \\ \hline
\!\!\! PL-3012wp~\cite{matpower} & 5 & 0.08 & 0.04 & 0.5 & 0.18 & 232.2 \\ \hline
\!\!\! PL-3120sp~\cite{matpower} & 5 & 0.25 & 0.19 & 0.5 & 0.38 & 232.6 \\ \hline
\!\!\! PEGASE-1354~\cite{pegase,josz-2016}\!\!\! & 12 & 0.27 & 0.18 & 0.5 & 0.15 & 199.2 \\ \hline
\!\!\! PEGASE-2869~\cite{pegase,josz-2016}\!\!\! & 38 & 0.91 & 0.69 & 0.5 & 0.15 & 2378.4 \\ \hline
\end{tabular}
}
\end{table}

\begin{table}[t!]
\centering
\caption{Descriptions of Large Test Cases Before and After Low-Impedance Line Preprocessing}
\begin{tabular}{|l|c|c|c|c|}
\hline
\multicolumn{1}{|c|}{\textbf{Case}} & \multicolumn{2}{c|}{\textbf{Before Preprocessing}} & \multicolumn{2}{c|}{\textbf{After Preprocessing}}\\\cline{2-5}
\multicolumn{1}{|c|}{\textbf{Name}} & \textbf{Num.} & \textbf{Num.} & \textbf{Num.} & \textbf{Num.}\\
\multicolumn{1}{|c|}{} & \textbf{\; Buses \;} & \textbf{Lines} & \textbf{\; Buses \;} & \textbf{Lines}\\
\hline
PL-2383wp & 2,383 & 2,869 & 2,177 & 2,690\\\hline
PL-3012wp & 3,012 & 3,572 & 2,292 & 2,851\\\hline
PL-3120sp & 3,120 & 3,693 & 2,314 & 2,886\\\hline
PEGASE-1354 & 1,354 & 1,991 & 1,179 & 1,803\\\hline
PEGASE-2869 & 2,869 & 4,582 & 2,120 & 4,164\\
\hline
\end{tabular}
\label{tab:size}
\end{table}

\begin{figure}[t]
\centering
\includegraphics[totalheight=0.33\textheight]{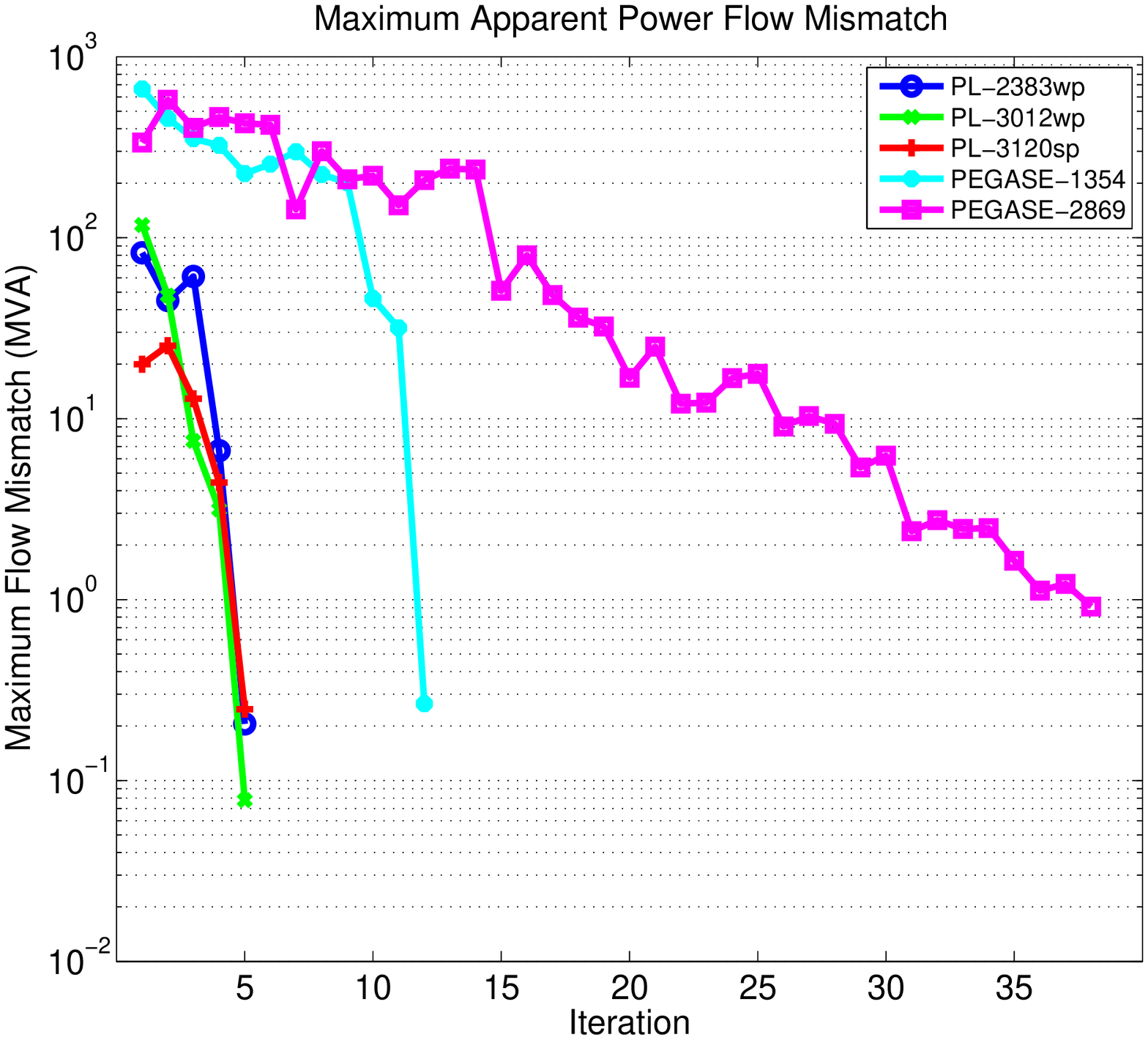}
\caption{Maximum Apparent Power Flow Mismatches versus Iteration of Algorithm~\ref{a:weights} for Active Power Loss Minimizing Test Cases}
\label{f:apparent_flow_mis}
\end{figure}

\begin{figure}[t]
\centering
\includegraphics[totalheight=0.33\textheight]{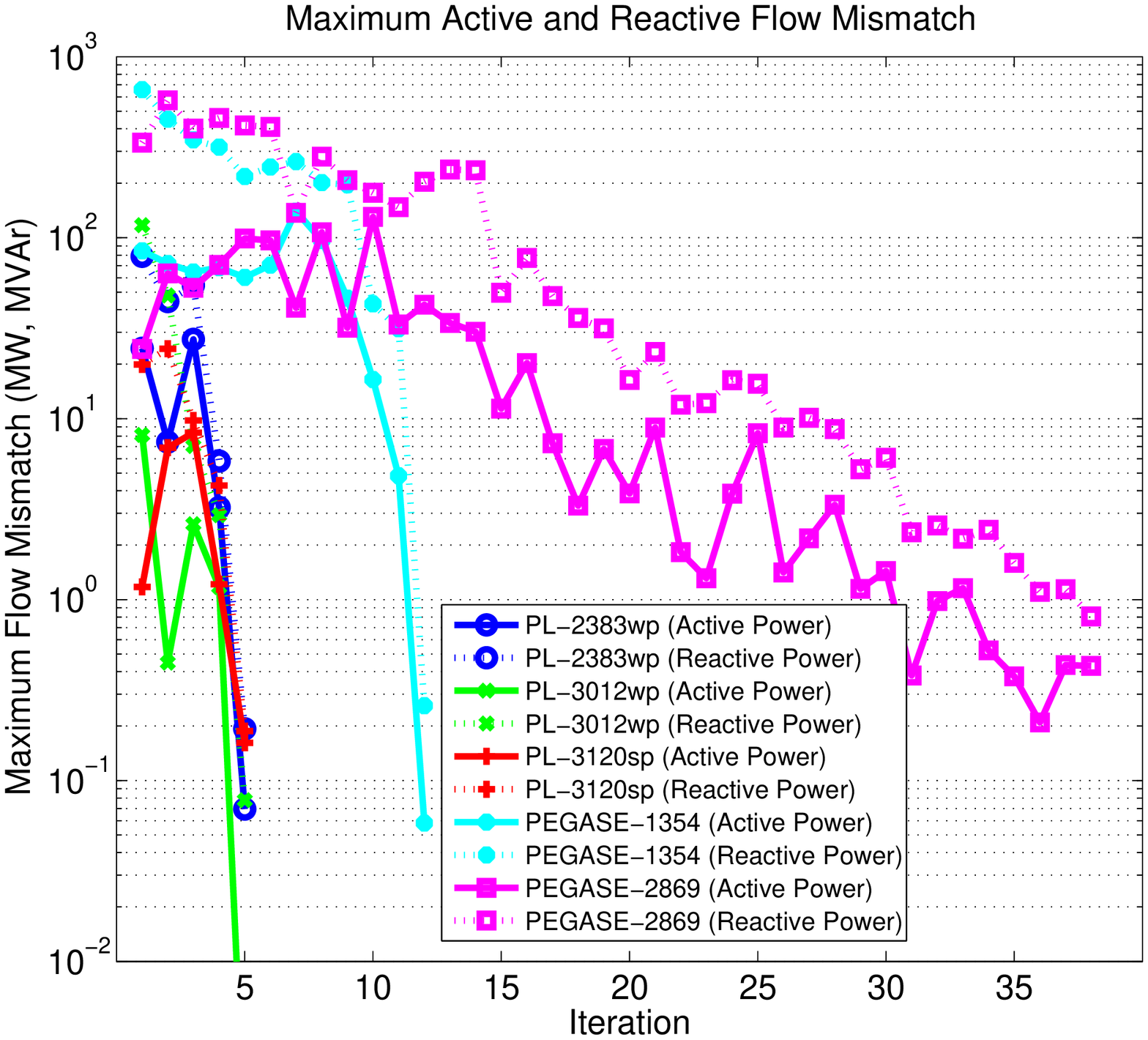}
\caption{Maximum Active and Reactive Power Flow Mismatches versus Iteration of Algorithm~\ref{a:weights} for Active Power Loss Minimizing Test Cases}
\label{f:active_reactive_flow_mis}
\end{figure}

\begin{figure}[t]
\centering
\includegraphics[totalheight=0.33\textheight]{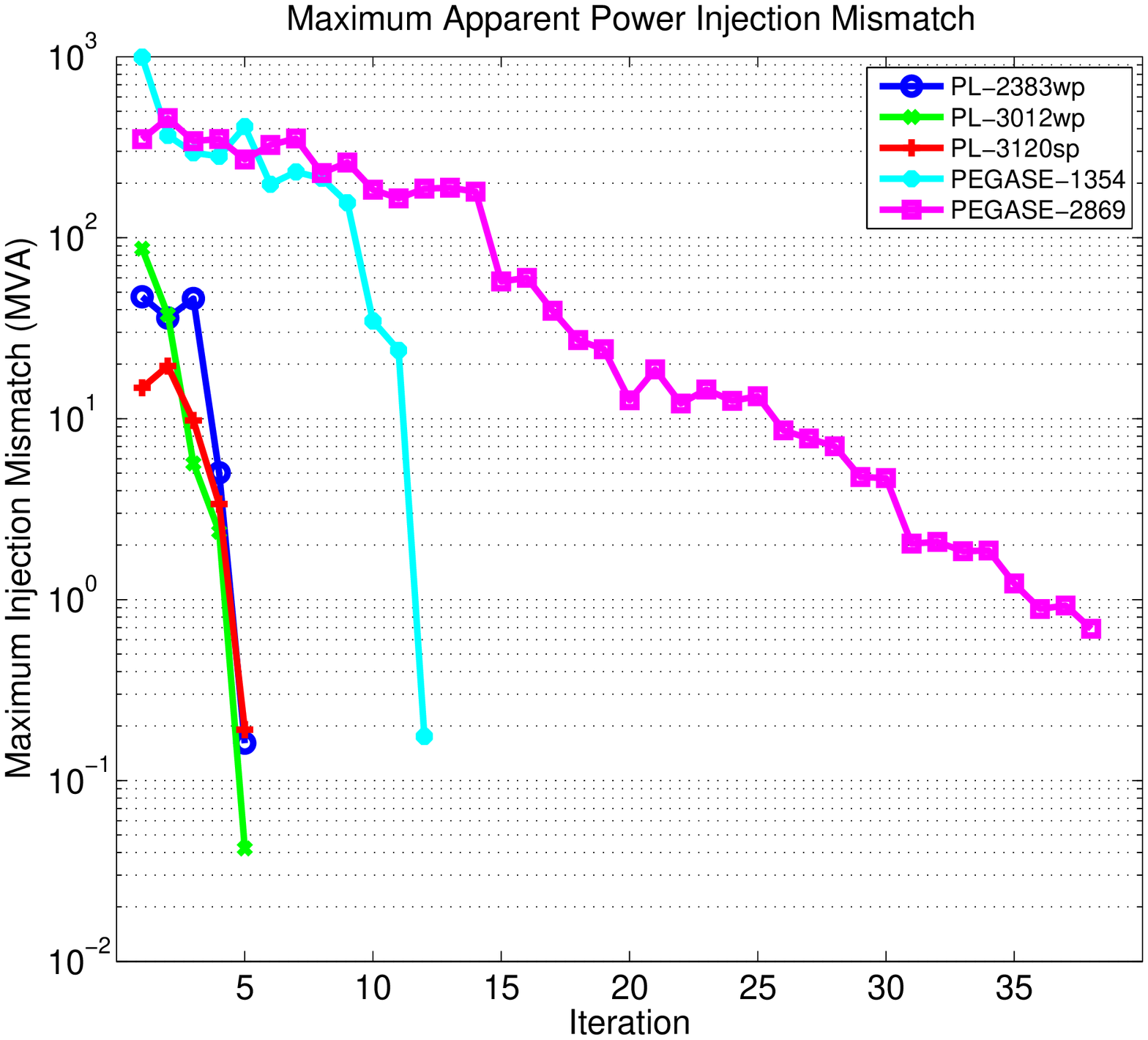}
\caption{Maximum Apparent Power Injection Mismatches versus Iteration of Algorithm~\ref{a:weights} for Active Power Loss Minimizing Test Cases}
\label{f:apparent_injection_mis}
\end{figure}

\begin{figure}[t]
\centering
\includegraphics[totalheight=0.33\textheight]{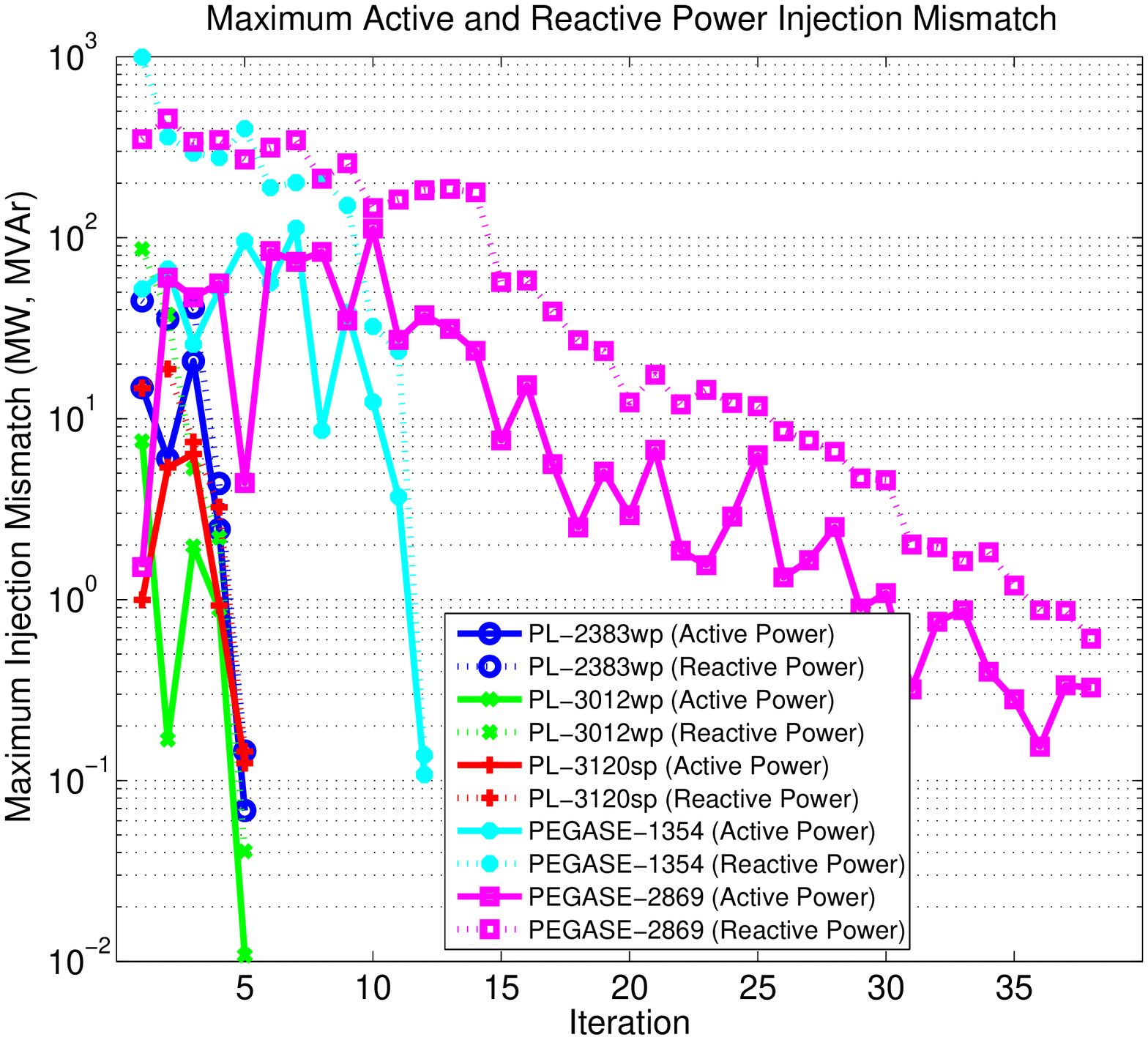}
\caption{Maximum Active and Reactive Power Injection Mismatches versus Iteration of Algorithm~\ref{a:weights} for Active Power Loss Minimizing Test Cases}
\label{f:active_reactive_injection_mis}
\end{figure}

We note that the interior point solver in \matpower{} obtained superior relaxation gaps for the test cases considered in this paper. Within approximately five seconds for the large test cases in Tables~\ref{t:large_gencost_results} and~\ref{t:large_loss_results}, \matpower{} obtained relaxation gaps that ranged from $0.14\%$ to $0.32\%$ smaller than those obtained with Algorithm~\ref{a:weights}.\footnote{Of course, \matpower{} cannot provide any measure of the quality of its solution in terms of a lower bound on the globally optimal objective value whereas Algorithm~\ref{a:weights} provides such guarantees.} This suggests that smaller values of $\delta$ could be used in Algorithm~\ref{a:weights}. Indeed, additional numerical experiments demonstrated that Algorithm~\ref{a:weights} converged with $\delta = 0.25\%$ (half the value used in previous numerical experiments) for all test cases for which the \matpower{} solution indicated that a value of $\delta = 0.25\%$ was achievable.

We select termination parameter values of
  $\epsilon_{flow}$ and $\epsilon_{inj}$ of 1~MVA. This tolerance is
  typically numerically achievable with MOSEK's SDP solver, which
  experience suggests is often a limiting factor to obtaining smaller mismatches.
  
  Note that the
maximum mismatches do not necessarily decrease monotonically with each
iteration of Algorithm~\ref{a:weights}. Figs.~\ref{f:apparent_flow_mis}
and~\ref{f:active_reactive_flow_mis} show the maximum flow mismatches (on a logarithmic scale) for the test cases that minimize active power losses (cf.~Table~\ref{t:large_loss_results}). Likewise, Figs.~\ref{f:apparent_injection_mis} and~\ref{f:active_reactive_injection_mis} show the maximum power injection mismatches for the same test cases. Although the mismatches do not always decrease monotonically, there is a generally decreasing trend which results in satisfaction of the termination criteria for each test case. At each iteration, Algorithm~\ref{a:weights} yields larger reactive power mismatches than active power mismatches for these test cases.

Note that it is not straightforward to compare the computational costs of the Laplacian objective approach and other penalization approaches in the literature~\cite{lavaei_mesh,lavaei_allerton2014}. A single solution of the penalized SDP relaxations in~\cite{lavaei_allerton2014} requires approximately the same computational effort as one iteration of Algorithm~\ref{a:weights}. Thus, if one knows appropriate penalty parameters, the method in~\cite{lavaei_allerton2014} is faster for problems where the SDP relaxation is not exact. However, the key advantage of the proposed approach is that there is no need to specify any parameters other than the value of $\delta$ used in the generation cost constraint. In contrast, the literature largely lacks systematic approaches for identifying appropriate parameter values for the penalization methods in~\cite{lavaei_mesh,lavaei_allerton2014}.

\section{Conclusion}
\label{subsec:Conclusion6}

The SDP relaxation of~\cite{lavaei-low-2012} is capable of globally solving
a variety of OPF problems. To address a broader class of OPF problems,
this paper has described an approach that finds feasible points with
objective values that are within a specified percentage of the global
optimum. Specifically, the approach in this paper adds a constraint to
ensure that the generation cost is within a small specified percentage
of the lower bound obtained from the SDP relaxation. This constraint
frees the objective function to be chosen to yield a \emph{feasible}
(i.e., rank-one) solution rather than a \emph{minimum-cost}
solution. Inspired by previous penalization approaches and results in
the optimization literature, an objective function based on a weighted
Laplacian matrix is selected. The weights for this matrix are
iteratively determined using ``line flow mismatches.'' The proposed
approach is validated through successful application to a variety of
both small and large test cases, including several OPF problems
representing large portions of European power systems. There are,
however, test cases for which the approach takes many iterations to
converge or does not converge at all.

Future work includes modifying the algorithm for choosing the weights
in order to more consistently require fewer iterations. Also, future
work includes testing alternative SDP solution approaches with ``hot
start'' capabilities to improve computational efficiency by leveraging
knowledge of the solution to a ``nearby'' problem from the previous
iteration of the algorithm. Future work also includes extension of the
algorithm to a broader class of OPF problems, such as the test case
WB39mod from~\cite{bukhsh2013} and several examples in~\cite{kocuk15}.

In this chapter and the one before that, two methods were proposed to find nearly global solutions to large-scale OPF problems with generation cost minimization. In the case of active power loss minimization, Lasserre's hierarchy finds the global solution to large-scale problems, provided the data is preprocessed to avoid bad conditioning. The difference between the two cases is that active power loss is a convex function of voltages, whereas generation cost is not. In the next chapter, Lasserre's hierarchy is transposed to complex numbers in order to reduce its computional cost for OPF problems. Moreover, the sparsity exploiting algorithm from~\cite{mh_sparse_msdp} designed for OPF problems is formalized for general polynomial optimization problems with either real or complex variables.

\chapter{Complex hierarchy for enhanced tractability}
\label{sec:Complex hierarchy for enhanced tractability}

We consider the problem of finding the global optimum of a real-valued complex polynomial $f : z \in \mathbb{C}^n \longmapsto \sum_{\alpha,\beta} f_{\alpha,\beta} \bar{z}^\alpha z^\beta \in \mathbb{R}$ ($z^\alpha := z_1^{\alpha_1} \hdots z_n^{\alpha_n} ,~ \overline{f_{\alpha,\beta}} = f_{\beta,\alpha}$) on a compact set defined by real-valued complex polynomial inequalities. It reduces to solving a sequence of complex semidefinite optimization relaxations that grow tighter and tighter thanks to D'Angelo's and Putinar's Positivstellenstatz discovered in 2008. In other words, the Lasserre hierarchy may be transposed to complex numbers. We propose a method for exploiting sparsity and apply the complex hierarchy to problems with several thousand complex variables. These problems consist of computing optimal power flows in the European high-voltage transmission network. The material presented in this chapter is based on the submitted manuscript:
\\\\
\noindent {\scshape C. Josz, D. K. Molzahn}, \textit{Moment/Sum-of-Squares Hierarchy for Complex Polynomial Optimization}, submitted to Society for Industrial and Applied Mathematics, Journal on Optimization. \href{http://arxiv.org/pdf/1508.02068v1.pdf}{[preprint]}

\section{Introduction}
\label{subsec:Introduction7}

Multivariate polynomial optimization where variables and data are complex numbers is a non-deterministic polynomial-time hard problem that arises in various applications such as electric power systems (Section \ref{subsec:Numerical Results7}), imaging science~\cite{singer-2011,candes-2013,bandeira-2014,fogel-2014}, signal processing~\cite{maricic-2003,aittomaki-2009,chen-2009,luo-2010,li-2012,aubry-2013}, automatic control~\cite{toker-1998}, and quantum mechanics~\cite{hilling-2010}. Complex numbers are typically used to model oscillatory phenomena which are omnipresent in physical systems. Although complex polynomial optimization problems can readily be converted into real polynomial optimization problems where variables and data are real numbers, efforts have been made to find \textit{ad hoc} solutions to complex problems~\cite{sorber-2012,jiang-2014,jiang-2015}. The observation that relaxing nonconvex constraints and converting from complex to real numbers are two non-commutative operations motivates our work. This leads us to transpose to complex numbers Lasserre's moment/sum-of-squares hierarchy~\cite{lasserre-2001} for real polynomial optimization.

The moment/sum-of-squares hierarchy succeeds to the vast development of real algebraic geometry during the twentieth century~\cite{prestel-2001}. In 1900, Hilbert's seventeenth problem~\cite{schmudgen-2012} raised the question of whether a non-negative polynomial in multiple real variables can be decomposed as a sum of squares of fractions of polynomials, to which Artin~\cite{artin-1927} answered in the affirmative in 1927. Later, positive polynomials on sets defined by a finite number of polynomial inequality constraints were investigated by Krivine~\cite{krivine-1964}, Stengle~\cite{stengle-1974}, Schm\"udgen~\cite{schmudgen-1991}, and Putinar~\cite{putinar-1993}. A theorem concerning such polynomials is referred to as \textit{Positivstellensatz}~\cite{scheiderer-2009}. For instance, Putinar proved under an assumption slightly stronger than compactness that they can be decomposed as a weighted sum of the constraints where the weights are sums of squares of polynomials. Lasserre~\cite{lasserre-2000,lasserre-2001,lasserre_book} used this result in 2001 to develop a hierarchy of semidefinite programs to solve real polynomial optimization problems with compact feasible sets, with Parrilo~\cite{parrilo-2000b,parrilo-2003} making a similar contribution independently. In order to satisfy the assumption made by Putinar, Lasserre proposed to add a redundant ball constraint $x_1^2 + \hdots + x_n^2 \leqslant R^2$ to the description of the feasible set when it is included in a ball of radius $R$. Subsequent work on the hierarchy includes its comparison with lift-and-project methods~\cite{laurent-2003}, a new proof of Putinar's Positivstellensatz via a 1928 theorem of P\'olya~\cite{schweighofer-2005}, and a proof of generically finite convergence~\cite{nie-2014}.

In 1968, Quillen~\cite{quillen-1968} showed that a real-valued bihomogenous complex polynomial that is positive away from the origin can be decomposed as a sum of squared moduli of holomorphic polynomials when it is multiplied by $(|z_1|^2 + \hdots + |z_{n}|^2)^r$ for some $r\in \mathbb{N}$. The result was rediscovered years later by Catlin and D'Angelo~\cite{catlin-1996} and ignited a search for complex analogues of Hilbert's seventeenth problem~\cite{angelo-2002,angelo-2010} and the ensuing Positivstellens\"atze~\cite{putinar-2006,putinar-2012,putinar-scheiderer-2012,putinar-2013}. Notably, D'Angelo and Putinar~\cite{angelo-2008} proved in 2008 that a positive complex polynomial on a sphere intersected by a finite number of polynomial inequality constraints can be decomposed as a weighted sum of the constraints where the weights are sums of squared moduli of holomorphic polynomials.
Similar to Lasserre, we use D'Angelo's and Putinar's Positivstellensatz to construct a complex moment/sum-of-squares hierarchy of semidefinite programs to solve complex polynomial optimization problems with compact feasible sets. To satisfy the assumption in the Positivstellensatz, we propose to add a slack variable $z_{n+1} \in \mathbb{C}$ and a redundant constraint $|z_1|^2 + \hdots + |z_{n+1}|^2 = R^2$ to the description of the feasible set when it is in a ball of radius $R$. The complex hierarchy is more tractable than the real hierarchy yet produces potentially weaker bounds. Computational advantages are shown using the optimal power flow problem in electrical engineering.

Below, Section \ref{subsec:Motivation} uses Shor and second-order conic relaxations to motivate the construction of a complex moment/sum-of-squares hierarchy in Section \ref{sec:Complex Moment/Sum-of-Squares Hierarchy}. Using a sparsity-exploiting method, numerical experiments on the optimal power flow problem are presented in Section \ref{subsec:Numerical Results7}. Section \ref{sec:Conclusion7} concludes our work.

\section{Motivation}
\label{subsec:Motivation}
Let $\mathbb{N}$, $\mathbb{N}^*$, $\mathbb{R}$, $\mathbb{R}_+$ and $\mathbb{C}$ denote the set of natural, positive natural, real, non-negative real, and complex numbers respectively. Also, let ``$\textbf{i}$'' denote the imaginary unit and $\mathbb{H}_n$ denote the set of Hermitian matrices of order $n\in \mathbb{N}^*$. Let's begin with the subclass of complex polynomial optimization composed of quadratically-constrained quadratic programs
\begin{subequations}
\label{eq:qcqp}
\begin{gather}
\text{QCQP-}\mathbb{C}~\text{:} ~~~
\inf_{z \in \mathbb{C}^n}~ z^H H_0 z, ~~~~~~~~~~~~~~~~~~~~~~~~~~~~~~~~~~~~~~ \label{eq:qcqpC1} \\
~~~~~~~~ \text{s.t.} ~~
z^H H_i z \leqslant h_i,~~~~ i=1, \hdots, m, \label{eq:qcqpC2}
\end{gather}
\end{subequations}
where $m \in \mathbb{N}^*$, $H_0,\hdots,H_m \in \mathbb{H}_n$, $h_0,\hdots,h_m \in \mathbb{R}$, and $\left(\cdot\right)^H$ denotes the conjugate transpose. The feasible set is not assumed to contain a point (i.e. it may be empty). The Shor~\cite{shor-1987b} and second-order conic relaxations of QCQP-$\mathbb{C}$ share the following property: it is better to relax nonconvex constraints before converting from complex to real numbers rather than to do the two operations in the opposite order.

\subsection*{Shor relaxation}
\label{subsec:Shor Relaxation}
For $H \in \mathbb{H}_n$ and $z \in \mathbb{C}^n$, the relationship $z^H H z = \text{Tr}(H z z^H)$ holds where $\text{Tr}\left(\cdot\right)$ denotes the trace\footnote{For all matrices $A,B \in \mathbb{C}^{n\times n}$, $\text{Tr}(AB) = \sum_{1 \leqslant i,j \leqslant n} A_{ij} B_{ji}$.} of a complex square matrix. Relaxing the rank of $Z = zz^H$ in $\eqref{eq:qcqp}$ yields 
\begin{subequations}
\begin{gather}
\text{SDP-}\mathbb{C}~\text{:} ~~~
\inf_{Z \in \mathbb{H}_n}~ \text{Tr}(H_0 Z), ~~~~~~~~~~~~~~~~~~~~~ \label{eq:sdpC1} \\
~~~~~~~~~~~~~~~~~~~~~~\text{s.t.} ~~~
\text{Tr}(H_i Z) \leqslant h_i,~~~~ i=1, \hdots, m,\\
Z \succcurlyeq 0, \label{eq:sdpC2} ~
\end{gather}
\end{subequations}
\noindent where $\succcurlyeq 0$ indicates positive semidefiniteness.

Let $\text{Re}Z$ and $\text{Im}Z$ denote the real and imaginary parts of the matrix $Z\in \mathbb{C}^{n\times n}$ respectively. Consider the ring homomorphism $\Lambda : (\mathbb{C}^{n\times n},+,\times) \longrightarrow (\mathbb{R}^{2n\times 2n},+,\times)$ defined by
\begin{equation}
\label{eq:conversion}
\Lambda(Z) :=
\left( \begin{array}{cr}
\text{Re}Z & - \text{Im} Z \\
\text{Im}Z & \text{Re}Z
\end{array} 
\right),
\end{equation}
whose relevant properties are proven in Appendix \ref{app:Ring Homomorphism}. To convert \sdpc{} into real numbers, real and imaginary parts of the complex matrix variable are identified using two properties: (1) a complex matrix $Z$ is positive semidefinite if and only if the real matrix $\Lambda(Z)$ is positive semidefinite, and (2) if $Z_1,Z_2 \in \mathbb{H}_n$, then $\text{Tr}\left[\Lambda(Z_1)\Lambda(Z_2)\right] = \text{Tr}\left[\Lambda(Z_1Z_2)\right]  = 2\text{Tr}(Z_1 Z_2)$. 
This yields the converted problem
\begin{subequations}
\begin{gather}
\text{CSDP-}\mathbb{R}~\text{:} ~~~ \inf_{X \in \mathbb{S}_{2n}} ~ \text{Tr}( \Lambda(H_0) X), ~~~~~~~~~~~~~~~~~~~~~~~~ \label{eq:csdpR1} \\
~~~~~~~~~~~~~~~~~~~~~~\text{s.t.} ~~~
\text{Tr}( \Lambda(H_i) X) \leqslant h_i,~~~~ i=1, \hdots, m, \label{eq:csdpR2}
\\
X  \succcurlyeq 0, \label{eq:csdpR3} ~~~~~~
\\ ~~~~~~~~~~~~~~~~~~~~~~~~~~~~~~~~~~~~~
X = 
\left(
\begin{array}{cl}
A & B^T \\
B & C
\end{array}
\right) ~~~\&~~~ 
\begin{array}{lcr}
A & = & C, \\
B^T & = & -B,
\end{array}
\label{eq:csdpR4}
\end{gather}
\end{subequations}
where $\mathbb{S}_{2n}$ denotes the set of real symmetric matrices of order~$2n$ and $\left(\cdot\right)^T$ indicates the transpose. Note that the set of matrices satisfying \eqref{eq:csdpR4} is isomorphic to $\mathbb{C}^{n \times n}$. A global solution to \qcqp{} can be retrieved from \csdpr{} if and only if $\text{rank}(X)\in \{0,2\}$ at optimality (proof in Appendix \ref{app:Rank-2 Condition}).

In order to convert \qcqp{} into real numbers, real and imaginary parts of the complex vector variable are identified. This is done by considering a new variable $x = \left( ~ (\text{Re}z)^T ~ (\text{Im}z)^T ~ \right)^T $ and observing that if $H \in \mathbb{H}_n$, then $z^H H z = x^T \Lambda(H) x = \text{Tr}(\Lambda(H)xx^T)$. This gives rise to a problem which we will call QCQP-$\mathbb{R}$. Relaxing the rank of $X=xx^T$ yields
\begin{subequations}
\begin{gather}
\text{SDP-}\mathbb{R}~\text{:} ~~~ \inf_{X \in \mathbb{S}_{2n}} ~ \text{Tr}( \Lambda(H_0) X), ~~~~~~~~~~~~~~~~~~  \label{eq:sdpR1} \\
~~~~~~~~~~~~~~~~~~~~~~~~~~\text{s.t.} ~~~
\text{Tr}( \Lambda(H_i) X) \leqslant h_i, ~~~~ i=1, \hdots, m, \label{eq:sdpR2}
\\
X  \succcurlyeq 0. \label{eq:sdpR3} ~~~ 
\end{gather}
\end{subequations}

A global solution to \qcqp{} can be retrieved from \sdpr{} if and only if $\text{rank}(X)\in \{0,1\}$ or $\text{rank}(X)=2$ and \eqref{eq:csdpR4} holds at optimality.

We have $\text{val}(\text{SDP-}\mathbb{C}) = \text{val}(\text{CSDP-}\mathbb{R}) = \text{val}(\text{SDP-}\mathbb{R})$
where ``val'' is the optimal value of a problem (proof in Appendix \ref{app:Invariance of Shor Relaxation Bound}).
The number of scalar variables of \csdpr{} is half that of \sdpr{} due to constraint \eqref{eq:csdpR4}. This constraint also halves the possible ranks of the matrix variable, which must be an even integer in \csdpr{} whereas it can be any integer between 0 and $2n$ in \sdpr{}. The number of variables in \sdpr{} can be reduced by a small fraction ($\frac{2}{2n+1}$ to be exact) by setting a diagonal element of $X$ to 0. This does not affect the optimal value (proof in Appendix \ref{app:Invariance of SDP-R Relaxation Bound}). Figure~\ref{fig:commutation} summarizes this section.
\begin{figure}[ht]
  \centering
    \includegraphics[width=.65\textwidth]{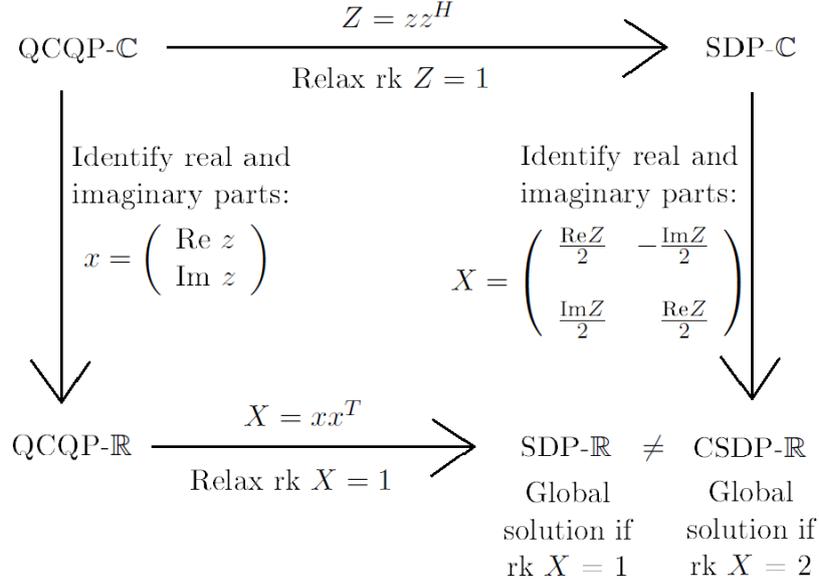}
  \caption{Non-Commutativity of Relaxation and Complex-to-Real Conversion}
  \label{fig:commutation}
\end{figure}

\subsection*{Second-order conic relaxation}
\label{subsec:Second-Order Conic Relaxation}
In \sdpc{}, assume that the semidefinite constraint \eqref{eq:sdpC2} is relaxed to the second-order cones
\begin{equation}
\left(
\begin{array}{cc}
Z_{ii} & Z_{ij} \\
Z_{ij}^H & Z_{jj}
\end{array}
\right)
 \succcurlyeq 0 ~~,~~ 1 \leqslant i \neq j \leqslant n. \label{eq:2x2}
\end{equation}
Equation \eqref{eq:2x2} is equivalent to constraining the determinant $Z_{ii} Z_{jj} - Z_{ij} Z_{ij}^H$ and diagonal elements $Z_{ii}$ to be non-negative. This yields
\begin{subequations}
\begin{gather}
\text{SOCP-}\mathbb{C}~\text{:} ~~~ \inf_{Z \in \mathbb{H}_n}~ \text{Tr}(H_0 Z), ~~~~~~~~~~~~~~~~~~~~~~ \label{eq:socpC1} \\
~~~~~~~~~~~~~~~~~~~~~~~~\text{s.t.} ~~~
\text{Tr}(H_i Z) \leqslant h_i, ~~~~ i=1, \hdots, m, \label{eq:socpC2}\\
~~~~~~~~~~~~~~~~~~~~~~~~~~~~~~~~~~~~|Z_{ij}|^2 \leqslant Z_{ii} Z_{jj}, ~~~~~ 1 \leqslant i \neq j \leqslant n, \label{eq:socpC3} \\
~~~~~~~~~~~~~~~~~~~~~~~~~~~~~~ Z_{ii} \geqslant 0, ~~~~~~~~~~~ i=1, \hdots, n, \label{eq:socpC4}
\end{gather}
\end{subequations}
where $|\cdot|$ denotes the complex modulus.
Identifying real and imaginary parts of the matrix variable $Z$ leads to 
\begin{subequations}
\begin{gather}
\text{CSOCP-}\mathbb{R}~\text{:} ~~~
\inf_{X \in \mathbb{S}_{2n}} ~ \text{Tr}( \Lambda(H_0) X), ~~~~~~~~~~~~~~~~~ \label{eq:csocpR1} ~~~~~ \\
~~~~~~~~~~~~~~~~~~~~~~~~~~~\text{s.t.} ~~~
\text{Tr}( \Lambda(H_i) X) \leqslant h_i, ~~~~ i=1, \hdots, m, \label{eq:csocpR2}
\\
~~~~~~~~~~~~~~~~~~~~~~~~~~~~~~~~~~~~~~~~~~~X_{ij}^2 + X_{n+i,j}^2 \leqslant X_{ii} X_{jj}, ~~~~ 1 \leqslant i \neq j \leqslant n, \label{eq:csocpR3} \\
~~~~~~~~~~~~~~~~~~~~~~~~~~~~~~~~~~X_{ii} + X_{n+i,n+i} \geqslant 0, ~~ i=1, \hdots, n, \label{eq:csocpR4} \\
~~~~~~~~~~~~~~~~~~~~~~~~~~~~~~~~~~~~~~~~~~~~X = 
\left(
\begin{array}{cl}
A & B^T \\
B & C
\end{array}
\right) ~~~\&~~~ 
\begin{array}{lcr}
A & = & C, \\
B^T & = & -B.
\end{array} \label{eq:csocpR5}
\end{gather}
\end{subequations}
In \sdpr{} of Section, assume that the semidefinite constraint \eqref{eq:sdpR3} is relaxed to the second-order cones
\begin{equation}
\left(
\begin{array}{cc}
X_{ii} & X_{ij} \\
X_{ij} & X_{jj}
\end{array}
\right)
 \succcurlyeq 0 ~~,~~ 1 \leqslant i \neq j \leqslant 2n.
\end{equation}
This leads to
\begin{subequations}
\begin{gather}
\text{SOCP-}\mathbb{R}~\text{:} ~~~ \inf_{X \in \mathbb{S}_{2n}} ~ \text{Tr}( \Lambda(H_0) X), ~~~~~~~~~~~~~~~~~~~~~~~~~ \label{eq:socpR1} \\
~~~~~~~~~~~~~~~~~~~~~\text{s.t.} ~~~
\text{Tr}( \Lambda(H_i) X) \leqslant h_i, ~~~~ i=1, \hdots, m, \label{eq:socpR2} 
\\
~~~~~~~~~~~~~~~~~~~~~~~~~~~ X_{ij}^2 \leqslant X_{ii} X_{jj}, ~~~~ 1 \leqslant i \neq j \leqslant 2n, \label{eq:socpR3} \\
~~~~~~~~~~~~~~~~ X_{ii} \geqslant 0, ~~~ i=1, \hdots, 2n. \label{eq:socpR4}
\end{gather}
\end{subequations}

Unlike in the Shor relaxation, we have
$
\text{val}(\text{SOCP-}\mathbb{C}) = \text{val}(\text{CSOCP-}\mathbb{R}) \geqslant \text{val}(\text{SOCP-}\mathbb{R})
$ (proof in Appendix \ref{app:Discrepancy Between Second-Order Conic Relaxation Bounds}).
The number of scalar variables of \csocpr{} is half that of \socpr{} due to constraint \eqref{eq:csocpR5}. The number of second-order conic constraints in \csocpr{}, equal to $\frac{n(n-1)}{2}$, is roughly a fourth of that in \socpr{}, equal to $\frac{2n(2n-1)}{2}$.

\subsection*{Exploiting sparsity}

\label{subsec:Exploiting sparsity}

Given an undirected graph $(\mathcal{V},\mathcal{E})$ where $\mathcal{V} \subset \{1,\hdots,n\}$ and $\mathcal{E}\subset \mathcal{V} \times \mathcal{V}$, define for all $Z \in \mathbb{H}_n$
\begin{equation}
\Psi_{(\mathcal{V},\mathcal{E})}(Z)_{ij} := \left\{ \begin{array}{cl} Z_{ij} & \text{if} ~~ (i,j) \in \mathcal{E} ~~ \text{or} ~~ i = j \in \mathcal{V}, \\ 0 & \text{else}. \end{array} \right.
\end{equation}

We associate an undirected graph $\mathcal{G}$ to \qcqp{} whose nodes are $\{1,\hdots,n\}$ and that satisfies 
$H_i = \Psi_\mathcal{G}(H_i) $ for $i = 0, \hdots, m$. Let $ \mathbb{H}_n^+$ denote the set of positive semidefinite Hermitian matrices of size $n$ and let ``Ker'' denote the kernel of a linear application. Given the definition of $\mathcal{G}$, constraint~\eqref{eq:sdpC2} of \sdpc{} can be relaxed to 
$
Z \in \mathbb{H}_n^+ + \text{Ker} ~ \Psi_{\tilde{\mathcal{G}}}
$
without changing its optimal value for any graph $\tilde{\mathcal{G}}$ whose nodes are $\{1,\hdots,n\}$ and where $\mathcal{G} \subset \tilde{\mathcal{G}}$. 
Consider a chordal extension $\mathcal{G}\subset \mathcal{G}^\text{ch}$, that is to say that all cycles of length four or more have a chord (edge between two non-consecutive nodes of the cycle). Let $\mathcal{C}_1,\hdots,\mathcal{C}_p \subset \mathcal{G}^\text{ch}$ denote the maximal cliques of $\mathcal{G}^\text{ch}$. (A clique is a subgraph where all nodes are linked to one another. The set of maximally sized cliques of a given graph can be computed in linear time~\cite{tarjan}). A chordal extension has a useful property for exploiting sparsity \cite{grone}: for all $Z \in \mathbb{H}_n$, we have that $Z \in \mathbb{H}_n^+ + \text{Ker}~ \Psi_{\mathcal{G}^\text{ch}}$ if and only if $\Psi_{\mathcal{C}_i}(Z) \succcurlyeq 0$ for $i=1,\hdots, p$. Note that $\Psi_{\mathcal{C}_i}(Z) \succcurlyeq 0$ if and only if
$\Lambda \circ \Psi_{\mathcal{C}_i}(Z) \succcurlyeq 0$, where ``$\circ$'' is the composition of functions. Given a graph $(\mathcal{V},\mathcal{E})$, define for $X \in \mathbb{S}_{2n}$
\begin{equation}
\tilde{\Psi}_{(\mathcal{V},\mathcal{E})}(X) :=
\left(
\begin{array}{cl}
\Psi_{(\mathcal{V},\mathcal{E})}(A) & \Psi_{(\mathcal{V},\mathcal{E})}(B^T)  \\
\Psi_{(\mathcal{V},\mathcal{E})}(B)  & \Psi_{(\mathcal{V},\mathcal{E})}(C) 
\end{array}
\right),
\end{equation}
using the block decomposition in the left hand part of \eqref{eq:csdpR4}. Notice that $\Lambda \circ \Psi_{(\mathcal{V},\mathcal{E})}= \tilde{\Psi}_{(\mathcal{V},\mathcal{E})} \circ \Lambda$. As a result, \eqref{eq:csdpR3} can be replaced by
$ \tilde{\Psi}_{\mathcal{C}_i} (X) \succcurlyeq 0$ for $i = 1,\hdots, p$ without changing the optimal value of \csdpr{}, with an analogous replacement for constraint \eqref{eq:sdpR3} in \sdpr{}. If in \sdpr{} we exploit the sparsity of matrices $\Lambda(H_i)$ instead of that of $H_i$, the resulting graph has twice as many nodes. Computing a chordal extension and maximal cliques is hence more costly.

Sparsity in the second-order conic relaxations is exploited using the fact that applying \eqref{eq:csocpR3} and \eqref{eq:socpR3} only for $(i,j)$ that are edges of $\mathcal{G}$ does not change the optimal values of \csocpr{} and \socpr{}.

\section{Complex moment/sum-of-squares hierarchy} 
\label{sec:Complex Moment/Sum-of-Squares Hierarchy}
We now transpose the work of Lasserre~\cite{lasserre-2001} from real to complex numbers.
Let $z^\alpha$ denote the monomial $z_1^{\alpha_1} \cdots z_n^{\alpha_n}$ where $z \in {\mathbb C}^n$ and $\alpha \in {\mathbb N}^n$ for some integer $n \in \mathbb N^*$. Define $|\alpha| := \alpha_1+\hdots+\alpha_n$ and $\overline{w}$ as the conjugate of $w \in \mathbb{C}$. Define $\bar{z} := (\bar{z}_1,\hdots,\bar{z}_n)^T$ where $z \in {\mathbb C}^n$.
Consider the sets
\begin{equation}
\label{eq:sets}
\begin{array}{rl}
\mathbb{C}[z] := & \{ ~ p : \mathbb{C}^n \rightarrow \mathbb{C} ~|~ p(z) = \sum_{|\alpha|\leqslant l} p_{\alpha} z^\alpha, ~ l\in \mathbb{N}, ~ p_{\alpha} \in \mathbb{C}  ~ \}, \\
\mathbb{C}[\bar{z},z]  := & \{ ~ f : \mathbb{C}^n \rightarrow \mathbb{C} ~|~ f(z) = \sum_{|\alpha|,|\beta|\leqslant l} f_{\alpha,\beta} \bar{z}^\alpha z^\beta,~ l\in \mathbb{N}, ~ f_{\alpha,\beta} \in \mathbb{C} ~ \}, \\
\mathbb{R}[\bar{z},z]  := & \{ ~ f \in \mathbb{C}[\bar{z},z] ~|~ \overline{f(z)} = f(z), ~ \forall z \in \mathbb{C}^n ~ \}, \\[0.25em]
\Sigma[z] := & \{ ~ \sigma : \mathbb{C}^n \rightarrow \mathbb{C} ~|~ \sigma = \sum_{j=1}^r |p_j|^2, ~ r\in \mathbb{N}^*, ~ p_j \in \mathbb{C}[z] ~ \},
\end{array}
\end{equation}
and for all $d\in \mathbb{N}$
\begin{equation}
\label{eq:setsd}
\begin{array}{rl}
\mathbb{C}_d[z] := & \{ ~ p : \mathbb{C}^n \rightarrow \mathbb{C} ~|~ p(z) = \sum_{|\alpha|\leqslant d} p_{\alpha} z^\alpha, ~ p_{\alpha} \in \mathbb{C}  ~ \}, \\
\mathbb{C}_d[\bar{z},z] := & \{ ~ f : \mathbb{C}^n \rightarrow \mathbb{C} ~|~ f(z) = \sum_{|\alpha|,|\beta|\leqslant d} f_{\alpha,\beta} \bar{z}^\alpha z^\beta, ~ f_{\alpha,\beta} \in \mathbb{C}  ~ \}, \\
\mathbb{R}_d[\bar{z},z]  := & \{ ~ f \in \mathbb{C}_d[\bar{z},z] ~|~ \overline{f(z)} = f(z), ~ \forall z \in \mathbb{C}^n ~ \}, \\[0.25em]
\Sigma_d[z] := & \{ ~ \sigma : \mathbb{C}^n \rightarrow \mathbb{C} ~|~ \sigma = \sum_{j=1}^r |p_j|^2, ~ r\in \mathbb{N}^*, ~ p_j \in \mathbb{C}_d[z] ~ \}.
\end{array}
\end{equation}
Note that the coefficients of a function $f \in \mathbb{R}[\bar{z},z]$ satisfy $\overline{f_{\alpha,\beta}} = f_{\beta,\alpha}$ for all $|\alpha|,|\beta|\leqslant l$ for some $l \in \mathbb{N}$. The set of complex polynomials $\mathbb{C}[\bar{z},z]$ is a $\mathbb{C}$-algebra (i.e. commutative ring and vector space over $\mathbb{C}$) 
and the set of holomorphic polynomials $\mathbb{C}[z]$ is a subalgebra of it (i.e. subspace closed under sum and product). The set of real-valued complex polynomials $\mathbb{R}[\bar{z},z]$ is an $\mathbb{R}$-algebra. The set of sums of squared moduli of holomorphic polynomials $\Sigma[z]$ and the set $\Sigma_d[z] \subset \mathbb{R}_d[z]$ are pointed cones (i.e. closed under multiplication by elements of $\mathbb{R}_+$) that are convex (i.e. $tu+ (1-t)v$ with $0\leqslant t \leqslant 1$ belongs to them if $u$ and $v$ do). 
Let $C(K,\mathbb{C})$ denote the Banach (i.e. complete) $\mathbb{C}$-algebra of continuous functions from a compact set $K \subset \mathbb{C}^n$ to $\mathbb{C}$ equipped with the norm $\|\varphi\|_{\infty} := \sup_{z\in K} |\varphi(z)|$. Consider $R_K : \mathbb{C}[\bar{z},z] \longrightarrow C(K,\mathbb{C})$ defined by  $f \longmapsto f_{|K}$
where $f_{|K}$ denotes the restriction of $f$ to $K$. $R_K(\mathbb{C}[\bar{z},z])$ is a unital subalgebra of $C(K,\mathbb{C})$ (i.e. contains multiplicative unit) that separates points of $K$ (i.e. $u\neq v \in K \Longrightarrow \exists \varphi \in R_K(\mathbb{C}[\bar{z},z]): \varphi(u)\neq \varphi(v)$) and that is closed under complex conjugation. It is hence a dense subalgebra due to the Complex Stone-Weiestrass Theorem.
Likewise, $C(K,\mathbb{R}) := \{ \varphi \in C(K,\mathbb{C}) ~|~ \overline{\varphi(z)} = \varphi(z), ~ \forall z \in \mathbb{C}^n  \}$ is a Banach $\mathbb{R}$-algebra of which $R_K(\mathbb{R}[\bar{z},z])$ is a dense subalgebra.
In other words, a continuous real-valued function of multiple complex variables can be approximated as close as desired by real-valued complex polynomials when restricted to a compact set. They are hence a powerful modeling tool in optimization. Speaking of which, let $m \in \mathbb{N}^*$ and $k,k_1,\hdots,k_m \in \mathbb{N}$. Consider $(f,g_1,\hdots,g_m) \in \mathbb{R}_k[\bar{z},z] \times \mathbb{R}_{k_1}[\bar{z},z] \times \hdots \times \mathbb{R}_{k_m}[\bar{z},z]$ where there exists $|\alpha|=k$ and $|\beta| \leqslant k$ such that $f_{\alpha,\beta} \neq 0$. In addition, for $i=1,\hdots,m$, there exists $|\alpha|=k_i$ and $|\beta| \leqslant k_i$ such that $g_{i,\alpha,\beta} \neq 0$. Consider the complex multivariate polynomial optimization problem
\begin{equation}
\boxed{
\label{eq:complexPOP}
\begin{array}{rcllll}
f^\text{opt} & := & \inf_{z \in {\mathbb C}^n} & f(z) & \mathrm{s.t.} & g_i(z) \geqslant 0,~~i=1,...,m,
\end{array}
}
\end{equation}
where by convention $f^\text{opt} := + \infty$ if the feasible set is empty.
The feasible set is a closed semi-algebraic set on which we make the following assumption from now on:
\begin{equation}
\boxed{
\label{eq:assumption}
K:= \{ ~ z \in \mathbb{C}^n ~|~ g_i(z) \geqslant 0,~i=1,...,m ~ \} ~~ \text{is compact.}
}
\end{equation}
Let $K^\text{opt}$ denote the set of optimal solutions to \eqref{eq:complexPOP}.
It may be empty because we do not assume $K$ to be non-empty. (Note that in practice, it is often hard to know whether there exists a feasible solution, as for the application of Section \ref{subsec:Numerical Results7}.) 

Let $\mathcal{M}(K)$ denote the Banach space over $\mathbb{R}$ of Radon measures on $K$. Bear in mind that since $K$ is compact, $\mathcal{M}(K)$ may be identified with the topological dual of $C(K,\mathbb{R})$ i.e. the Banach space over $\mathbb{R}$ of linear continuous applications from $C(K,\mathbb{R})$ to $\mathbb{R}$ equipped with the operator norm. (This is due to the Riesz-Markov-Kakutani Representation Theorem.) For $\varphi \in C(K,\mathbb{C})$, define $\int_K \varphi d\mu := \int_K \text{Re}(\varphi)d\mu + \textbf{i} \int_K \text{Im}(\varphi)d\mu$~\cite[1.31 Definition]{rudin-1987}\footnote{We wish to thank Bruno Nazaret for bringing this reference to our attention.}. Next, consider the convex pointed cone 
$
\mathcal{P}(K) := \{ ~ \varphi \in C(K,\mathbb{R}) ~|~ \varphi(z) \geqslant 0,~\forall z\in K ~ \}$.
A Radon measure $\mu$ is positive (denoted $\mu \geqslant 0$) if $\varphi \in \mathcal{P}(K)$ implies that $\int_K \varphi d\mu \geqslant 0$. Let $\mathcal{M}_+(K)$ denote the set of positive Radon measures. 
With these definitions, we have
\begin{equation}
\label{eq:moment}
\begin{array}{rcllll}
f^\text{opt} & = & \inf_{\mu \in \mathcal{M}(K)}  & \int_K f d\mu & \mathrm{s.t.} & \int_K d\mu = 1 ~~\&~~ \mu \geqslant 0.
\end{array}
\end{equation}
Indeed, if $z \in K$, then the Dirac\footnote{The Dirac measure $\delta_z$ with $z\in K$ may be identified with the continuous linear application from $C(K,\mathbb{R})$ to $\mathbb{R}$ defined by $\varphi \longmapsto \varphi(z)$. This is one way to interpret the fact that $\int_K f d\delta_z = f(z)$.} 
measure $\delta_z$ is a feasible point of \eqref{eq:moment} for which the objective value is equal to $f(z)$. Hence the optimal value of \eqref{eq:moment} is less than or equal to $f^\text{opt}$. Conversly, if $\mu$ is a feasible point of \eqref{eq:moment}, then $\int_K (f - f^\text{opt})d\mu \geqslant 0$ and hence $\int_K f d\mu \geqslant \int_K f^\text{opt} d\mu = f^\text{opt} \int_K d\mu = f^\text{opt}$. 
\begin{proposition}
\label{prop:measure}
\normalfont
\textit{The set of optimal solutions to \eqref{eq:moment} is}
\begin{equation}
\label{eq:solutions}
\{ ~ \mu \in \mathcal{M}_+(K) ~|~ \mu(K^\text{opt}) = 1 ~~ \& ~~ \mu(K \setminus K^\text{opt}) = 0 ~ \}.
\end{equation}
\textit{As a consequence, if }$K^\text{opt}$\textit{ is a finite set of $S\in \mathbb{N}^*$ points $z(1),\hdots,z(S) \in \mathbb{C}^n$, then the set optimal solutions to \eqref{eq:moment} is} $\{ ~ \sum_{j=1}^S \lambda_j \delta_{z(j)} ~ | ~ \sum_{j=1}^S \lambda_j = 1 ~~  \& ~~ \lambda_1,\hdots,\lambda_S \in \mathbb{R}_+ \}$.
\end{proposition}
\begin{proof}
Consider $\mu$ an optimal solution to \eqref{eq:moment}. It must be that $\int_K (f-f^\text{opt})d\mu = 0$. Thus $\int_{K\setminus K^\text{opt}} (f-f^\text{opt}) d\mu = 0$ and $\mu(K \setminus K^\text{opt})  = \int_{K\setminus K^\text{opt}} d\mu = 0$. Therefore $\mu(K^\text{opt}) =  \int_{K^\text{opt}} d\mu = \mu(K) - \mu(K \setminus K^\text{opt}) =1$. Conversly, if $\mu$ belongs to the set in \eqref{eq:solutions}, then it is feasible for \eqref{eq:moment} and $\int_K (f - f^\text{opt}) d\mu = \int_{ K \setminus K^\text{opt} } (f - f^\text{opt}) d\mu = 0$. Hence $\int_K fd\mu=\int_K f^\text{opt} d\mu = f^\text{opt}\int_K d\mu = f^\text{opt}$.
\end{proof}

In order to dualize the equality constraint in \eqref{eq:moment}, consider the Lagrange function $\mathcal{L} : \mathcal{M}_+(K) \times \mathbb{R} \longrightarrow \mathbb{R}$ defined by
$
(\mu,\lambda) \longmapsto \int_K f d\mu + \lambda \left(1-\int_K d\mu \right)
$.
We have $\mathcal{L}(\mu,\lambda)  =  \lambda + \int_K (f - \lambda) d\mu$ and
\begin{equation}
\inf_{\mu \in \mathcal{M}_+(K)} \int_K (f- \lambda) d\mu =
\left\{
\begin{array}{cl}
\hphantom{-}0 & \text{if} ~ f(z)-\lambda \geqslant 0, ~~~ \forall z \in K, \\
- \infty & \text{else},
\end{array}
\right.
\end{equation}
since, in the second case, we may consider $t \delta_z$ for a $z\in K$ such that $f(z)-\lambda< 0$ and $t \rightarrow + \infty$.
This leads to the dual problem
\begin{equation}
\label{eq:sos}
\begin{array}{rcllll}
f^\text{opt} & = & \sup_{\lambda \in \mathbb{R}}  & \lambda & \mathrm{s.t.} & f(z)- \lambda \geqslant 0,~~ \forall z \in K.
\end{array}
\end{equation}
Primal problem \eqref{eq:moment} gives rise to the complex moment hierarchy below. Dual problem \eqref{eq:sos} gives rise to the complex sum-of-squares hierarchy below.

\subsection*{Complex moment hierarchy} 
\label{subsec:Complex Moment Hierarchy}

Let $\mathcal{H}$ (respectively $\mathcal{H}_d$) denote the set of sequences of complex numbers $(y_{\alpha,\beta})_{\alpha,\beta \in \mathbb{N}^n}$ (respectively $(y_{\alpha,\beta})_{|\alpha|,|\beta| \leqslant d}$)
such that $\overline{y_{\alpha,\beta}} = y_{\beta,\alpha}$ for all $\alpha,\beta \in \mathbb{N}^n$ (respectively $|\alpha|,|\beta| \leqslant d$).
\begin{definition}
\label{def:moment}
An element $y \in \mathcal{H}$ is said to have a representing measure $\mu$ on $K$ if $\mu \in \mathcal{M}_+(K)$ and $y_{\alpha,\beta} = \int_K \bar{z}^\alpha z^\beta d\mu$ for all $\alpha,\beta \in \mathbb{N}^n$.
In that case, $y_{\alpha,\beta}$ is called the $(\alpha,\beta)$\text{-moment} of $\mu$.
\end{definition} 

When $y \in \mathcal{H}$ has a representing measure on $K$, the measure is unique because $R_K(\mathbb{C}[\bar{z},z])$ is dense in $C(K,\mathbb{C})$. The complex moment problem consists in characterizing the sequences that are representable by a measure on $K$ and is connected to other branches of mathematics such as functional analysis and spectral theory of operators~\cite{akhiezer-1965}. It has been studied by Atzmon~\cite{atzmon-1975}, Schm\"udgen~\cite{schmudgen-1991}, Putinar~\cite{putinar-1998}, Curto and Fialkow~\cite{curto-1996,curto-2000,curtoquad-2000}, Stochel~\cite{stochel-2001}, and Vasilescu~\cite{vasilescu-2003}. 
For example, Atzmon~\cite[Theorem 2.1]{atzmon-1975} proved that the solutions to the complex moment problem where $K=\{ z \in \mathbb{C} ~|~ |z|=1\}$ are the sequences $y\in \mathcal{H}$ such that 
$\sum_{m,n,j,k \in \mathbb{N}} c_{n,j} ~ \overline{c}_{m,k} ~ y_{m+j,n+k} \geqslant 0$
and $\sum_{m,n \in \mathbb{N} } w_m \overline{w}_n ~ (y_{m,n} - y_{m+1,n+1}) \geqslant 0$
for all complex numbers $(c_{j,k})_{j,k \in \mathbb{N}}$ and $(w_n)_{n \in \mathbb{N}}$ with only finitely many non-zero terms. A generalization to the multidimensional case is considered in Section \ref{subsec:Convergence of the Complex Hierarchy}. We conclude our presentation of the complex moment problem by noting that the case where $K$ is not compact is an open problem.

Consider a feasible point $\mu$ of \eqref{eq:moment} and the sequence $y \in \mathcal{H}$ that has representation measure $\mu$ on $K$. Notice that 
$\int_K f d\mu = \int_K \sum_{|\alpha|,|\beta|\leqslant k} f_{\alpha,\beta} \bar{z}^\alpha z^\beta  d\mu =
\sum_{|\alpha|,|\beta|\leqslant k} f_{\alpha,\beta} \int_K \bar{z}^\alpha z^\beta d\mu 
= \sum_{|\alpha|,|\beta|\leqslant k} ~ f_{\alpha,\beta} y_{\alpha,\beta} =: L_y(f)$ and
$\int_K d\mu = \int_K \bar{z}^0 z^0 d\mu = y_{0,0} = 1$.
For all $p \in \mathbb{C}[z]$, we have $|p|^2 g_i \geqslant 0$ on $K$.  Since $\mu \geqslant 0$, this implies that $\int_K  |p|^2 g_i d\mu \geqslant 0$. Naturally, we also have $\int_K  |p|^2 g_0 d\mu \geqslant 0$ if we define $g_0 := 1$. Define $k_0 :=0$ and $d^\text{min}:=\max\{k, k_1 \hdots,k_m\}$. Consider $d\geqslant d^\text{min}$, $0 \leqslant i \leqslant m$, and $p \in \mathbb{C}_{d-k_i}[z]$. We have $\int_K |p|^2 g_i d\mu = \hdots$
\begin{equation}
\label{eq:positivecon}
\begin{array}{l}
 = \int_K  |\sum_{|\alpha|\leqslant d-k_i} ~ p_{\alpha} z^{\alpha}|^2 ~ ( \sum_{|\gamma|,|\delta| \leqslant k_i} ~ g_{i,\gamma,\delta} ~ \bar{z}^{\gamma} z^{\delta}) ~ d\mu \\
= \int_K (\sum_{|\alpha|,|\beta| \leqslant d-k_i} ~ \overline{p}_{\alpha}  p_{\beta} \bar{z}^\alpha z^\beta) ~ ( \sum_{|\gamma|,|\delta| \leqslant k_i} ~ g_{i,\gamma,\delta} ~ \bar{z}^{\gamma} z^{\delta}) ~ d\mu \\
= \int_K \sum_{|\alpha|,|\beta| \leqslant d-k_i} \overline{p}_\alpha  p_\beta ~ \sum_{|\gamma|,|\delta| \leqslant k_i}  ~ g_{i,\gamma,\delta} ~ \bar{z}^{\alpha+\gamma} z^{\beta+\delta} ~ d\mu \\
= \sum_{|\alpha|,|\beta| \leqslant d-k_i} \overline{p}_{\alpha}  p_{\beta} ~ \sum_{|\gamma|,|\delta| \leqslant k_i}  ~ g_{i,\gamma,\delta} ~ \int_K \bar{z}^{\alpha+\gamma} z^{\beta+\delta} ~ d\mu \\
= \sum_{|\alpha|,|\beta| \leqslant d-k_i} \overline{p}_{\alpha}  p_{\beta} ~ (\sum_{|\gamma|,|\delta| \leqslant k_i}  ~ g_{i,\gamma,\delta} ~ y_{\alpha+\gamma,\beta+\delta}) =: M_{d-k_i}(g_iy)(\alpha,\beta) \\
= \sum_{|\alpha|,|\beta| \leqslant d-k_i} \overline{p}_{\alpha}  p_{\beta} ~ M_{d-k_i}(g_iy)(\alpha,\beta)\\
= \vec{p}^H M_{d-k_i}(g_iy) \vec{p},
\end{array}
\end{equation}
where $\vec{p} := ( p_\alpha )_{|\alpha|\leqslant d-k_i}$ and $M_{d-k_i}(g_iy)$ is a Hermitian matrix indexed by $|\alpha|,|\beta| \leqslant d-k_i$. As a result
\begin{equation}
\label{eq:necessary}
M_{d-k_i}(g_iy) \succcurlyeq 0, ~~~~~~ i = 0,\hdots,m, ~~~~~~ \forall d \geqslant d^\text{min}.
\end{equation}
To sum up, $y$ is a feasible point of 
\begin{equation}
\label{eq:momentinf}
\begin{array}{rclll}
\rho & := & \inf_{y\in \mathcal{H}} & L_y(f), & \\
 & & \text{s.t.} &  y_{0,0} = 1, &  \\
 & & &  M_{d-k_i}(g_iy) \succcurlyeq 0, &i = 0, \hdots, m, ~~ \forall d \geqslant d^\text{min},
\end{array}
\end{equation}
with same objective value as $\mu$ in \eqref{eq:moment}. Automatically, $\rho \leqslant f^\text{opt}$. Consider the relaxation of \eqref{eq:momentinf} defined by
\begin{equation}
\label{eq:momentd}
\boxed{
\begin{array}{rclll}
\rho_d & := & \inf_{y\in \mathcal{H}_d} & L_y(f),                                         & \\
           &     & \text{s.t.}                          & y_{0,0} = 1,                                 &  \\
           &     &                                         &  M_{d-k_i}(g_iy) \succcurlyeq 0, & i = 0, \hdots, m,
\end{array}
}
\end{equation}
which we name the \textit{complex moment relaxation of order} $d$ for reasons that will become clear with Theorem \ref{th:represent}. In Section \ref{subsec:Complex Sum-of-Squares Hierarchy}, we will introduce its dual counterpart.
\begin{remark}
\label{rem:Ly}
\normalfont
Given $y \in \mathcal{H}$, the function $L_y$ in this section can be formally be defined by the $\mathbb{C}$-linear operator $L_y : \mathbb{C}[\bar{z},z] \longrightarrow \mathbb{C}$ such that $L_y(\bar{z}^\alpha z^\beta) = y_{\alpha,\beta}$ for all $\alpha,\beta \in \mathbb{N}$. If $\varphi \in \mathbb{C}[\bar{z},z]$ and $\overline{\varphi} = \varphi$, then $\overline{L_y(\varphi)} = L_y(\varphi)$.
Given $l,d \in \mathbb{N}$ and $\varphi \in \mathbb{R}_{l}[\bar{z},z]$, the matrix $M_d$ in \eqref{eq:positivecon} can be formally be defined as the Hermitian matrix indexed by $|\alpha|,|\beta| \leqslant d$ such that
$M_{d}(\varphi y)(\alpha,\beta) := L_y(\varphi(z) \bar{z}^\alpha z^\beta) = \sum_{|\gamma|,|\delta| \leqslant l}  ~ \varphi_{\gamma,\delta} ~ y_{\alpha+\gamma,\beta+\delta}$. Notice that $M_{d}(\varphi y)(0,0) = L_y(\varphi)$. Lastly, define $M_d(y):=M_d(g_0y)$ which we refer to as \textit{complex moment matrix of order $d$}. 
\end{remark}

\subsection*{Complex sum-of-squares hierarchy} 
\label{subsec:Complex Sum-of-Squares Hierarchy} 
We introduced notation $\vec{p}$ for $p \in \mathbb{C}_d[z]$ where $d\in \mathbb{N}$ and will now extend it to $\sigma \in \Sigma_d[z]$. For such an element, there exists $r\in \mathbb{N}^*$ and $p_j \in \mathbb{C}_d[z]$ such that $\sigma = \sum_{j=1}^r |p_j|^2$. Let $\vec{\sigma} := \sum_{j=1}^r \vec{p}_j \vec{p}_j^H$. Also, define $\langle A , B \rangle_{\mathcal{H}_d} := \text{Tr} (AB)$ where $A,B \in \mathcal{H}_d$.
Given $d\geqslant d^\text{min}$, consider the Lagrange function
$\mathcal{L}_d : \mathcal{H}_d \times \mathbb{R} \times \Sigma_{d-k_0}[z] \times \hdots \times \Sigma_{d-k_m}[z] \longrightarrow \mathbb{R}$ defined by
$( y , \lambda, \sigma_0, \hdots , \sigma_m ) \longmapsto L_y(f) + \lambda(1-y_{0,0}) - \sum_{i=0}^m \langle M_{d-k_i}(g_iy) , \vec{\sigma}_i \rangle_{\mathcal{H}_{d-k_i}}$.
Compute $\mathcal{L}_d(y , \lambda, \sigma_0, \hdots , \sigma_m) = \lambda + L_y(f - \lambda) - \sum_{i=0}^m \sum_{j=0}^{r_i} (\vec{p}_j^{\hphantom{.}i})^H M_{d-k_i}(g_iy) \vec{p}_j^{\hphantom{.}i} = \lambda + L_y(f - \lambda) -  \sum_{i=0}^m \sum_{j=0}^{r_i} L_y(|p_j^i|^2 g_i) = \lambda + L_y(f - \lambda - \sum_{i=0}^m \sigma_i g_i)$.
Observe that 
\begin{equation}
\inf_{y \in \mathcal{H}} ~~ L_y\left(f - \lambda - \sum_{i=0}^m \sigma_i g_i\right) =
\left\{
\begin{array}{cl}
\hphantom{-}0 & \text{if} ~ f(z) - \lambda - \sum_{i=0}^m \sigma_i(z) g_i(z) = 0, \\
& \text{for all} ~ z \in \mathbb{C}^n, \\
- \infty & \text{else}.
\end{array}
\right.
\end{equation}
Indeed, in the second case, there exists $z\in \mathbb{C}^n$ such that $f(z) - \lambda - \sum_{i=0}^m \sigma_i(z) g_i(z) \neq 0$. With $(y_{\alpha,\beta})_{\alpha,\beta \in \mathbb{N}} := (\bar{z}^\alpha z^\beta)_{\alpha,\beta \in \mathbb{N}}$, $L_{ty}(f - \lambda - \sum_{i=0}^m \sigma_i g_i) \longrightarrow - \infty$ for either $t\longrightarrow -\infty$ or $t\longrightarrow +\infty$.
The associated dual problem of \eqref{eq:momentd} is thus
\begin{equation}
\label{eq:sosd}
\boxed{
\begin{array}{rcll}
\rho_d^* &:= & \sup_{\lambda,\sigma} & \lambda, \\
              &    & \text{s.t.} &  f-\lambda = \sum_{i=0}^m \sigma_i g_i, \\ 
              &    &                &  \lambda \in \mathbb{R},~ \sigma_i \in \Sigma_{d-k_i}[z] ,~ i=0, \hdots, m,
\end{array}
}
\end{equation}
which we name the \textit{complex sum-of-squares relaxation of order} $d$. Consider
\begin{equation}
\label{eq:sosinf}
\begin{array}{rcll}
\rho^* &:= & \sup_{\lambda,\sigma} & \lambda, \\
              &    & \text{s.t.} &  f-\lambda = \sum_{i=0}^m \sigma_i g_i, \\ 
              &    &                & \lambda \in \mathbb{R},~ \sigma_i \in \Sigma[z] ,~ i=0, \hdots, m,
\end{array}
\end{equation}
whose relationship with \eqref{eq:momentinf} is touched upon in Proposition \ref{prop:dualconverge} below.
\begin{proposition}
\label{prop:dualconverge}
\normalfont
\textit{We have} $\rho_d^* \leqslant \rho_d$ for all $d\geqslant d^\text{min}$ and $\rho_d^* \longrightarrow \rho^* \leqslant \rho \leqslant f^\text{opt}$.
\end{proposition}
\begin{proof}
The sequence $(\rho_d^*)_{d\geqslant d^\text{min}}$ is non-decreasing and upper bounded by $\rho^* \in \mathbb{R}\cup\{\pm \infty\}$. Thus it converges towards some limit $\rho^*_\text{lim} \in \mathbb{R}\cup\{\pm \infty\}$ such that $\rho^*_\text{lim} \leqslant \rho^*$. If $\rho^* = - \infty$, then $\rho_d^* = - \infty$ for all $d\geqslant d^\text{min}$ and $\rho_d^* \longrightarrow \rho^*$. If not, by definiton of the optimum $\rho^*$, there exists a sequence $(\lambda^l,\sigma_0^l,\hdots,\sigma_m^l)$ of feasible points such that $\lambda^l \leqslant \rho^*$ and $\lambda^l \longrightarrow \rho^*$. To each $l\in\mathbb{N}$, we may associate an integer $d(l)\in\mathbb{N}$ such that $(\lambda^l,\sigma_0^l,\hdots,\sigma_m^l)$ is a feasible point of the complex sum-of-squares relaxation of order $d(l)$. Thus $\lambda^l \leqslant \rho_{d(l)}^* \leqslant \rho^*$. As a result, $\rho_\text{limit}^* = \rho^*$. Moreover, $(\rho_d)_{d\geqslant d^\text{min}}$ is non-decreasing and upper bounded by $\rho \in \mathbb{R}\cup\{\pm \infty\}$. Thus it converges towards some limit $\rho_\text{lim} \in \mathbb{R}\cup\{\pm \infty\}$ such that $\rho_\text{lim} \leqslant \rho$. Moreover, weak duality implies that $\rho_d^* \leqslant \rho_d ~ (\leqslant \rho)$. Thus $\rho^* \leqslant \rho_\text{lim} \leqslant \rho$. It was already shown that $\rho \leqslant f^\text{opt}$.
\end{proof}
\begin{remark}
\label{rem:duality}
\normalfont
Problems \eqref{eq:sosinf} and \eqref{eq:momentinf} may be interpreted as a pair of primal-dual linear programs in infinite-dimensional spaces~\cite{anderson-1987}. Indeed, consider the duality bracket $\langle .,. \rangle$ defined from $\mathbb{R}[\bar{z},z] \times \mathcal{H}$ to $\mathbb{R}$ by $\langle \varphi , y \rangle := L_y(\varphi)$. A sequence $(\varphi^n)_{n\in \mathbb{N}}$ in $\mathbb{R}[\bar{z},z]$ is said to converge weakly towards $\varphi \in \mathbb{R}[\bar{z},z]$ if for all $y \in \mathcal{H}$, we have $\langle \varphi^n , y \rangle \longrightarrow \langle \varphi , y \rangle$. Consider the weakly continuous $\mathbb{R}$-linear operator $A : \mathbb{R}[\bar{z},z] \longrightarrow \mathbb{R}[\bar{z},z]$ defined by $\varphi \longmapsto \varphi - \varphi_{0,0}$. Its dual $A^* : \mathcal{H} \longrightarrow \mathcal{H}$ is defined by $y \longmapsto y - y_{0,0} \delta_{0,0}$ where $(\delta_{0,0})_{0,0} = 1$ and $(\delta_{0,0})_{\alpha,\beta} = 0$ if $(\alpha,\beta)\neq(0,0)$. Indeed, $\langle A\varphi , y \rangle = \langle \varphi , A^* y \rangle$ for all $(\varphi,y) \in \mathbb{R}[\bar{z},z] \times \mathcal{H}$. Consider the convex pointed cone defined by $C:= \Sigma[z] g_0 + \hdots + \Sigma[z]g_m$ and its dual cone $C^* := \{ y \in \mathcal{H} ~ | ~ \forall \varphi \in C, ~ \langle \varphi , y \rangle \geqslant 0 \}$. With $b:=Af$, notice that 
\begin{equation}
\label{eq:pd}
\begin{array}{rcllll}
f_{0,0} - \rho^* & = & \inf_{\varphi \in \mathbb{R}[\bar{z},z]} & \langle \varphi , \delta_{0,0} \rangle & \text{s.t.} & A\varphi = b  ~~ \& ~~ \varphi \in C, \\
f_{0,0} - \rho\hphantom{^*}  & = & \sup_{y \in \mathcal{H}} & \langle b , y \rangle & \text{s.t.} & \delta_{0,0} - A^*y \in C^*.
\end{array}
\end{equation}
Let $\text{cl}(C)$ denote the weak closure of $C$ in $\mathbb{R}[\bar{z},z]$. According to \cite[5.91 Bipolar Theorem]{aliprantis-1999}\footnote{We wish to thank Jean-Bernard Baillon for bringing this reference to our attention.}, we have $\text{cl}(C) = C^{**}$. In the next section, Theorem \ref{th:angelo} and Theorem \ref{th:represent} provide a sufficient condition ensuring no duality gap in \eqref{eq:pd} and $\text{cl}(C) = \{ \varphi \in \mathbb{R}[\bar{z},z] ~|~ \varphi_{|K} \geqslant 0 \}$ respectively.
\end{remark}

\subsection*{Convergence of the complex hierarchy} 
\label{subsec:Convergence of the Complex Hierarchy} 
We now turn our attention to a result from algebraic geometry discovered in 2008.
\begin{theorem}[D'Angelo's and Putinar's Positivstellenstatz~\cite{angelo-2008}]
\label{th:angelo}
If one of the constraints that define $K$ in \eqref{eq:assumption} is a sphere constraint $|z_1|^2 + \hdots + |z_n|^2 = 1$, then 
\begin{equation}
\label{eq:sumangelo}
f_{|K} > 0 ~~~~~~ \Longrightarrow ~~~~~~ \exists \sigma_0, \hdots, \sigma_m \in \Sigma[z]: ~~~ 
f = \sum_{i=0}^m \sigma_i g_i.
\end{equation}
\end{theorem}
\begin{proof}
D'Angelo and Putinar wrote the theorem slightly differently so we provide an explanation. Say that constraints $g_{m-1}$ and $g_m$ are such that $g_{m-1}=s$ and $g_m=-s$ where $s(z):= 1 - |z_1|^2 - \hdots - |z_n|^2$. With the assumptions of Theorem \ref{th:angelo}, the authors of \cite[Theorem 3.1]{angelo-2008} show that there exists $\sigma_0, \hdots, \sigma_{m-2} \in \Sigma[z]$ and $r \in \mathbb{R}[\bar{z},z]$ such that
$f(z) =\sum_{i=0}^{m-2} \sigma_i(z) g_i(z) +  r(z)s(z)$ for all $z\in \mathbb{C}^n$. Thanks to~\cite[Proposition 1.2]{angelo-2010}, there exists $\sigma_{m-1},\sigma_m \in \Sigma[z]$ such that $r = \sigma_{m-1} - \sigma_m$ hence the desired result. 
\end{proof}

Theorem \ref{th:angelo} can easily be generalized to any sphere $|z_1|^2 + \hdots + |z_n|^2 = R^2$ of radius $R > 0$. With scaled variable $w = \frac{z}{R} \in \mathbb{C}^n$, the sphere constraint has radius 1 and a monomial of \eqref{eq:complexPOP} with coefficient $c_{\alpha,\beta}\in \mathbb{C}$  reads $c_{\alpha,\beta} \bar{z}^\alpha z^\beta = c_{\alpha,\beta} (R\bar{w})^\alpha (Rw)^\beta = R^{|\alpha|+|\beta|} c_{\alpha,\beta} \bar{w}^\alpha w^\beta$. With the scaled coefficients $R^{|\alpha|+|\beta|} c_{\alpha,\beta}$, Theorem \ref{th:angelo} can then be applied. Reverting back to the old scale $z = Rw$ in \eqref{eq:sumangelo} leads to the desired result. Accordingly, we define the following statement which we will consider true only when explicitly stated:
\begin{equation}
\label{eq:sa}
\textbf{Sphere Assumption:} ~~~
\boxed{
\begin{array}{l}
\text{One of the constraints of polynomial} \\
\text{optimization problem \eqref{eq:complexPOP} is a sphere} \\ 
\text{constraint} ~ |z_1|^2 + \hdots + |z_n|^2 = R^2 ~ \text{for} \\
\text{some radius} ~ R > 0.
\end{array}
}
\end{equation}
Under the sphere assumption, $K$ is compact so assumption \eqref{eq:assumption} holds.
\begin{corollary}
\label{cor:convergence}
Under the sphere assumption \eqref{eq:sa}, $\rho_d^* \rightarrow f^\text{opt}$ and $\rho_d \rightarrow f^\text{opt}$.
\end{corollary}
\begin{proof} Theorem \ref{th:angelo} implies that $\rho^* = f^\text{opt}$
because for all $\epsilon > 0$, function \mbox{$f-(f^\text{opt}-\epsilon)$} is positive on $K$. The sequences $(\rho_d^*)_{d\geqslant d^\text{min}}$ and $(\rho_d)_{d\geqslant d^\text{min}}$ converge towards $f^\text{opt}$ due to Proposition \ref{prop:dualconverge}.
\end{proof}

To require a sphere constraint in a complex polynomial optimization problem seems very restrictive and irrelevant for many problems. But in fact, a sphere constraint can be applied to any complex polynomial optimization problem \eqref{eq:complexPOP} with a feasible set contained in a ball $|z_1|^2 + \hdots + |z_{n}|^2 \leqslant R^2$ of known radius $R > 0$. Indeed, simply add a slack variable $z_{n+1} \in \mathbb{C}$ and the constraint
\begin{equation}
\label{eq:sphere}
\boxed{
|z_1|^2 + \hdots + |z_{n+1}|^2 = R^2.
}
\end{equation}
Let $\hat{K}$ denote the feasible set of the problem in $n+1$ variables.
If $(z_1,\hdots,z_{n+1}) \in \hat{K}$, then $(z_1,\hdots,z_n) \in K$  and has the same objective value. Conversly, if $(z_1,\hdots,z_n) \in K$, then $(z_1,\hdots,z_{n+1}) \in \hat{K}$ for all $z_{n+1} \in \mathbb{C}$ such that $|z_{n+1}|^2 = R^2 - |z_1|^2  \hdots - |z_n|^2$. Again, the objective value is unchanged.
To ensure a bijection between $K$ and $\hat{K}$, add yet two more constraints $\textbf{i}z_{n+1}  - \textbf{i}\overline{z}_{n+1} =  0$ and $z_{n+1} + \overline{z}_{n+1} \geqslant 0$, thereby preserving the number of global solutions. In that case, the application from $K$ to $\hat{K}$ defined by $(z_1,\hdots,z_n) \longmapsto (z_1,\hdots,z_n,\sqrt{R^2-|z_1|^2-\hdots-|z_n|^2})$ is a bijection. Adding the two extra constraints is optional and not required for convergence of optimal values. 

As seen in Theorem \ref{th:angelo}, an equality constraint may be enforced via two opposite inequality constraints. Let $h_1,\hdots,h_e$ denote $e\in \mathbb{N}^*$ equality constraints in polynomial optimization problem \eqref{eq:complexPOP}.
Putinar and Scheiderer~\cite[Propositions 6.6 and 3.2 (iii)]{putinar-2013} show that the sphere assumption in D'Angelo's and Putinar's Positivstellensatz may be weakened to the existence of $r_1,\hdots,r_e \in \mathbb{R}[\bar{z},z]$, $\sigma \in \Sigma[z]$, and $a\in \mathbb{R}$ such that
\begin{equation}
\label{eq:weak}
\sum_{j=1}^e r_j(z) h_j(z) = \sum_{i=1}^n |z_i|^2 + \sigma(z)+a, ~~~~~~ \forall z\in \mathbb{C}^n.
\end{equation}
If a problem contains the constraints $|z_1|^2 - 1 = \hdots = |z_n|^2 - 1 = 0$, then the assumption is satisfied by $r_1 = \hdots = r_n = 1$, $\sigma = 0$ and $a=-n$. In particular, there is no need to add a slack variable in the non-bipartite Grothendieck problem over the complex numbers~\cite{bandeira-2014}.
\begin{example}
\label{ex:angelo}
\normalfont
D'Angelo and Putinar~\cite{angelo-2008} consider $\frac{1}{3} < a < \frac{4}{9}$ and problem
\begin{equation}
\label{eq:ex1}
\begin{array}{llcl}
 \inf_{z \in \mathbb{C}} & f(z) & := & 1-\frac{4}{3}|z|^2+a|z|^4, \\
            \text{s.t.}          & g(z) & := & 1 - |z|^2 \geqslant 0,
\end{array}
\end{equation}
whose set of global solutions is $K^\text{opt} = \{ z \in \mathbb{C}~|~ |z| =1\}$ and $f^\text{opt} = a-\frac{1}{3}>0$. They prove that the decomposition $f = \sigma_0 + \sigma_1g ~ (\sigma_0,\sigma_1 \in \Sigma[z])$ of Theorem \ref{th:angelo} does not hold. As a result, the optimal values of the complex sum-of-squares relaxations cannot exceed 0 even though $f^\text{opt}>0$. Indeed, if $\rho_d^* > 0$ for some order $d\geqslant d^\text{min}$, then there exists $\lambda \geqslant \frac{\rho_d^*}{2}$ and $\sigma_0,\sigma_1 \in \Sigma_d[z]$ such that $f - \lambda = \sigma_0 + \sigma_1 g$. Thus $f = \lambda + \sigma_0 + \sigma_1 g$ where $\lambda + \sigma_0 \in \Sigma_d[z]$, which is a contradiction. We suggest solving
\begin{equation}
\label{eq:ex1sphere}
\begin{array}{llcl}
 \inf_{z_1,z_2 \in \mathbb{C}} & \hat{f}(z_1,z_2) & := & 1-\frac{4}{3}|z_1|^2+a|z_1|^4, \\
            \text{s.t.}          & \hat{g}(z_1,z_2) & := & 1 - |z_1|^2 - |z_2|^2 = 0,
\end{array}
\end{equation}
for which the decomposition of Theorem \ref{th:angelo} holds thereby ensuring convergence of the complex moment/sum-of-squares hierarchy (Corollary \ref{cor:convergence}). In other words, for all $\lambda < f^\text{opt}$ there exists $\hat{\sigma}_0 \in \Sigma[z_1,z_2]$ and $\hat{r}\in \mathbb{R}[\overline{z}_1,\overline{z}_2,z_1,z_2]$ such that
\begin{equation}
\label{eq:pol}
\hat{f}(z_1,z_2) - \lambda = \hat{\sigma}_0(z_1,z_2) + \hat{r}(z_1,z_2) \hat{g}(z_1,z_2),~~~~~~ \forall z_1,z_2 \in \mathbb{C}.
\end{equation}
Plug in $z_1=z$ and $z_2=0$ and obtain $ f(z) - \lambda = \hat{\sigma}_0(z,0) + \hat{r}(z,0) g(z)$ for all $z \in \mathbb{C}$.
While function $z \longmapsto \hat{\sigma}_0(z,0)$ belongs to $\Sigma[z]$, function $z \longmapsto \hat{r}(z,0)$ does not! Hence we do not contradict the fact that $f = \sigma_0 + \sigma_1g ~ (\sigma_0,\sigma_1 \in \Sigma[z])$ is impossible. Consider $a = \frac{1}{2}(\frac{1}{3} + \frac{4}{9}) = \frac{7}{18}$ so that $f^\text{opt} = \frac{1}{18}$. Notice that $d^\text{min} = 2$ for \eqref{eq:ex1} and \eqref{eq:ex1sphere}. The complex relaxations of orders $2 \leqslant d \leqslant 3$ of \eqref{eq:ex1} both yield\footnote{MATLAB 2013a, YALMIP \mbox{2015.06.26}~\cite{yalmip}, and MOSEK are used for the numerical experiments.} the value $-0.3333$. The complex relaxation of order 2 of \eqref{eq:ex1sphere} yields the value $0.0556 ~ (\approx f^\text{opt})$ and optimal polynomials $\hat{\sigma}_0(z_1,z_2) = 0.2780 |z_2|^2 + 0.2776 |z_1 z_2|^2 + 0.6667 |z_2|^4$ and $\hat{r}(z_1,z_2) = 0.9444 - 0.3889|z_1|^2 + 0.6665 |z_2|^2$, all of which satisfy \eqref{eq:pol}. 
\end{example}

\begin{example}
\label{ex:putinar}
\normalfont
Putinar and Scheiderer~\cite{putinar-scheiderer-2012} consider parameters $0 < a < \frac{1}{2}$ and $C>\frac{1}{1-2a}$, and problem
\begin{equation}
\label{eq:ex2}
\begin{array}{llcl}
 \inf_{z \in \mathbb{C}} & f(z) & := & C - |z|^2, \\
            \text{s.t.}          & g(z) & := & |z|^2 - a z^2 - a \bar{z}^2 - 1 = 0,
\end{array}
\end{equation}
whose set of global solutions is $K^\text{opt} = \left\{ \pm \frac{1}{\sqrt{1-2a}} \right\}$ and $f^\text{opt} = C-\frac{1}{1-2a}>0$. They prove that the decomposition of Theorem \ref{th:angelo} does not hold. Since the feasible set is included in the Euclidean ball of radius $\sqrt{C}$, we suggest solving
\begin{equation}
\label{eq:ex2sphere}
\begin{array}{llcl}
 \inf_{z_1,z_2 \in \mathbb{C}} & \hat{f}(z_1,z_2) & := & C - |z_1|^2, \\
            \text{s.t.}          & \hat{g}_1(z_1,z_2) & := & |z_1|^2 - a z_1^2 - a \bar{z}_1^2 - 1 = 0, \\
                                    & \hat{g}_2(z_1,z_2) & := & C - |z_1|^2 - |z_2|^2 = 0, \\
                                    & \hat{g}_3(z_1,z_2) & := & \textbf{i} z_2 - \textbf{i} \overline{z}_2 = 0, \\
                                    & \hat{g}_4(z_1,z_2) & := & z_2 + \overline{z}_2 \geqslant 0.
\end{array}
\end{equation}
Consider $a = \frac{1}{4}$ and $C = \frac{1}{1-2a} + 1 = 3$ so that $f^\text{opt} = 1$. Notice that $d^\text{min} = 2$ for \eqref{eq:ex2} and \eqref{eq:ex2sphere}. The complex relaxations of orders $2\leqslant d \leqslant 3$ of \eqref{eq:ex2} are unbounded. The complex relaxation of order 2 of \eqref{eq:ex2sphere} yields the value 0.6813. That of order 3 yields the value 1.0000 and the complex moment matrix\footnote{It so happens that the Hermitian matrix $M_3(y)$, indexed by $(\alpha,\beta) \in \mathbb{N}^2 \times \mathbb{N}^2$ with $|\alpha|,|\beta| \leqslant 3$, is real-valued in this example.} with $10^{-4}$ precision
\begin{equation}
\label{eq:M3(y)}
M_3(y) ~~ = ~~
\begin{array}{ccccccccccc}
 & \begin{turn}{90} (0,0) \end{turn} & \begin{turn}{90} (1,0) \end{turn} & \begin{turn}{90} (0,1) \end{turn} & \begin{turn}{90} (2,0) \end{turn} & \begin{turn}{90} (1,1) \end{turn} & \begin{turn}{90} (0,2) \end{turn} & \begin{turn}{90} (3,0) \end{turn} & \begin{turn}{90} (2,1) \end{turn} & \begin{turn}{90} (1,2) \end{turn} & \begin{turn}{90} (0,3) \end{turn} \\
(0,0) & 1 & 0 & 1 & 2 & 0 & 1 & 0 & 2 & 0 & 1 \\
(1,0) & 0 & 2 & 0 & 0 & 2 & 0 & 4 & 0 & 2 & 0 \\
(0,1) & 1 & 0 & 1 & 2 & 0 & 1 & 0 & 2 & 0 & 1 \\
(2,0) & 2 & 0 & 2 & 4 & 0 & 2 & 0 & 4 & 0 & 2 \\
(1,1) & 0 & 2 & 0 & 0 & 2 & 0 & 4 & 0 & 2 & 0 \\
(0,2) & 1 & 0 & 1 & 2 & 0 & 1 & 0 & 2 & 0 & 1 \\
(3,0) & 0 & 4 & 0 & 0 & 4 & 0 & 8 & 0 & 4 & 0 \\
(2,1) & 2 & 0 & 2 & 4 & 0 & 2 & 0 & 4 & 0 & 2 \\
(1,2) & 0 & 2 & 0 & 0 & 2 & 0 & 4 & 0 & 2 & 0 \\
(0,3) & 1 & 0 & 1 & 2 & 0 & 1 & 0 & 2 & 0 & 1
\end{array} ~~~~~~
\end{equation}
which satisfies $\text{rank} ~ M_3(y) = \text{rank} ~ M_1(y) = 2$.\\
\end{example}
Examples \ref{ex:angelo} and \ref{ex:putinar} show the importance of the modeling of the feasible set of the optimization problem. Depending on what equations are used to define the feasible set, the complex moment/sum-of-squares hierarchy may or may not converge towards the globally optimal value. If one of the constraints is a sphere, convergence is guaranteed. The real moment/sum-of-squares hierarchy also depends on how the feasible set is modeled. In that case, convergence is guaranteed if one of the constraints is a ball.

As a by-product of Corollary \ref{cor:convergence}, we propose a solution to the complex moment problem in Theorem \ref{th:represent} below. To that end, consider Lemma \ref{lemma:sphere} below where we transpose~\cite[Lemma 3]{josz-2015} from real to complex numbers.
\begin{lemma}
\normalfont
\label{lemma:sphere}
Let $s : \mathbb{C}^n \longrightarrow \mathbb{R}$ be defined by $s(z):= R^2 - |z_1|^2 - \hdots - |z_n|^2$. Given $d \in \mathbb{N}^*$ and $y\in \mathcal{H}_d$, we have
\begin{equation}
~~ (~M_d(g_0y) \succcurlyeq 0 ~~~ \& ~~~ M_{d-1}(sy) = 0~) ~~~~~~ \Longrightarrow ~~~~~~  \text{Tr}(M_{d}(g_0y)) \leqslant y_{0,0} \sum_{l=0}^d R^{2l}.
\end{equation}
\end{lemma}
\begin{proof}
Given $1 \leqslant l \leqslant d$, we have $\text{Tr}(M_{l-1}(sy)) = \hdots$
\begin{equation}
\begin{array}{cl}
 = & \sum_{|\alpha|\leqslant l-1} ~ M_{l-1}(sy)(\alpha,\alpha) \\
                                    = & \sum_{|\alpha|\leqslant l-1} ~ L_y(s(z) \bar{z}^\alpha z^\alpha) \\
                                    = & \sum_{|\alpha|\leqslant l-1} \sum_{|\gamma| \leqslant 1} ~ s_{\gamma,\gamma} ~ y_{\gamma+\alpha,\gamma+\alpha} \\
                                    = & \sum_{|\alpha|\leqslant l-1,|\gamma| = 0}  ~ s_{\gamma,\gamma} ~ y_{\gamma+\alpha,\gamma+\alpha} ~+~ \sum_{|\alpha|\leqslant l-1,|\gamma| = 1} ~ s_{\gamma,\gamma} ~ y_{\gamma+\alpha,\gamma+\alpha} \\
 = & \sum_{|\alpha|\leqslant l-1} R^2 ~ y_{\alpha,\alpha} - \sum_{|\alpha|\leqslant l-1,|\gamma| = 1} y_{\gamma+\alpha,\gamma+\alpha}. \\                                 
\end{array}
\end{equation}
$M_{d-1}(sy) = 0$ implies that $M_{l-1}(sy) = 0$ for all $1 \leqslant l \leqslant d$ and hence $ \text{Tr}(M_{l-1}(sy)) = 0 $. In addition, $\sum_{0<|\alpha|\leqslant l}  y_{\alpha,\alpha} \leqslant \sum_{|\alpha|\leqslant l-1,|\gamma| = 1} y_{\gamma+\alpha,\gamma+\alpha}$. Thus
\begin{equation}
\label{eq:recurrence}
\sum_{|\alpha|\leqslant l}  y_{\alpha,\alpha}  \leqslant y_{0,0} + R^2  \sum_{|\alpha|\leqslant l-1} y_{\alpha,\alpha},~~~~~~ l = 1,\hdots,d,
\end{equation}
which proves the lemma.
\end{proof}
The next theorem can be deduced from~\cite{putinar-2013} but we provide a different proof.
\begin{theorem}
\label{th:represent} 
Under the sphere assumption \eqref{eq:sa}, a sequence $y\in \mathcal{H}$ has a representing measure on $K$ if and only if
\begin{equation}
\label{eq:sol}
M_{d}(g_iy) \succcurlyeq 0, ~~~~~~ i = 0,\hdots,m, ~~~~~~ \forall d \in \mathbb{N}.
\end{equation}
\end{theorem}
\begin{proof} 
It was already shown that if $y \in \mathcal{H}$ has a representing measure on $K$, then \eqref{eq:necessary} holds. Notice that \eqref{eq:necessary} and \eqref{eq:sol} are equivalent, hence the ``only if'' part. Concerning the ``if'' part, assume that $y \in \mathcal{H}$ satisfies \eqref{eq:sol}. If $y_{0,0} = 0$, then Lemma \ref{lemma:sphere} implies that $y = 0$ which can be represented by $\mu = 0$ on $K$. Otherwise $y_{0,0}>0$ and $y/y_{0,0}$ is a feasible point of problem \eqref{eq:momentinf} whose optimal value is $f^\text{opt}$ for all $f \in \mathbb{R}[\bar{z},z]$ according to Corollary \ref{cor:convergence}. If moreover $f_{|K} \geqslant 0$, then $L_{ y/y_{0,0} } (f) \geqslant f^\text{opt} \geqslant 0$. In particular, if $f_{|K} = 0$, then $L_{ y/y_{0,0} } (f) = 0$. We may therefore define $\tilde{L}_{ y/y_{0,0} } : R_K( \mathbb{C}[\bar{z},z] ) \longrightarrow \mathbb{C}$ such that $\tilde{L}_{ y/y_{0,0} } ( \varphi_{|K} ) :=  L_{ y/y_{0,0} } (\varphi)$ (similarily to Schweighofer~\cite[Proof of Theorem 2]{schweighofer-2005}). If $\varphi \in R_K(\mathbb{R}[\bar{z},z])$, then $\tilde{L}_{ y/y_{0,0} }(\|\varphi\|_\infty - \varphi) \geqslant 0$ and $\tilde{L}_{ y/y_{0,0} }(\varphi) \leqslant \|\varphi\|_\infty$. Linearity implies that $|\tilde{L}_{ y/y_{0,0} }(\varphi)| \leqslant \|\varphi\|_\infty$. 
As a result, for all $\varphi \in R_K(\mathbb{C}[\bar{z},z])$, we have
$ |\tilde{L}_{ y/y_{0,0} }(\varphi)| = | \tilde{L}_{ y/y_{0,0} }(\text{Re}(\varphi) + \textbf{i} \text{Im}(\varphi)) | = | \tilde{L}_{ y/y_{0,0} }(\text{Re}(\varphi)) + \textbf{i} \tilde{L}_{ y/y_{0,0} }(\text{Im}(\varphi)) | \leqslant | \tilde{L}_{ y/y_{0,0} }(\text{Re}(\varphi)) | + | \tilde{L}_{ y/y_{0,0} }(\text{Im}(\varphi)) | \leqslant \| \text{Re}(\varphi) \|_\infty + \| \text{Im}(\varphi) \|_\infty 
\leqslant 2 \| \varphi \|_\infty
$.
Moreover, $R_K(\mathbb{C}[\bar{z},z])$ is dense in $C(K,\mathbb{C})$. Therefore $\tilde{L}_{y/y_{0,0}}$ may be extended to a continous linear functional on $C(K,\mathbb{C})$ (we preserve the same name for the extension). $K$ is compact thus the Riesz-Markov-Kakutani Representation Theorem implies that there exists a unique Radon measure $\mu$ such that
$\tilde{L}_{y/y_{0,0}}(\varphi) = \int_K \varphi d\mu$ for all $\varphi \in C(K,\mathbb{C})$. It is positive because $\varphi \in \mathcal{P}(K)$ implies that $\tilde{L}_{ y/y_{0,0} }(\varphi) \geqslant 0$ (density argument). Finally, if $\alpha,\beta \in \mathbb{N}^n$, $y_{\alpha,\beta} / y_{0,0}  = L_{y/y_{0,0}} ( \bar{z}^\alpha z^\beta )$ (c.f. Remark \ref{rem:Ly}) hence $y$ has representing measure $y_{0,0}\mu$ on $K$.
\end{proof}

Vasilescu~\cite[Theorem I.2.17]{vasilescu-2003} has already proposed a different solution to the complex moment problem on $K$. We now transpose the proof of~\cite[Theorem 1]{josz-2015} from real to complex numbers.
\begin{proposition}
\label{prop:duality}
\normalfont
\textit{Under the sphere assumption \eqref{eq:sa},} $\rho_d^* = \rho_d \in \mathbb{R}\cup\{+\infty\}$ \textit{for all} $d \geqslant d^\text{min}$.
\end{proposition}
\begin{proof}
Given $A \in \mathcal{H}_d$, consider the operator norm  $\| A \| $, the greatest eigenvalue of $A$ in absolute value, and the Frobenius norm $\| A \|_{\mathcal{H}_d} := \sqrt{ \langle A , A \rangle_{\mathcal{H}_d} }$. Consider $d \geqslant d^\text{min}$. Two cases can occur. The first is that the feasible set of the complex moment relaxation of order $d$ is non-empty, in which case we consider a feasible point $(y_{\alpha,\beta})_{|\alpha|,|\beta|\leqslant d}$. All norms are equivalent in finite dimension so there exists a constant $C_d \in \mathbb{R}$ such that
$\sqrt{ \sum_{|\alpha|,|\beta|\leqslant d} |y_{\alpha,\beta}|^2 }  = \| M_{d}(g_0y) \|_{\mathcal{H}_d} \leqslant C_d ~ \| M_{d}(g_0y) \| \leqslant C_d ~ \sum_{l=0}^d R^{2l}$, according to Lemma \ref{lemma:sphere}.                                                          
As a result, the feasible set of the complex moment relaxation of order $d$ is a non-empty compact set and so is its image by $\Lambda$ (defined in \eqref{eq:conversion}). We can thus apply Trnovsk\'a's result \cite{maria-2005} which states that in a semidefinite program in real numbers, if the primal feasible set is non-empty and compact, then there exists a dual interior point and there is no duality gap.

The second case is that the feasible set of the complex moment relaxation of order $d$ is empty, i.e. $\rho_d = + \infty$. It must be strongly infeasible because it cannot be weakly infeasible (see~\cite[Section 5.2]{klerk-2000} for definitions). Indeed, if it is weakly infeasible, then there exists a sequence $(y^j)_{j\in \mathbb{N}}$ of elements of $\mathcal{H}$ such that for all $j \in \mathbb{N}$, we have $|y^j_{0,0}-1| \leqslant \frac{1}{j+1}$ and
$\lambda_\text{min} ( M_{d-k_i} (g_iy^j) ) \geqslant - \frac{1}{j+1}$ where $i = 0,\hdots,m$.
Define $c:= (n+d)!/(n!d!)$. We now mimick the computations in Lemma \ref{lemma:sphere} using $y^j_{0,0} \leqslant 1 + \frac{1}{j+1} \leqslant 2$ and $|\text{Tr}(M_{l-1} (sy^j))|  \leqslant \frac{c}{j+1} \leqslant c$ if $1\leqslant l \leqslant d$. Consider $j_0 \in \mathbb{N}$ such that for all $j\geqslant j_0$ and $1\leqslant l \leqslant d$, we have $\sum_{|\alpha|\leqslant l-1,|\gamma| = 1} y_{\gamma+\alpha,\gamma+\alpha}^j - \sum_{0<|\alpha|\leqslant l}  y_{\alpha,\alpha}^j \geqslant -1$. Equation \eqref{eq:recurrence} then becomes $\sum_{|\alpha|\leqslant l}  y_{\alpha,\alpha}^j \leqslant 2 + R^2 \left(\sum_{|\alpha|\leqslant l-1}  y_{\alpha,\alpha}^j \right) + c+1$.
As a result, $\text{Tr}(M_{d} (g_0y^j)) = \sum_{|\alpha|\leqslant d}  y_{\alpha,\alpha}^j \leqslant ( 3 + c ) \sum_{l=0}^d R^{2l}$, which, together with $\lambda_\text{min} ( M_{d} (g_0y^j) ) \geqslant -\frac{1}{j+1} \geqslant- 1$, yields
$\lambda_\text{max} ( M_{d} (g_0y^j) ) \leqslant ( 3 + c ) \sum_{l=0}^d R^{2l} + c - 1$. Hence for all $j \geqslant j_0$, the spectrum of $M_{d} (g_0y^j)$ is lower bounded by $-1$ and upper bounded by $B_d :=  ( 3 + c ) \sum_{l=0}^d R^{2l} + c - 1 \geqslant 1$. We therefore have
$\sqrt{ \sum_{|\alpha|,|\beta|\leqslant d} |y_{\alpha,\beta}^j|^2 }  \leqslant C_d ~ \| M_{d}(g_0y) \| \leqslant  C_d \times B_d$.
The sequence $(y^j)_{j\geqslant j_0}$ is thus included in a compact set. Hence there exists a subsequence that converges towards a limit $y^\text{lim}$ which satisfies $y^\text{lim}_{0,0} =1$ and the constraints $\lambda_\text{min} ( M_{d-k_i} (g_iy^\text{lim}) ) \geqslant 0, ~ i = 0,\hdots,m$. Therefore $y^\text{lim}$ is a feasible point of the complex moment relaxation of order $d$, which is a contradiction. Strong infeasibility means that the dual feasible set contains an improving ray~\cite[Definition 5.2.2]{klerk-2000}. Moreover, $\inf_{y\in \mathcal{H}_d} L_y(f)$ subject to $y_{0,0} = 1, ~ M_d(g_0y) \succcurlyeq 0, ~\text{and} ~ M_{d-1}(sy) = 0$ is a semidefinite program with a non-empty compact feasible set hence the dual feasible set contains a point $(\lambda,\sigma_0,\sigma_1)$. As result $(\lambda,\sigma_0,\sigma_1,0,\hdots,0)$ is a feasible point of the complex sum-of-squares relaxation of order $d$. Together with the improving ray, this means that $\rho_d^* = + \infty$.
To conclude, $\rho_d^* = \rho_d$ in both cases.
\end{proof}

\begin{proposition}
\label{prop:saddle}
\normalfont
\textit{Assume that complex polynomial optimization problem \eqref{eq:complexPOP} satisfies \eqref{eq:weak} and has a global solution} $z^\text{opt} \in K^\text{opt}$. \textit{In addition, assume that} $(\sigma_0^\text{opt},\hdots,\sigma_m^\text{opt}) \in \Sigma[z]^{m+1}$ \textit{is an optimal solution to the sum-of-squares problem \eqref{eq:sosinf}. 
Then} $(z^\text{opt},\sigma_1^\text{opt},\hdots,\sigma_m^\text{opt})$ \textit{is a saddle point of $\phi : \mathbb{C}^n \times \Sigma[z]^m \longrightarrow \mathbb{R}$ defined by $(z,\sigma) \longmapsto f(z) - \sum_{i=1}^m \sigma_i(z) g_i(z)$.}
\end{proposition}
\begin{proof}
The optimality of $(\sigma_0^\text{opt},\hdots,\sigma_m^\text{opt})$ means that $f - f^\text{opt} = \sum_{i=0}^m \sigma_i^\text{opt} g_i$. With $f(z^\text{opt}) - f^\text{opt} = \sum_{i=0}^m \sigma_i^\text{opt}(z^\text{opt}) g_i(z^\text{opt}) = 0$, $\sigma_i^\text{opt}(z^\text{opt}) \geqslant 0$, and $g_i(z^\text{opt}) \geqslant 0$, we have $\sigma_i^\text{opt}(z^\text{opt}) g_i(z^\text{opt}) = 0$ for $i=0,\hdots,m$. It follows that $\phi(z^\text{opt},\sigma) \leqslant \phi(z^\text{opt},\sigma^\text{opt})$ for all $\sigma \in \Sigma[z]$. For all $z \in \mathbb{C}^n$, $\phi(z^\text{opt},\sigma^\text{opt}) \leqslant \phi(z,\sigma^\text{opt})$ because $f(z) - f^\text{opt} - \sum_{i=1}^m \sigma_i^\text{opt}(z) g_i(z) = \sigma_0^\text{opt}(z) \geqslant 0$.
\end{proof}

Given an application $\varphi : \mathbb{C}^n \longrightarrow \mathbb{R}$, define $\tilde{\varphi} : \mathbb{R}^{2n} \longrightarrow \mathbb{R}$ by $(x,y) \longmapsto \varphi(x+\textbf{i}y)$. 
If $\tilde{\varphi}$ is $\mathbb{R}$-differentiable at point $(x,y) \in \mathbb{R}^{2n}$, consider the Wirtinger derivative~\cite{wirtinger-1927} defined by $\nabla \varphi (x+\textbf{i}y):= \frac{1}{2}(\nabla_x \tilde{\varphi} (x,y) - \textbf{i} \nabla_y \tilde{\varphi} (x,y) ) \in \mathbb{C}^{n}$.
\begin{corollary}
\label{cor:KKT}
\normalfont
\textit{With the same assumptions as in} Proposition \ref{prop:saddle}\textit{, we have}
\begin{equation}
\boxed{
\begin{array}{l}
\nabla f (z^\text{opt}) = \sum_{i=1}^m \sigma_i^\text{opt}(z^\text{opt}) \nabla g_i(z^\text{opt}), \\
\sigma_i^\text{opt}(z^\text{opt}), g_i(z^\text{opt}) \geqslant 0, ~~~ i = 1,\hdots,m, \\
\sigma_i^\text{opt}(z^\text{opt})  g_i(z^\text{opt}) = 0, ~~~ i = 1,\hdots,m.
\end{array}
}
\end{equation}
\end{corollary}
\begin{proof}
$z^\text{opt}$ is a minimizer of $z \in \mathbb{C}^n \longmapsto \phi(z,\sigma^\text{opt})$ thus $\nabla_z \phi(z^\text{opt},\sigma^\text{opt}) = \nabla f (z^\text{opt}) - \sum_{i=1}^m \nabla \sigma_i^\text{opt}(z^\text{opt}) g_i(z^\text{opt}) - \sum_{i=1}^m \sigma_i^\text{opt}(z^\text{opt}) \nabla g_i(z^\text{opt}) = 0$. Consider $1\leqslant i \leqslant m$. Since $\sigma_i^\text{opt}(z^\text{opt})= 0$ and $\sigma_i^\text{opt} \in \Sigma[z]$, it must be that $|z_k-z_k^\text{opt}|^2$ divides $\sigma_{i,k}^\text{opt}: z_k \in \mathbb{C} \longmapsto \sigma_i^\text{opt}(z_1^\text{opt},\hdots,z_{k-1}^\text{opt},z_k,z_{k+1}^\text{opt},\hdots,z_n^\text{opt})$. With $z_k^\text{opt} =: x_k^\text{opt}+\textbf{i}y_k^\text{opt}$, the real number $x_k^\text{opt}$ is a root of multiplicity 2 of $x_k \in \mathbb{R} \longmapsto \sigma_{i,k}^\text{opt} (x_k+\textbf{i}y_k^\text{opt})$, with an analogous remark for $y_k^\text{opt}$. Thus $\nabla \sigma_i^\text{opt}(z^\text{opt}) = 0$ which leads to the desired result.
\end{proof}

\subsection*{Comparison of real and complex hierarchies} 
\label{subsec:Comparison of Real and Complex Hierarchies} 

Similar to Shor relaxation and the second-order conic relaxation, the following notations will be used: POP-$\mathbb{C}$ denotes the complex polynomial optimization problem \eqref{eq:complexPOP}; POP-$\mathbb{R}$ denotes the real polynomial optimization problem after conversion of POP-$\mathbb{C}$ into real numbers; $\text{MSOS}_d$-$\mathbb{C}$ denotes the complex moment/sum-of-squares relaxation of order $d$ applied to POP-$\mathbb{C}$; $\text{CMSOS}_d$-$\mathbb{R}$ denotes the conversion of $\text{MSOS}_d$-$\mathbb{C}$ into real numbers; and $\text{MSOS}_d$-$\mathbb{R}$ denotes the real moment/sum-of-squares relaxation of order $d$ applied to POP-$\mathbb{R}$. Let $d^\text{min}$-$\mathbb{R}$ and $d^\text{min}$-$\mathbb{C}$ respectively denote the minimum orders of the real and complex hierarchies.
Consider the sets
\begin{equation}
\label{eq:realsets}
\begin{array}{rl}
\mathbb{R}[x,y] := & \{  ~  q : \mathbb{R}^{2n} \rightarrow \mathbb{R} ~|~ q(x,y) = \sum_{|\kappa|\leqslant j} q_{\kappa} (x,y)^\kappa,  j \in \mathbb{N}, q_{\kappa} \in \mathbb{R}  \}, \\
\mathbb{R}_d[x,y] := & \{ ~ q : \mathbb{R}^{2n} \rightarrow \mathbb{R} ~|~ q(x,y) = \sum_{|\kappa|\leqslant d} q_{\kappa} (x,y)^\kappa, ~ q_{\kappa} \in \mathbb{R}  ~ \}, \\
\Sigma_{d}[x,y]:= & \{ ~ \sigma : \mathbb{R}^{2n} \rightarrow \mathbb{R} ~|~ \sigma = \sum_{i=1}^r q_i^2, ~ \text{with} ~ r\in \mathbb{N}^*, q_i \in \mathbb{R}_d[x,y] ~ \},
\end{array}
\end{equation}
where $\kappa \in \mathbb{N}^{2n}$ and $(x,y)^\kappa := x_1^{\kappa_1} \hdots x_n^{\kappa_n} y_1^{\kappa_{n+1}} \hdots y_n^{\kappa_{2n}}$. 
\begin{proposition}
\label{prop:compare}
\normalfont
\textit{Under the sphere assumption \eqref{eq:sa}, for all integer $d$ greater than or equal to} $\max \{  \text{$d^\text{min}$-$\mathbb{R}$},\text{$d^\text{min}$-$\mathbb{C}$} \}$, \textit{we have}
\begin{equation}
\label{eq:compare}
\text{val}(\text{MSOS}_d\text{-$\mathbb{C}$}) = \text{val}(\text{CMSOS}_d\text{-$\mathbb{R}$}) \leqslant \text{val}(\text{MSOS}_d\text{-$\mathbb{R}$}).
\end{equation}
\end{proposition}
\begin{proof}
It suffices to compare the optimal values of the real and complex sum-of-squares relaxations. This is due to Proposition \ref{prop:duality} and~\cite{josz-2015} where the ball constraint can be replaced by the sphere constraint to ensure no duality gap. We have
\begin{equation}
\begin{array}{rcll}
\text{val}(\text{POP-$\mathbb{C}$}) 	               & = & \sup_{\lambda \in \mathbb{R}}  & \lambda, \\
       						 		      &    & \mathrm{s.t.}                              &  f(z)- \lambda \geqslant 0, ~\forall z \in K, \\\\
\text{val}(\text{MSOS}_d\text{-$\mathbb{C}$})    & = & \sup_{\lambda,\sigma}               & \lambda, \\
              						              &    & \text{s.t.}                                     &  f-\lambda = \sum_{i=0}^m \sigma_i g_i,\\ 
            							     &    &             				    &  \lambda \in \mathbb{R},~ \sigma_i \in \Sigma_{d-k_i}[z] ,~ i=0, \hdots, m, \\\\
\text{val}(\text{CMSOS}_d\text{-$\mathbb{R}$}) & = & \sup_{\lambda,\sigma}                & \lambda, \\
             						 	     &    & \text{s.t.} 				    &  \tilde{f}-\lambda = \sum_{i=0}^m \sigma_i \tilde{g_i}, \\ 
             							     &    &               				    &  \lambda \in \mathbb{R},~ \sigma_i \in \Sigma_{d-k_i}[x+\textbf{i}y] ,~ i=0, \hdots, m, \\\\
\text{val}(\text{POP-$\mathbb{R}$}) 		    & = & \sup_{\lambda \in \mathbb{R}}    & \lambda,  \\
       							 	    &     & \mathrm{s.t.}                               &  \tilde{f}(x,y)- \lambda \geqslant 0, ~ \forall (x+\textbf{i}y) \in K, \\\\
\text{val}(\text{MSOS}_d\text{-$\mathbb{R}$})  & = & \sup_{\lambda,\sigma}                 & \lambda, \\
              							    &    & \text{s.t.}                                      &  \tilde{f}-\lambda = \sum_{i=0}^m \sigma_i \tilde{g_i}, \\ 
             							   &    &               				    &  \lambda \in \mathbb{R},~ \sigma_i \in \Sigma_{d-k_i}[x,y] ,~ i=0, \hdots, m.
\end{array}
\end{equation}

We now conclude because for all $d\in \mathbb{N}$, $\Sigma_{d}[x+\textbf{i}y] \subset \Sigma_d[x,y]$. Indeed, if $\sigma = \sum_{j=1}^r |p_j|^2$ with $r\in \mathbb{N}^*$ and $p_1,\hdots,p_r \in \mathbb{C}_d[z]$, then $\tilde{\sigma}(x,y) =  \sum_{j=1}^r |\tilde{p}_j(x,y)|^2 = \frac{1}{4} \sum_{j=1}^r \left(\tilde{p}_j(x,y)+\overline{\tilde{p}_j(x,y)}\right)^2 +  \left(\textbf{i} \tilde{p}_j(x,y)+\overline{\textbf{i}\tilde{p}_j(x,y)}\right)^2 \in \Sigma_d[x,y]$. 
\end{proof}

We may suspect the inequality in \eqref{eq:compare} to be strict in some cases because $\Sigma_{d}[x+\textbf{i}y]$ is a strict subset of $\Sigma_d[x,y]$ for all $d \in \mathbb{N}^*$. Indeed, for $i=1,\hdots,n$, we have $x_i^2 = \left(\frac{z_i+\bar{z}_i}{2}\right)^2 = \frac{1}{4} ( z_i^2 + 2|z_i|^2 + \bar{z}_i^2) \in \Sigma_d[x,y] \setminus \Sigma_{d}[x+\textbf{i}y]$. 
According to numerical experiments\footnote{We attempted a formal proof but it is difficult even on such a small example.}, the inequality is strict for \eqref{eq:ex2sphere} in Example \ref{ex:putinar} ($\text{val}(\text{CMSOS}_2\text{-$\mathbb{R}$}) \approx 0.6813$ and $\text{val}(\text{MSOS}_2\text{-$\mathbb{R}$}) \approx 1.0000$). 

Proposition \ref{eq:compare} seems to imply that the real moment/sum-of-squares hierarchy is better than the complex one.
However, the size of the largest semidefinite constraint of $\text{CMSOS}_d$-$\mathbb{R}$, equal to $2(n+d)!/(n!d!)$, is far inferior to that of $\text{MSOS}_d$-$\mathbb{R}$, equal to $(2n+d)!/((2n)!d!)$. For instance, if $n=10$ and $d=3$, the former is 572 and the latter is 1,771.

\begin{proposition}
\label{prop:oscillatory}
\normalfont
\textit{Given $l \in \mathbb{N}$ and $\varphi \in \mathbb{C}_l[\bar{z},z]$, we have}
\begin{equation}
\label{eq:oscillatory}
\forall z \in \mathbb{C}^n,~ \forall \theta \in \mathbb{R}, ~ \varphi(e^{\textbf{i}\theta} z) = \varphi(z) ~~~~ \Longleftrightarrow ~~~~ \forall |\alpha|,|\beta|\leqslant l, ~ |\alpha-\beta|\varphi_{\alpha,\beta} = 0.
\end{equation}
\end{proposition}
\begin{proof}
($\Longrightarrow$) Notice that $\varphi(e^{\textbf{i}\theta} z) = \sum_{|\alpha|,|\beta|\leqslant l} \varphi_{\alpha,\beta} \overline{(e^{\textbf{i}\theta} z)}^\alpha (e^{\textbf{i}\theta} z)^\beta = \hdots$ \\ $\sum_{|\alpha|,|\beta|\leqslant l} \varphi_{\alpha,\beta} e^{\textbf{i}(|\alpha|-|\beta|)\theta} \bar{z}^\alpha z^\beta$. Polarization implies that for all $z,w\in \mathbb{C}^n$, we have\\ $\sum_{|\alpha|,|\beta|\leqslant l} \varphi_{\alpha,\beta} e^{\textbf{i}(|\alpha|-|\beta|)\theta} \bar{z}^\alpha w^\beta = \sum_{|\alpha|,|\beta|\leqslant l} \varphi_{\alpha,\beta} \bar{z}^\alpha w^\beta$ and hence for all $|\alpha|,|\beta|\leqslant l$, $\varphi_{\alpha,\beta} e^{\textbf{i}(|\alpha|-|\beta|)\theta} = \varphi_{\alpha,\beta}$. If $\varphi_{\alpha,\beta} \neq 0$, then for all $\theta \in \mathbb{R}$, $|\alpha-\beta|\theta \equiv 0 [2\pi]$ and thus $|\alpha-\beta|=0$.
($\Longleftarrow$) Simply compute $\varphi(e^{\textbf{i}\theta} z)$.
\end{proof}

\begin{definition}
\label{def:oscillatory}
Complex polynomial optimization problem \eqref{eq:complexPOP} is said to be oscillatory if $f,g_1,\hdots$, and $g_m$ satisfy either of the two equivalent properties in \eqref{eq:oscillatory}.
\end{definition}

\begin{proposition}
\label{prop:order}
\normalfont
\textit{If complex polynomial optimization problem \eqref{eq:complexPOP} is oscillatory, then} $\text{$d^\text{min}$-$\mathbb{R}$} = \text{$d^\text{min}$-$\mathbb{C}$}$.
\end{proposition}
\begin{proof}
Observe that $d^\text{min}$-$\mathbb{C} = \max \{ |\alpha| , |\beta| ~|~ f_{\alpha,\beta}~g_{1,\alpha,\beta} ~ \hdots ~ g_{m,\alpha,\beta} \neq 0 \}$ and $\text{$d^\text{min}$-$\mathbb{R}$} = \max \{ \lceil (|\alpha| + |\beta|) / 2 \rceil ~|~ f_{\alpha,\beta}~g_{1,\alpha,\beta} ~ \hdots ~ g_{m,\alpha,\beta} \neq 0 \}$ where $\lceil . \rceil$ denotes the ceiling of a real number. Both are equal if the problem is oscillatory.
\end{proof}

\begin{conjecture}
\label{con}
\normalfont
\textit{Under the sphere assumption \eqref{eq:sa}, if complex polynomial optimization problem \eqref{eq:complexPOP} is oscillatory, then for all} $d \geqslant \text{$d^\text{min}$-$\mathbb{R}$} = \text{$d^\text{min}$-$\mathbb{C}$}$, \textit{we have}
\begin{equation}
\label{eq:conjecture}
\text{val}(\text{MSOS}_d\text{-$\mathbb{C}$}) = \text{val}(\text{CMSOS}_d\text{-$\mathbb{R}$}) = \text{val}(\text{MSOS}_d\text{-$\mathbb{R}$}).
\end{equation}
\end{conjecture}
In Section \ref{subsec:Numerical Results7}, we consider problems for which Conjecture \ref{con} seems to hold numerically. This suggests that for oscillatory problems, the complex hierarchy is more tractable than the real hierarchy at no loss of bound quality.

\subsection*{Exploiting sparsity in real and complex hierarchies}
\label{subsec:Exploiting Sparsity in Real and Complex Hierarchies}
The chordal sparsity technique described in Section~\ref{subsec:Motivation} has been extended to the real hierarchy by Waki~\cite{waki-2006} and may readily be transposed to the complex hierarchy. Each positive semidefinite constraint in~\eqref{eq:momentd} is replaced by a set of positive semidefinite constraints on certain submatrices of $M_{d-k_i}\left(g_i y\right)$. These submatrices are defined by the maximal cliques of a chordal extension of the graph associated with the objective and constraint equations. Equivalently, the sum-of-squares variables $\sigma_i$ in the dual formulation~\eqref{eq:sosd} are restricted to be functions of a subset (defined by the same maximal cliques) of the decision variables $z_1,\hdots,z_n$. These sparse relaxation hierarchies provide potentially lower bounds than their dense counterparts yet retain convergence guarantees~\cite{lasserre_book}. However, further size reduction is often necessary. We propose to selectively apply computationally intensive higher-order constraints in the sparse relaxations. In other words, rather than a single relaxation order applied to all constraints, each constraint has an associated relaxation order. This allows for solving many large-scale problems. 

We now formalize our approach applied to the complex hierarchy.\footnote{See~\cite{mh_sparse_msdp} for the details of this approach as applied to \mbox{$\text{MSOS}_d$-$\mathbb{R}$} in the context of the optimal power flow problem.} Objective function $f$ and constraints $(g_i)_{1 \leqslant i \leqslant m}$ in \eqref{eq:complexPOP} have an associated undirected sparsity graph $\mathcal{G}=\left(\mathcal{N},\mathcal{E}\right)$ with nodes $\mathcal{N} = \left\lbrace 1,\ldots , n\right\rbrace$ corresponding to each variable and edges $\mathcal{E} \subset \mathcal{N}\times \mathcal{N}$ for each pair of variables that appear together in any monomial that has a non-zero coefficient in the objective function or constraints.

Each constraint function $g_i$ has an associated relaxation order $d_i$ so that $d \in \mathbb{N}^m$. When $d_i > 1$, there must exist at least one clique of a chordal extension of $\mathcal{G}$ that contains all variables with non-zero coefficients in $g_i$. To ensure this, define a supergraph $\hat{\mathcal{G}}=(\mathcal{N},\hat{\mathcal{E}})$ where $\hat{\mathcal{E}}$ is composed of $\mathcal{E}$ augmented with edges connecting all variables with non-zero coefficients in $g_i$, not necessarily in the same monomial. For example, $g_1\left(z \right) = \bar{z}_1 z_2 + \bar{z}_2 z_1  + \bar{z}_1 z_3 +  \bar{z}_3 z_1+ \bar{z}_1 z_4 + \bar{z}_4 z_1$ with $d_1 > 1$ implies $\mathcal{E} \supset \left\lbrace\left(1,2\right),\, \left(1,3\right), \left(1,4\right)\right\rbrace$ and $\hat{\mathcal{E}} \supset \left\lbrace\left(1,2\right),\, \left(1,3\right),\, \left(1,4\right),\,\left(2,3\right),\,\left(2,4\right),\,\left(3,4\right)\right\rbrace$.

To exploit sparsity, construct a chordal extension $\mathcal{G}^{\mathrm{ch}}$ of $\hat{\mathcal{G}}$.\footnote{One approach to creating a chordal extension is to use the sparsity pattern of a Cholesky factorization (employing a minimum degree ordering to maintain sparsity) of the Laplacian matrix associated with $\hat{\mathcal{G}}$ plus an identity matrix.} Denote the set of maximally sized cliques of the chordal extension by $\mathcal{C}_1,\ldots,\mathcal{C}_p \subset \mathcal{G}^{\mathrm{ch}}$. By construction of $\hat{\mathcal{G}}$, each constraint function $g_i$ for which $d_i > 1$ has all associated variables contained in at least one clique. For each $g_i$ for which $d_i > 1$, denote as $\mathcal{C}^{\left(i\right)}$ the \emph{minimal covering clique} (i.e., the smallest clique in $\left\lbrace\mathcal{C}_1,\ldots,\mathcal{C}_p\right\rbrace$ that contains all variables in $g_i$). (If not unique, a single clique $\mathcal{C}^{\left(i\right)}$ is chosen arbitrarily among the smallest cliques.) Associate an order $\tilde{d} \in \mathbb{N}^p$ with each clique $\mathcal{C}_\gamma,\, \gamma=1,\ldots,p$ defined such that $\tilde{d}_\gamma$ is the maximum relaxation order $d_i$ among all constraints for which the clique $\mathcal{C}_\gamma$ is the minimal covering clique. If a clique $\mathcal{C}_\gamma$ is not a minimal covering clique for any constraints, then $\tilde{d}_\gamma = 1$. See Appendix~\ref{app:five-bus example} for a small illustrative example.

For all $1\leqslant i \leqslant m$ such that $d_i > 1$, the positive semidefinite constraints $M_{d-k_i}\left(g_iy\right) \succcurlyeq 0$ in the moment hierarchy~\eqref{eq:momentd} are replaced by $N_i\left(g_i y\right) \succcurlyeq 0,\, i=1,\ldots,m$, where $N_i\left(g_i y\right)\left(\alpha,\beta\right) := M_{d_i-k_i}\left(g_iy\right)\left(\alpha,\beta\right)$ such that all non-zero entries of $\alpha$ and $\beta$ correspond to variables in $\mathcal{C}^{\left(i\right)}$. For $i=0$, the positive semidefinite constraint $M_{d}\left(y\right) \succcurlyeq 0$ (recall that $g_0 = 1$ and $k_0 = 0$) is replaced by constraints defined by each maximal clique: $\tilde{N}_\gamma\left(y\right) \succcurlyeq 0,\, \gamma=1,\ldots,p$, where $\tilde{N}_\gamma\left(y\right)\left(\alpha,\beta\right) =: M_{\tilde{d}_\gamma}\left(y\right)\left(\alpha,\beta\right)$ such that all non-zero entries of $\alpha$ and $\beta$ correspond to variables in $\mathcal{C}_\gamma$. 

For the sum-of-squares representation of the hierarchy, the polynomials $\sigma_i \in \Sigma_{d-k_i}[z] ,~ i=1, \hdots, m$ in~\eqref{eq:sosd} are replaced by sums-of-squares polynomials $\omega_i \in \Sigma_{d_i-k_i}[z_{\mathcal{C}^{\left(i\right)}}],\; i=1, \hdots, m$, where $z_{\mathcal{C}^{\left(i\right)}}$ denotes the subset of variables $z$ that are in the clique $\mathcal{C}^{\left(i\right)}$. The polynomial $\sigma_0 \in \Sigma_{d-k_0}[z]$ is replaced by the polynomial $\sum_{\gamma=1}^p \tau_\gamma$ where $\tau_\gamma \in \Sigma_{\tilde{d}_\gamma}[z_{\mathcal{C}_{\gamma}}],\, \gamma=1,\ldots,p$.

The sparse version of the real hierarchy \mbox{$\text{MSOS}_d$-$\mathbb{R}$} converges to the global optimum of a polynomial optimization problem when the constraints include ball constraints on all decision variables $x$ included in each clique: $\sum_{k\in \mathcal{C}_i} x_k^2 \leqslant \left(R_{\mathcal{C}_i}\right)^2 ,\; i=1,\ldots,p$, where $R_{\mathcal{C}_i}$ is the radius of a ball enclosing all decision variables in clique~$\mathcal{C}_i$~\cite{lasserre_book}. A similar result holds for the complex hierarchy \mbox{$\text{MSOS}_d$-$\mathbb{C}$} with sphere constraints enforced for the variables included in each clique. Due to~\eqref{eq:weak}, the sparse version of the complex hierarchy is guaranteed to converge to the global optimum of~\eqref{eq:complexPOP} with increasing relaxation order when the constraints include $(\sum_{k\in \mathcal{C}_i} \left|z_k\right|^2) + \left| z_{n+i} \right|^2 = \left(R_{\mathcal{C}_i}\right)^2 ,\; i=1,\ldots,p$, where $z_{n+i}$ is a slack variable associated with clique~$\mathcal{C}_i$.

Selectively applying the higher-order constraints requires a method for determining the relaxation order $d_i$ for each constraint. We use a heuristic based on ``mismatches'' to the closest rank-one matrix~\cite{mh_sparse_msdp}.  
The idea is to extract the largest eigenvalue $\lambda_1$ and its associated unit-length eigenvector $\eta_1$ from $(y_{\alpha,\beta})_{|\alpha|=|\beta|=1}$, hence defining an ``approximate'' solution $z^{\text{approx}} := \sqrt{\lambda_1}\,\eta_1$ to the polynomial optimization problem. Define  ``mismatches'' $\zeta \in \mathbb{R}$ and $\Delta \in \mathbb{R}^m$ between the solution $y$ to the relaxation and $z^{\text{approx}}$:
\begin{subequations}
\begin{align}
\label{eq:mismatch_f}
\zeta & := \left| f\left(z^{\text{approx}}\right) - L_y(f) \right|,\\
\label{eq:mismatch_g}
\Delta_i & := \left| g_i\left(z^{\text{approx}}\right) - L_y(g_i) \right|, ~~~~~~ i=1,\ldots,m.
\end{align}
\end{subequations}
We use the iteration in Algorithm~\ref{alg:heuristic} to determine relaxation orders $d_i,\, i=1,\ldots,m$. Each iteration solves the moment/sum-of-squares relaxation after increasing the relaxation orders $d_i$ in a manner that is dependent on the largest associated $\Delta_i$ values. Denote $d^{\text{max}} := \max_i \left\lbrace d_i \right\rbrace$.\footnote{Note that $d^{\text{max}}$ is not a specified maximum relaxation order but can change at each iteration.} At each iteration of the algorithm, increment $d_i$ at up to $h$ constraints, where $h$ is a specified parameter, that have the largest mismatches $\Delta_i$ among constraints satisfying two conditions: (1)~$d_i < d^{\text{max}}$ and (2)~$\Delta_i > \epsilon_g$, where $\epsilon_g$ is a specified mismatch tolerance. If no constraints satisfy both of these conditions, increment $d_i$ at up to $h$ constraints with the largest $\Delta_i$ greater than the specified tolerance and increment $d^{\text{max}}$. That is, in order to prevent unnecessarily increasing the size of the matrices, the heuristic avoids incrementing the maximum relaxation order $d^{\text{max}}$ until $d_i = d^{\text{max}}$ at all constraints $g_i$ with mismatch $\Delta_i > \epsilon_g$.

There is a computational trade-off in choosing the value of $h$. Larger values of $h$ likely result in fewer iterations of the algorithm but each iteration is slower if more buses than necessary have high-order relaxations. Smaller values of $h$ result in faster solution at each iteration, but may require more iterations.

The algorithm terminates upon satisfaction of two conditions: First, $\left|\Delta\right|_{\infty} \leqslant \epsilon_g$, where $\left|\;\cdot\;\right|_{\infty}$ denotes the infinity norm (maximum absolute value), which indicates that the iterate is a numerically feasible point of polynomial optimization problem \eqref{eq:complexPOP}. Second, $\zeta \leqslant \epsilon_f$, which indicates global optimality to within a relative tolerance $\epsilon_f$. If the relaxation satisfies the former but not the latter termination condition (which was never observed in practice for the problem in Section \ref{subsec:Numerical Results7}), the algorithm increases $d_i$ at the $h$ constraints with largest mismatch $\Delta_i$ and continues iterating.

The moment/sum-of-squares hierarchy is successively tightened in a manner that preserves computational tractability. For sufficiently small tolerances $\epsilon_f$ and $\epsilon_g$, Algorithm~\ref{alg:heuristic} eventually proceeds to build the complete moment/sum-of-squares hierarchies. Thus, Algorithm~\ref{alg:heuristic} inherits the theoretical convergence guarantees of \mbox{$\text{MSOS}_d$-$\mathbb{C}$}. The same can be said of the real version of Algorithm~\ref{alg:heuristic} applied to \mbox{$\text{MSOS}_d$-$\mathbb{R}$}.

\begin{algorithm}
\caption{Iterative Solution for Sparse Moment/Sum-of-Squares Relaxations}\label{alg:heuristic}
\begin{algorithmic}[1]
\State Set $d_i = 1,\, i=1,\ldots, m$.
\Repeat
\State Solve relaxation with order $d_i$ for constraints $g_i\left(z\right) \geqslant 0,\, i=1,\ldots,m$.
\State Calculate mismatches $\Delta_i,\, i=1,\ldots,m$ using~\eqref{eq:mismatch_g}.
\State Increase entries of $d$ according to the mismatch heuristic.
\Until{$\left|\Delta\right|_\infty < \epsilon_g $ and $\zeta < \epsilon_f $}
\State Extract solution $z^{\text{opt}}$.
\end{algorithmic}
\end{algorithm}

\section{Numerical results} 
\label{subsec:Numerical Results7} 

The optimal power flow problem is an instance of complex polynomial optimization. Since 2006, the power systems literature has been studying the ability of the Shor and second-order conic relaxations to find global solutions~\cite{jabr2006, bai-fujisawa-wang-wei-2008, lavaei-low-2012, dan2013, zhang2013, taylor2013, bukhsh2013, hicss2014, bose2014, low_tutorial, lavaei_allerton2014, andersen2014, coffrin2015, bose-chandy-gayme-low-2015, taylor2015, dan2015, cdc2015, laplacian_obj}. Some relaxations are presented in real numbers~\cite{lavaei-low-2012, dan2013} and some in complex numbers~\cite{zhang2013, bose2014, bose-chandy-gayme-low-2015}. Nevertheless, in all numerical applications, standard solvers such as SeDuMi, SDPT3, and MOSEK are used which currently handle only real numbers. Modeling languages such as YALMIP and CVX do handle inputs in complex numbers, but the data is transformed into real numbers before calling the solver~\cite[Example 4.42]{boyd2009}. We use the European network to illustrate the fact that it is beneficial to relax nonconvex constraints before converting from complex to real numbers. The Shor relaxation, the second-order conic relaxation, and the moment/sum-of-squares hierarchy are considered. 

We consider large test cases representing portions of European electric power systems. They represent Great Britain~(GB)~\cite{gbnetwork} and Poland~(PL)~\cite{matpower} power systems as well as other European systems from the PEGASE project~\cite{pegase,josz-2016}. The test cases were preprocessed to remove low-impedance lines as described in~\cite{cdc2015} in order to improve the solver's numerical convergence, which is a typical procedure in power system analyses.\footnote{Low-impedance lines often model connections between buses in the same physical location.} A \mbox{$1\times 10^{-3}$}~per~unit low-impedance line threshold was used for all test cases except for PEGASE-1354 and PEGASE-2869 which use a \mbox{$3\times 10^{-3}$}~per~unit threshold. The processed data is described in Table~\ref{tab:sizebis}. This table also includes the at-least-locally-optimal objective values obtained from the interior point solver in M{\sc atpower}~\cite{matpower} for the problems after preprocessing. Note that the PEGASE systems specify generation costs that minimize active power losses, so the objective values in both columns are the same.

\begin{table}[ht]
\centering
\caption{Size of Data (After Low-Impedance Line Preprocessing)}
\begin{tabular}{|l|c|c|c|c|}
\hline
\multicolumn{1}{|c|}{\textbf{Test}}  & \multicolumn{1}{c|}{\textbf{Number of}} & \multicolumn{1}{c|}{\textbf{Number of}} & \multicolumn{2}{c|}{M{\sc atpower} \textbf{Solution~\cite{matpower}}}\\ \cline{4-5} 
\multicolumn{1}{|c|}{\textbf{Case}}  & \multicolumn{1}{c|}{\textbf{Complex}} & \multicolumn{1}{c|}{\textbf{Edges}} & \multicolumn{1}{c|}{\textbf{Gen. Cost}} & \multicolumn{1}{c|}{\textbf{Loss Min.}}\\
\multicolumn{1}{|c|}{\textbf{Name}}  & \multicolumn{1}{c|}{\textbf{Variables}} & \multicolumn{1}{c|}{\textbf{in Graph}} & \multicolumn{1}{c|}{\textbf{(\$/hr)}} & \multicolumn{1}{c|}{\textbf{(MW)}}\\
\hline
GB-2224            & 2,053 & \hphantom{1}2,581 & 1,942,260  & \hphantom{1}60,614\\
PL-2383wp        & 2,177 & \hphantom{1}2,651 &   1,868,350 & \hphantom{1}24,991\\
PL-2736sp         & 2,182 & \hphantom{1}2,675 &  1,307,859 & \hphantom{1}18,336\\
PL-2737sop       & 2,183 & \hphantom{1}2,675 &    \hphantom{1,}777,617 & \hphantom{1}11,397\\
PL-2746wop      & 2,189 & \hphantom{1}2,708 &    1,208,257 & \hphantom{1}19,212\\
PL-2746wp        & 2,192 & \hphantom{1}2,686 &   1,631,737 & \hphantom{1}25,269\\
PL-3012wp        & 2,292 & \hphantom{1}2,805 &   2,592,462 & \hphantom{1}27,646\\
PL-3120sp         & 2,314 & \hphantom{1}2,835 &  2,142,720 & \hphantom{1}21,513\\
PEGASE-89      & \hphantom{1,1}70 & \hphantom{11,}185 & \hphantom{1,11}5,819 & \hphantom{11}5,819\\
PEGASE-1354  & \hphantom{1,}983 & \hphantom{1}1,526 &          \hphantom{1,1}74,043 & \hphantom{1}74,043\\
PEGASE-2869  & 2,120 & \hphantom{1}3,487 &         \hphantom{1,}133,945 & 133,945\\
PEGASE-9241  & 7,154 & 12,292 &                    \hphantom{1,}315,749 & 315,749\\
PEGASE-9241R\tablefootnote{\label{pegase9241R}PEGASE-9241 contains negative resistances to account for generators at lower voltage levels. In PEGASE-9241R these are set to 0.} 
		        & 7,154 & 12,292 & \hphantom{1,}315,785 & 315,785\\
\hline
\end{tabular}
\label{tab:sizebis}
\end{table}

Implementations use YALMIP \mbox{2015.06.26}~\cite{yalmip}, Mosek 7.1.0.28, and MATLAB 2013a on a computer with a quad-core 2.70~GHz processor and 16~GB of RAM. The results do not include the typically small formulation times.

\subsection*{Shor relaxation} 
\label{subsubsec:Shor Relaxation NUM} 

Table~\ref{tab:SDP} shows the results of applying \sdpr{} and \sdpc{} to the test cases. For some problems, the Shor relaxation is \emph{exact} and yields the globally optimal decision variables and objective values. To practically identify such problems, solutions for which all power injection mismatches
$S_{k}^{\text{inj mis}}$ (see Section~\ref{subsubsec:Moment/Sum-of-Squares Hierarchy NUM}) are less than a tolerance of 1~MVA are considered exact. These problems are identified with an asterisk (*) in Table~\ref{tab:SDP}. 

The lower bounds in Table~\ref{tab:SDP} suggest that the corresponding M{\sc atpower} solutions in Table~\ref{tab:sizebis} are at least very close to being globally optimal. The gap between the M{\sc atpower} solutions and the lower bounds from \sdpc{} for the generation cost minimizing problems are less than 0.72\% for GB-2224, 0.29\% for the Polish systems, and 0.02\% for the PEGASE systems with the exception of PEGASE-9241. The non-physical negative resistances in PEGASE-9241 result in weaker lower bounds from the relaxations, yielding a gap of 1.64\% for this test case.

\begin{table}[ht]
\centering
\caption{Real and Complex SDP (Generation Cost Minimization)}
\begin{tabular}{|l|c|c|c|c|}
\hline
\multicolumn{1}{|c|}{\textbf{Case}} & \multicolumn{2}{c|}{\textbf{SDP}-$\mathbb{R}$} & \multicolumn{2}{c|}{\textbf{SDP}-$\mathbb{C}$} \\
\cline{2-5}
\multicolumn{1}{|c|}{\textbf{Name}} & \multicolumn{1}{c|}{Val. (\$/hr)} & \multicolumn{1}{c|}{\!\! Time (sec)\!\!} & \multicolumn{1}{c|}{Val. (\$/hr)} & \multicolumn{1}{c|}{\!\!Time (sec)\!\!} \\
\hline
GB-2224       & 1,928,194 &  \hphantom{1}10.9 & 1,928,444 &  \hphantom{11}6.2 \\
PL-2383wp     & 1,862,979 &  \hphantom{1}48.1 & 1,862,985 &  \hphantom{1}23.0 \\
PL-2736sp*    & 1,307,749 &  \hphantom{1}35.7 & 1,307,764 &  \hphantom{1}22.0 \\
PL-2737sop*   &   \hphantom{1,}777,505 &  \hphantom{1}41.7 &   \hphantom{1,}777,539 &  \hphantom{1}19.5 \\
PL-2746wop*   & 1,208,168 &  \hphantom{1}51.1 & 1,208,182 &  \hphantom{1}22.8 \\
PL-2746wp     & 1,631,589 &  \hphantom{1}43.8 & 1,631,655 &  \hphantom{1}20.0 \\
PL-3012wp     & 2,588,249 &  \hphantom{1}52.8 & 2,588,259 &  \hphantom{1}24.3 \\
PL-3120sp     & 2,140,568 &  \hphantom{1}64.4 & 2,140,605 &  \hphantom{1}25.5 \\
PEGASE-89*    &     \hphantom{1,11}5,819 &   \hphantom{11}1.5 &     \hphantom{1,11}5,819 &   \hphantom{11}0.9 \\
PEGASE-1354   &    \hphantom{1,1}74,035 &  \hphantom{1}11.2 &    \hphantom{1,1}74,035 &   \hphantom{11}5.6 \\
PEGASE-2869   &   \hphantom{1,}133,936 &  \hphantom{1}38.2 &   \hphantom{1,}133,936 &  \hphantom{1}20.6 \\
PEGASE-9241   &   \hphantom{1,}310,658 & 369.7 &   \hphantom{1,}310,662 & 136.1 \\
PEGASE-9241R &    \hphantom{1,}315,848 & 317.2 &   \hphantom{1,}315,731 &  \hphantom{1}95.9 \\
\hline
\end{tabular}
\label{tab:SDP}
\end{table}

As shown in Appendices~\ref{app:Invariance of Shor Relaxation Bound} and~\ref{app:Invariance of SDP-R Relaxation Bound}, the optimal objective values for \sdpr{} and \sdpc{} should be identical. With all objective values in Table~\ref{tab:SDP} matching to within $0.037\%$, this is numerically validated.

For these test cases, \sdpc{} is significantly faster (between a factor of 1.60 and 3.31) than \sdpr{}. This suggests that exploiting the isomorphic structure of complex matrices in \sdpc{} is better than eliminating a row and column in \sdpr{}.

\subsection*{Second-order conic relaxation}
\label{subsubsec:Second-Order Conic Relaxation NUM}

Table~\ref{tab:SOCP} shows the results of applying \socpr{} and \socpc{} to the test cases. Unlike the Shor relaxation, the second-order conic relaxation is not exact for any of the test cases. (\socpc{} is generally not exact with the exception of radial systems for which the relaxation is provably exact when certain non-trivial technical conditions are satisfied~\cite{low_tutorial}.)

\begin{table}[ht]
\centering
\caption{Real and Complex SOCP (Generation Cost Minimization)}
\begin{tabular}{|l|c|c|c|c|}
\hline
\multicolumn{1}{|c|}{\textbf{Case}} & \multicolumn{2}{c|}{\textbf{SOCP}-$\mathbb{R}$} & \multicolumn{2}{c|}{\textbf{SOCP}-$\mathbb{C}$} \\
\cline{2-5}
\multicolumn{1}{|c|}{\textbf{Name}} & \multicolumn{1}{c|}{Val. (\$/hr)} & \multicolumn{1}{c|}{\!\! Time (sec)\!\!} & \multicolumn{1}{c|}{Val. (\$/hr)} & \multicolumn{1}{c|}{\!\!Time (sec)\!\!} \\
\hline
GB-2224      & 1,855,393 &  \hphantom{1}3.5 & 1,925,723 &  \hphantom{1}1.4 \\
PL-2383wp    & 1,776,726 &  \hphantom{1}8.5 & 1,849,906 &  \hphantom{1}2.4 \\
PL-2736sp    & 1,278,926 &  \hphantom{1}4.8 & 1,303,958 &  \hphantom{1}1.7 \\
PL-2737sop   &  \hphantom{1,}765,184 &  \hphantom{1}5.5 &   \hphantom{1,}775,672 &  \hphantom{1}1.6 \\
PL-2746wop   & 1,180,352 &  \hphantom{1}5.1 & 1,203,821 &  \hphantom{1}1.7 \\
PL-2746wp    & 1,586,226 &  \hphantom{1}5.5 & 1,626,418 &  \hphantom{1}1.7 \\
PL-3012wp    & 2,499,097 &  \hphantom{1}5.9 & 2,571,422 &  \hphantom{1}2.0 \\
PL-3120sp    & 2,080,418 &  \hphantom{1}6.2 & 2,131,258 &  \hphantom{1}2.2 \\
PEGASE-89    &     \hphantom{1,11}5,744 &  \hphantom{1}0.5 &     \hphantom{1,11}5,810 &  \hphantom{1}0.4 \\
PEGASE-1354  &   \hphantom{1,1}73,102 &  \hphantom{1}3.4 &    \hphantom{1,1}73,999 &  \hphantom{1}1.5 \\
PEGASE-2869  &   \hphantom{1,}132,520 &  \hphantom{1}9.0 &   \hphantom{1,}133,869 &  \hphantom{1}2.7 \\
PEGASE-9241  &   \hphantom{1,}306,050 & 35.3 &   \hphantom{1,}309,309 & 10.0 \\
PEGASE-9241R &   \hphantom{1,}312,682 & 36.7 &   \hphantom{1,}315,411 &  \hphantom{1}5.4 \\
\hline
\end{tabular}
\label{tab:SOCP}
\end{table}

\socpc{} provides better lower bounds and is computationally faster than \socpr{}. Specifically, lower bounds from \socpc{} are between $0.87\%$ and $3.96\%$ larger and solver times are faster by between a factor of $1.24$ and $6.76$ than those from \socpr{}. 


\subsection*{Moment/sum-of-squares hierarchy}
\label{subsubsec:Moment/Sum-of-Squares Hierarchy NUM}

Relaxations from the real 
moment/sum-of-squares hierarchy globally solve a broad class of optimal power flow problems~\cite{pscc2014,cedric_tps,mh_sparse_msdp,ibm_paper}. Previous work uses \mbox{$\text{MSOS}_d$-$\mathbb{R}$} by first converting the complex formulation of the optimal power flow problem to real numbers. 



We next summarize computational aspects of both the real and complex hierarchies. The dense formulations of the hierarchies solve small problems (up to approximately ten buses). Without also selectively applying the higher-order relaxations' constraints (i.e., $d_1=d_2=\ldots = d_m = d$), exploiting network sparsity enables solution of the second-order relaxations for problems with up to approximately 40 buses.

Scaling to larger problems is accomplished by both exploiting network sparsity and selectively applying the computationally intensive higher-order relaxation constraints to specific ``problematic'' buses. To better match the structure of the optimal power flow constraint equations, we use the algorithm in~\cite{mh_sparse_msdp}, which is slightly different than that described in Section~\ref{subsec:Exploiting Sparsity in Real and Complex Hierarchies}. Rather than consider each constraint individually, we use the mismatch in apparent power injections at each bus rather than the active and reactive power injection equations separately. The relaxation orders $d_i$ associated with all constraints at a bus are changed together.

Specifically, mismatches for the active and reactive power injection constraints at bus~$i$, denoted as $P_i^{\text{inj}\,\text{mis}}$ and $Q_i^{\text{inj}\,\text{mis}}$, are calculated using~\eqref{eq:mismatch_g}. Problematic buses are identified as those with large apparent power injection mismatch $S_i^{\text{inj}\,\text{mis}} = | P_i^{\text{inj}\,\text{mis}} + \mathbf{i} Q_i^{\text{inj}\,\text{mis}} |$. Application of the higher-order relaxation's constraints to these problematic buses using the iterative algorithm described in~\cite{mh_sparse_msdp} (cf Section~\ref{subsec:Exploiting Sparsity in Real and Complex Hierarchies}) results in global solutions to many optimal power flow problems and enables computational scaling to systems with thousands of buses~\cite{mh_sparse_msdp,cdc2015}. This section extends this approach to the complex hierarchy.

Tables~\ref{tab:MSOSR} and~\ref{tab:MSOSC} show the results of applying the algorithm from~\cite{mh_sparse_msdp} for both the real and complex hierarchies to several test cases with tolerances $\epsilon_g = 1~\text{MVA}$ and $\epsilon_f = 0.05\%$.\footnote{The algorithm in~\cite{mh_sparse_msdp} has a parameter $h$ specifying the maximum number of buses to increase the relaxation order $d_i$ at each iteration. This parameter is set to two for these examples. Additionally, bounds on the lifted variables $y$ derived from the voltage magnitude limits are enforced to improve numeric convergence.} The optimal objective values in these tables match to at least 0.007\%, which is within the expected solver tolerance. Further, the solutions for both the real and complex hierarchies match the optimal objective values for the loss minimizing problems obtained from M{\sc atpower} shown in Table~\ref{tab:sizebis} to within 0.013\%, providing an additional numerical proof that these solutions are globally optimal. Note, however, that local solvers do not always globally solve optimal power flow problems~\cite{bukhsh2013,mh_sparse_msdp,ferc5}.

The test cases considered in Tables~\ref{tab:MSOSR} and~\ref{tab:MSOSC} minimize active power losses rather than generation costs. Although the moment/sum-of-squares hierarchy solves many small- and medium-size test cases which minimize generation cost, application of the algorithm in~\cite{mh_sparse_msdp} to larger generation-cost-minimizing test cases often requires too many higher-order constraints for tractability. See~\cite{lavaei_allerton2014,cdc2015,laplacian_obj} for related algorithms which often find feasible points that are nearly globally optimal for such problems.

The feasible set of the optimal power flow problem is included is the ball of radius $ \sum_{k\in \mathcal{B}} \left(v_k^{\max}\right)^2$ so a slack variable and a sphere constraint may be added as suggested in Section \ref{subsec:Convergence of the Complex Hierarchy}. In order to preserve sparsity, a slack variable and a sphere constraint may be added \emph{for each maximal clique} of the chordal extension of the network graph. Global convergence is then guaranteed due to \eqref{eq:weak}. However, the sphere constraint tends to introduce numerical convergence challenges in problems with several thousand buses, resulting in the need for higher-order constraints at more buses and correspondingly longer solver times. 

Interestingly, the examples in Table~\ref{tab:MSOSC} converged without the slack variables and sphere constraints, and the results therein correspond to relaxations without sphere constraints. A potential way to account for the success of the complex hierarchy without sphere constraints would be to compute the Hermitian complexity~\cite{putinar-2012} of the ideal generated by the polynomials associated with equality constraints. A step in that direction would be to assess the greatest number of distinct points (possibly infinite) $v^i \in \mathbb{C}^n, 1 \leqslant i \leqslant p,$ such that $(v^i)^H (H_k + \textbf{i}\tilde{H_k}) v^j = -p_k^\text{dem} - \textbf{i}q_k^\text{dem}$ for all buses $k$ not connected to a generator and for all $1\leqslant i,j \leqslant p$. Note that the Hermitian complexity of the ideal generated by $\sum_{i=1}^n |z_i|^2 + \sigma(z)+a$ as defined in \eqref{eq:weak} with $a<0$ is equal to 1.

Despite being unnecessary for convergence of the hierarchies in Table~\ref{tab:MSOSC}, the sphere constraint can tighten the relaxations of some optimal power flow problems. Consider, for instance, the 9-bus example in~\cite{bukhsh2013}. The dense second-order relaxations from the real and complex hierarchies (both with and without the sphere constraint) yield the global optimum of \$3088/hour. Likewise, with second-order constraints enforced at all buses, the sparse versions of the real hierarchy and the complex hierarchy with the sphere constraint yield the global optimum. However, the sparse version of the second-order complex hierarchy without the sphere constraint only provided a lower bound of \$2939/hour. Thus, the sphere constraint tightens the sparse version of the second-order complex hierarchy for this test case. Since the sparse version of the third-order complex hierarchy without the sphere constraint yields the global optimum, the sphere constraint is unnecessary for convergence in this example.

\begin{table}[ht]
\centering
\caption{Real Moment/Sum-of-Squares Hierarchy $\text{MSOS}_d$-$\mathbb{R}$ (Active Power Loss Minimization)}
\begin{tabular}{|l|c|c|c|c|}
\hline
\multicolumn{1}{|c|}{\textbf{Case}} & \textbf{Num.} & \multicolumn{1}{c|}{\textbf{Global Obj.}} & \textbf{Max} $S^\text{mis}$ & \multicolumn{1}{c|}{\textbf{Solver}} \\
\multicolumn{1}{|c|}{\textbf{Name}} & \textbf{Iter.} & \multicolumn{1}{c|}{\textbf{Val. (MW)}} & \textbf{(MVA)} & \multicolumn{1}{c|}{\textbf{Time (sec)}} \\
\hline
PL-2383wp    & 3  & \hphantom{1}24,990 & 0.25 & \hphantom{1,}583.4 \\
PL-2736sp    & 1  & \hphantom{1}18,334 & 0.39 & \hphantom{1,1}44.0 \\
PL-2737sop   & 1  & \hphantom{1}11,397 & 0.45 & \hphantom{1,1}52.4 \\
PL-2746wop   & 2  & \hphantom{1}19,210 & 0.28 & 2,662.4 \\
PL-2746wp    & 1  & \hphantom{1}25,267 & 0.40 & \hphantom{1,1}45.9 \\
PL-3012wp    & 5  & \hphantom{1}27,642 & 1.00 & \hphantom{1,}318.7 \\
PL-3120sp    & 7  & \hphantom{1}21,512 & 0.77 & \hphantom{1,}386.6 \\
PEGASE-1354  & 5  & \hphantom{1}74,043 & 0.85 & \hphantom{1,}406.9 \\ 
PEGASE-2869  & 6 & 133,944 & 0.63 & \hphantom{1,}921.3 \\ 
\hline
\end{tabular}
\label{tab:MSOSR}
\end{table}

\begin{table}[ht]
\centering
\caption{Complex Moment/Sum-of-Squares Hierarchy $\text{MSOS}_d$-$\mathbb{C}$ (Active Power Loss Minimization)}
\begin{tabular}{|l|c|c|c|c|}
\hline
\multicolumn{1}{|c|}{\textbf{Case}} & \textbf{Num.} & \multicolumn{1}{c|}{\textbf{Global Obj.}} & \textbf{Max} $S^\text{mis}$ & \multicolumn{1}{c|}{\textbf{Solver}} \\
\multicolumn{1}{|c|}{\textbf{Name}} & \textbf{Iter.} & \multicolumn{1}{c|}{\textbf{Val. (MW)}} & \textbf{(MVA)} & \multicolumn{1}{c|}{\textbf{Time (sec)}} \\
\hline
PL-2383wp    & 3  & \hphantom{1}24,991 & 0.10 & \hphantom{1,1}53.9 \\
PL-2736sp    & 1  & \hphantom{1}18,335 & 0.11 & \hphantom{1,1}17.8 \\
PL-2737sop   & 1  & \hphantom{1}11,397 & 0.07 & \hphantom{1,1}25.7 \\
PL-2746wop   & 2  & \hphantom{1}19,212 & 0.12 & \hphantom{1,}124.3 \\
PL-2746wp    & 1  & \hphantom{1}25,269 & 0.05 & \hphantom{1,1}18.5 \\
PL-3012wp    & 7  & \hphantom{1}27,644 & 0.91 & \hphantom{1,}141.0 \\ 
PL-3120sp    & 9  & \hphantom{1}21,512 & 0.27 & \hphantom{1,}193.9 \\
PEGASE-1354  & 11  & \hphantom{1}74,042 & 1.00 & 1,132.6 \\ 
PEGASE-2869  & 9 & 133,939 & 0.97 & \hphantom{1,}700.8 \\ 
\hline
\end{tabular}
\label{tab:MSOSC}
\end{table}

Similar to the second-order conic relaxation, the results in Tables~\ref{tab:MSOSR} and~\ref{tab:MSOSC} show that the complex hierarchy generally has computational advantages over the real hierarchy. For all the test cases except PEGASE-1354, \mbox{$\text{MSOS}_d$-$\mathbb{C}$} solves between a factor of 1.31 and 21.42 faster than \mbox{$\text{MSOS}_d$-$\mathbb{R}$}. The most significant computational speed improvements for the complex hierarchy over the real hierarchy are seen for cases (e.g., \mbox{PL-2383wp} and \mbox{PL-2746wop}) where the higher-order constraints account for a large portion of the solver times. The complex hierarchy for these cases has significantly fewer terms in the higher-order constraints than the real hierarchy. 

Observe that several of the test cases (\mbox{PL-3012wp}, \mbox{PL-3120sp}, \mbox{PEGASE-1354}, and \mbox{PEGASE-2869}) require more iterations of the algorithm from~\cite{mh_sparse_msdp} for \mbox{$\text{MSOS}_d$-$\mathbb{C}$} than for \mbox{$\text{MSOS}_d$-$\mathbb{R}$}. Nevertheless, the improved speed per iteration results in faster overall solution times for all of these test cases except for \mbox{PEGASE-1354}, for which six additional iterations result in a factor of 2.78 slower solver time.

Both hierarchies were also applied to a variety of small test cases (less than ten buses) from~\cite{hicss2014,bukhsh2013,lesieutre_molzahn_borden_demarco-allerton2011,iscas2015} for which the first-order relaxations failed to yield the global optima. For all these test cases, the dense versions of both \mbox{$\text{MSOS}_d$-$\mathbb{C}$} and \mbox{$\text{MSOS}_d$-$\mathbb{R}$} converged at the same relaxation order. Section~\ref{subsec:Comparison of Real and Complex Hierarchies} demonstrates that the \mbox{$\text{MSOS}_d$-$\mathbb{R}$} is at least as tight as \mbox{$\text{MSOS}_d$-$\mathbb{C}$}. The results for small problems suggest that the hierarchies have the same tightness for some class of polynomial optimization problems which includes the optimal power flow problem with the sphere constraint (cf Conjecture~\ref{con}). The numerical results for some large test cases have different numbers of iterations between the real and complex hierarchies. Rather than differences in the theoretical tightness of the relaxation hierarchies, we attribute this discrepancy in the number of iterations to numerical convergence inaccuracies; not enforcing the sphere constraint for the sparse complex hierarchy; and, in some cases, the algorithm from~\cite{mh_sparse_msdp} selecting different buses at which to enforce the higher-order constraints.

\section{Conclusion}
\label{sec:Conclusion7}
We construct a complex moment/sum-of-squares hierarchy for complex polynomial optimization and prove convergence toward the global optimum. Theoretical and experimental evidence suggest that relaxing nonconvex constraints before converting from complex to real numbers is better than doing the operations in the opposite order. 
We conclude with the question: is it possible to gain efficiency by transposing convex optimization algorithms from real to complex numbers? \\\\
This chapter contains several appendices that may be found after Chapter \ref{sec:Conclusion and perspectives}.

\chapter{Conclusion and perspectives}
\label{sec:Conclusion and perspectives}

The main challenge that prompted this doctoral project was to be able to provide global solutions to the optimal power flow problem using semidefinite programming when the Shor relaxation fails. Having realized that the Lasserre hierarchy offers a solution to this challenge for small networks, the goal of the dissertation became to apply the Lasserre hierarchy to solve large-scale networks. The main contribution was to adapt the Lasserre hierarchy to the complex structure of our problem to enhance its tractability. This yielded a new general approach, the complex moment/sum-of-squares hierarchy.

In Chapter \ref{sec:Lasserre hierarchy for small-scale networks}, it was shown that the Lasserre hierarchy solves small-scale networks to global optimality. These networks could not be solved using the Shor relaxation. Surprisingly, the hierarchy solves them for low orders, generally the second or third order. However, the second order relaxation can only be applied to about a dozen of variables. With more variables, it becomes intractable. 

In Chapter \ref{sec:Zero duality gap in the Lasserre hierarchy}, it was proven that there is no duality gap at each order of the Lasserre hierarchy provided one the constraints is a ball constraint. This result is relevant because Lasserre proposes to add a redundant ball constraint to bounded feasible sets in order to guarantee convergence of the hierarchy. As a corrolary, we obtained that there is no duality gap at each order of the hierarchy applied to the optimal power flow problem, without having to add a redundant ball constraint. Note that the ball constraint is not needed in the case of our problem of interest due to upper bound constraints on the variables, that is to say upper voltage bounds. The property we've proven is necessary for interior-point solvers to converge to solutions of the semidefinite programs in the Lasserre hierarchy.

In Chapter \ref{sec:Data of European high-voltage transmission network}, new large-scale test cases are presented. They correspond to sections of the European high-voltage transmission network and the entire network. They can be viewed as quadratically-constrained quadratic programs where the variables and data are complex numbers. They consist in sparse problems with several thousand complex variables, with 9,241 variables in the biggest test case. The new data are representative of the size and complexity of real world power systems. They can hence be used to validate new methods and tools, such as those developed in this dissertation.
 
In Chapter \ref{sec:Penalized Lasserre hierarchy for large-scale networks}, the Lasserre hierarchy is applied to large-scale networks by combining it with a penalization approach. As a result, nearly global solutions are found to generation cost minizimation problems, with a guarantee of how far the value is from the global value. For power loss minimization problems, the objective function is convex. In those cases, the Lasserre hierarchy finds the global solution. In all cases, the sparsity of the problem is exploited using the notions of chordal graph and maximal clique, as well as a technique to identify problematic constraints. Higher-orders of the Lasserre hierarchy are then only applied to those constraints, reducing computation time. Moreover, low impedance lines are removed to cope with the inherent bad conditionning of power systems data.

In Chapter \ref{sec:Laplacian matrix gets rid of penalization parameter}, a method for finding nearly global solutions to the optimal power flow problem is proposed. It does so without having to specify a parameter, which a major disadvantage of penalization approaches. It is inspired by successful penalizations of the optimal power flow, which we observed to be linked with Laplacian matrices of the graph of the power network. Minimizing a quadratic form defined by such a Laplacian matrix over the power flow equations promotes low rank solutions. In fact, by iterative update of the weights of the Laplacian matrix, the rank can be reduced to one for many large-scale test cases. To guarantee near global optimality, the original objective function is set as a constraint. It is constrained to be less than or equal to the lower bound obtained by the Shor relaxation, plus a small fraction of it. This is founded because in all practical test cases, the Shor relaxation computes a lower bound of very high quality.

In Chapter \ref{sec:Complex hierarchy for enhanced tractability}, the Lasserre hierarchy is transposed to complex numbers in order to reduce the computional burden when solving polynomial problems with complex data and variables. The motivation for this is that the optimal power flow problem is a special case of complex polynomial optimization. We introduce a complex hierarchy and prove its convergence to the global solution for any complex polynomial problem with a feasible set of known radius. The proof relies on recent developments in algebraic geometry. The global solution to problems with several thousand complex variables is retrieved with the complex hierarchy. Sparsity is exploited by using chordal graphs techniques and a mismatch procedure to identify problematic constraints. 

There are various future research directions as a result of this dissertation. One direction is to enhance the tractability of the complex moment/sum-of-squares hierarchy. A way to accomplish this may be to develop a solver in complex numbers. Interior point solvers involve Cholesky factorizations, and Cholesky factorizations could be carried out on the Hermitian matrices. Another speed-up could come from developing a randomized complex hierarchy. This is based on an idea proposed by Lasserre. Two polynomials are equal to one another with high probability if they are equal on a randomly generated set of points. Thus far, the complex hierarchy is only able to solve optimal power flow problems which minimize active power loss, which is convex function of voltage. By enhancing the tractability of the complex hierarchy, it will hopefully be possible to tackle more general objective functions such as generation cost minimization or minimum deviation from a generation plan.

The complex hierarchy entails a trade-off. It is more tractable than the real hierarchy at a given order, but provides a potentially lower bound. It would interesting to know when the bounds generated by both hierarchies are the same. That would correspond to the cases for which it is certainly advantageous to use the complex hierarchy. In the case of the optimal power flow problem, numerical results show it is advantageous. It would be enviable to better understand why this is so.

Another research direction is to answer the following question. Do the power flow equations possess the Quillen property? In other words, is the complex hierarchy guaranteed to converge without having to add a slack variable and a redundant sphere constraint? Numerical experiments seem to show that this is true, but it is not clear why.

Lastly, transmission system operators are interested in optimization tools that cope with discrete variables. Indeed, there are many decisions which must be made from a finite number of possibilities: unit commitment, tap of phase-shifting transformers, and changes in network topology. The framework of real and complex polynomial optimization encompasses such cases, so real and complex hierarchies are relevant from a theoritical perspective. The Lasserre hierarchy is known to provide the best bounds to hard combinatorial problems, so it makes sense to try to apply real and complex hierarchies to the optimal power flow problem with discrete variables.

\appendix

\chapter{Ring Homomorphism}
\label{app:Ring Homomorphism}
It is shown here that the application $\Lambda$ defined by \eqref{eq:conversion} is a ring homomorphism.\\\\
Let $I_p$ denote the identity matrix of order $p\in \mathbb{N}$. $\Lambda(I_n)= I_{2n}$ and if $Z_1,Z_2 \in \mathbb{C}^{n\times n}$, $\Lambda(Z_1+Z_2) = \Lambda(Z_1) + \Lambda(Z_2)$ and
$$
\begin{array}{rcl}
\Lambda(Z_1) \Lambda(Z_2) & = &
\left(
\begin{array}{cr}
\text{Re}Z_1 & -\text{Im}Z_1 \\
\text{Im}Z_1 & \text{Re}Z_1 \\
\end{array}
\right)
\left(
\begin{array}{cr}
\text{Re}Z_2 & -\text{Im}Z_2 \\
\text{Im}Z_2 & \text{Re}Z_2 \\
\end{array}
\right) \\
 & = &  
\left(
\begin{array}{cr}
\text{Re}Z_1 \text{Re}Z_2 -\text{Im}Z_1\text{Im}Z_2  & - \text{Re}Z_1 \text{Im}Z_2 - \text{Im}Z_1 \text{Re}Z_2 \\
\text{Im}Z_1\text{Re}Z_2 +  \text{Re}Z_1\text{Im}Z_2 & \text{Re}Z_1 \text{Re}Z_2 - \text{Im}Z_1 \text{Im}Z_2  \\
\end{array}
\right) 
\\
 & = &  
\Lambda [ \text{Re}Z_1 \text{Re}Z_2 -\text{Im}Z_1\text{Im}Z_2  + \textbf{i} (\text{Im}Z_1\text{Re}Z_2 +  \text{Re}Z_1\text{Im}Z_2)  ] \\
 & = &  
\Lambda [ (\text{Re}Z_1 +\textbf{i} \text{Im}Z_1)(\text{Re}Z_2 + \textbf{i} \text{Im}Z_2)  ] \\
 & = & \Lambda(Z_1Z_2).
\end{array}
$$

\chapter{Rank-2 Condition}
\label{app:Rank-2 Condition}
It is proven here that a Hermitian matrix $Z$ is positive semidefinite and has rank~1 if and only if $\Lambda(Z)$ is positive semidefinite and has rank~2.\\\\
$(\Longrightarrow)$ Say $Z = z z^H$ where real and imaginary parts are defined by $z = x_1 + \textbf{i} x_2$ and $(x_1,x_2) \neq (0,0)$. Then 
\begin{subequations}
\label{eq:comp1}
\begin{align}
\Lambda(Z) & = 
\left(
\begin{array}{cc}
x_1 x_1^T + x_2 x_2^T & x_1 x_2^T - x_2 x_1^T \\
x_2 x_1^T - x_1 x_2^T & x_1 x_1^T + x_2 x_2^T
\end{array}
\right) \\ \label{eq:comp2}
& = 
\left(
\begin{array}{c}
x_1 \\
x_2 
\end{array}
\right)
\left(
\begin{array}{c}
x_1 \\
x_2 
\end{array}
\right)^T
+
\left(
\begin{array}{r}
-x_2 \\
x_1 
\end{array}
\right)
\left(
\begin{array}{r}
-x_2 \\
x_1 
\end{array}
\right)^T. 
\end{align}
\end{subequations}
The rank of $\Lambda(Z)$ is equal to 2 since $( ~ x_1^T ~ x_2^T ~)^T$ and $(~ (-x_2)^T ~ x_1^T ~ )^T$ are non-zero orthogonal vectors.

$(\Longleftarrow)$ Say $\Lambda(Z) = x x^T + y y^T$ where $x$ and $y$ are non-zero real vectors. Consider the block structure $x = ( ~ x_1^T ~ x_2^T ~ )^T$ and $y = ( ~ y_1^T ~ y_2^T ~ )^T$. For $i = 1, \hdots, n$, it must be that
\begin{subequations}
\begin{gather}
x_{1i}^2 + y_{1i}^2 = x_{2i}^2 + y_{2i}^2, \label{eq:sym} \\
x_{1i} x_{2i} + y_{1i} y_{2i} = 0. \label{eq:asym}
\end{gather}
\end{subequations}
Two cases can occur. The first is that $x_{1i} x_{2i} \neq 0$ in which case there exists a real number $\lambda_i \neq 0$ such that 
\begin{equation}
\left\{
\begin{array}{rcr}
y_{1i} & = & - \lambda_i ~ x_{2i}, \\
y_{2i} & = & \frac{1}{\lambda_i} ~ x_{1i}.
\end{array}
\right.
\end{equation}
Equation \eqref{eq:sym} implies that $(1-\lambda_i^2) x_{1i}^2 = (1-\frac{1}{\lambda_i^2}) x_{2i}^2$ thus $(1-\lambda_i^2)(1-\frac{1}{\lambda_i^2}) \geqslant 0$ and $\lambda_i = \pm 1$. The second case is that $x_{1i} x_{2i} = 0$. Then, according to \eqref{eq:asym}, $y_{1i} y_{2i} = 0$. If either $x_{1i} = y_{1i} = 0$ or $x_{2i} = y_{2i} = 0$, then \eqref{eq:sym} implies that $x_{1i} = x_{2i} = y_{1i} = y_{2i} = 0$. If $x_{1i} = y_{2i} = 0$, then \eqref{eq:sym} implies that $y_{1i} = \pm x_{2i}$. If $x_{2i} = y_{1i} = 0$, then \eqref{eq:sym} implies that $y_{2i} = \pm x_{1i}$.

In any case, there exists $\epsilon_i = \pm 1$ such that
\begin{equation}
\left\{
\begin{array}{rcr}
y_{1i} & = & - \epsilon_i ~ x_{2i}, \\
y_{2i} & = & \epsilon_i ~ x_{1i}.
\end{array}
\right.
\end{equation}
For $i,j=1,\hdots, n$ it must be that 
\begin{subequations}
\begin{gather}
(1-\epsilon_i \epsilon_j)(x_{1i} x_{1j}-x_{2i} x_{2j} ) = 0, \label{eq:sym1} \\
(1-\epsilon_i \epsilon_j)(x_{1j}x_{2i} + x_{1i} x_{2j}) = 0. \label{eq:asym1}
\end{gather}
\end{subequations}
Moreover 
\begin{equation}
\left\{
\begin{array}{rcr}
x_{1i} x_{1j} + y_{1i} y_{1j} & = & x_{1i} x_{1j} + \epsilon_i \epsilon_j x_{2i} x_{2j}, \\
x_{1i} x_{2j} + y_{1i} y_{2j} & = & x_{1i} x_{2j} -\epsilon_i \epsilon_j x_{2i} x_{1j}.
\end{array}
\right.
\end{equation} 
It will now be shown that
\begin{equation}
\left\{
\begin{array}{rcr}
x_{1i} x_{1j} + y_{1i} y_{1j} & = & x_{1i} x_{1j} + x_{2i} x_{2j}, \\
x_{1i} x_{2j} + y_{1i} y_{2j} & = & x_{1i} x_{2j} - x_{2i} x_{1j}.
\end{array}
\right.
\label{eq:conclusion}
\end{equation}
It is obvious if $\epsilon_i \epsilon_j = 1$. If $\epsilon_i \epsilon_j = -1$, then \eqref{eq:sym1}--\eqref{eq:asym1} imply
\begin{subequations}
\begin{gather}
x_{1i} x_{1j}-x_{2i} x_{2j} = 0, \label{eq:sym2}\\
x_{1j}x_{2i} + x_{1i} x_{2j} = 0. \label{eq:asym2}
\end{gather}
\end{subequations}
If $x_{1i} x_{1j} x_{2i} x_{2j} = 0$, it can be seen that \eqref{eq:conclusion} holds. If not, \eqref{eq:sym2} implies that there exists a real number $\mu_{ij} \neq 0$ such that 
\begin{equation}
\left\{
\begin{array}{rcr}
x_{2i} & = & \mu_{ij} ~ x_{1i}, \\
x_{2j} & = & \frac{1}{\mu_{ij}} ~ x_{1j}.
\end{array}
\right.
\end{equation}
Further, \eqref{eq:asym2} implies that $(\mu_{ij} + \frac{1}{\mu_{ij}}) x_{1j} x_{2i} = 0$. This is impossible ($\mu_{ij} + \frac{1}{\mu_{ij}} \neq 0$ and $x_{1j} x_{2i} \neq 0$). Thus,~\eqref{eq:conclusion} holds.

With the left hand side corresponding to~$\Lambda(Z) = x x^T + y y^T$ and the right hand side corresponding to~\eqref{eq:comp2}, equation \eqref{eq:conclusion} implies that $\Lambda(Z)$ is equal to \eqref{eq:comp2}. Since the function $\Lambda$ is injective, it must be that $Z = (x_1+\textbf{i}x_2)(x_1+\textbf{i}x_2)^H$.

\chapter{Invariance of Shor Relaxation Bound}
\label{app:Invariance of Shor Relaxation Bound}

It is shown here that the Shor relaxation bound obtained by relaxing nonconvexities then converting from complex to real numbers is the same as that obtained by converting from complex to real number then relaxing nonconvexities.\\\\
We have val(\csdpr{}) $\geqslant$ val(\sdpr{}) since the feasible set is more tightly constrained due to $\eqref{eq:csdpR4}$. To prove the opposite inequality, define $\tilde{\Lambda}(X) := (A + C)/2 + \textbf{i}(B-B^T)/2$ for all $X \in \mathbb{S}_{2n}$
using the block decomposition in the left hand part of \eqref{eq:csdpR4}.
It is proven here that if $X$ is a feasible point of \sdpr{}, then $\Lambda \circ \tilde{\Lambda}(X)$ is a feasible point of \csdpr{} with same objective value as $X$. Firstly, $\Lambda \circ \tilde{\Lambda}(X)$ satisfies \eqref{eq:csdpR4} because $\tilde{\Lambda}(X)$ is a Hermitian matrix.
Secondly, in order to show that $\Lambda \circ \tilde{\Lambda}(X)$ satisfies \eqref{eq:csdpR3}, notice that if $x = ( ~ x_1^T ~ x_2^T ~ )^T$ then
\begin{equation}\small
\begin{array}{c}
\left(
\begin{array}{r}
x_1 \\
x_2
\end{array}
\right)^T
\left(
\begin{array}{lc}
\hphantom{-}C & -B \\
-B^T & \hphantom{-}A
\end{array}
\right)
\left(
\begin{array}{r}
x_1 \\
x_2
\end{array}
\right) 
\\
= 
\\
\left(
\begin{array}{r}
-x_2 \\
x_1
\end{array}
\right)^T
\left(
\begin{array}{cl}
A & B^T \\
B & C
\end{array}
\right)
\left(
\begin{array}{r}
-x_2 \\
x_1
\end{array}
\right). 
\end{array}
\end{equation}
Hence $\Lambda \circ \tilde{\Lambda}(X)$ is equal to the sum of two positive semidefinite matrices.
Finally, to prove that $\Lambda \circ \tilde{\Lambda}(X)$ satisfies \eqref{eq:csdpR2} and has same objective value as $X$, notice that if $H \in \mathbb{H}_n$ and $Y\in \mathbb{S}_{2n}$, then $\text{Tr} \left[ \Lambda(H) Y \right] = \sum_{1 \leqslant i,j \leqslant 2n} \Lambda(H)_{ij} Y_{ji} = \sum_{1 \leqslant i,j \leqslant 2n} \Lambda(H)_{ij} Y_{ij} =  \sum_{1 \leqslant i,j \leqslant n} \text{Re}(H)_{ij} A_{ij} + \text{Im}(H)_{ij} B_{ij} + (-\text{Im}(H)_{ij}) (B^T)_{ij} + \text{Re}(H)_{ij} C_{ij} = \sum_{1 \leqslant i,j \leqslant n} \text{Re}(H_{ij}) (A + C)_{ij} + \text{Im}(H_{ij})(B - B^T)_{ij} =  2 \sum_{1 \leqslant i,j \leqslant n} \text{Re} [ H_{ij} (\tilde{\Lambda}(Y)_{ij})^H ] = 2 \sum_{1 \leqslant i,j \leqslant n} H_{ij} (\tilde{\Lambda}(Y)_{ij})^H = 2\text{Tr} [ H \tilde{\Lambda}(Y) ]$. Completing the proof, for all $H \in \mathbb{H}_n$, $\text{Tr} [ \Lambda(H) ~ \Lambda \circ \tilde{\Lambda}(X) ] = 2 \text{Tr} [ H \tilde{\Lambda}(X) ] = \text{Tr} \left[ \Lambda(H) X \right]$. 

\chapter{Invariance of SDP-$\mathbb{R}$ Relaxation Bound}
\label{app:Invariance of SDP-R Relaxation Bound}
We consider the semidefinite problem obtained by converting from complex to real numbers then relaxing nonconvexities. It is proven here that setting the phase of one of the variables to zero does not affect the relaxation bound.\\\\
We assume that $X$ is a feasible point of SDP-$\mathbb{R}$ and construct a feasible point of SDP-$\mathbb{R}$ with same objective value and first diagonal entry equal to 0. Consider the eigenvalue decomposition $X = \sum_{k=1}^p x_k x_k^T$ for some $x_k \in \mathbb{R}^{2n}$ and $p\in \mathbb{N}$. For all $\theta \in \mathbb{R}$, define
\begin{equation}
R_\theta := \Lambda [ \cos (\theta) I_n + \textbf{i} \sin (\theta) I_n] =
\left(
\begin{array}{cr}
\cos (\theta) I_n & -\sin (\theta) I_n \\
\sin (\theta) I_n & \cos (\theta) I_n
\end{array}
\right).
\end{equation}
For $k=1,\hdots,p$, define $\theta_k \in \mathbb{R}$ such that $x_{k,n+1} + \textbf{i} x_{k,1} =: \sqrt{x_{k, n+1}^2+ x_{k1}^2} e^{\textbf{i}\theta_k}$. Construct $\tilde{X} := \sum_{k=1}^p (R_{\theta_k} x_k) (R_{\theta_k} x_k)^T \succcurlyeq 0$ whose first diagonal entry is equal to 0. If $H \in \mathbb{H}_n$, then $\text{Tr}(\Lambda(H)\tilde{X}) = \sum_{k=1}^p \text{Tr}[ \Lambda(H) R_{\theta_k} x_k x_k^T R_{\theta_k}^T ] = \sum_{k=1}^p \text{Tr}[ R_{\theta_k}^T \Lambda(H) R_{\theta_k} x_k x_k^T  ] = \sum_{k=1}^p \text{Tr}[ \Lambda\{(\cos (\theta_k) I_n - \textbf{i} \sin (\theta_k) I_n) H (\cos (\theta_k) I_n + \textbf{i} \sin (\theta_k) I_n)\} x_k x_k^T ] = \sum_{k=1}^p \text{Tr}[ \Lambda( H ) x_k x_k^T  ] = \text{Tr}(\Lambda(H)X)$.

\chapter{Discrepancy Between Second-Order Conic Relaxation Bounds}
\label{app:Discrepancy Between Second-Order Conic Relaxation Bounds}
It is shown here that the second-order conic relaxation bound obtained by relaxing nonconvexities then converting from complex to real numbers is different from that obtained by converting from complex to real number then relaxing nonconvexities.\\\\
We have val(\csocpr{}) $\geqslant$ val(\socpr{}) since the feasible set is more tightly constrained.
The opposite inequality between optimal values does not hold, and this can be proven by considering the example QCQP-$\mathbb{C}$ defined by $\inf_{z_1,z_2 \in \mathbb{C}} ~ (1+\textbf{i}) \bar{z}_1 z_2 + (1-\textbf{i}) \bar{z}_2 z_1 ~ \mathrm{s.t.} ~ \bar{z}_1 z_1 \leqslant 1, ~ \overline{z}_2 z_2 \leqslant 1$. \csocpr{} yields the globally optimal value of $-2\sqrt{2}$, while \socpr{} yields $-4$,
as can be seen below.
\begin{equation}
\begin{array}{ccll}
-2\sqrt{2} & = & \inf_{X \in \mathbb{S}_4} & 2X_{12} + 2X_{34} + 2X_{23} - 2X_{14}, \\
& & \mathrm{s.t.} & X_{11} + X_{33} \leqslant 1, ~~~
X_{22} + X_{44} \leqslant 1, \\
& & &  X_{12}^2 + X_{23}^2 \leqslant X_{11} X_{22} ,\\
& & & X_{11} = X_{33}, ~~~
X_{12} = X_{34}, ~~~
X_{22} = X_{44}, \\
& & & X_{13} = X_{24} = 0, ~~~
X_{23} + X_{14} = 0, \\\\\\
-4 & = & \inf_{X \in \mathbb{S}_4} & 2X_{12} + 2X_{34} + 2X_{23} - 2X_{14}, \\
& & \mathrm{s.t.} & X_{11} + X_{33} \leqslant 1, ~~~
X_{22} + X_{44} \leqslant 1, \\
& & &  X_{12}^2 \leqslant X_{11} X_{22}, ~~~
X_{13}^2 \leqslant X_{11} X_{33}, \\
& & & X_{14}^2 \leqslant X_{11} X_{44}, ~~~
X_{23}^2 \leqslant X_{22} X_{33}, \\
& & & X_{24}^2 \leqslant X_{22} X_{44}, ~~~
X_{34}^2 \leqslant X_{33} X_{44}. 
\end{array}
\end{equation}

\chapter{Five-Bus Illustrative Example for Exploiting Sparsity}
\label{app:five-bus example}

To illustrate the selective application of second-order constraints, consider the five-bus optimal power flow problem in~\cite{bukhsh2013} which is an instance of QCQP-$\mathbb{C}$. Let $\text{ind}(\cdot)$ denote the set of indices corresponding to monomials of either the objective $f$ or constraint functions $(g_{i})_{1\leqslant 20}$. We have
\begin{align*}
\text{ind}(f) =\; & \{ (1,1),(1,2),(1,3),(3,5),(4,5),(5,5) \}, \\
\text{ind}(g_{1}) = \text{ind}(g_{2}) =\; & \{ (1,1),(1,2),(1,3)\} & \left[P_1^{\min}, Q_1^{\min}\right], \\
\text{ind}(g_{3}) = \text{ind}(g_{4}) =\; & \{ (1,2),(2,2),(2,3),(2,4) \} & \left[P_2, Q_2\right],\\
\text{ind}(g_{5}) = \text{ind}(g_{6}) =\; & \{ (1,3),(2,3),(3,3),(3,5) \} & \left[P_3,Q_3\right],\\
\text{ind}(g_{7}) = \text{ind}(g_{8}) =\; & \{ (2,4),(4,4),(4,5) \} & \left[P_4,Q_4\right], \\ \addtocounter{equation}{1}\tag{\theequation}\label{eq:5bus}
\text{ind}(g_{9}) = \text{ind}(g_{10}) =\; & \{ (3,5),(4,5),(5,5) \} & \left[P_5^{\min},Q_5^{\min}\right],\\
\text{ind}(g_{11}) = \text{ind}(g_{12}) =\; & \{ (1,1) \} & \left[V_1^{\min},V_1^{\max}\right], \\
\text{ind}(g_{13}) = \text{ind}(g_{14}) =\; & \{ (2,2) \} & \left[V_2^{\min},V_2^{\max}\right], \\
\text{ind}(g_{15}) = \text{ind}(g_{16}) =\; & \{ (3,3) \} & \left[V_3^{\min},V_3^{\max}\right], \\
\text{ind}(g_{17}) = \text{ind}(g_{18}) =\; & \{ (4,4) \} & \left[V_4^{\min},V_4^{\max}\right], \\ 
\text{ind}(g_{19}) = \text{ind}(g_{20}) =\; & \{ (5,5) \} & \left[V_5^{\min},V_5^{\max}\right],
\end{align*}
\noindent where the text in brackets indicates the origin of the constraint: $P_i$ and $Q_i$ for active and reactive power injection equality constraints, $P_i^{\min}$ and $Q_i^{\min}$ for lower limits on active and reactive power injections, and $V_i^{\min}$ and $V_i^{\max}$ for squared voltage magnitude limits at bus~$i$.

For brevity, the sphere constraints discussed in Section~\ref{subsec:Convergence of the Complex Hierarchy} are not enforced in this example. Regardless, the complex hierarchy with $d_i = 1,\; \forall i \in \left\lbrace 1,2,3,4,5,6,11,12,13,14,15,16 \right\rbrace$, $d_i = 2,\; \forall i \in \left\lbrace 7,8,9,10,17,18,19,20 \right\rbrace$ converges to the global solution. The second-order constraints are identified using the maximum power injection mismatch heuristic in~\cite{mh_sparse_msdp}.

The graph $\mathcal{G} = \left(\mathcal{N},\mathcal{E}\right)$ corresponding to~\eqref{eq:5bus} is shown in Fig.~\ref{fig:5bus graph}. The nodes correspond to the complex variables $\mathcal{N} = \left\lbrace 1,\hdots,5\right\rbrace$. Edges $\mathcal{E}$, which are denoted by solid lines in Fig.~\ref{fig:5bus graph}, connect variables that appear in the same monomial in any of the constraint equations or objective function. The supergraph $\hat{\mathcal{G}} = \left(\mathcal{N},\hat{\mathcal{E}}\right)$ has edges $\hat{\mathcal{E}}$ comprised of $\mathcal{E}$ (solid lines in Fig.~\ref{fig:5bus graph}) augmented with edges connecting all variables within each constraint with $d_i > 1$ (dashed lines in Fig.~\ref{fig:5bus graph}). In this case, $\hat{\mathcal{G}}$ is already chordal, so there is no need to form a chordal extension $\mathcal{G}^{\text{ch}}$. 

The maximal cliques of $\hat{\mathcal{G}}$ are $\mathcal{C}_1 = \left\lbrace 1, 2, 3\right\rbrace$ and $\mathcal{C}_2 = \left\lbrace 2, 3, 4, 5\right\rbrace$. Clique $\mathcal{C}_2$ is the minimal covering clique for all second-order constraints $ g_i\left(z\right),\, \forall i \in \left\lbrace 7,8,9,10,17,18,19,20\right\rbrace$. The order associated with $\mathcal{C}_2$ is two ($\tilde{d}_2 = 2$) since the highest order $d_i$ among all constraints for which $\mathcal{C}_2$ is the minimal covering clique is two. Clique $\mathcal{C}_1$ is not the minimal covering clique for any constraints with $d_i > 1$, so $\tilde{d}_1 = 1$. 

The globally optimal objective value obtained from the complex hierarchy specified above is 946.8 with corresponding decision variable $z = (1.0467 + 0.0000\mathbf{i}, 0.9550 - 0.0578\mathbf{i}, 0.9485 - 0.0533\mathbf{i}, 0.7791 + 0.6011\mathbf{i}, 0.7362 + 0.7487\mathbf{i})^T$.

\begin{figure}[H]
  \centering
    \includegraphics[width=.45\textwidth]{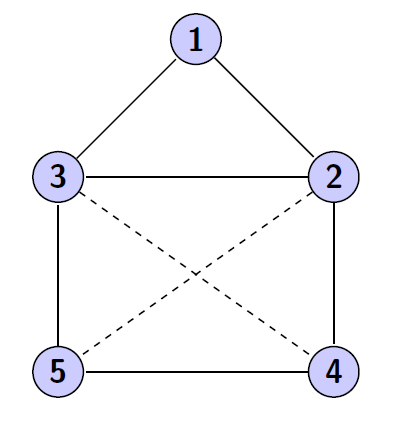}
  \caption{Graph Corresponding to Equations~\eqref{eq:5bus} from Five-Bus System in~\cite{bukhsh2013}}
  \label{fig:5bus graph}
\end{figure}

\chapter{Complex Hierarchy Applied to Optimal Power Flow}
\label{app:Complex Hierarchy Applied to Optimal Power Flow}
We consider an example of power loss minimization. 
The system 
of Figure~\ref{fig:WB2bis} 
links a generator to a load via a line of admittance $g+\textbf{i}b$ while respecting upper voltage constraints.
\begin{figure}[H]
  \centering
    \includegraphics[width=.65\textwidth]{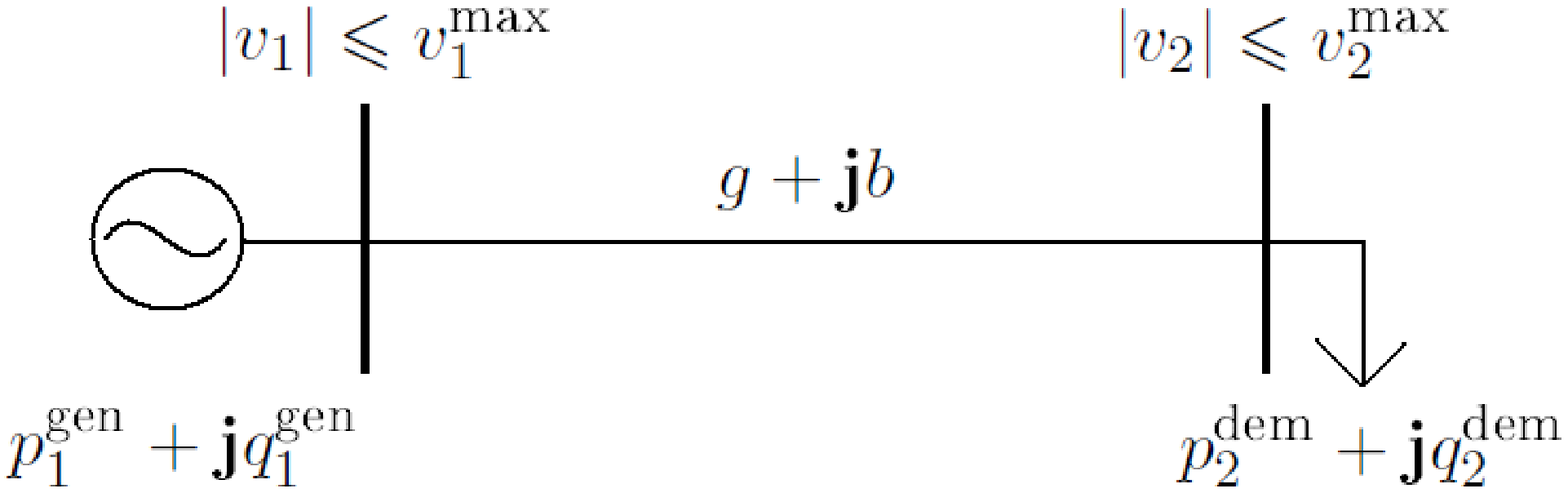}
  \caption{Two-Bus System}
  \label{fig:WB2bis}
\end{figure}
Minimizing power loss reads
\begin{gather}
\inf_{v_1,v_2 \in \mathbb{C}} ~~ g ~ |v_1|^2 - g ~ \overline{v}_1 v_2 - g ~ \overline{v}_2 v_1 + g ~ |v_2|^2,  \\
\text{subject to} ~~~~~~~~~~~~~~~~~~~~~~~~~~~~~~~~~~~~~~~~~~~~~~~~~~~~~~~~~~~~~~~~~~~~~~~~~~~~~~~~~~~~~~~~~~~~~~~~~~~~~~~~~~~ \notag \\
 -\frac{g-\textbf{i}b}{2} ~ \overline{v}_1 v_2 -\frac{g+\textbf{i}b}{2} ~ \overline{v}_2 v_1 + g ~ |v_2|^2 = -p_2^\text{dem}, \\
~~~~~ \frac{b+\textbf{i}g}{2} ~ \overline{v}_1 v_2 + \frac{b-\textbf{i}g}{2} ~ \overline{v}_2 v_1 -b ~ |v_2|^2 = -q_2^\text{dem}, ~~ \\
|v_1|^2 \leqslant (v_1^{\text{max}})^2, \\
|v_2|^2 \leqslant (v_2^{\text{max}})^2.
\end{gather}
The feasible set is included in the ball defined by $|v_1|^2 + |v_2|^2 \leqslant (v_1^{\text{max}})^2 + (v_2^{\text{max}})^2$. In accordance with Section \ref{subsec:Convergence of the Complex Hierarchy}, let's add a slack variable $v_3 \in \mathbb{C}$ and a constraint
\begin{equation}
|v_1|^2 + |v_2|^2 +   |v_3|^2 = (v_1^{\text{max}})^2 + (v_2^{\text{max}})^2.
\end{equation}
The first and second orders (i.e., $\text{MSOS}_1$-$\mathbb{C}$ and $\text{MSOS}_2$-$\mathbb{C}$) are written below where the notation $y^{\alpha}_{\beta} := y_{\alpha,\beta} ~ (\alpha,\beta \in \mathbb{N}^3)$ is used to save space.
\\
\text{}
\\
\textbf{Example of $\text{MSOS}_1$-$\mathbb{C}$:}
\begin{gather}
\inf_{y} ~ g ~ y^{100}_{100} - g ~ y^{100}_{010} - g ~ y^{010}_{100} + g ~ y^{010}_{010}, \\
\text{subject to} ~~~~~~~~~~~~~~~~~~~~~~~~~~~~~~~~~~~~~~~~~~~~~~~~~~~~~~~~~~~~~~~~~~~~~~~~~~~~~~~~~~~~~~~~~~~~~~~~~~~~~~~~~~~~~~\notag  \\
-\frac{g-\textbf{i}b}{2} ~ y^{100}_{010} -\frac{g+\textbf{i}b}{2} ~ y^{010}_{100} + g ~ y^{010}_{010} = -p_2^\text{dem} ~ y^{000}_{000}, ~~~  \\[0.5em]
\frac{b+\textbf{i}g}{2} ~ y^{100}_{010}  + \frac{b-\textbf{i}g}{2} ~ y^{010}_{100} - b ~ y^{010}_{010} = -q_2^\text{dem} ~ y^{000}_{000}, ~ \\[0.5em]
y^{100}_{100} \leqslant (v_1^{\text{max}})^2 y^{000}_{000}, \\[0.5em]
y^{010}_{010} \leqslant (v_2^{\text{max}})^2 y^{000}_{000}, \\[0.5em]
y^{100}_{100} + y^{010}_{010} + y^{001}_{001}  = \left( (v_1^{\text{max}})^2 + (v_2^{\text{max}})^2 \right)y^{000}_{000}, \\[0.5em]
\left(
\begin{array}{cccc}
y^{000}_{000} & y^{000}_{100} & y^{000}_{010} &  y^{000}_{001} \\[0.5em]
y^{100}_{000} & y^{100}_{100} & y^{100}_{010} & y^{100}_{001} \\[0.5em]
y^{010}_{000} & y^{010}_{100} & y^{010}_{010} & y^{010}_{001} \\[0.5em]
y^{001}_{000} & y^{001}_{100} & y^{001}_{010} & y^{001}_{001}
\end{array}
\right) \succcurlyeq 0, \\[0.5em]
y^{000}_{000} = 1.
\end{gather}
\\
\text{}
\\
\textbf{Example of $\text{MSOS}_2$-$\mathbb{C}$:}
\begin{gather}
\inf_{y} ~ g ~ y^{100}_{100} - g ~ y^{100}_{010} - g ~ y^{010}_{100} + g ~ y^{010}_{010}, \label{eq:2order1}\\
\text{subject to} ~~~~~~~~~~~~~~~~~~~~~~~~~~~~~~~~~~~~~~~~~~~~~~~~~~~~~~~~~~~~~~~~~~~~~~~~~~~~~~~~~~~~~~~~~~~~~~~~~~~~~~~~~~~~~~~~~~~~~~~~~~~~~~~~~~~~~~~~~~~~~~~\notag \\
p_2^\text{dem} \left(
\begin{array}{cccc}
y^{000}_{000} & y^{000}_{100} & y^{000}_{010} &  y^{000}_{001} \\[0.5em]
y^{100}_{000} & y^{100}_{100} & y^{100}_{010} & y^{100}_{001} \\[0.5em]
y^{010}_{000} & y^{010}_{100} & y^{010}_{010} & y^{010}_{001} \\[0.5em]
y^{001}_{000} & y^{001}_{100} & y^{001}_{010} & y^{001}_{001}
\end{array}
\right) \hdots \notag \\[0.25em]
- \frac{g-\textbf{i}b}{2}
\left(
\begin{array}{cccc}
y^{100}_{010} & y^{100}_{110} & y^{100}_{020} &  y^{100}_{011} \\[0.5em]
y^{200}_{010} & y^{200}_{110} & y^{200}_{020} & y^{200}_{011} \\[0.5em]
y^{110}_{010} & y^{110}_{110} & y^{110}_{020} & y^{110}_{011} \\[0.5em]
y^{101}_{010} & y^{101}_{110} & y^{101}_{020} & y^{101}_{011}
\end{array}
\right) \hdots \notag \\[0.25em]
- \frac{g+\textbf{i}b}{2}
\left(
\begin{array}{cccc}
y^{010}_{100} & y^{010}_{200} & y^{010}_{110} &  y^{010}_{101} \\[0.5em]
y^{110}_{100} & y^{110}_{200} & y^{110}_{110} & y^{110}_{101} \\[0.5em]
y^{020}_{100} & y^{020}_{200} & y^{020}_{110} & y^{020}_{101} \\[0.5em]
y^{011}_{100} & y^{011}_{200} & y^{011}_{110} & y^{011}_{101}
\end{array}
\right) \hdots \notag \\[0.25em]
+g
\left(
\begin{array}{cccc}
y^{010}_{010} & y^{010}_{110} & y^{010}_{020} &  y^{010}_{011} \\[0.5em]
y^{110}_{010} & y^{110}_{110} & y^{110}_{020} & y^{110}_{011} \\[0.5em]
y^{020}_{010} & y^{020}_{110} & y^{020}_{020} & y^{020}_{011} \\[0.5em]
y^{011}_{010} & y^{011}_{110} & y^{011}_{020} & y^{011}_{011}
\end{array}
\right) = 0, \\[2em]
q_2^\text{dem} \left(
\begin{array}{cccc}
y^{000}_{000} & y^{000}_{100} & y^{000}_{010} &  y^{000}_{001} \\[0.5em]
y^{100}_{000} & y^{100}_{100} & y^{100}_{010} & y^{100}_{001} \\[0.5em]
y^{010}_{000} & y^{010}_{100} & y^{010}_{010} & y^{010}_{001} \\[0.5em]
y^{001}_{000} & y^{001}_{100} & y^{001}_{010} & y^{001}_{001}
\end{array}
\right) \hdots \notag \\[0.25em]
+\frac{b+\textbf{i}g}{2}
\left(
\begin{array}{cccc}
y^{100}_{010} & y^{100}_{110} & y^{100}_{020} &  y^{100}_{011} \\[0.5em]
y^{200}_{010} & y^{200}_{110} & y^{200}_{020} & y^{200}_{011} \\[0.5em]
y^{110}_{010} & y^{110}_{110} & y^{110}_{020} & y^{110}_{011} \\[0.5em]
y^{101}_{010} & y^{101}_{110} & y^{101}_{020} & y^{101}_{011}
\end{array}
\right) \hdots \notag \\[0.25em]
+\frac{b-\textbf{i}g}{2}
\left(
\begin{array}{cccc}
y^{010}_{100} & y^{010}_{200} & y^{010}_{110} &  y^{010}_{101} \\[0.5em]
y^{110}_{100} & y^{110}_{200} & y^{110}_{110} & y^{110}_{101} \\[0.5em]
y^{020}_{100} & y^{020}_{200} & y^{020}_{110} & y^{020}_{101} \\[0.5em]
y^{011}_{100} & y^{011}_{200} & y^{011}_{110} & y^{011}_{101}
\end{array}
\right) \hdots \notag \\[0.25em]
-b
\left(
\begin{array}{cccc}
y^{010}_{010} & y^{010}_{110} & y^{010}_{020} &  y^{010}_{011} \\[0.5em]
y^{110}_{010} & y^{110}_{110} & y^{110}_{020} & y^{110}_{011} \\[0.5em]
y^{020}_{010} & y^{020}_{110} & y^{020}_{020} & y^{020}_{011} \\[0.5em]
y^{011}_{010} & y^{011}_{110} & y^{011}_{020} & y^{011}_{011}
\end{array}
\right) = 0, \\[2em]
(v_1^{\text{max}})^2 \left(
\begin{array}{cccc}
y^{000}_{000} & y^{000}_{100} & y^{000}_{010} &  y^{000}_{001} \\[0.5em]
y^{100}_{000} & y^{100}_{100} & y^{100}_{010} & y^{100}_{001} \\[0.5em]
y^{010}_{000} & y^{010}_{100} & y^{010}_{010} & y^{010}_{001} \\[0.5em]
y^{001}_{000} & y^{001}_{100} & y^{001}_{010} & y^{001}_{001}
\end{array}
\right) \hdots \notag \\[0.25em]
- 
\left(
\begin{array}{cccc}
y^{100}_{100} & y^{100}_{200} & y^{100}_{110} &  y^{100}_{101} \\[0.5em]
y^{200}_{100} & y^{200}_{200} & y^{200}_{110} & y^{200}_{101} \\[0.5em]
y^{110}_{100} & y^{110}_{200} & y^{110}_{110} & y^{110}_{101} \\[0.5em]
y^{101}_{100} & y^{101}_{200} & y^{101}_{110} & y^{101}_{101}
\end{array}
\right)
\succcurlyeq 0, \\[2em]
(v_2^{\text{max}})^2 \left(
\begin{array}{cccc}
y^{000}_{000} & y^{000}_{100} & y^{000}_{010} &  y^{000}_{001} \\[0.5em]
y^{100}_{000} & y^{100}_{100} & y^{100}_{010} & y^{100}_{001} \\[0.5em]
y^{010}_{000} & y^{010}_{100} & y^{010}_{010} & y^{010}_{001} \\[0.5em]
y^{001}_{000} & y^{001}_{100} & y^{001}_{010} & y^{001}_{001}
\end{array}
\right) \hdots \notag \\[0.25em]
- 
\left(
\begin{array}{cccc}
y^{010}_{010} & y^{010}_{110} & y^{010}_{020} &  y^{010}_{011} \\[0.5em]
y^{110}_{010} & y^{110}_{110} & y^{110}_{020} & y^{110}_{011} \\[0.5em]
y^{020}_{010} & y^{020}_{110} & y^{020}_{020} & y^{020}_{011} \\[0.5em]
y^{011}_{010} & y^{011}_{110} & y^{011}_{020} & y^{011}_{011}
\end{array}
\right)
\succcurlyeq 0, \\[2em]
\left( (v_1^{\text{max}})^2 + (v_2^{\text{max}})^2 \right) \left(
\begin{array}{cccc}
y^{000}_{000} & y^{000}_{100} & y^{000}_{010} &  y^{000}_{001} \\[0.5em]
y^{100}_{000} & y^{100}_{100} & y^{100}_{010} & y^{100}_{001} \\[0.5em]
y^{010}_{000} & y^{010}_{100} & y^{010}_{010} & y^{010}_{001} \\[0.5em]
y^{001}_{000} & y^{001}_{100} & y^{001}_{010} & y^{001}_{001}
\end{array}
\right) \hdots \notag \\[0.25em]
- 
\left(
\begin{array}{cccc}
y^{100}_{100} & y^{100}_{200} & y^{100}_{110} &  y^{100}_{101} \\[0.5em]
y^{200}_{100} & y^{200}_{200} & y^{200}_{110} & y^{200}_{101} \\[0.5em]
y^{110}_{100} & y^{110}_{200} & y^{110}_{110} & y^{110}_{101} \\[0.5em]
y^{101}_{100} & y^{101}_{200} & y^{101}_{110} & y^{101}_{101}
\end{array}
\right) \hdots \notag \\[0.25em]
-
\left(
\begin{array}{cccc}
y^{010}_{010} & y^{010}_{110} & y^{010}_{020} &  y^{010}_{011} \\[0.5em]
y^{110}_{010} & y^{110}_{110} & y^{110}_{020} & y^{110}_{011} \\[0.5em]
y^{020}_{010} & y^{020}_{110} & y^{020}_{020} & y^{020}_{011} \\[0.5em]
y^{011}_{010} & y^{011}_{110} & y^{011}_{020} & y^{011}_{011}
\end{array}
\right) \hdots \notag \\[0.25em]
-
 \left(
\begin{array}{cccc}
y^{001}_{001} & y^{001}_{101} & y^{001}_{011} &  y^{001}_{002} \\[0.5em]
y^{101}_{001} & y^{101}_{101} & y^{101}_{011} & y^{101}_{002} \\[0.5em]
y^{011}_{001} & y^{011}_{101} & y^{011}_{011} & y^{011}_{002} \\[0.5em]
y^{002}_{001} & y^{002}_{101} & y^{002}_{011} & y^{002}_{002}
\end{array}
\right)
= 0, \\[2em]
\left(
\begin{array}{cccccccccc}
y^{000}_{000} & y^{000}_{100} & y^{000}_{010} &  y^{000}_{001} & y^{000}_{200} & y^{000}_{110} & y^{000}_{101} & y^{000}_{020} & y^{000}_{011} & y^{000}_{002} \\[0.5em]
y^{100}_{000} & y^{100}_{100} & y^{100}_{010} & y^{100}_{001} & y^{100}_{200} & y^{100}_{110} & y^{100}_{101} & y^{100}_{020} & y^{100}_{011} & y^{100}_{002} \\[0.5em]
y^{010}_{000} & y^{010}_{100} & y^{010}_{010} & y^{010}_{001} & y^{010}_{200} & y^{010}_{110} & y^{010}_{101} & y^{010}_{020} & y^{010}_{011} & y^{010}_{002} \\[0.5em]
y^{001}_{000} & y^{001}_{100} & y^{001}_{010} & y^{001}_{001} & y^{001}_{200} & y^{001}_{110} & y^{001}_{101} & y^{001}_{020} & y^{001}_{011} & y^{001}_{002} \\[0.5em]
y^{200}_{000} & y^{200}_{100} & y^{200}_{010} &  y^{200}_{001} & y^{200}_{200} & y^{200}_{110} & y^{200}_{101} & y^{200}_{020} & y^{200}_{011} & y^{200}_{002} \\[0.5em]
y^{110}_{000} & y^{110}_{100} & y^{110}_{010} &  y^{110}_{001} & y^{110}_{200} & y^{110}_{110} & y^{110}_{101} & y^{110}_{020} & y^{110}_{011} & y^{110}_{002} \\[0.5em]
y^{101}_{000} & y^{101}_{100} & y^{101}_{010} &  y^{101}_{001} & y^{101}_{200} & y^{101}_{110} & y^{101}_{101} & y^{101}_{020} & y^{101}_{011} & y^{101}_{002} \\[0.5em]
y^{020}_{000} & y^{020}_{100} & y^{020}_{010} &  y^{020}_{001} & y^{020}_{200} & y^{020}_{110} & y^{020}_{101} & y^{020}_{020} & y^{020}_{011} & y^{020}_{002} \\[0.5em]
y^{011}_{000} & y^{011}_{100} & y^{011}_{010} &  y^{011}_{001} & y^{011}_{200} & y^{011}_{110} & y^{011}_{101} & y^{011}_{020} & y^{011}_{011} & y^{011}_{002} \\[0.5em]
y^{002}_{000} & y^{002}_{100} & y^{002}_{010} &  y^{002}_{001} & y^{002}_{200} & y^{002}_{110} & y^{002}_{101} & y^{002}_{020} & y^{002}_{011} & y^{002}_{002}
\end{array}
\right) \succcurlyeq 0, \\[2em]
y^{000}_{000} = 1. \label{eq:2order2}
\end{gather}

\bibliography{mybib}
\bibliographystyle{siam}

\end{document}